\documentclass[oneside,16pt]{amsart}
\pdfoutput=1
\usepackage[fontsize=14pt]{scrextend}
\usepackage{mathptmx}
\usepackage[scale={.85,.85}]{geometry}
\usepackage{mathrsfs}
\usepackage{mathptmx,amssymb}
\usepackage{amsmath}
\usepackage[dvipsnames]{xcolor}
\usepackage{color}
\usepackage{graphicx}
\usepackage{mathtools}
\usepackage[colorlinks,citecolor=Plum,urlcolor=Periwinkle,linkcolor=DarkOrchid]{hyperref}

\usepackage{enumitem} 
\newcounter{Cond}
\usepackage{stmaryrd} 

\allowdisplaybreaks

\newcommand{\TNo}{{_{\displaystyle\Longrightarrow\atop N\to\infty}}}
\newcommand{\Tmo}{{_{\displaystyle\Longrightarrow\atop m\to\infty}}}
\newcommand{\tNo}{{_{\displaystyle\longrightarrow\atop N\to\infty}}}

\usepackage{subcaption}

\newtheorem{theorem}{Theorem}
\newtheorem{lemma}{Lemma}[section]
\newtheorem{corollary}[lemma]{Corollary}

\newtheorem{propo}[lemma]{Proposition}
\newtheorem{proposition}[lemma]{Proposition}

\theoremstyle{definition}
\newtheorem{definition}[lemma]{Definition}
\newtheorem{remark}[lemma]{Remark}

\newcommand{\bb}[1]{\mathbb{#1}}
\newcommand{\E}{\ensuremath{ \bb{E} } }

\usepackage[mathlines]{lineno}

\newcommand{\re}{\ensuremath{\mathbb{R}}}
\newcommand{\paren}[1]{\ensuremath{\left( #1\right) }}

\newcommand{\p}{\mathbb{P}}

\newcommand{\na}{\ensuremath{\mathbb{N}}}

\newcommand{\floor}[1]{\ensuremath{\lfloor #1\rfloor}}

\allowdisplaybreaks

\usepackage{todonotes}

\newcommand{\Osvaldo}[1]{\todo[color=green!70,inline]{Osvaldo: #1}}

\begin{document}
\title{Scaling limit of the Aldous-Broder chain on regular graphs: the transient regime}

\author{Osvaldo Angtuncio Hern\'andez}
\address{Osvaldo Angtuncio Hern\'andez\\
Centro de Investigaci\'on en Matem\'aticas\\
Departamento de Probabilidad y Estad\'istica\\ Guanajuato\\ M\'exico} \email{osvaldo.angtuncio@cimat.mx}

\author{Gabriel Berzunza Ojeda}
\address{Gabriel Berzunza Ojeda\\University of Liverpool, Department of Mathematical Sciences, United Kingdom.}\email{gabriel.berzunza-ojeda@liverpool.ac.uk}

\author{Anita Winter}
\address{Anita Winter\\
Universit\"at Duisburg-Essen\\
Fakult\"at f\"ur Mathematik\\Universit\"atsstr. 2\\ D-45141 Essen\\
Germany} \email{anita.winter@uni-due.de}

\maketitle

\vspace{0.1in}
\begin{abstract}
The continuum random tree is the scaling limit of the uniform spanning tree on the  complete graph with $N$ vertices. The Aldous-Broder chain on a graph $G=(V,E)$ is a discrete-time stochastic process with values in the space of rooted trees whose vertex set is a subset of $V$ which is stationary under the uniform distribution on the space of rooted trees spanning $G$. In Evans, Pitman and Winter (2006) \cite{EvansPitmanWinter2006} the so-called root growth with re-grafting process (RGRG) was constructed. Further it was shown that the suitable rescaled Aldous-Broder chain converges to the RGRG weakly with respect to the Gromov-Hausdorff topology. It was shown in Peres and Revelle (2005) \cite{PeresRevelle} that (upto a dimension depending constant factor) the continuum random tree is, with respect to the Gromov-weak topology, the scaling limit of the uniform spanning tree on $\mathbb{Z}_N^d$, $d\ge 5$. This result was recently strengthens in Archer, Nachmias and Shalev (2024) \cite{MR4712854} to convergence with respect to the Gromov-Hausdorff-weak topology, and therefore also with respect to the Gromov-Hausdorff topology. In the present paper we show that also the suitable rescaled Aldous-Broder chain converges to the RGRG weakly with respect to the Gromov-Hausdorff topology when initially started in the trivial rooted tree. We give conditions on the increasing graph sequence under which the result extends to regular graphs and give probabilistic expressions scales at which time has to be sped up and edge lengths have to be scaled down. 
\end{abstract}

\bigskip

{\scriptsize
\tableofcontents
}

{\section{Introduction}
\label{S:introduction}
In this paper we study the scaling limit of an algorithm generating trees that span simple, connected graphs. We say that a graph is simple when it does not contain multiple edges between a pair of vertices nor a self-edge from one vertex to itself. 
A spanning tree of a finite simple, connected graph $G=(V,E)$ is a subtree, $T=(V,E^{\prime})$, with $E^{\prime}\subseteq E$}. The uniform spanning tree of $G$ (denoted by UST($G$)) is the random tree which is uniformly distributed on the set of all trees spanning $G$. {It is closely related to several other topics in probability theory, including
 loop-erased random walks (\cite{Wilson1996}), potential theory \cite{BenjaminiLyonsSchramm2001}, conformally invariant
 scaling limits (\cite{Schramm2000,LawlerSchrammWerner2004}), domino tiling \cite{BurtonPemantle1993,Kenyon2000}, the Abelian sandpile model \cite{Dhar1990,Jarai2018}, and 
Sznitman’s interlacement process \cite{Teixeira2009,Sznitman2010,Hutchcroft2018}.}

A simplest way to generate the UST($G$) is the following: 
run the random walk, $W=(W(n))_{n\in\mathbb{N}_{0}}$, on $G$, where $\mathbb{N}_0:=\{0,1,2,..., \}$, until the first time the walk has visited all the vertices. Obviously, the random subgraph ${\mathcal T}:=(V,E^{\prime})$ with edge set given as 
\begin{equation}
\begin{aligned} 
   E^{\prime}\coloneqq \big\{\{W(\sigma_v-1),v\};\,v\in V\setminus\{W(0)\}\big\},
\end{aligned}
\end{equation} 
where $\sigma_{v}:=\min\big\{k\in \mathbb{N}_{0}:\,W(k)=v\big\}$,
is a spanning tree of $G$. 
Less obvious though, Aldous (\cite[Proposition~1]{Aldous1990}) and Broder (\cite[Corollary~4]{Broder1989}) showed independently that ${\mathcal T}$ is the UST($G$), or equivalently, $(V,E^{\prime},W(0))$ is the uniform rooted spanning tree provided that $W(0)$ is distributed according to the stationary distribution of $W$ (compare also with \cite{AnanthrahamTsoucas1989}).
{A related algorithm generating the UST($G$) faster than the cover time is the {\em Wilson algorithm} (\cite{Wilson1996}) {which allows us to sample the UST$(G)$ by joining together
 loop-erased random walk paths on $G$. Given a path segment $\gamma([0,n]):=(\gamma(0),...,\gamma(n))\in {V^{[0,n+1]}}$, the loop erased path segment $\mathrm{LE}(\gamma([0,n]))$ is the loop free path segment obtained by erasing all loops in chronological order.
 Recently, {some} modifications of the Aldous-Broder algorithm have been studied as well (\cite{Hutchcroft2018,HuaLyonsTang2021}).}}

Exploiting the reversibility of the driving random walk the idea behind the two above algorithms can be turned into 
{{\em coupling from the past}}. For that we consider yet another version of the Aldous-Broder algorithm
(\cite{Symer1984,AnanthrahamTsoucas1989,Aldous1990}).  
Consider now the random walk $(W(n))_{n\in\mathbb{Z}\cap(-\infty,0]}$ on the finite simple, connected graph $G=(V,E)$ run from time $-\infty$ to zero, and let for $v\in V$ and for an interval $I\subseteq\mathbb{Z}$ bounded by above,
\begin{equation}
\label{e:043}
    L_v(I):=\max\big\{k\in I:\,W(k)=v\big\}
\end{equation}
be the time of the last visit of $v$ during $I$. Then {it is obvious that} the subgraph $G''=(V,E'')$ with
\begin{equation} 
\label{e:029}
   E'':=\big\{\{v,W(L_v(0)+1)\};\,v\in V\setminus\{W(0)\}\big\},
\end{equation}
is a tree spanning $G$. Moreover, it is shown in \cite{AnanthrahamTsoucas1989} that the stationary distribution of the Aldous-Broder chain is for each rooted tree {${\mathcal T}=(V,\widetilde{E},\widetilde{\varrho})$ spanning $G=(V,E)$ given as
{\begin{equation}
\label{e:201}
   \mathbb{P}({\mathcal T})=C^{-1}\prod_{x\in V\setminus\{\varrho\}}\big(\mathrm{deg}_{G}(x)\big)^{-1}
\end{equation}}
with a normalizing constant $C$. }{Here we denote by $\mathrm{deg}_{G}(x)$ the degree in $G$ of the vertex $x$.}
In particular, if $G$ is a regular graph, i.e., if all vertices have the same degree in $G$, then the {UST($G$)} is stationary for the Aldous-Broder chain.

We can use the idea behind {this algorithm, and provide a stationary Markov chain, $Y^{G}=(Y^{G}(n))_{n\in\mathbb{N}_{0}}$, with values in the space of rooted trees which we start that time $0$ in the trivial tree and which has the UST$(G)$ as its invariant distribution. For that, we define} the Aldous-Broder map (AB-map) that sends the random walk path $W=(W(n))_{n\in\mathbb{N}_{0}}$ to a rooted tree $(T(n),E(n),\varrho(n))_{n\in\mathbb{N}_{0}}$ given for every $k\in\mathbb{N}_0$ as follows:
\begin{itemize}
\item $T(k):=\{W(0),\dots,W(k)\}$,
\item $E(k):=\big\{\{v,W(L_v(k)+1)\};\,v\in T(k)\setminus\{W(k)\}\big\}$, and
\item $\varrho(k):=W(k)$
\end{itemize}
with $L_v(k)=L_v([0,k])$ as defined in (\ref{e:043}).
This yields a discrete-time stochastic process,
\begin{equation}
\label{f:001}Y^{G}=(Y^{G}(n))_{n\in\mathbb{N}_{0}},
\end{equation}
taking values in the space of rooted trees of $G$ whose invariant distribution is the UST rooted according to the stationary distribution of the  random walk on $G$.

\begin{figure}
\includegraphics[width=12cm]{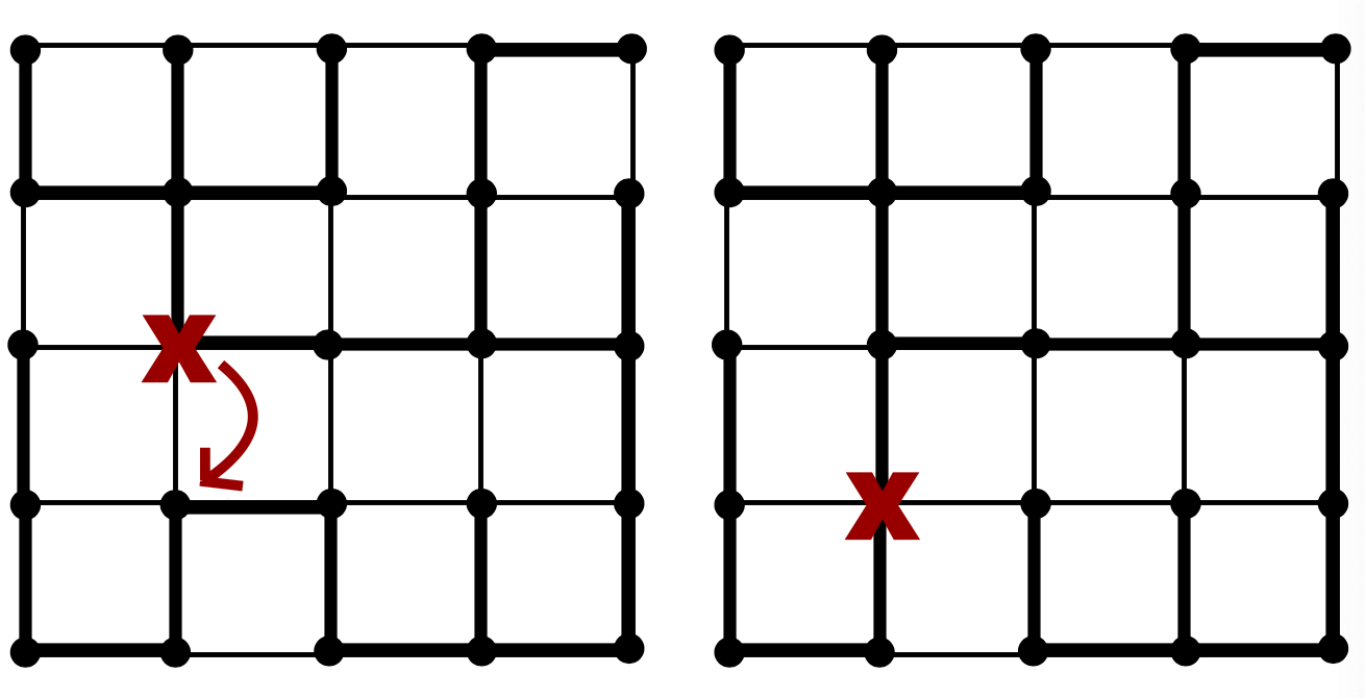}
\caption{Illustrates an AB-move in the 
Aldous-Broder {algorithm} on a lattice graph in discrete time. The red cross indicates the current root. {One} can see that the 
new root was already contained in the tree, and thus it gets connected with the old root. At the same time, we erase the edge connecting the new root with the vertex visited one step after the last visit of the new root. Note that here we jump from a spanning tree of intrinsic height $13$ to one of intrinsic height $9$.}
\label{Fig:001}
\end{figure}

{This so-called {\em Aldous-Broder chain} starts in the trivial tree $Y^{G}(0)=(\{W(0)\},W(0))$ (and trivially rooted at $Y^{G}(0):=W(0)$), and
has the following one step transition (compare with Figure~\ref{Fig:001}):} given the current state $(T,E,\varrho)$ at say time $n$, then in the next step {with probability $\frac{1}{2}$ we {remain in $(T,E,\varrho)$}, while with probability $\frac12$ we do the following:}
\begin{itemize}\label{definitionAldousBroderChain}
\item[] {{\bf (AB-chain)}} We pick a vertex $\upsilon$ {adjacent to $\varrho$ in $G$} at random according to the simple random walk transition probability ${\mathrm{P}}(\varrho,\upsilon)=\frac{1}{\mathrm{deg}_{G}(\varrho)}$, where $\varrho$ is the current root,
    and then do the following:
    \begin{itemize}
\item {{\bf Root Growth. }} If $\upsilon\not\in T$, {we add $\upsilon$ to the set of vertices and we insert} a new edge between $\upsilon$ and $\varrho$.
\item {{\bf AB-move. }} If $\upsilon\in T$,  we insert a new edge between $\upsilon$ and $\varrho$ but erase the edge {that connects $\upsilon$ with $W(L_{\upsilon}(n)+1)$.}
\item {{\bf New root. }} {In any case we let $\upsilon$ be the new root.    } 
\end{itemize}
\end{itemize}

The main goal of the present paper is to present a scaling limit of the Aldous-Broder chain $Y^{G}$ for a class of increasing regular graphs including $d$-dimensional tori of side length $N$ as $N\to\infty$, $d\ge 5$. For that purpose we encode our trees as rooted metric trees. For that we refer to a pointed metric space $(T,d,\varrho)$ as a rooted
metric tree if $\varrho\in T$ and $(T,d)$ satisfies the following two conditions (\cite[Definition~1.1]{AthreyaLohrWinter2017}):
\begin{enumerate}
\item[{\bf (4pc)}] $(T,d)$ is {\em $0$-hyperbolic}, or equivalently, satisfies the {\em $4$ point condition}, i.e.,
\begin{equation}
\begin{aligned}
	\label{e:007}
	&\mathrm{d}(x_1,x_2)+\mathrm{d}(x_3,x_4) \\
	& \quad \le\max\big\{\mathrm{d}(x_1,x_3)+\mathrm{d}(x_2,x_4),\mathrm{d}(x_1,x_4)+\mathrm{d}(x_2,x_3)\big\},
\end{aligned}\end{equation}
for all $x_1,x_2,x_3,x_4\in T$.
\item[{\bf (bp)}] $(T,d)$ contains all branch points, i.e., for all  $x_1,x_2,x_3\in T$ there is $c=c(x_1,x_2,x_3)\in T$ such that for $1\le i<j\le 3$,
\begin{equation}\label{e:008}
\mathrm{d}(x_i,c)+\mathrm{d}(c,x_j)=\mathrm{d}(x_i,x_j).
\end{equation}
\end{enumerate}
\noindent We will refer to a metric tree $(T,d)$ as an {\em $\mathbb{R}$-tree} if it is in addition path-connected.

{To measure the distance between any two rooted metric trees, $(T,d,\varrho)$ and $(T',d',\varrho')$, we introduce the notion of correspondences and their distortion (compare e.g. with \cite{MR1699320} or
 \cite[Theorem~7.3.25]{BuragoBuragoIvanov2001}). We call a subset $\mathfrak{R}\subseteq T\times T'$ a {\em correspondence} if $(\varrho,\varrho')\in \mathfrak{R}$, and if  
for all $x\in T$ there is $x'\in T'$ with $(x,x')\in\mathfrak{R}$ and for all $y'\in T'$ there is $y\in T$ with $(y,y')\in\mathfrak{R}$, and denote its {\em distortion} by 
\begin{equation}
\begin{aligned} 
\label{distortion}
   \mathrm{dis}({\mathfrak R}):=\sup\big\{\big|d(x,y)-d'(x',y')\big|;\,(x,x'),(y,y')\in\mathfrak{R}\big\}.
\end{aligned}
\end{equation}
Obviously, two metric spaces are isometric if and only if there exists a correspondence between them of distortion zero. We consider 
two rooted metric trees $(T,d,\varrho)$ and $(T',d',\varrho')$ to be {\em equivalent}} if there exists an isometry $\varphi:T\to T'$ with $\varphi(\varrho)=\varrho'$, and denote by $\mathbb{T}_{\mbox{\tiny metric}}$ the space of equivalence classes of pointed compact metric trees. {Define the  Gromov-Hausdorff distance (GH-distance) between two equivalence classes $[(T,d,\varrho)]$ and  $[(T',d',\varrho')]$ as
\begin{equation}
\label{e:009b}
   d_{\rm GH}\big([(T,d,\varrho)],[(T',d',\varrho')]\big):=\frac{1}{2}\inf_{{\mathfrak R}}\mathrm{dis}({\mathfrak R}),
\end{equation}
where the infimum is taken over all correspondences ${\mathfrak R}$ between any representatives from the two equivalence classes. 
We say that a sequence $([(T_N,d_N,\varrho_N)])_{N\in\mathbb{N}}$ converges to $([(T,d,\varrho)])$ in {the topology of {\em Gromov-Hausdorff-convergence (GH-convergence)} if 
\begin{equation}
\label{f:010}
   d_{\rm GH}\big([(T_N,d_N,\varrho_N)],[(T'_N,d'_N,\varrho'_N)]\big)\tNo 0.
\end{equation}
}    
Note that $\mathbb{T}_{\mbox{\tiny metric}}$ equipped with the topology of Gromov-Hausdorff convergence 
is a Polish space (\cite[Theorem~2]{EvansPitmanWinter2006}).}

{As for a scaling limit, it was shown in \cite[Proposition~7.1]{EvansPitmanWinter2006} that the family of Aldous-Broder Markov chains $\{Y^{\mathbb{K}_m};\,m\in\mathbb{N}\}$ on the complete graph $\mathbb{K}_m$ with $m$ vertices} can be suitably rescaled as follows: suppose that each tree ${Y^{\mathbb{K}_m}}(0)$ is a non-random spanning tree of $\mathbb{K}_m$ and such that
$m^{-1/2}{Y^{\mathbb{K}_m}}(0)$ converges in the pointed Gromov-Hausdorff topology to some rooted compact metric tree, $X(0)$, as $m\to\infty$. Then
there exists a piecewise deterministic Markov process, $X=(X(t))_{t\ge 0}$, such that
\begin{equation} 
\label{e:003}
   {\mathcal L}_{Y^{\mathbb{K}_m}}(0)\Big(\big(m^{-\frac12}Y^{m}(\lfloor m^{\frac12} t\rfloor) \big)_{t\ge 0}\Big) 
 \TNo 
    {\mathcal L}_{X(0)}\Big(\big(X(t)\big)_{t\ge 0}\Big),
\end{equation}
\noindent where convergence means weak convergence in Skorokhod topology of {random} c\`{a}dl\`{a}g paths with values in $\mathbb{T}_{\mbox{\tiny metric}}$ equipped with the topology of pointed {GH-c}onvergence. The so-called {\em Root Growth with Re-grafting {d}ynamics (RGRG)}, $X=(X(t))_{t\ge 0}$, appearing  in the limit is the piecewise deterministic Markovian jump process with the following dynamics: 
given {the current state} $(T,d,\varrho)$, 
\begin{itemize}
\item {\bf Re-grafting. } For each $\upsilon\in T$, a cut point occurs at unit rate $\lambda^{(T,d)}(\mathrm{d}\upsilon)$ with $\lambda$ being the length measure on $(T,d)$ (see \cite[Section~2.4]{EvansPitmanWinter2006}).
    As a result $T\setminus\{\upsilon\}$ is decomposed into two subtree components. One of them containing the root $\varrho$. We reconnect by identifying $\upsilon$ with $\varrho$, i.e., redefining all mutual distances accordingly. 
\item  {\bf Root growth. } In between two jumps the root grows away from the tree at unit speed. 
\end{itemize}

{RGRG is a deterministic pure jump process if started in a tree of finite length. A rigorous Poisson process construction of the RGRG for general initial trees is given in \cite{EvansPitmanWinter2006}. 
The RGRG is invariant under the random real tree known as the Brownian continuum random tree (CRT). The CRT appears as the scaling limit of suitably rescaled Bienaym\'{e} branching trees with finite variance offspring distribution given a fixed number of $m$ vertices, as $m\to\infty$ (\cite{Aldous1991a,Aldous1993}).} 
{If we choose for the offspring distribution a critical Poisson distribution and condition on its size to be $m$, then the resulting tree equals in distribution the uniform tree with $m$ vertices, or equivalently, the UST$(\mathbb{K}_m)$. This implies that }
\begin{equation}
\label{e:059}
   {\mathcal L}\Big(m^{-\frac{1}{2}}\mathrm{UST}\big(\mathbb{K}_m\big)\Big)\Tmo{\mathcal L}\Big(\mathrm{CRT}\Big),
\end{equation}
where here $\Rightarrow$ denotes weak convergence of probability distributions on $\mathbb{T}_{\mbox{\tiny metric}}$
equipped with the topology of the {GH}-convergence.

Apparently, as conjectured by Pitman, the CRT is also the scaling limit of the UST {for a class of graphs that are relatively fast mixing}. This class includes the $d$-dimensional discrete torus
\begin{equation} \label{e:001}
   \mathbb{Z}^{d}_{N} = \{0,1,\dots,N-1\}^{d}\big|_{\mbox{{\tiny mod} }N}.
\end{equation}
{And i}ndeed, it was shown in \cite[Theorem~1.1]{PeresRevelle}, that for $d\ge 5$ there exists a constant $\beta(d)$ such that
\begin{equation}\label{e:059b}
{\mathcal L}\Big(N^{-\frac{d}{2}}\mathrm{UST}\big(\mathbb{Z}^d_N\big)\Big)\TNo{\mathcal L}\Big(\beta(d)\cdot\mathrm{CRT}\Big),
\end{equation}
where convergence is  with respect {to a related topology of Gromov-weak convergence}.
The constant $\beta(d)$ is of the form $\beta(d)=\frac{\gamma(d)}{\sqrt{\alpha(d)}}$ with
\begin{equation}
\label{e:062}
   \gamma(d)=\mathbb{P}\big(\mathrm{LE}({\hat{W}^1([0,\infty))})\cap\hat{W}^2([1,\infty))=\emptyset\big)
\end{equation}
\noindent and
\begin{equation}
\label{e:063}
\begin{aligned}
   &\alpha(d)=\mathbb{P}\big(Q_{1,2}\cap Q_{1,3}\cap Q_{2,3}\big)
\end{aligned}
\end{equation}
where for $i,j\in\{1,2,3\}$, $Q_{i,j}$ denotes the event $\{\mathrm{LE}(\hat{W}^i([0,\infty)))\cap\hat{W}^j([1,\infty))=\emptyset\}$, and with {$W^1$, $W^2$ and $W^3$} independent lazy random walks on $\mathbb{Z}^d$ starting at the origin {(see \cite[Lemma~8.1]{PeresRevelle})}. {The result was extended in \cite[Theorem~1.1]{Schweinsberg2009} to $d=4$ up to some logarithmic term correction the scaling factor.}
Recently, the CRT up to a scaling factor was also shown to be  the scaling limit of the UST on sequences of dense graphs that converge in graphon sense and thus also for the UST of the Erd\"os-Renyi graph in the dense regime (\cite{ArcherShalev2024}).  

This convergence was recently stated in \cite[Theorem~1.1]{MR4712854} to even hold, for  $d\ge 5$, with respect to a stronger notion of convergence, namely in the {GH-weak sense, and thereby also in the topology of GH-convergence} (compare~\cite{MR3522292}). {This suggests that it might be possible to also rescale the Aldous-Broder chain on the high-dimensional in GH-topology such that the scaling limit is invariant under the CRT. }

If we observe how the distance of the initial root to the current root is changing over time, then (\ref{e:003}) implies that on the complete graph for all initial positions $k\in\mathbb{K}_m$,
\begin{equation}
\label{e:065}
{\mathcal L}_k\Big(\big(m^{-\frac{1}{2}}d^{\mathbb{K}_m}(\varrho^{\mathbb{K}_m}(\lfloor m^{\frac{1}{2}}t\rfloor),\varrho^{\mathbb{K}_m}(0))\big)_{t\ge 0}\Big)\Tmo{\mathcal L}_0\Big(\big(R(t)\big)_{t\ge 0}\Big),
\end{equation}
where here $\Rightarrow$ is weak convergence with respect to the Skorokhod topology for $[0,\infty)$-valued c\`{a}dl\`{a}g paths, and where $R=(R(t))_{t\ge 0}$ is the {\em Rayleigh process}, i.e., the piecewise deterministic $[0,\infty)$-valued Markovian jump process with generator
\begin{equation}
\label{e:066}
\Omega f(x)=f'(x)+\int_0^1\big(f(ux)-f(x)\big)\,\mathrm{d}u, \hspace{1cm}f\in{\mathcal C}^1([0,\infty)).
\end{equation}
\noindent Its invariant distribution is the {\em Rayleigh distribution} which has density $xe^{-\frac{x}{2}}$, $x \in [0, \infty)$.

In \cite[Theorem~1.1,Remark~1.2]{Schweinsberg2008} the dynamics of the length of a loop erased path of a random walk on $\mathbb{Z}^d_N$ was studied. It was {shown that} this process has in $d\ge 4$ also the Rayleigh process as its limit. In $d\ge 5$, the result reads as follows: if $W_N$ is the simple random walk on $\mathbb{Z}^d_N$ {starting in $0$}, then
\begin{equation}
\label{e:066}
{\mathcal L}\Big(\big(N^{-\frac{d}{2}}\#{\rm LE}(W_N([0,N^{\frac{d}{2}}t]))\big)_{t\ge 0}\Big)\TNo{\mathcal L}_0\Big(\beta(d)\big(R(\sqrt{\alpha(d)}t)\big)_{t\ge 0}\Big),
\end{equation}
where here $\Rightarrow$ is once more weak convergence with respect to the Skorokhod topology for $[0,\infty)$-valued c\`{a}dl\`{a}g paths. \\

Our main result generalizes this to the whole Aldous Broder chain as follows:
\begin{theorem}[Convergence of the Aldous-Broder chain on $\mathbb{Z}^d_N$, $d\ge 5$]
Let $\{Y^{\mathbb{Z}_N^d};\,N\in\mathbb{N}\}$ be a family of Aldous-Broder chains with values in rooted metric trees embedded in $\mathbb{Z}^d_N$, $d\ge 5$ which start in the rooted {tree $Y^{\mathbb{Z}_N^d}(0)=(\{0\}, 0)$}.
Then with the constants $\beta(d)$ and $\alpha(d)$ from (\ref{e:059b}) and (\ref{e:063}), respectively,
\begin{equation}
\label{e:005a}
   {\mathcal L}_{{(\{0\},0)}}\Big(\Big( N^{- \frac{d}{2}} Y^{\mathbb{Z}^d_N}\big(\lfloor N^{\frac{d}{2}} t \rfloor\big)\Big)_{t\ge 0}\Big) 
 \TNo 
   {\mathcal L}_{(\{\varrho\},\varrho)}\Big(\Big(\beta(d)\cdot X\big(\sqrt{\alpha(d)} t\big)\Big)_{t\ge 0}\Big),
\end{equation}
where $\Rightarrow$ stands for weak convergence in {Skorokhod} space on $\mathbb{T}_{\mbox{\tiny metric}}$ in the topology of GH-convergence.
\label{T:001}
\end{theorem}

{As we see in Theorem~\ref{T:001}, the spatial dependencies inherent in the trees spanning {$\mathbb{Z}^d_N$, $d\ge 5$,} are not relevant on the space and time scale $N^{\frac{d}{2}}$.
{Notice that $d\ge 5$ is considered as {\em transient regime} in the sense that} if $W_N^1$ and $W^2_N$ are two independent lazy random walk on the torus $\mathbb{Z}^d_N$, $d\ge 5$, then the mean number of intersections, $H^{\mathbb{Z}^d_N}$ until the mixing time  (see (\ref{e:transient}) below), 
is uniformly bounded as $N\to\infty$. 
In this transient regime the lazy random walk on $\mathbb{Z}^d_N$ allows until a time of the order $N^{\frac{d}{2}}$
for a {\em separation between short versus long loops} (compare with Assumption~\ref{assumptionNoLoopsOfIntermediateLength}). That is, lazy random walk paths  frequently  self-intersect 
by closing loops much shorter than the mixing time, while rarely they are 
closing a loop longer than the mixing time, and if so its length is of macroscopic order $N^{\frac{d}{2}}$.

On $\mathbb{Z}^d_N$, the rate at which the latter happens equals $N^{-\frac{d}{2}}$. Thus, running the Aldous-Broder chain on the time scale $N^{\frac{d}{2}}$ results in a finite number of macroscopically long loops. Due to the transience of the 
random walk on the whole lattice $\mathbb{Z}^d$, $d\ge 5$, loops of length $N^{\frac{d}{2}}$  would not be observable on $\mathbb{Z}^d$
on our time scale $N^{\frac{d}{2}}$. This explains why the constants $\beta(d)$ and $\alpha(d)$ can be expressed as intersection probabilities of the random walk on $\mathbb{Z}^d$. This is similar to what is known as the {\em finite system scheme} which compares the behavior of large finite particle systems or interacting diffusions on the torus with the infinite system on the whole lattice (\cite{CoxGreven1994,CoxGrevenShiga1995,MR2074427,GrevenLimicWinter2005}).  

To sketch the main ideas of the proof, notice first that due to the local geometry of the torus, the Aldous-Broder chain is very hard to handle on a microscopic scale. We therefore introduce for each $s\ge 1$ with the {\em $s$-skeleton chain} (or skeleton chain) a simpler chain which is still very close to the Aldous-Broder chain in GH-distance. In the $s$-skeleton chain, at many of the time indices at which the lazy random walk closes a loop of length at most $s$, we are going to erase this loop completely rather than doing only
the required AB-move, i.e, just erasing the edge from the current position of the random walk to the vertex it went to after the last visit of this position. However, in doing so we need to be careful, as a loop of length at most $s$ might be twisted with a loop longer than $s$ in such a way that erasing the smaller loop would disconnect the tree grown so far into two disconnected components (see Figure~\ref{figGhostIndexPretzelV2}). We want to avoid the latter. For that purpose we 
introduce the notion of {\em $s$-ghost indices} (see Definition~\ref{DefGhostI}), 
and define the rooted $s$-skeleton tree at a given time $n\in\mathbb{N}$ as the subgraph obtained from the rooted Aldous-Broder tree at time $n$ restricted to the non-$s$-ghost indices at time $n$. We then assume that the Aldous-Broder chain is driven by a good path $\gamma:\,\mathbb{N}\to V$ that is a path with the following two properties: for given $s',s,r\in\mathbb{N}$ with $r\ge 3s+1\ge 18s'+1$, 
\begin{itemize}
\item {\bf separating loop lengths. }there are no loops of length larger than $s'$ and shorter than $r$,  
\item {\bf dense set of $2s'$-local cut points. } for each $n\in\mathbb{N}_0$ with $n\ge 2s'$ there is a so-called {\em $2s'$-local cut point} $\ell\in[n,n+s]$ satisfying $\mathrm{R}^\gamma([\ell-2s',\ell])\cap \mathrm{R}^\gamma([\ell+1,\ell+2s'])=\emptyset$,
\end{itemize}
where once more we write $\mathrm{R}^\gamma(A):=\{\gamma(n),\,n\in A\}\subseteq V$ for the range of the path $\gamma$ over the time index set $A\subseteq\mathbb{N}_0$. Under these conditions, 
loops of length at most $s$ 
get erased in the $s$-skeleton chain if they are not twisted with a loop of macroscopic length. The conditions also ensure that we can rule out such a twist if  
these loops of lengths at most $s$ are not closing too shortly after a time index at which a macroscopically long loop was closed. 
We can further show that for all $n\in\mathbb{N}$ the $s$-skeleton chain at time $n$ is connected and thus a rooted tree that is in GH-distance at most $r$ from the Aldous-Broder tree at time $n$ (Proposition~\ref{DerterLemma1}). Finally, the dynamics of the $s$-skeleton is easy to describe: it consists of {\em root growth}, {\em AB-moves} mostly only when a macroscopically long loop is closed, and erasure of loops no longer than $s$.

 Consider next a sequence of simple, connected regular graphs $\{G_N;\,N\in\mathbb{N}\}$ in the {\em transient regime}, i.e., such that $\limsup_{N\to\infty}H^{G_N}<\infty$, where 
\begin{equation}
\label{e:transient}
   H^{G_N}:=\sup_{\varrho_N\in V_N}\sum\nolimits_{i,j=1}^{\tau^{G_N}_{\mathrm{mix}}}\mathbb{P}_{(\varrho_N,\varrho_N)}\big(W^1_N(i)=W^2_N(j)\big),
\end{equation}   
for two independent lazy random walks $W^1_N$ and $W^2_N$ on $G_N$ (see (\ref{e:015}) for our definition of $\tau^{G_N}_{\mathrm{mix}}$). 
Put
\begin{equation}
\label{e:ccNN}
   \big(c^{G_N}(r_N)\big)^2:=\mathbb{P}_{\pi^{G_N}\otimes\pi^{G_N}}\big(\mathrm{LE}(W_N^1[1,r_N-s_N])\mbox{ intersects with }W^2_N([1,r_N-s_N])\big),
\end{equation}
and let
\begin{equation}
\label{e:GammaNN}
   \Gamma^{G_N}(r_N)
   :=
   \mathbb{E}_{\pi^{G_N}}\big[\#\mathrm{LE}(W_N(A_1^N))\big].
\end{equation}

Assume that we can find a sequence $(r_N)\subset \na$, 
such that
\begin{equation}
\label{f:021}
   \big(\ln(\# V_{N})\big)^2\,t^{G_N}_{\mathrm{mix}}\ll r_N\ll\frac{\Gamma^{G_N}(r_N)}{c^{G_N}(r_N)}={\mathcal O}\big(\sqrt{\# V_N}\big)
\end{equation}
(compare Corollary~\ref{corollaryUniformBoundsOfNonErasedVertices}). In order to establish a scaling limit for the skeleton chain on $G_N$, we further choose $(s_N)\subset \na$ such that 
\begin{equation}
\label{f:s_N}\big(\ln(\# V_{N})\big)^2\,t^{G_N}_{\mathrm{mix}}\ll s_N\ll \big(\log{(c^{G_N}(r_N))}\big)^6 r_N,
\end{equation} and
rely on the fact that for any time intervals $A^N_1$, $A^N_2$, ... of equal size, $r_N-s_N$, 
and separated by at least $s_N$,
the segments $W_N(A^N_i)$ are up to time of order $\frac{r_N}{c^{G_N}(r_N)}$ nearly independent and identically distributed like the path segment $W_N([0,r_N])$ of the lazy random walk started in the stationary distribution. Moreover, with high probability lazy random walk paths are good paths up to time of order $\frac{r_N}{c^{G_N}(r_N)}$. We are then using these segments to mimic points in the complete graph. 
To do so, we refer to a segment $\gamma(I')$ as the $s$-local loop erased path segment $\gamma(I)$, for a path $\gamma:\,\mathbb{N}\to V$, a finite interval $I\subset\mathbb{N}_0$ and $I'\subseteq I$, if $\gamma(I')$ is obtained from $\gamma(I)$ by erasing all loops of size at most $s$ in chronological order. 
We then say that the $j^{\mathrm{th}}$ segment performs
\begin{itemize}
\item {\bf Root growth, } if the range $\mathrm{R}^{W_N}(A_j^N)$ does not intersect with the range of any of the $s_N$-local loop erased path segments
$W_N(A^N_i)$, $1\le i<j$. In this case we glue the $s_N$-local loop erased path segment $W_n(A_j^N)$  to the tree grown so far.
\item an {\bf AB-move, } if the range $\mathrm{R}^{W_N}(A_j^N)$ intersects with, say  
the $s_N$-local loop erased path
segment $W_N(A^N_i)$ for some $1\le i<j$. In this case we perform an AB-move on the 
segments, i.e., 
we take away the $s_N$-local loop erased path
segment which had latest been glued to the $s_N$-local loop erased path $W_N(A^N_i)$ before. 
\end{itemize}

With the choice of our scaling parameters we can couple the Poisson process driving the limiting RGRG and the lazy random walk driving the skeleton chain in such a way that with high probability up to time of order $\frac{r_N}{c^{G_N}(r_N)}$ such that we find a Poisson point in the square $B^{c^{G_N}(r_N)}_{i,j}$ 
where for $c>0$
\begin{equation}
\label{f:022}
 B^{c}_{i,j}:=\big\{(x,y):\,(j-1)c\le x<jc,(i-1)c\le y< ic\big\} 
\end{equation}  
for some $1\le i<j$, 
if and only if the $j^{\mathrm{th}}$-segment performs an AB-move involving the $s_N$-loop erased $i^{\mathrm{th}}$-segment.  
On such an event we obtain good bounds on the GH-distance between the $s_N$-skeleton chain and the RGRG.  
As in between two AB-moves branch lengths grow like those of a loop erased random walk and these lengths concentrate around the mean, after rescaling edge length in the Aldous-Broder chain by $\frac{c^{G_N}(r_N)}{\Gamma^{G_N}(r_N)}$, we converge to the RGRG in the  
Skorokhod topology.

We shall therefore generalize Theorem~\ref{T:001} as follows:
\begin{theorem}[Scaling the Aldous-Broder chain on regular graphs: transient regime]\label{T:002}
Assume that $(G_{N})_{N \in \mathbb{N}}$ is a sequence of finite simple, connected, regular graphs in the transient regime, i.e.,  $\limsup_{N\to\infty}H^{G_N}<\infty$, and such that we can choose $(s_N)$ and $(r_N)$ such that (\ref{f:021}) and (\ref{f:s_N}) holds. 
Let $\{Y^{G_N};\,N\in\mathbb{N}\}$ be a family of Aldous-Broder chains with values in rooted metric trees embedded in $G_N$ which start in the rooted trivial tree, i.e., $Y^{G_N}(0)=(\{\varrho_N\}, \varrho_N)$. 
Then, 
\begin{equation}
\label{f:T002}{\mathcal L}_{(\{\varrho_N\},\varrho_N)}\Big(\Big( \frac{c_N}{\mathrm{E}_{\pi^{G_N}}[\#\mathrm{LE}^{W_N}([0,r_N])]} Y^{G_N}\big(\Big\lfloor \frac{t}{c_N}\Big\rfloor r_N\big)\Big)_{t\ge 0}\Big) 
 \TNo 
   {\mathcal L}_{(\{\varrho\},\varrho)}\Big(\big(X\big(t\big)\big)_{t\ge 0}\Big),
\end{equation}   
where $\TNo$ here means weak convergence in the Skorokhod space of  c\`{a}dl\`{a}g paths with values in $\mathbb{T}_{\mathrm{metric}}$ equipped with the pointed GH-convergence.
\end{theorem}
}

\bigskip

\noindent {\em Outline. }The rest of the paper is organized as follows: {In Section~\ref{S:RW} we 
introduce the random walk driving the Aldous-Broder chain as well as a locally loop erased version. We then get bounds on the intersection probabilities of a random walk segment with  an earlier and locally loop erased segment (Lemma~\ref{corollaryExt3}). We also show a concentration inequality of a locally loop erased segment around its mean (Lemma~\ref{PereRevellConcentration}). In Section~\ref{S:RGRG}
we recall the Poisson point process construction of the RGRG from \cite{EvansPitmanWinter2006}.
In Section~\ref{S:AldousBroder} we construct with the {\em skeleton chain} an auxiliary discrete-time chain with values in rooted tree graphs that is close in Skorokhod distance with respect to the pointed GH-distance when rooted tree graphs, and  
for which it is much easier to prove the convergence to the RGRG dynamics after our rescaling. In Section~\ref{Sub:decomposablepaths} 
we show that with high probability
the random walk paths on the macroscopic time scale can  be decomposed  into nearly independent and identically distributed segments separated by gaps with lengths of a slightly larger order than the mixing time. 
 In Section~\ref{sectionBoundingGHBEtweenABChainAndRGRGR} we bound the GH-distance between the skeleton chain and the RGRG.  
In Section~\ref{Sub:regraftindex} we present the coupling between the Poisson process driving the RGRG and the lazy random walk driving the skeleton chain.  
In Section~\ref{S:proof} we collect all our results and prove Theorem~\ref{T:002}. Finally, we discuss that high-dimensional tori in $d\ge 5$ satisfy all of our assumptions.
}

\section{Estimates on RWs and locally erased RWs on regular graphs}
\label{S:RW}

In this section we provide estimates on the lazy random walk that will be useful for the proof of our main result.  
We start in Subsection~\ref{Sub:SRWestimates} with estimates that allow to compare a family of path segments of the lazy random walk with path segments of an  i.i.d.\ family of lazy random walks. In Subsection~\ref{Sub:rangeintersect} we give bounds on the mean number of certain self-intersection events involving the lazy random walk. Finally, in Subsection~\ref{Sub:LERW}, we introduce loop-erased and locally loop-erased random walks and state a uniform concentration inequality for the length of (locally) loop-erased path segments (Lemma~\ref{PereRevellConcentration}).

Let $G = (V, E)$ be a finite simple, connected graph with a finite vertex set $V = V(G)$ and edge set $E = E(G)$. Write $x\sim y$ if the vertices $x,y \in V$ are connected by an edge, i.e., if $\{x,y\}\in E$. As usual the {\em degree} of a vertex $x \in V$ is given as $\mathrm{deg}(x):=\sum_{y \in V}\mathbf{1}_E(\{x,y\})$. We then say that $G=(V,E)$ is a finite simple, connected, {\em regular graph} if the map $x\in V\mapsto \mathrm{deg}(x)$ is constant. 
To avoid trivialities, we always assume $\# V\geq 2$ (and thus $\# E\geq 1$), where for any finite set $A$, we denote by $\# A$ its \emph{cardinality}. 

We consider the time-homogeneous discrete-time Markov chain $W:=(W_n)_{n\in\mathbb{N}_0}$ such that
\begin{equation}
\begin{aligned} \label{f:003}
\mathbb{P}_x\big(W(1) = y\big)
&= 
\mathbb{P}\big(W(1) = y| W(0) = x\big) 
:=
\left\{\begin{array}{cc}\frac{1}{2} & \mbox{ if }x=y, \\[2mm] \frac{1}{2\mathrm{deg}(x)}, & x\sim y \in V.\end{array}\right.
\end{aligned}
\end{equation}
We refer to $W$ as the {\em lazy random walk} on $G$.

Obviously, $W$ is irreducible and reversible. Also, the degree distribution
\begin{equation}
\label{e:014}
\pi^G(x)\coloneqq \frac{\mathrm{deg}(x)}{2\# E},\quad x\in V,
\end{equation}
is the stationary distribution. Specifically, if $G=(V,E)$ is a regular graph, then $\pi^G$ is the uniform distribution on $V$, i.e., $\pi^G(x)\equiv(\# V)^{-1}$ for $x\in V$. 
Moreover, $W$ is aperiodic. Thus, $W(n)$ converges, in distribution, to $\pi^G$, as $n\to\infty$.

For all $\varepsilon\in(0,1)$, we denote by
\begin{equation}\begin{aligned} \label{e:108}
t_{\mbox{{\tiny mix}},\ell_\infty}^{G}\big(W;\varepsilon\big) \coloneqq
\min\big\{n\in\mathbb{N}_{0}:\, \max_{x,y \in V} \big| \tfrac{\p_x(W(n)=y)}{\pi^{G}(y)} - 1 \big| \leq \varepsilon \big\}.
\end{aligned}
\end{equation}		

Put
\begin{equation}
\label{e:015}
t_{\mbox{\tiny mix}}^{G}=t_{\mbox{\tiny mix}}^{G}\big(W\big):=t_{\mbox{\tiny mix},\ell_\infty}^{G}\big(W;\tfrac{1}{4}\big).
\end{equation}
That is, for any $n \geq t_{{\mbox{\tiny mix}}}^{G}$ and $x,y\in V$,
\begin{equation}
\begin{aligned} 
	\label{Exteq2}
	\tfrac{3}{4}\pi^{G}(y) \leq \p_x(W(n)=y)
	\leq \tfrac{5}{4} \pi^{G}(y).
\end{aligned}
\end{equation}

\subsection{Decomposition into nearly independent random walk path segments}
\label{Sub:SRWestimates}
In this subsection we decompose the path of a lazy random walk into nearly independent segments. 

We refer to $\gamma:\,\mathbb{N}_0\to V$ as a path on $G$ iff $\{\gamma(n),\gamma(n+1)\}\in E$ or $\gamma(n+1)=\gamma(n)$, for all $n \in \mathbb{N}_0$. 
 For any $m,n\in \na_0$ with $m\le n$, abbreviate
\begin{equation}
\label{e:012}
[m,n]:=\big\{j\in\na_0:\,m\leq j\leq n\big\}.
\end{equation}We use the convention $[m,n]=\emptyset$ if $m>n$. 

If $I\subseteq\mathbb{N}_0$, we write 
\begin{equation}
\label{e:002} 
\gamma(I) \coloneqq \big(\gamma(n):\,n\in I\big) \in V^{\# I},
\end{equation}
for the map which sends each time index $n\in I$ to a vertex $\gamma(n)\in V$. If $I$ is a finite interval, i.e., $I=[m,n]$ for some $m,n\in\mathbb{N}_0$ with $m\le n$, then $\gamma(I)$ yields a {\em path segment}. 
By convention we define $\gamma(\emptyset):=\emptyset$. 

For two lazy random walks $(W_1,W_2)$, we denote by $\mathbb{P}_{\pi^G\otimes \pi^G}$ their joint law when each starts independently from the stationary distribution, and $\mathbb{P}_{(\rho,\rho)}$ starting each on the same vertex $\rho\in V$. 
Similarly, for several lazy  random walks, the notation $\mathbb{P}_{\pi^G\otimes \cdots \otimes \pi^G}$ will be used.

The first result 
extends \cite[Lemma~2.5]{Schweinsberg2009}.
It says that probabilities of path segments $W(A_1), \ldots, W(A_k)$ of a lazy  random walk, each happening after a long period of time, can be bounded by path segments of \emph{independent} random walks $W_1(A_1),\ldots, W_k(A_k)$ that start in the stationary distribution.
After this lemma, we give a more precise estimate to \emph{approximate} such quantities.

\begin{lemma}[Nearly independent after mixing]\label{lemma1NewE}
Let $G=(V,E)$ be a finite simple, connected graph, $k\in\mathbb{N}$, and $W,W_1,...,W_k$ independent lazy  random walks on $G$. Consider  $s\in\mathbb{N}$ with 
$s \geq t_{\mbox{\tiny mix}}^{G}+1$, and finite intervals 
$A_1,...,A_k$ with $\min(A_1)\ge s$, $\min(A_{i+1})-\max{(A_i)}\ge s$ for all $i\in\{0,\ldots, k-1\}$. 
Then for all $F:V^{\sum_{i=1}^k\# A_i}\to\mathbb{R}_+$ and $\varrho\in V$, 
\begin{equation}
\begin{aligned} 
	\label{Exteq4}
	\mathbb{E}_{\varrho}\Big[F\big(W(A_1), \dots, W(A_k) \big) \Big] \leq 2^k\mathbb{E}_{\pi^G\otimes \cdots\otimes\pi^G} \Big[F \big(W_{1}(A_1), \dots, W_{k}(A_k) \big) \Big],
\end{aligned}
\end{equation}
and
\begin{equation}
\begin{aligned} 
	\label{Exteq5}
	\mathbb{E}_{\varrho}\Big[F\big(W(A_1), \dots, W(A_k) \big) \Big] \geq 2^{-k}
	\mathbb{E}_{\pi^G\otimes\cdots\otimes\pi^G} \Big[F\big(W_{1}(A_1), \dots, W_{k}(A_k)\big) \Big].
\end{aligned}
\end{equation}
\end{lemma}
\begin{remark}\label{rmrklemma1NewE}
In the setting of Lemma \ref{lemma1NewE}, note that since each $W_1,\ldots, W_k$ starts in the stationary distribution, we have
\begin{equation}
\mathbb{E}_{\pi^G\otimes \cdots\otimes\pi^G} \Big[F \big(W_{1}(A_1), \dots, W_{k}(A_k) \big) \Big]=\mathbb{E}_{\pi^G\otimes \cdots\otimes\pi^G} \Big[F \big(W_{1}([1,\#A_1]), \dots, W_{k}([1,\#A_k]) \big) \Big].
\end{equation}The latter implies that the expectation, only depends on the length of the intervals $A_1,\ldots, A_k$. 
\end{remark}
\begin{proof}
We will show that for all $\gamma_i\in V^{\# A_i}$, $i=1,...,k$, and all $\varrho\in V$,
\begin{equation}
\label{e:055}
\mathbb{P}_\varrho\Big(W\big(A_1\big)=\gamma_1,...,W\big(A_k\big)=\gamma_k\Big)
\le
2^k\prod_{i=1}^k \mathbb{P}_{\pi^G}\Big(W\big(A_i\big)=\gamma_i\Big).
\end{equation}  
For that we shall proceed by induction on $k$. Let $k=1$ 
and define $A-b:=\{a-b:a\in A\}$, for any finite non-empty interval $A$ and  $b\in \na_{0}$. Since $\min A_1-1\geq s-1\geq t^G_{\mbox{\tiny mix}}$, then for all $\gamma\in V^{\# A_1}$ and all $\varrho\in V$, by \eqref{Exteq2}, 
\begin{equation}
\label{e:055k1}
\begin{aligned}
	\mathbb{P}_\varrho\Big(W\big(A_1\big)=\gamma\Big)
	&=
	\sum_{v\in V}\mathbb{P}_\varrho\Big(W\big(\min A_1-1\big)=v\Big)\mathbb{P}_{v}\Big(W\big(A_1-\min A_1+1\big)=\gamma\Big)
	\\
	&\le
	2\sum_{v\in V}\pi^G(v)\mathbb{P}_{v}\Big(W\big([1,\#A_1]\big)=\gamma\Big)
	\\
	&=
	2\mathbb{P}_{\pi^G}\Big(W\big(A_1\big)=\gamma\Big),
\end{aligned}  
\end{equation}
where the last inequality follows by stationarity. 
\\

Assume next that (\ref{e:055}) holds for all $k\in\{1,\ldots, m-1\}$ for some $m\geq 2$. Then, writing $\gamma_{i}:=(\gamma_{i}(1),\ldots, \gamma_{i}(\# A_{i})) \in V^{\# A_{i}}$ for all $i\in\{1,\ldots,m \}$, then for any $\varrho\in V$,
\begin{equation} 
\label{Exteq7}
\begin{aligned}
	&\mathbb{P}_\varrho\Big(W\big(A_1\big)=\gamma_1,...,W\big(A_{m}\big)=\gamma_{m}\Big)
	\\
	&=
	\mathbb{P}_{\rho}\Big(W\big(A_1\big)=\gamma_1,...,W\big(A_{m-1}\big)=\gamma_{m-1}\Big)
	\mathbb{P}_{\gamma_{m-1}(\# A_{m-1})}\Big(W\big(A_{m}-\max A_{m-1}\big)=\gamma_m\Big)
	\\
	&\le
	2^{m-1}\prod_{i=1}^{m-1} \mathbb{P}_{\pi^G}\Big(W\big(A_i\big)=\gamma_i\Big)2\mathbb{P}_{\pi^G}\Big(W\big(A_m-\max A_{m-1}\big)=\gamma_1\Big)
	\\
	&=
	2^{m}\prod_{i=1}^{m} \mathbb{P}_{\pi^G}\Big(W\big(A_i\big)=\gamma_i\Big),
\end{aligned}
\end{equation}
where we have applied the induction hypothesis and 
\eqref{e:055k1} in the third line together with the assumption $\min A_{m}-\max A_{m-1}\geq s$.

The proof of (\ref{Exteq5}) is similar by using the lower bound in (\ref{Exteq2}).
\end{proof}

Recall from (\ref{e:108}) the uniform mixing time $t_{{\rm mix},\ell_\infty}^{G}\big(W;\varepsilon\big)$ of a lazy  random walk on $G$.
For all $\varepsilon\in(0,1)$, we denote by
\begin{equation}\begin{aligned} \label{e:108b}
{t_{{\rm mix},\ell_1}^{G}\big(W;\varepsilon\big)} 
&\coloneqq
\min\big\{n\in\mathbb{N}_{0}:\, \max_{x\in V}\frac{1}{2}\sum_{y\in V}\big|\p_x\big(W(n)=y\big)-\pi^{G}(y)\big|\leq\varepsilon\big\}
\\
&=
\min\big\{n\in\mathbb{N}_{0}:\, \max_{x\in V,\,A\subseteq V}\big|\p_x\big(W(n)\in A\big)-\pi^{G}(A)\big|\leq\varepsilon\big\}.
\end{aligned}
\end{equation}	

Note that if $G=(V,E)$ is a finite simple, connected, regular graph such that $\#V\ge 2$, then for all $\varepsilon\in(0,1)$,
\begin{equation}
\label{e:024}
t_{{\rm mix},\ell_1}^{G}\big(W; \frac{\varepsilon}{2} \big)	\leq t_{{\rm mix},\ell_\infty}^{G}\big(W;\varepsilon\big) \leq t_{{\rm mix},\ell_1}^{G}\big(W;\frac{\varepsilon}{ \# V}\big)
\end{equation} 
 We are thus in a position to apply  
\cite[(4.32)]{Levin2017} together (\ref{Exteq2}), to conclude that for all $q \in \mathbb{N}$ and $n \geq q t_{{\mbox{\tiny mix}}}^{G}$,
\begin{equation}\begin{aligned} \label{eq1supb}
   \max_{x\in V}\sum_{y\in V}\big|\p_x\big(W(n)=y\big)-\pi^{G}(y)\big|
 &\leq 
   2 \cdot 4^{-q}.
\end{aligned}\end{equation}

\begin{lemma}[Distance to independence]\label{strongestima2}  Let $G=(V,E)$ be a finite simple, connected, regular graph, $k\in\mathbb{N}$ and 
$W$, $W_1$, ..., $W_k$ independent lazy  random walks on $G$. If $s,q\in\mathbb{N}$ with 
$s \geq qt_{{\mbox{\tiny mix}}}^{G}+1$, then for all finite non-empty intervals $A_1,...,A_k\subset\mathbb{N}$ such that $\min(A_1)\ge s$, $\min(A_{i+1})-\max(A_i)\ge s$ for all $i\in\{0,\ldots, k-1\}$, 
all functions $F:\,V^{\sum_{i=1}^k\# A_i} \to \mathbb{R}$, and all $\varrho\in V$,
\begin{equation}
\begin{aligned} 
	\label{eq3supb2}
	&\Big |\mathbb{E}_\varrho\big[F\big(W(A_1), \dots, W(A_k)\big)\big] 
	- 
	\mathbb{E}_{\pi^G\otimes \cdots\otimes\pi^G}\big[F\big(W_{1}(A_1),\dots, W_{k}(A_k)\big)\big]\Big| 
	\leq 
	2k 4^{-q}\|F\|_{\infty}.
\end{aligned}
\end{equation}
\end{lemma}

\begin{proof} 
Note that by the triangle inequality it is enough to show that for all $k\in\mathbb{N}$ and $\varrho \in V$, we have
\begin{equation}
\begin{aligned} 
	\label{eq5supb0}
	&\sum_{\gamma_1 \in V^{\# A_1},\ldots, \gamma_k \in V^{\# A_k}}\Big|\mathbb{P}_\varrho\big(W(A_1)=\gamma_1,...,W(A_{k})=\gamma_{k}\big) 
	-\prod_{j=1}^{k}\mathbb{P}_{\pi^G}\big(W(A_j)=\gamma_j\big)\Big|\le 2k 4^{-q}.
\end{aligned}
\end{equation}
Indeed, for $k=1$, we see that 
\begin{equation}
\begin{aligned} 
	& \Big |\mathbb{E}_\varrho\big[F\big(W(A_1)\big)\big] 
	- 
	\mathbb{E}_{\pi^G}\big[F\big(W_{1}(A_1)\big)\big]\Big| \\
	&=\Big|\sum_{\gamma_1 \in V^{\# A_i}}F(\gamma_1)\mathbb{P}_\varrho\big(W(A_1)=\gamma_1\big) 
	-\sum_{\gamma_1 \in V^{\# A_i}}F(\gamma_1)\mathbb{P}_{\pi^G}\big(W(A_1)=\gamma_1\big)\Big|\\
	& \leq \|F\|_{\infty}\sum_{\gamma_1 \in V^{\# A_i}}\Big|\mathbb{P}_\varrho\big(W(A_1)=\gamma_1\big) 
	-\mathbb{P}_{\pi^G}\big(W(A_1)=\gamma_1\big)\Big|,
\end{aligned}
\end{equation}and a similar argument for general $k$ shows that \eqref{eq5supb0} is sufficient.

We shall proceed by induction over $k\in\mathbb{N}$.
Let $k=1$, and consider $A_1\subset\mathbb{N}$ a non-empty interval with $\min(A_1)\ge s$. 
To prove \eqref{eq5supb0}, we sum over all $\gamma:=(\gamma(1),\ldots, \gamma(\# A_1))\in V^{\# A_1}$ and apply the Markov property at time $\min A_1$. 
Recall the notation,  $A_1-b:=\{a-b:a\in A_1\}$, for $b \in \mathbb{N}_{0}$, , and set 
$\widetilde \gamma = (\gamma(2),\ldots, \gamma(\# A_1)) \in V^{\# A_1 -1}$.
Thus, for all $\varrho\in V$ we have that
\begin{equation}
\begin{aligned} 
	\label{eq5supb2}
	&
	\sum_{\gamma \in V^{\# A_1}}\Big|\mathbb{P}_\varrho\big(W(A_1)=\gamma\big) 
	-\mathbb{P}_{\pi^G}\big(W(A_1)=\gamma\big)\Big|
	\\
	&=  
	\sum_{\gamma\in V^{\# A_1}}\big|\mathbb{P}_\varrho\big(W(\min A_1)=\gamma(1))\p_{\gamma(1)}\big(W(A_1-\min A_1)=\gamma\big) 
   \\
	&
	\qquad \qquad \qquad \qquad 
	- \frac{1}{\#V}\mathbb{P}_{\gamma(1)}\big(W(A_1-\min A_1) = \gamma\big)\big| 
	\\
	&\le
	\sum_{\gamma(1)\in V}
	\big|\mathbb{P}_\varrho\big(W(\min A_1)=\gamma(1))- \frac{1}{\#V}\big| \sum_{\widetilde \gamma \in V^{\# A_1-1}}\p_{\gamma(1)}\big(W(A_1\setminus\{\min A_1\}-\min A_1)=\widetilde \gamma\big)
	\\ 
	&=
	\sum_{\gamma(1)\in V}\Big |\p_{\varrho}\big(W(\min A_1)=\gamma(1)\big)-\frac{1}{\# V} \Big| 
	\\
	&\leq  2\cdot 4^{-q},
\end{aligned}
\end{equation}
\noindent 
where we have used in the second line
that used that $\pi^{G}$ is the stationary distribution, and \eqref{eq1supb} in the last line.
This implies \eqref{eq5supb0} for $k=1$. 

Assume next that the statement \eqref{eq5supb0} holds for all $k\in\{1,\ldots, m\}$ for some $m\in\mathbb{N}$. We shall show that it also holds with $k=m+1$. Fix $\varrho\in V$.
Then conditioning 
on the history up to time $\max A_{m}$, noting that $\min A_{m+1}-\max A_{m}\geq s$ (which implies no overlap between the sets), adding a zero and using the triangle inequality, we have
\begin{equation}
\begin{aligned} 
	\label{eq5supb2m}
	&
	\sum_{\gamma_1 \in V^{\# A_1},\ldots, \gamma_{m+1} \in V^{\# A_{m+1}}}\Big|\mathbb{P}_\varrho\big(W(A_1)=\gamma_1,...,W(A_{m+1})=\gamma_{m+1}\big) 
	-\prod_{j=1}^{m+1}\mathbb{P}_{\pi^G}\big(W(A_j)=\gamma_j\big)\Big|
	\\
	&\leq 
	\sum_{\gamma_1 \in V^{\# A_1},\ldots, \gamma_{m+1} \in V^{\# A_{m+1}}}\Big\{ \Big|\mathbb{P}_\varrho\big(W(A_{1})=\gamma_{1},...,W(A_{m})=\gamma_{m}\big)
	-\prod_{j=1}^{m}\mathbb{P}_{\pi^G}\big(W(A_j)=\gamma_j\big)\Big|
	\\
	&\hspace{6cm}
	\p_{W(\max A_{m})}\big(W(A_{m+1}-\max A_{m})=\gamma_{m+1}\big)
	\\
	&\hspace{.1cm}+
	\Big|\p_{W(\max A_{m})}\big(W(A_{m+1}-\max A_{m})=\gamma_{m+1}\big)-\p_{\pi^G}\big(W(A_{m+1})=\gamma_{m+1}\big)
	\Big|\prod_{j=1}^{m}\mathbb{P}_{\pi^G}\big(W(A_j)=\gamma_j\big)\Big\}
	\\
	&\le
	2(m+1)4^{-q}.
\end{aligned}
\end{equation}Here we used in the last line the induction hypothesis with $k=m$ and $k=1$. 
As the right hand side in the second last line does not depend on $\varrho$ the claim \eqref{eq5supb0}, and thus \eqref{eq3supb2} follows. 
\end{proof}

\subsection{Estimates on intersections probabilities of random walks}
\label{Sub:rangeintersect}
In this subsection, we bound probabilities on self-intersection of random walks. For a path $\gamma:\,\mathbb{N}_0\to V$ and $A\subseteq\mathbb{N}_0$,
we write
\begin{equation}
\label{e:011}
   {\rm R}^\gamma(A):=\big\{\gamma(n):\,n\in A\big\}\subseteq V.
\end{equation}
for the {\em range} of $\gamma(A)$, i.e., the set of distinct vertices visited by $\gamma$ during times in $A$. By convention we define ${\rm R}^\gamma(\emptyset) = \emptyset$. 

\begin{proposition}[Range self-intersection probabilities]
Let $W$ be he lazy random walk on  a finite simple, connected, regular graph $G = (V,E)$. If $s\in \mathbb{N}_{0}$ with 
$s\geq t_{{\mbox{\tiny mix}}}^{G}+1$, and if  $A_1,A_2,A_3\subset\mathbb{N}$ are finite intervals
with $\min(A_1)\ge s$ and $\min(A_{i+1})-\max(A_i)\ge s$, $i=1,2$, 
then for all $\varrho\in V$,
\begin{equation} 
\label{Exteq12}
\mathbb{P}_\varrho\big( {\rm R}^W(A_{1}) \cap {\rm R}^W(A_{2}) \neq \emptyset \big)  \leq \frac{4\# A_1\# A_2}{\# V},
\end{equation}
and 
\begin{equation} \label{Exteq13}
\mathbb{P}_\varrho\big( {\rm R}^W(A_{1}) \cap {\rm R}^W(A_{2}) \neq \emptyset, {\rm R}^W(A_{1}) \cap {\rm R}^W(A_{3}) \neq \emptyset \big)  \leq \frac{8(\# A_1)^{2}\# A_2\# A_3}{(\# V)^{2}}.
\end{equation}
\label{P:004}
\end{proposition}

\begin{proof}Let $W_{1}, W_{2}, W_{3}$ be i.i.d.\ lazy random walks on $G$ and independent of $W$. Recall that since we assume that $G$ is a regular graph, for all $i =1,2,3$ and $n \in \mathbb{N}_{0}$,
\begin{equation}\begin{aligned} \label{Exteq14}
	\mathbb{P}_{\pi^G}\big(W_{i}(n) = x\big) =  \pi^{G}(x)= \frac{1}{\# V}, \quad x \in V.
\end{aligned}\end{equation} 

We start with the proof of (\ref{Exteq12}). It follows from (\ref{Exteq4}) in Lemma~\ref{lemma1NewE} that
\begin{equation}\begin{aligned} \label{Exteq15}
	\mathbb{P}_\varrho\big( {\rm R}^W(A_{1}) \cap {\rm R}^W(A_{2}) \neq \emptyset \big) 
	& \leq 
	4 \mathbb{P}_{\pi^G\otimes\pi^G}\big({\rm R}^{W_{1}}(A_1) \cap {\rm R}^{W_{2}}(A_2)\neq\emptyset\big) 
	\\
	& \leq 
	4 \sum_{n\in A_1,m\in A_2}\mathbb{P}_{\pi^G\otimes\pi^G}\big( W_{1}(n) = W_{2}(m)\big).
\end{aligned}\end{equation} 

By independence  and (\ref{Exteq14}), we have that for all $m\in A_1,n\in A_2$,
\begin{equation}\begin{aligned} \label{Exteq16}
	\mathbb{P}_{\pi^G\otimes\pi^G }\big( W_{1}(n) = W_{2}(m)\big) = \sum_{x \in V} \mathbb{P}_{\pi^G}\big( W_{1}(n) = x\big) \mathbb{P}_{\pi^G }\big(W_{1}(m) = x\big) = \frac{1}{\# V}.
\end{aligned}\end{equation} 
Thus (\ref{Exteq12}) follows from (\ref{Exteq15}).

Next we prove (\ref{Exteq13}). It follows from (\ref{Exteq4}) in Lemma~\ref{lemma1NewE} that 
\begin{equation}\begin{aligned} \label{Exteq17}
	&\mathbb{P}_\varrho\big({\rm R}^W(A_{1}) \cap {\rm R}^W(A_{2}) \neq \emptyset, {\rm R}^W(A_{1}) \cap {\rm R}^W(A_{3}) \neq \emptyset \big) 
	\\
	&\quad \leq 
	8 \mathbb{P}_{\pi^G\otimes \pi^G\otimes\pi^G}\big({\rm R}^{W_{1}}(A_1) \cap {\rm R}^{W_{2}}(A_2) \neq \emptyset, {\rm R}^{W_{1}}(A_1) \cap {\rm R}^{W_{3}}(A_3) \neq \emptyset \big) 
	\\
	&\quad \leq 
	8 \sum_{n,u\in A_1,m\in A_2,v\in A_3}\mathbb{P}_{\pi^G\otimes \pi^G\otimes\pi^G}\big( W_{1}(n) = W_{2}(m), W_{1}(u) = W_{3}(v) \big).
\end{aligned}\end{equation} 

On the other hand, by independence and (\ref{Exteq14}),
\begin{equation}\begin{aligned} \label{Exteq18}
	&\mathbb{P}_{\pi^G\otimes \pi^G\otimes\pi^G}\big( W_{1}(n) = W_{2}(m), W_{1}(u) = W_{3}(v) \big) 
	\\
	&= 
	\sum_{x_{1},x_{2} \in V}\mathbb{P}_{\pi^G}\big( W_{2}(m) = x_{1}\big) \mathbb{P}_{\pi^G}\big( W_{3}(v) =x_{2} \big) \mathbb{P}_{\pi^G}\big( W_{1}(n) = x_{1},  W_{1}(u) = x_{2} \big) 
	= 
	\frac{1}{(\# V)^{2}}.
\end{aligned}\end{equation} 
Thus,  (\ref{Exteq13}) follows from (\ref{Exteq17}).
\end{proof}

We can prove an analogous formula to \eqref{Exteq12} for which in Proposition~\ref{P:004} 
the assumption  $\min (A_1)\geq s$ is dropped.

\begin{corollary}[Range self-intersection probabilities starting at zero]\label{coroRangeSelfIntersectionProbabilitiesStartingAtZero}
Let $W$ be the lazy random walk on  a finite simple, connected, regular graph $G = (V,E)$. If $s\in \mathbb{N}_{0}$ with 
$s\geq t_{{\mbox{\tiny mix}}}^{G}+1$, and if  $A_1,A_2\subset\mathbb{N}$ are finite intervals
with $\min(A_1)=0$ and $\min(A_{2})-\max(A_1)\ge s$, 
then for all $\varrho\in V$ we have
\begin{equation} 
\label{eqnRangeSelfIntersectionProbabilitiesStartingAtZero1}
\mathbb{P}_\varrho\big( {\rm R}^W(A_{1}) \cap {\rm R}^W(A_{2}) \neq \emptyset \big)  \leq \frac{2\# A_1\# A_2}{\# V}.
\end{equation}
\end{corollary}
\begin{proof}The proof follows the same lines as the proof of \eqref{Exteq12} in Proposition \ref{P:004}. 
Indeed, let $W_1$ and $W_2$ be two independent lazy random walks on $G$. Then, by a proof analogous to that of  \eqref{Exteq4} in Lemma \ref{lemma1NewE}, we have that
\begin{align} \label{eqnRangeSelfIntersectionProbabilitiesStartingAtZero2}
& \p_{\rho}\big(R^W(A_1)\cap  R^W(A_2) \not=\emptyset \big) \leq  2 \p_{(\rho,\pi^G)}\big(R^{W_1}(A_1)\cap  R^{W_2}(A_2)\not=\emptyset
\big).
\end{align}
\noindent Note that \eqref{eqnRangeSelfIntersectionProbabilitiesStartingAtZero2} and an argument analogous to the proof of  \eqref{Exteq12} in Proposition \ref{P:004} (using also \eqref{Exteq14}) imply that
\begin{align} \label{eqnRangeSelfIntersectionProbabilitiesStartingAtZero3}
& \p_{\rho}\big(R^W(A_1)\cap  R^W(A_2) \not=\emptyset \big) \nonumber \\
& \quad \quad  \quad \leq  2\sum_{n\in A_1} \sum_{m\in A_2} \sum_{x\in V}\p_{\rho}\big(W_1(n)=x\big)\p_{\pi^G}\big(W_2(m)=x\big)  = \frac{2\# A_1\# A_2}{\# V}.
\end{align} 
\end{proof}

For a given path $\gamma:\,\mathbb{N}_{0}\to V$ on a finite simple, connected, regular graph $G =(V, E)$, and $m,n\in\mathbb{N}$ such that $m<n$ and $\gamma(n)=\gamma(m)$, we refer to $\gamma([m,n])$ as a loop of {\em length} $n-m$. The next lemma estimates the probability that up to a given time there are no loops of an intermediate length. 

\begin{lemma}[Separation of $s'$-short and $r$-long loops] Fix $s',r\in\mathbb{N}$ with $r\ge s'+1$ and $s' \geq t_{{\mbox{\tiny mix}}}^{G}+1$, and 
let $W$ be the lazy random walk on  $G =(V,E)$. Then, for all $N \in \mathbb{N}_{0}$ and $\rho\in V$,
\begin{equation}
\begin{aligned} 
	\label{e:113AB}
	\mathbb{P}_\varrho\Big(W(n)\notin{\rm R}^W\big([n+s',n+r]\big),\forall \, n\in[0,N]\Big) 
	\geq 
	1- \frac{2 (r-s'+1)(N+1)}{\# V}.
\end{aligned}
\end{equation} 
\label{lemmaLongLoop}
\end{lemma}

\begin{proof}
For all $\varrho\in V$, $n,k\in\mathbb{N}$ with $k\ge s'$,  by the Markov property and Lemma~\ref{lemma1NewE}, 
\begin{equation}\begin{aligned} \label{e:112}
	\mathbb{P}_\varrho\big(W(n)=W(n+k)\big) 
	&= 
	\sum_{x \in V} \p_x\big(W(k)=x\big) \p_\varrho\big(W(n)=x\big)
	\\
	&\le
	2\sum_{x \in V} \p_{\pi^G}\big(W(k)=x\big) \p_\varrho\big(W(n)=x\big)
	\\
	&=\frac{2}{\# V}\sum_{x \in V}\p_\varrho\big(W(n)=x\big)\\
	&= \frac{2}{\# V}.
\end{aligned}\end{equation} 

Hence by the union bound,
\begin{equation}\begin{aligned} \label{e:107}
	\mathbb{P}_\varrho\Big(\bigcup_{n\in[0,N]}\big\{W(n)\in{\rm R}^W\big([n+s',n+r]\big)\big\}\Big) 
	&\le
	\sum_{n=0}^{N}\sum_{k=s'}^{r}\mathbb{P}_\varrho\big(W(n)=W(n+k)\big) 
	\\
	&\leq 
	\frac{2(r-s'+1) (N+1)}{\# V},
\end{aligned}\end{equation} 
which yields the claim. 
\end{proof}

We next define \emph{local cut points} that will later be useful in the approximation of the AB chain by the so-called skeleton chain. Note that our definition is slightly different from that in \cite[p.\ 13]{PeresRevelle} or \cite[p.\ 337]{Schweinsberg2009}. These points will be crucial to approximate the Aldous-Broder chain with a simpler chain (see Proposition \ref{DerterLemma1}).

\begin{definition}[$s'$-local cut point] Fix $s^{\prime} \in \mathbb{N}$. 
Let $\gamma:\,\mathbb{N}_0\to V$ be a path on a finite simple, connected graph $G=(V,E)$. We say that $\ell \geq s^{\prime}$ is an $s^{\prime}$-local cut point of the path $\gamma$ if 
\begin{equation}
\begin{aligned}
	\label{e:cut}
	{\rm R}^\gamma\Big(\big[\ell-s^{\prime},\ell\big]\Big) \cap {\rm R}^\gamma\Big(\big[\ell+1,\ell+s^{\prime}\big]\Big) = \emptyset.
\end{aligned}
\end{equation}
\label{LocaCutP} 
\end{definition}

Note that in accordance with the terminology of \cite{PeresRevelle} and \cite{Schweinsberg2009}, we used the name $s'$-local cut points. Nevertheless, a more appropriate name is $s'$-separation points, since from their Definition \ref{LocaCutP}, they separate a part of the path into two segments that do not intersect.

For all $s' \in \mathbb{N}$ and a finite interval $A\subset \na_0$ of length at least $2s'+1$, set
\begin{equation}
\begin{aligned}
\label{e:cutset}
   {\rm CP}^{\gamma,s'}(A):=\big\{\ell\in [\min A+s^{\prime},\max A-s^{\prime}]:\,\mbox{$\ell$ is an $s'$-local cut point} \big\}\subseteq A.
\end{aligned}
\end{equation}
Next we define the probability that segments of length $s'\in\mathbb{N}$ of two lazy random walks,  $W_{1}$ and $W_{2}$, on $G =(V,E)$ that start both in the same uniformly distributed vertex and are conditionally independent do not intersect. That is,
\begin{equation}
\begin{aligned}  
\label{LPeq8}
\overline{q}^{G}\big(s^{\prime}\big) := \frac{1}{\# V}\sum_{\varrho\in V}\mathbb{P}_{(\varrho,\varrho)}\Big({\rm R}^{W_{1}}\big([0,s']\big)  \cap {\rm R}^{W_{2}}\big([1,s']\big) \not= \emptyset\Big).
\end{aligned}
\end{equation} 

The next result bounds the probability that up to a given time $N$ all times on $[0,N]$ are \emph{at most at distance $s$} from a  $2s^{\prime}$-local cut point.
It corresponds to \cite[Corollary 4.1]{PeresRevelle}; see also \cite[Proposition 2.13]{Schweinsberg2009}.
\begin{lemma}[Existence of local cut points]
Let $W$ be the lazy random walk on a finite simple, connected, regular graph $G = (V,E)$.  Fix $s, s^{\prime}, N, q \in \mathbb{N}$ such that $N \geq s$, 
$s \geq 6s^{\prime}$ and $2s^{\prime} \geq qt_{{\mbox{\tiny mix}}}^{G}+1$.
Then for all $\varrho\in V$,
\begin{equation}
\begin{aligned} 
	\label{LPeq6}
	\mathbb{P}_\varrho\big({\rm CP}^{W,2s'}\big([m,m+s])\neq \emptyset,\ \forall\ m\in [0,N-s] \big) 
	\geq 
	1- \big(N-s+1\big) \Big( \big(\overline{q}^{G}(2s^{\prime}) \big)^{\lfloor\frac{s}{6s'}\rfloor} + \frac{s}{3 \cdot  4^{q} s'} \Big).
\end{aligned} 
\end{equation} 
\label{Lemmacut point} 
\end{lemma}
We prepare the proof of Lemma~\ref{Lemmacut point} by showing the following identity. 

\begin{lemma}[Non-intersection probabilities identity]\label{lemmaKIsALocalCut point} Let $W$ be lazy random walk on a finite simple, connected, regular graph $G=(V,E)$. Then for all $\ell\ge s'\ge 1$,
\begin{equation}
\begin{aligned}
	\label{e:025}
	\overline{q}^{G}\big(s'\big)=
    {\rm R}^{W}([\ell-s,\ell])\cap {\rm R}^W([\ell+1,\ell+s])\not=\emptyset\big).
\end{aligned}
\end{equation}
\end{lemma}

\begin{proof}
Set $\gamma \coloneqq (\gamma(0),\ldots, \gamma(s))\in V^{s'+1}$. By the Markov property, conditioning on $W([\ell-s',\ell])$, 
\begin{equation}
\begin{aligned}  
\label{e:026}
	&\mathbb{P}_{\pi^G}\big({\rm R}^{W}([\ell-s',\ell])\cap {\rm R}^W([\ell+1,\ell+s'])\not=\emptyset\big) 
  \\
	&= \frac{1}{\#V}\sum_{\varrho\in V}\mathbb{P}_{\varrho}\big({\rm R}^{W}([\ell-s',\ell])\cap {\rm R}^W([\ell+1,\ell+s'])\not=\emptyset\big) 
\\
	&= 
   \frac{1}{\#V}\sum_{\varrho\in V}\sum_{\gamma\in V^{s'+1}}\mathbb{P}_{\varrho}\big(W([\ell-s',\ell])=\gamma\big)
	\mathbb{P}_{\gamma(s')}\big({\rm R}^{\gamma}([0,s'])\cap  {\rm R}^{W_2}([1,s'])\not=\emptyset\big),
\end{aligned}
\end{equation}
where $W_2$ is a lazy random walk on $G$. 
By time reversibility (see e.g., \cite[(1.30)]{Levin2017}), for all $\gamma\in V^{s'+1}$,
\begin{equation}
\label{e:Levine}
   \pi^{G}(\rho) \mathbb{P}_{\varrho}\big(W([\ell-s',\ell])=\gamma\big) = \pi^{G}(\gamma(s')) \mathbb{P}_{\gamma(s)}\big(W([0,s'])=
   \gamma_{\mbox{\tiny rev}},W(\ell) = \rho\big),
\end{equation}
\noindent where $\gamma_{\mbox{\tiny rev}}:=(\gamma(s),\dots,\gamma(0))$. Since $\pi^{G}(v) = (\# V)^{-1}$, for $v \in V$, we have that
\begin{equation}
\begin{aligned} 
\label{f:011} 
	&\mathbb{P}_{\pi^G}\big({\rm R}^{W}([\ell-s',\ell])\cap {\rm R}^W([\ell+1,\ell+s'])\not=\emptyset\big) 
 \\
	&= 
     \frac{1}{\#V}  \sum_{\gamma_{\mbox{\tiny rev}} \in V^{s'+1}} \sum_{\varrho\in V} \mathbb{P}_{\gamma(s)}\big(W([0,s'])=\gamma_{\mbox{\tiny rev}}, W(\ell) = \rho\big) \mathbb{P}_{\gamma(s')}\big({\rm R}^{\gamma}([0,s'])\cap  {\rm R}^{W_2}([1,s'])\not=\emptyset\big) 
 \\
	&= 
        \frac{1}{\#V}  \sum_{\gamma\in V^{s'+1}}  \mathbb{P}_{\gamma(s')}\big(W([0,s'])=\gamma_{\mbox{\tiny rev}}\big) \mathbb{P}_{\gamma(s')}\big({\rm R}^{\gamma_{\mbox{\tiny rev}}}([0,s'])\cap  {\rm R}^{W_2}([1,s'])\not=\emptyset\big) 
  \\
	&= 
   \frac{1}{\#V} \sum_{\gamma(s')\in V}  \mathbb{P}_{(\gamma(s'), \gamma(s'))}\big({\rm R}^{W_{1}}([0,s'])\cap  {\rm R}^{W_{2}}([1,s'])\not=\emptyset\big),
\end{aligned}
\end{equation}
\noindent where $W_{1}$ and $W_{2}$ are two conditionally independent lazy random walks on $G$ starting uniformly in the same vertex. 
\end{proof}

\begin{proof}[Proof of Lemma~\ref{Lemmacut point}] 
Let $s, s^{\prime}, N, q \in \mathbb{N}$ such that $N \geq s$ and $s \geq 6s^{\prime}$ and $2s^{\prime} \geq qt_{{\mbox{\tiny mix}}}^{G}+1$. Then by the union bound
\begin{equation}
\begin{aligned} 
	\label{e:056}
	\mathbb{P}_\varrho\big({\rm CP}^{W,2s'}([m,m+s])\neq \emptyset,\ \forall\ m\in [0,N-s] \big)  &\geq 1- \sum_{m=0}^{N-s}\mathbb{P}_\varrho\big({\rm CP}^{W,2s'}([m,m+s])= \emptyset\big).
\end{aligned}
\end{equation} 
\noindent Moreover, by Definition \ref{LocaCutP}, we have that, for all $m\in [0,N-s]$, 
\begin{equation}
\begin{aligned} \label{e:067}
	&\mathbb{P}_\varrho\big({\rm CP}^{W,2s'}([m,m+s])= \emptyset\big) \\
	& \quad = \mathbb{P}_\varrho\Big({\rm R}^W\big(\big[\ell-2s^{\prime}, \ell \big]\big) \cap {\rm R}^W\big(\big[\ell+1,\ell+2s^{\prime}\big]\big)\not= \emptyset,\,\forall\,\ell\in [m+2s',m+s-2s']\Big) \\
	& \quad   \leq \mathbb{P}_\varrho\Big({\rm R}^W\big(\big[\ell_k-2s^{\prime}, \ell_k \big]\big) \cap {\rm R}^W\big(\big[\ell_k+1,\ell_k+2s^{\prime}\big]\big)\not= \emptyset,\,\forall\, k \in \big \{1,\ldots, \lfloor s/6s' \rfloor \big\}\Big),
\end{aligned}
\end{equation} 
\noindent where $\ell_k:=m+(6k-2)s'$, for $k\in\{1,\ldots, \lfloor \frac{s}{6s'}\rfloor\}$. Next, we use \eqref{e:067}, apply \eqref{eq3supb2} in Lemma~\ref{strongestima2} with $A_k:=[\ell_k-2s',\ell_k+2s']$, for $k\in\{1,\ldots, \lfloor \frac{s}{6s'}\rfloor\}$, and Lemma \ref{lemmaKIsALocalCut point} to obtain that 		
\begin{equation}
\begin{aligned}  \label{e:067III}
	&\mathbb{P}_\varrho\big({\rm CP}^{W,2s'}([m,m+s])= \emptyset\big) \\
	&\quad \le \prod_{\ell=1}^{\lfloor \frac{s}{6s'}\rfloor}\mathbb{P}_{\pi^G}\Big({\rm R}^W\big(\big[\ell_k-2s^{\prime}, \ell_k \big]\big) \cap {\rm R}^W\big(\big[ \ell_k+1,  \ell_k+2s^{\prime}\big]\big)\not= \emptyset\Big) + 2 \lfloor \frac{s}{6s'}\rfloor 4^{-q} \\
	&\quad \leq \Big(\overline{q}^G\big(2s^{\prime}\big)\Big)^{\lfloor \frac{s}{6s'} \rfloor} + \frac{s}{3 \cdot  4^{q} s'}.
\end{aligned}
\end{equation}
\noindent Note that we used that $\min A_{1} =  m+2s^{\prime} \geq qt_{{\mbox{\tiny mix}}}^{G}+1$ and for $k\in\{1,\ldots, \lfloor \frac{s}{6s'}\rfloor-1\}$, $\min A_{k+1}-\max A_{k}=2s'\geq qt_{{\mbox{\tiny mix}}}^{G}+1$ and that $[m,m+s]$ has length at least $6s'+1$ by hypothesis.

Therefore, the combination of \eqref{e:056}, \eqref{e:067} and \eqref{e:067III} imply our claim.
\end{proof}

We will close this subsection with an upper bound on $\bar{q}^G(s)$.
For that we introduce the following quantity. 
Let $W$ be the lazy random walk on a finite simple, connected, regular graph $G=(V,E)$. For $m\in\mathbb{N}_{0}$ and $\rho \in V$, define
\begin{equation}\begin{aligned} \label{Hfunct}
H^W_\varrho(m,n) \coloneqq  \sum_{i=0}^m\sum_{j=0}^{n} \mathbb{P}_{\varrho}\big(W(i+j)=\varrho\big)\qquad \mbox{and}\qquad H^W_\varrho(m) \coloneqq  \sum_{i=0}^m\sum_{j=0}^{m} \mathbb{P}_{\varrho}\big(W(i+j)=\varrho\big).
\end{aligned}\end{equation}
We also define
\begin{equation}
\begin{aligned}
\label{e:036}
H^G(m,n):=\sum_{\varrho\in V}\pi^G(\varrho)H_\varrho^{W}(m,n)\qquad \mbox{and}\qquad H^G(m):=\sum_{\varrho\in V}\pi^G(\varrho)H_\varrho^{W}(m). 
\end{aligned}
\end{equation}
\noindent Note that for any $\rho\in V$ and $m,n\in \mathbb{N}_{0}$ we have $H^W_\rho(m,n)\geq H^W_\rho(0,0)=1 >0$, and thus, $H^G(m,n) >0$. 
The next result provides upper bounds for the functions $H^W_\rho$ and $H^G$ which will be useful later in this work.

\begin{lemma}\label{LemmaUpperBoundsForH}
Let $W$ be the lazy random walk on a finite simple, connected, regular graph $G = (V,E)$. Fix $\rho \in V$ an $m,n\in \na$ such that $m\geq n\geq t_{\mbox{\tiny mix}}^{G}+1$. 
Then
\begin{align}\label{eqnUniformBoundsOfNonErasedVertices3III}
H^W_\varrho(m,n)\leq H^W_\rho(m)\leq \sup_{\varrho \in V} H^W_\rho(t_{\mbox{\tiny mix}}^{G})+ \frac{5}{2}\frac{m^2}{\# V}
\end{align}
\noindent and in particular,
\begin{align} \label{eqnUniformBoundsOfNonErasedVertices3IV}
H^G(m,n)\leq H^{G}(m) \leq \sup_{\varrho \in V} H^W_\rho(t_{\mbox{\tiny mix}}^{G})+ \frac{5}{2}\frac{m^2}{\# V}. 
\end{align}
\end{lemma}

\begin{proof}
Fix $\rho \in V$ and $m,n\in \na$ such that $m\geq n\geq t_{\mbox{\tiny mix}}^{G} +1$. 
Note that $H^W_\varrho(m,n)\leq H^W_\varrho(m)$ by \eqref{Hfunct}, and clearly \eqref{eqnUniformBoundsOfNonErasedVertices3IV} follows from \eqref{eqnUniformBoundsOfNonErasedVertices3III}. 
Thus, it is enough to upper bound $H^W_\varrho(m)$.
By symmetry, \eqref{Exteq2} and \eqref{Exteq14}, we have that
\begin{equation}\label{eqnUniformBoundsOfNonErasedVertices3}
\begin{split}
H^W_\rho(m) &\leq H^W_\rho(t_{\mbox{\tiny mix}}^{G})+ 2\sum_{i=t_{\mbox{\tiny mix}}^{G}+1}^{m} \sum_{j=0}^{m} \p_\rho\big(W(i+j)=\rho\big)\\
&\leq H^W_\rho(t_{\mbox{\tiny mix}}^{G})+\frac{5}{2}\frac{m^2}{\# V}.
\end{split}
\end{equation}
\noindent (Note that $i+j\geq t_{\mbox{\tiny mix}}^{G}+1$, for $t_{\mbox{\tiny mix}}^{G}+1 \leq i \leq m$ and $0 \leq j \leq m$.) 
\end{proof}

\begin{lemma}\label{P:HGq}
Let $W$ be the lazy random walk on a finite simple, connected, regular graph $G = (V,E)$. Then,
for any $s \geq t_{\mbox{\tiny mix}}^{G}+1$,
\begin{equation}
\label{f:012}
   \overline{q}^{G}\big(s\big)
 \leq 
    1-\paren{1-\frac{10s^2}{\# V}}\frac{1}{\sup_{\rho \in V} H^W_\rho (s)}
 \leq 
    1-\paren{1-\frac{10s^2}{\# V}}\frac{1}{\sup_{\rho \in V} H^W_\rho (s)+\frac{5}{2}\frac{s^2}{\# V}}. 
\end{equation}
\end{lemma}

\begin{proof} We call a pair $(i,j)\in [0,s]\times [0,s]$ an \emph{$s$-last-intersection pair} if
\begin{equation} 
\label{NewLemmaSizeLceq1}
   W_1(i)=W_2(j) \quad \text{and} \quad {\rm R}^{W_1}([i,2s])\cap {\rm R}^{W_2}([j+1,2s]) = \emptyset,   
\end{equation}
and denote by ${\rm LIP}^{s}$ the set of $s$-last intersection pairs on $[0,s] \times [0,s]$. Define also the set
\begin{equation}
\label{NewLemmaSizeLceq1II}
	O^c_{s} 
  :=
    \big\{W_1(k)=W_2(\ell)\mbox{ for some }(k,\ell)\in [0,2s]\times [0,2s]\mbox{ with }\max(k,\ell)>s\big\}.
\end{equation}
By the union bound, symmetry and \eqref{Exteq2}, for all $\varrho\in V$,
\begin{equation} 
\label{NewLemmaSizeLceq2}
\begin{aligned}
   \mathbb{P}_{(\varrho, \varrho)}\big(O^c_{s}\big)  
 &\leq 
   2\sum_{\ell =s+1}^{2s}\sum_{k=0}^{\ell}\mathbb{P}_{(\varrho,\varrho)}\big(W_1(k)=W_2(\ell)\big)
   \\
   &=
   2\sum_{\ell =s+1}^{2s}\sum_{k=0}^{\ell}\sum_{x\in V}\mathbb{P}_{\varrho}\big(W_1(k)=x\big)\mathbb{P}_{\varrho}\big(W_2(\ell)=x\big) 
   \\
    &\leq 
   \frac{5}{2}\sum_{\ell =s+1}^{2s}\sum_{k=0}^{\ell}\sum_{x\in V}\mathbb{P}_{\varrho}( W_1(k)=x) \pi^{G}(x)  
   \\
 &= 
   \frac{10s^2}{\# V}.
   \end{aligned}
\end{equation}
We now claim that
\begin{equation} 
\label{NewLemmaSizeLceq4II2}
   \big\{W_{1}(0)=W_{2}(0)\big\}\cap\big\{\#{\rm LIP}^{s}\geq 1\big\}\setminus O^c_{s} =\big\{W_{1}(0)=W_{2}(0)\big\}\setminus O^c_{s}. 
\end{equation}
\noindent To see this, first choose the largest values $i \in [0,s]$ and $j \in [0,s]$ such that $W_{1}(i) = W_{2}(j)$ and ${\rm R}^{W_1}\big([i,s]\big)\cap {\rm R}^{W_1}\big([j+1,s]\big)=\emptyset$ (which exist in the event $\{ W_{1}(0) = W_{2}(0)\}\setminus O^c_s$). Then by \eqref{NewLemmaSizeLceq1II},
\begin{equation}
\label{NewLemmaSizeLceq20}
	{\rm R}^{W_1}([i,2s])\cap {\rm R}^{W_2}([j+1,2s])\cap \Big(\{ W_{1}(0) = W_{2}(0)\}\setminus O^c_{s}\Big)=\emptyset,
\end{equation}
which implies $(i,j)$ is an $s$-last intersection pair.		
Therefore, on the event $\{ W_{1}(0) = W_{2}(0)\} \setminus O^c_{s}$ we have  $\# {\rm LIP}^{s} \geq 1$ and \eqref{NewLemmaSizeLceq4II2} follows. Thus
\begin{equation}  
\label{NewLemmaSizeLceq5}
   \mathbb{E}_{(\varrho, \varrho)}( \# {\rm LIP}^{s} )  \geq \mathbb{P}_{(\varrho, \varrho)}(\# {\rm LIP}^{s} \geq 1) \geq \mathbb{P}_{(\varrho, \varrho)}(O_{s}) \geq  1- \frac{10s^2}{\# V}.   
\end{equation}

Note that, by the Markov property, for any $(i,j)\in [0,s]\times [0,s]$ and $\rho\in V$ we have
\begin{equation}  
\label{NewLemmaSizeLceq6}
\begin{aligned}
	&\p_{(\rho,\rho)}\big((i,j)\in {\rm LIP}^s \big) 
 \\
	& \quad \quad =
   \sum_{x\in V}\p_{(\rho,\rho)}\big(W_1(i)=W_2(j)=x \big)\p_{(x,x)}\big({\rm R}^{W_1}([0,2s-i])\cap {\rm R}^{W_2}([1,2s-j])=\emptyset \big)
   \\
	& \quad \quad \leq 
  \sum_{x\in V}\p_{(\rho,\rho)}\big(W_1(i)=W_2(j)=x \big)\p_{(x,x)}\big({\rm R}^{W_1}([0,s])\cap {\rm R}^{W_2}([1,s])=\emptyset \big),
\end{aligned}
\end{equation}
\noindent since $s\leq 2s-i$ and $s\leq 2s-j$. Therefore by \eqref{NewLemmaSizeLceq5},
\begin{equation}  
\label{NewLemmaSizeLceq7}
\begin{aligned}
	1- \frac{10s^2}{\# V}
	&\leq 
    \frac{1}{\# V} \sum_{\rho \in V}\mathbb{E}_{(\varrho, \varrho)}( \# {\rm LIP}^{s} ) 
  \\
	& \leq 
    \frac{1}{\# V} \sum_{\rho \in V} \sum_{i=0}^{s}\sum_{j=0}^{s} \p_{(\rho,\rho)}\big((i,j)\in {\rm LIP}^s \big)
  \\
	&\leq 
   \frac{1}{\# V}\sum_{x\in V}\p_{(x,x)}\big({\rm R}^{W_1}([0,s])\cap {\rm R}^{W_2}([1,s])=\emptyset \big) 
   \sum_{\rho \in V} \sum_{i=0}^{s}\sum_{j=0}^{s}  \p_{(\rho,\rho)}\big(W_1(i)=W_2(j)=x \big)
   \\
   &\le
   \bar{q}^G(s)\sup_{x\in V}\sum_{\varrho \in V} \sum_{i=0}^{s}\sum_{j=0}^{s}  \p_{(\rho,\rho)}\big(W_1(i)=W_2(j)=x \big).
\end{aligned}
\end{equation}
\noindent On the other hand, $\p_{(\rho,\rho)}(W_1(i)=W_2(j)=x)=\p_{\rho}(W_1(i)=x) \p_{\rho}(W_2(j)=x)$ by independence. Applying Lemma~\ref{LemmaUpperBoundsForH} together with (\ref{e:Levine}), we obtain for any $x\in V$, 
\begin{equation}  
\label{NewLemmaSizeLceq8}
\begin{aligned}
	\sum_{\varrho \in V} \sum_{i=0}^{s}\sum_{j=0}^{s} \p_{(\rho,\rho)}\big(W_1(i)=W_2(j)=x \big) 
 &= 
    \sum_{\rho \in V} \sum_{i=0}^{s}\sum_{j=0}^{s} \p_{x}\big(W_1(i)=\varrho \big) \p_{\varrho }\big(W_2(j)=x\big) 
    \\
 &= 
    \sum_{i=0}^{s}\sum_{j=0}^{s}\p_{x}\big(W_1(i+j)=x\big)
    \\
 &\le \sup_{\rho \in V} H^W_\rho (s)+\frac{5}{2}\frac{s^2}{\# V}.
\end{aligned}
\end{equation}
Finally, the combination of \eqref{NewLemmaSizeLceq7} and \eqref{NewLemmaSizeLceq8} implies the claim. 
\end{proof}

\subsection{Locally non-erased time indices on regular graphs}
\label{Sub:LERW}

In this subsection we give estimates on intersection probabilities of a random walk path with an earlier loop erased or locally loop erased path segment.
We will also state a concentration inequality for the length of (locally) loop erased path segments. As commonly known, given a path segment $\gamma([0,n])$, $n\in\mathbb{N}$, the ($s$-locally) loop erased path is obtained by erasing all loop (of length at most $s$) in chronological order (compare with Figure~\ref{figLEExample}).
Given a path $\gamma:\,\mathbb{N}_0\to V$ on a finite, simple, connected graph $G=(V,E)$, and a finite non-empty interval $A\subset \na_0$, we define a function ${\rm NE}^{\gamma,A}:A\to A$ as follows. Set ${\rm NE}^{\gamma,A}(\min A)=\{\min A\}$, and for $n\in [\min A+1,\max A]$, define recursively
\begin{equation}
\label{e:102}
{\rm NE}^{\gamma,A}(n):=\big\{m\in {\rm NE}^{\gamma,A}(n-1):\,\gamma(n)\notin {\rm R}^\gamma\big({\rm NE}^{\gamma,A}(n-1)\cap [\min A,m]\big)\big\}\cup \{n\}\subseteq A.
\end{equation}
\noindent The name ${\rm NE}^{\gamma,A}$ stands for \emph{non-erased} time indices by the loop-erased path on $A$, that is, ${\rm LE}(\gamma(A\cap [\min A,n])) = {\rm R}^\gamma({\rm NE}^{\gamma,A}(n))$ for $n\in A$.
Thus ${\rm NE}^{\gamma,A}(n)$ represents all the time indices that were not erased by the loop-erased path of $\gamma$ on $A$ up to time $n-1$, and that are not inside a loop created by the last step $\gamma(n)$ (if any). 
On the other hand, the erased times indices can be defined recursively by letting ${\rm E}^{\gamma,A}(\min A)=\emptyset$, and for $n\in [\min A+1,\max A]$,
\begin{equation}\label{eqnErasedIndexes}
{\rm E}^{\gamma,A}(n):={\rm E}^{\gamma,A}(n-1)\cup \big\{m\in {\rm NE}^{\gamma,A}(n-1):\,\gamma(n)\in {\rm R}^\gamma\big({\rm NE}^{\gamma,A}(n-1)\cap [\min A,m]\big)\big\},
\end{equation}i.e., ${\rm E}^{\gamma,A}(n)=A\setminus {\rm NE}^{\gamma,A}(n)$.

\begin{figure}
\includegraphics[width=12cm]{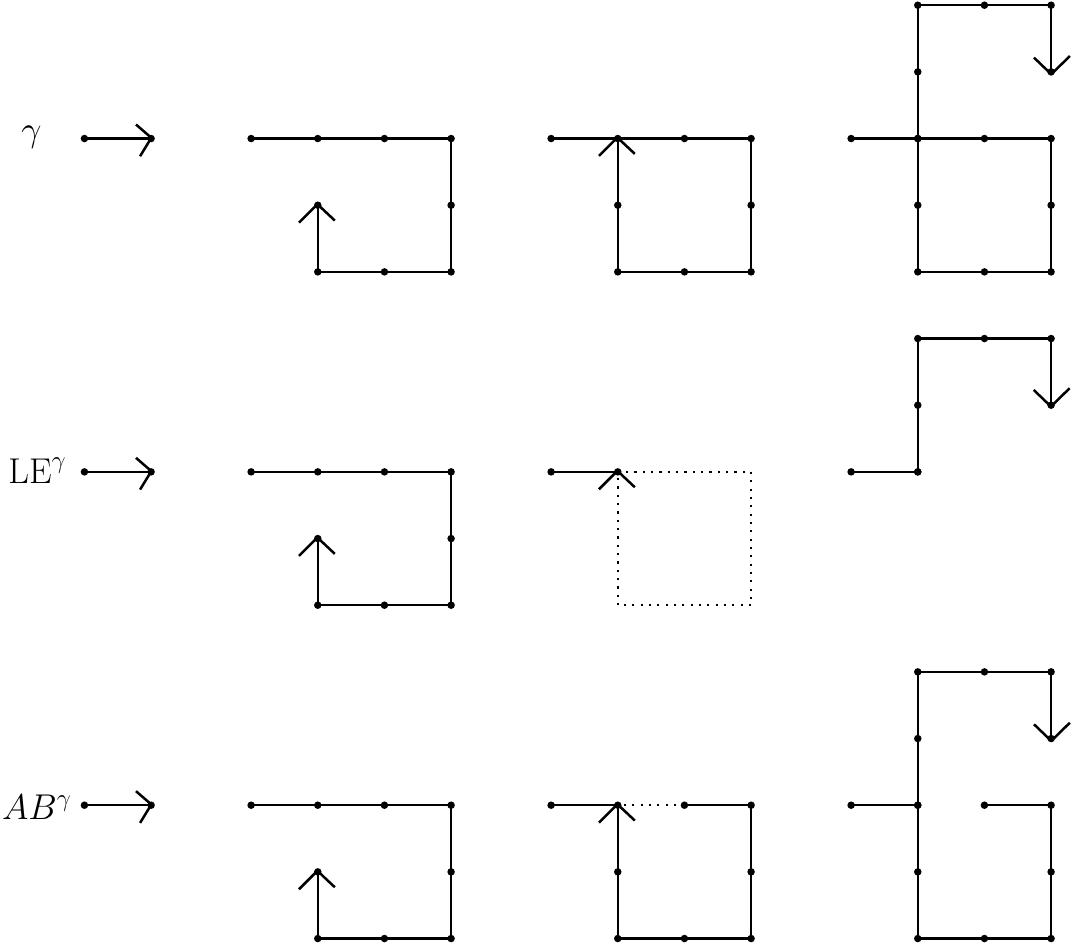}
\caption{A path segment $\gamma=(\gamma(0),\gamma(1),...,\gamma(n))$, $n=1,...,8,9,...,14$ together with its loop erasure. Here ${\rm NE}^\gamma(1)=[0,1]$, ${\rm NE}^\gamma(8)=[0,8]$, ${\rm NE}^\gamma(9)=\{0,9\}$ and ${\rm NE}^\gamma(14)=\{0\}\cup[9,14]$. }
\label{figLEExample}
\end{figure}

For a finite interval $A\subset\mathbb{N}_0$ of length $\# A\geq 2s+1$, with $s \in  \mathbb{N}_{0}$, and a path $\gamma:\na_0\to V$, put 

\begin{equation}
\begin{aligned} 
\label{e:023}
&{\rm NE}^{\gamma,s}(A)\\
& \quad \quad  := \Big\{m\in [\min A+s,\max A-s]:\, {\rm R}^\gamma\big({\rm NE}^{\gamma,[m-s,m]}(m)\big)\cap {\rm R}^\gamma\big([m+1,m+s]\big) = \emptyset\Big\}.
\end{aligned}
\end{equation}
\noindent We refer to indices in ${\rm NE}^{\gamma,s}(A)$ as {\em $(A,s)$-locally non-erased}, which are called \emph{locally retained} in  \cite[Definition~3]{PeresRevelle}. We also define the so-called {\em $(A,s)$-locally non-erased path} as
\begin{align} \label{e:023II}
{\rm R}^{\gamma}({\rm NE}^{\gamma, s}(A)),
\end{align}
\noindent compare with \cite[Definition~4]{PeresRevelle}.
\begin{remark}
Note that ${\rm NE}^{\gamma,s}(A)$ only depends on the values of the path segment $\gamma(A)$. 
\end{remark}

As a preparation we state the following:

\begin{lemma}[Intersection of independent walks]\label{lemmaProbabilityOfIntersectionOfTwoWalks}
Let $W_1$ and $W_2$ be two independent lazy random walks on a finite simple, connected, regular graph $G=(V,E)$. 
Then, for any $n,m\in \mathbb{N}_0$ and $\varrho\in V$, 
\begin{equation}
\begin{aligned}
\label{e:}
\p_{(\varrho,\varrho)}\big(W_1(n)=W_2(m)\big)=\p_{\varrho}\big(W_1(n+m)=\varrho \big).
\end{aligned}
\end{equation}
\end{lemma}

\begin{proof} 
On the one hand, by independence, we get that
\begin{equation}
\begin{aligned} 
\label{ExtIne7c}
\mathbb{P}_{(\varrho,\varrho)}\big(W_{1}(n)=W_{2}(m)\big) 
&= 
\sum_{x\in V}\mathbb{P}_\varrho\big(W_{1}(n)=x\big)\p_\varrho\big(W_{2}(m)=x\big).
\end{aligned}
\end{equation} 
On the other hand, recall that $\pi^G(x)\equiv (\# V)^{-1}$, for $x \in V$, is the stationary distribution. Then, by time reversibility (see e.g., \cite[(1.30)]{Levin2017}), $\p_\varrho\big(W_{2}(m)= x\big) = \p_x\big(W_{2}(m)= \varrho\big)$. Thus,
\begin{equation}
\begin{aligned} 
\mathbb{P}_{(\varrho,\varrho)}\big(W_{1}(n)=W_{2}(m)\big) 
&= 
\sum_{x\in V}\mathbb{P}_\varrho\big(W_{1}(n)=x)\p_{x}\big(W_{2}(m)=\varrho\big)
\\
&= 
\p_{\varrho}\big(W_1(n+m)=\varrho\big).
\end{aligned}
\end{equation} 
\end{proof}

Recall the definition of $(A,s)$-locally non-erased time indices in \eqref{e:023}. The next result bounds the probability that a segment of the lazy random walk intersects with a locally non-erased path segment of an independent lazy random walk.

\begin{lemma}[Intersection probability bounds]\label{corollaryExt3} 
Let $W_{1}, W_{2}$ be independent lazy random walks on a finite simple, connected, regular graph $G = (V,E)$. 
Then, for any finite intervals $A,B\subseteq \mathbb{N}_0$ such that $\# A\wedge \# B \geq 2s+1$ we have 
\begin{equation}
\begin{aligned} 
\label{ExtIne1}
\mathbb{P}_{\pi^G\otimes\pi^G}\big(\mathrm{R}^{W_{1}}(B)\cap{\rm R}^{W_2}({\rm NE}^{W_{2},s}(A))\neq \emptyset\big)  
\leq 
\frac{\# B}{\# V} \mathbb{E}_{\pi^G}\big[\# {\rm R}^{W_2}({\rm NE}^{W_{2},s}(A))\big].
\end{aligned}\end{equation} 
\noindent Furthermore,
\begin{equation}
\begin{aligned} 
\label{ExtIne1Low}
& \mathbb{P}_{\pi^G\otimes\pi^G}\big(\mathrm{R}^{W_{1}}(B) \cap {\rm R}^{W_2}({\rm NE}^{W_{2},s}(A))\neq \emptyset\big) \\
&\qquad \geq 
\frac{\# B}{4\# A\# V H^{G}(\max\{\# A-1,\# B-1\})}\Big(\mathbb{E}_{\pi^G}\big[\# {\rm R}^{W_2}({\rm NE}^{W_{2},s}(A))\big]\Big)^{2}.
\end{aligned}
\end{equation} 
\end{lemma}

\begin{proof}
Define
\begin{equation}
\begin{aligned} 
	\label{e:032}
	J=J(W_1,W_2)&:=\sum_{n\in A}\mathbf{1}_{{\rm R}^{W_2}({\rm NE}^{W_{2},s}(A))}\big(W_{1}(n)\big), 
\end{aligned}
\end{equation} 
and note  that 
\begin{equation}
\begin{aligned} 
	\label{ExtIne4a}
	\mathbb{P}_{\pi^G\otimes\pi^G}\big(\mathrm{R}^{W_{1}}(A) \cap {\rm R}^{W_2}({\rm NE}^{W_{2},s}(A)) \neq \emptyset \big) 
	&= 
	\mathbb{P}_{\pi^G\otimes\pi^G}\big(J\ge 1\big).
\end{aligned}
\end{equation}

The {\em upper bound \eqref{ExtIne1}} follows from independence, the fact that $\pi^{G}$ is the stationary distribution (see \eqref{Exteq14}) and the union bound. To be precise,
\begin{equation}
\begin{aligned} 
\label{ExtIne2}
\mathbb{P}_{\pi^G\otimes\pi^G}\big(J\ge 1\big)
& \leq 
\sum_{n\in B}\sum_{\gamma \in V^{\# A}}\mathbb{P}_{\pi^G}\big(W_{1}(n)\in{\rm R}^{\gamma}({\rm NE}^{\gamma ,s}(A))\big)
\mathbb{P}_{\pi^G}\big(W_{2}(A)=\gamma\big) 
\\
&  = 
\sum_{n\in B}\sum_{\gamma\in V^{\# A}}\frac{\# {\rm R}^{\gamma}({\rm NE}^{\gamma ,s}(A))}{\# V}\mathbb{P}_{\pi^G}\big(W_{2}(A)=\gamma\big) 
\\
&  = 
\frac{\# B}{\# V} \mathbb{E}_{\pi^G}\Big[\# {\rm R}^{W_2}\big({\rm NE}^{W_{2},s}(A)\big)\Big].
\end{aligned}\end{equation} 

We now prove the {\em lower bound (\ref{ExtIne1Low})} we use the second moment method, i.e.
\begin{equation}
\begin{aligned} 
\label{ExtIne4}
\mathbb{P}_{\pi^G\otimes\pi^G}\big(J \geq 1\big)\geq\frac{(\mathbb{E}_{\pi^G\otimes\pi^G}[J])^{2}}{\mathbb{E}_{\pi^G\otimes\pi^G}[J^{2}]}.
\end{aligned}
\end{equation} 

Note that by \eqref{Exteq14},
\begin{equation}
\begin{aligned} 
\label{ExtIne5}
\mathbb{E}_{\pi^G\otimes\pi^G}[J] 
&= 
\sum_{n\in B}\sum_{\gamma \in V^{\#A}} 
\mathbb{P}_{\pi^G\otimes\pi^G}\big( W_{1}(n) \in {\rm R}^{\gamma}({\rm NE}^{\gamma,s}(A))\big)\mathbb{P}_{\pi^G}\big(W_{2}(A)=\gamma\big) 
\\
&= 
\frac{\# B}{\# V} \mathbb{E}_{\pi^G}\Big[\# {\rm R}^{W_2}\big({\rm NE}^{W_{2},s}(A)\big)\Big].
\end{aligned}
\end{equation} 
\noindent On the one hand, by symmetry, we deduce that
\begin{equation}
\begin{aligned}
\label{ExtIne6}
\mathbb{E}_{\pi^G\otimes\pi^G}\big[J^{2}\big]  
&\leq 
2\sum_{n\in B}\sum_{m\in B,m\le n}\mathbb{P}_{\pi^G\otimes\pi^G}\big(W_{1}(n)\in{\rm R}^{W_{2}}(A),W_{1}(m)\in{\rm R}^{W_{2}}(A)\big)  
\\
&\leq 
4\sum_{n\in B}\sum_{m\in B,m\le n}\sum_{v\in A}\sum_{u\in A,u\le v}\mathbb{P}_{\pi^G\otimes\pi^G}\big(W_{1}(n) = W_{2}(v), W_{1}(m) = W_{2}(u)\big).
\end{aligned}
\end{equation} 
\noindent On the other hand, note that, for $m\leq n$ and $u\leq v$,
\begin{equation}
\begin{aligned} 
\label{ExtIne7bd}
& \mathbb{P}_{\pi^G\otimes\pi^G}\big( W_{1}(n) = W_{2}(v), W_{1}(m) = W_{2}(u)\big) \\
& = \sum_{\rho \in V}  \mathbb{P}_{\pi^G\otimes\pi^G}\big( W_{1}(n) = W_{2}(v)| W_{1}(m) = \rho, W_{2}(u) = \rho \big) \mathbb{P}_{\pi^G\otimes\pi^G}\big( W_{1}(m) = \rho, W_{2}(u) = \rho \big)\\
&= 
\frac{1}{(\# V)^2}\sum_{\varrho\in V}\mathbb{P}_{(\varrho,\varrho)}\big(W_{1}(n-m)=W_{2}(v-u)\big)
\\
&=
\frac{1}{(\#V)^2}\sum_{\varrho\in V}\p_\varrho\big(W(n-m+v-u)=\varrho\big),
\end{aligned}
\end{equation} 
\noindent where $W$ is the lazy random walk on $G$. Hence, \eqref{e:036} and \eqref{ExtIne6} and  \eqref{ExtIne7bd} imply that
\begin{equation}
\begin{aligned} 
\label{ExtIne7}
\mathbb{E}_{\pi^G\otimes\pi^G}[J^{2}]  
&\leq 
\frac{4}{(\#V)^2} \sum_{n\in B} \sum_{m\in B,m\le n} \sum_{v\in A}  \sum_{u\in A,u\le v}\sum_{\varrho\in V}\p_\varrho(W(n-m+v-u)=\varrho) 
\\
&= 
\frac{4\# A\# B}{(\#V)^2} \sum_{\varrho\in V} \sum_{i=0}^{\# B-1}  \sum_{j=0}^{\#A-1} \p_\varrho(W(i+j)=\varrho)
\\
& =\frac{4\# A\# B}{\#V} H^G\big(\# A-1,\# B-1\big).
\end{aligned}\end{equation} 
Thus, the combination of \eqref{eqnUniformBoundsOfNonErasedVertices3III}, \eqref{ExtIne4}, \eqref{ExtIne5} and  \eqref{ExtIne7} show \eqref{ExtIne1Low}.
\end{proof}

Recall the definition of the function $H^W_\rho$ in \eqref{Hfunct}. The following result bounds the expected number of $(A,s)$-locally non-erased indices by a fraction of the number of indices in the set.
\begin{lemma}[Bounding the expected number of $(A,s)$-locally non-erased indices]\label{NewLemmaSizeLc}
Let $W$ be the lazy random walk on a finite simple, connected, regular graph $G = (V,E)$. Fix $s \in \mathbb{N}$ such that $s \geq t_{\mbox{\tiny mix}}^{G}+1$. Let $A\subset \mathbb{N}_{0}$ be any finite interval such that $\# A \geq 2s +1$. Then,
\begin{equation}
\begin{aligned} \label{e:037}
\Big(1- \frac{10s^2}{\# V} \Big)\frac{\# A -2s}{\sup_{\rho \in V} H^W_\rho(s)}\leq \mathbb{E}_{\pi^G}\big[\# {\rm NE}^{W,s}(A)\big] \leq\# A.
\end{aligned}\end{equation} 
\end{lemma}
\begin{proof}
The {\em upper bound} is obvious. We therefore only need to prove the {\em lower bound}. By \eqref{e:023}, 
\begin{equation}
\begin{aligned}  
\label{LowLEeq1II}
&   \mathbb{E}_{\pi^G}\big[\# {\rm NE}^{W,s}(A)\big] \\
& \quad \quad = \sum_{\ell\in [\min A+s,\max A-s]} \mathbb{P}_{\pi^G}\big({\rm R}^W({\rm NE}^{W,[\ell-s,\ell]}(m))\cap {\rm R}^W([\ell+1,\ell+s])=\emptyset\big) \\
& \quad \quad \geq \sum_{\ell\in [\min A+s,\max A-s]} \mathbb{P}_{\pi^G}\big({\rm R}^W([\ell-s,\ell])\cap {\rm R}^W([\ell+1,\ell+s])=\emptyset\big).
\end{aligned}
\end{equation}
\noindent Then, by \eqref{LowLEeq1II} and Lemma~\ref{lemmaKIsALocalCut point}, we deduce that, for $W_1$ and $W_2$ independent lazy random walks on $G$, 
\begin{equation}
\begin{aligned}  
\label{LowLEeq1}
\mathbb{E}_{\pi^G}\big[\# {\rm NE}^{W,s}(A)\big] &\geq 
\frac{\# A-2s}{\# V}\sum_{\varrho\in V} \mathbb{P}_{(\varrho,\varrho)}\big({\rm R}^{W_1}([0,s])\cap {\rm R}^{W_2}([1,s])=\emptyset\big)\\
& =(\# A-2s)\bar{q}^G(s). 
\end{aligned}
\end{equation}
Thus the claim follows from Lemma \eqref{P:HGq}. 
\end{proof}

We continue by transferring the bounds of Lemma \ref{NewLemmaSizeLc} on  the expected number of $(A,s)$-locally non-erased indices to the expected number of elements in a $(A,s)$-locally non-erased chain.
The following result corresponds to \cite[Corollary 4.2]{PeresRevelle}. 
\begin{lemma}[Bounding the expected number of points in an $(A,s)$-locally non-erased chain] \label{NewLemmaSizeRangeL}
Let $W$ be the lazy random walk on a finite simple, connected, regular graph $G = (V,E)$. Fix $s \in \mathbb{N}$ such that $s \geq t_{\mbox{\tiny mix}}^{G}+1$. Let $A\subset \mathbb{N}_{0}$ be any finite interval such that $\# A \geq 2s +1$. Then,
\begin{equation} \label{NewLemmaSizeRangeLeq1}
\begin{aligned} 
\Big(1- \frac{10s^2}{\# V} \Big)\frac{\# A -2s}{\sup_{\rho \in V} H^W_\rho(s)} - \frac{2 (\# A-2s)^{2}}{\# V}\leq \mathbb{E}_{\pi^G}\big[\# {\rm R}^{W}({\rm NE}^{W,s}(A))\big] \leq \mathbb{E}_{\pi^G}\big[\# {\rm NE}^{W,s}(A)\big] \leq \# A.
\end{aligned}
\end{equation} 
\end{lemma}
\begin{proof}
The upper bound is obvious. We therefore only need to prove the lower bound. Let $\gamma:\,\mathbb{N}_0\to V$ be a path on a finite, simple, connected graph $G=(V,E)$. Note that
\begin{align} \label{NewLemmaSizeRangeLeq2}
\# {\rm R}^{\gamma}({\rm NE}^{\gamma,s}(A)) = \sum_{k \in {\rm NE}^{\gamma,s}(A)} \mathbf{1}_{\{ \gamma(k) \not \in {\rm R}^{\gamma}({\rm NE}^{\gamma,s}(A) \cap \{\min A+s, k-1] ) \}}.
\end{align}
\noindent To see this, note that the sum on the right-hand side counts the time indices $k \in {\rm NE}^{\gamma,s}(A)$ such that there is not $j \in {\rm NE}^{\gamma,s}(A)$ with $j <k$ and $\gamma(j) = \gamma(k)$ (that is, all different values in ${\rm R}^{\gamma}({\rm NE}^{\gamma,s}(A))$). Thus,
\begin{align} \label{NewLemmaSizeRangeLeq3}
\# {\rm R}^{\gamma}({\rm NE}^{\gamma,s}(A)) = \# {\rm NE}^{\gamma,s}(A) - \sum_{k \in {\rm NE}^{\gamma,s}(A)} \mathbf{1}_{\{ \gamma(k) \in {\rm R}^{\gamma}({\rm NE}^{\gamma,s}(A) \cap \{\min A+s, k-1] ) \}}.
\end{align}
\noindent On the other hand, note that
\begin{align} \label{NewLemmaSizeRangeLeq4}
& \sum_{k \in {\rm NE}^{\gamma,s}(A)} \mathbf{1}_{\{ \gamma(k) \in {\rm R}^{\gamma}({\rm NE}^{\gamma,s}(A) \cap \{\min A+s, k-1] ) \}} \nonumber \\
& \quad \quad = \sum_{k \in [\min A + s, \max A -s]} \mathbf{1}_{\{ \gamma(k) \in {\rm R}^{\gamma}({\rm NE}^{\gamma,s}(A) \cap \{\min A+s, k-1]), k \in {\rm NE}^{\gamma,s}(A) \}} \nonumber \\
& \quad \quad \leq  \sum_{k \in [\min A + s, \max A -s]} \sum_{j \in [\min A + s, k-1]} \mathbf{1}_{\{ \gamma(k) = \gamma(j), k \in {\rm NE}^{\gamma,s}(A), j \in {\rm NE}^{\gamma,s}(A) \}}.
\end{align}
\noindent Since for any $j \geq s$, $\gamma(j) \in {\rm R}^{\gamma}({\rm NE}^{\gamma,[j-s,j]}(j))$ by \eqref{e:102}, it follows that if $\gamma(k) = \gamma(j)$ for $k \in [\min A + s, \max A -s]$ and $j \in [k-s, k-1]$, then 
\begin{align} \label{NewLemmaSizeRangeLeq5}
{\rm R}^{\gamma}({\rm NE}^{\gamma,[j-s,j]}(j)) \cap {\rm R}^{\gamma}([j+1, j+s]) \neq \emptyset.
\end{align}
This implies that such an index $j$ satisfies $j\notin {\rm NE}^{\gamma,s}(A)$, by \eqref{e:023}.
\noindent The latter together with \eqref{NewLemmaSizeRangeLeq4} implies
\begin{align} \label{NewLemmaSizeRangeLeq6}
& \sum_{k \in {\rm NE}^{\gamma,s}(A)} \mathbf{1}_{\{ \gamma(k) \in {\rm R}^{\gamma}({\rm NE}^{\gamma,s}(A) \cap \{\min A+s, k-1] ) \}} \nonumber \\
& 
\quad \quad \leq  \sum_{k \in [\min A + s, \max A -s]} \sum_{j \in [\min A + s, k-s-1]} \mathbf{1}_{\{ \gamma(k) = \gamma(j), k \in {\rm NE}^{\gamma,s}(A), j \in {\rm NE}^{\gamma,s}(A) \}}   \nonumber 
\\
& \quad \quad \leq  \sum_{k \in [\min A + s, \max A -s]} \sum_{j \in [\min A + s, k-s-1]} \mathbf{1}_{\{ \gamma(k) = \gamma(j) \}}. 
\end{align}
By combining \eqref{NewLemmaSizeRangeLeq3} and \eqref{NewLemmaSizeRangeLeq6}, we obtain that
\begin{align} \label{NewLemmaSizeRangeLeq7}
& \mathbb{E}_{\pi^G}\big[ \# {\rm R}^{W}({\rm NE}^{W,s}(A)) \big] \nonumber \\& \quad \quad \geq  \mathbb{E}_{\pi^G}\big[ \# {\rm NE}^{W,s}(A) \big] - \sum_{k \in [\min A + s, \max A -s]} \sum_{j \in [\min A + s, k-s-1]} \mathbb{P}_{\pi^G}(W(k) = W(j))
\end{align}
\noindent Note that, by the Markov property and \eqref{Exteq14}, 
\begin{align} \label{NewLemmaSizeRangeLeq8}
\mathbb{P}_{\pi^G}(W(k) = W(j)) & =  \sum_{x \in V}\mathbb{P}_{x}(W(k-j) = x) \mathbb{P}_{\pi^G}(W(j) = x) \nonumber \\
& =\frac{1}{\# V} \sum_{x \in V}\mathbb{P}_{x}(W(k-j) = x)   \nonumber \\
& \leq \frac{2}{\# V} \sum_{x \in V}\mathbb{P}_{\pi^G}(W(k-j) = x) \nonumber \\
& = \frac{2}{\# V},
\end{align}
\noindent where to obtain the third inequality we have applied \eqref{Exteq4} in Lemma \ref{lemma1NewE} since $k -j \geq s+1 \geq  t_{\mbox{\tiny mix}}^{G}+2$.

Finally, the upper bound follows from \eqref{NewLemmaSizeRangeLeq7}, \eqref{NewLemmaSizeRangeLeq8} and Lemma \ref{NewLemmaSizeLc}. 
\end{proof}

We finish this subsection with a concentration inequality for the length of a locally non-erased chain, on a finite interval. It corresponds to \cite[(25) in Lemma 5.3]{PeresRevelle}.

\begin{lemma}[Concentration of the length of a locally non-erased chain]\label{PereRevellConcentration} 
Let $W$ be the lazy random walk on a finite simple, connected, regular graph $G = (V,E)$. Let $s,q, q^{\prime} \in \mathbb{N}$ and $A$ be a finite interval $A\subset\mathbb{N}_0$ such that $\# A \geq 2s +1$,
\begin{equation}
\begin{aligned} 
\label{conditionsKey}
\min A \geq s, \quad 3s +1 \leq q \leq \# A +s \quad \text{and} \quad s \geq q^{\prime} t_{\mbox{\tiny mix}}^{G}+1, \quad \text{with} \quad q^{\prime} \geq \frac{ \ln( 2 \# V )}{\ln 4}. 
\end{aligned}
\end{equation} 
\noindent Then, for $y > 2( 3s +1 +q^{2}/(\# A +s))$ and $\rho\in V$, we have
\begin{equation}
\begin{aligned} 
\label{e:039}
& \mathbb{P}_\varrho\Big(\Big|\# {\rm R}^W\big({\rm NE}^{W, s}(A)\big)- \mathbb{E}_{\pi^{G}}\big[ \# {\rm R}^W\big({\rm NE}^{W, s}(A)\big)\big]\Big| \geq (\#A+s)\frac{y}{q}\Big)  \\
& \quad \quad \quad \quad  \leq  2\exp\Big( - \frac{\#A+s}{2q} \Big(\frac{y}{q}\Big)^{2}\Big) + \frac{\#A+s}{q \# V}.
\end{aligned}
\end{equation} 
\end{lemma}
\begin{proof} 
This lemma will be proved by breaking down $A$ into smaller segments that are separated by a distance of at least $s$ from each other. Then we approximate the original lazy random walk on each of the smaller segments, with i.i.d.\ copies of lazy random walks starting from the stationary distribution. This will allow us to apply Hoeffding's inequality and conclude the result.

For $q, s \in \mathbb{N}$ such that $\#A \geq 2s+1$ and $3s+1 \leq q \leq \#A +s$, we define
\begin{equation}
\begin{aligned} 
\label{newSets1T}
A_{j}^{(q,s)} \coloneqq \big[\min A+(j-1)q,\min A+ jq -s -1\big]\subseteq A,
\end{aligned}
\end{equation}
\noindent for $j\in\{1,\ldots, k\}$, where $k = \lfloor (\#A+s)/q \rfloor$. Since $q \geq 3s+1$, the sets  $A_{1}^{(q,s)}, \dots, A_{k}^{(q,s)}$ are disjoint. Consequently, by \eqref{e:023},
${\rm NE}^{\gamma, s}( A_{1}^{(q,s)}), \dots, {\rm NE}^{\gamma, s}( A_{k}^{(q,s)})$ are also disjoint. Note also that $\min(A_{1}^{(q,s)}) = \min(A) \geq s$ and $\min(A_{j+1}^{(q,s)}) - \max(A_{j}^{(q,s)}) = s+1$, for all $j \in \{1,\dots,k-1\}$. We therefore have
\begin{equation}
\begin{aligned} 
\label{NewIneExt1}
\sum_{j=1}^{k} \# {\rm R}^\gamma\Big({\rm NE}^{\gamma, s}\big( A_{j}^{(q,s)}\big)\Big) \leq 
\# {\rm R}^\gamma\big({\rm NE}^{\gamma, s}(A)\big)
\end{aligned}
\end{equation} 
\noindent and 		
\begin{equation}
\begin{aligned} 
\label{NewIneExt1II}
\# {\rm R}^\gamma\big({\rm NE}^{\gamma, s}(A)\big) & \leq 
\sum_{j=1}^{k}\# {\rm R}^\gamma\Big({\rm NE}^{\gamma, s}\big( A_{j}^{(q,s)}\big)\Big)  + \sum_{j=1}^{k-1} \Big(3s+1 \Big) + \#A +s-kq  \\
& \leq \sum_{j=1}^{k}\# {\rm R}^\gamma\Big({\rm NE}^{\gamma, s}\big( A_{j}^{(q,s)}\big)\Big)  + \frac{\#A+s}{q} \Big(3s+1 \Big) + q.
\end{aligned}
\end{equation}
\noindent To obtain the first inequality in \eqref{NewIneExt1II} note  that  the number of points between two consecutive intervals $[\min A_j^{(q,s)} + s, \max A_j^{(q,s)} -s]$ and $[\min A_{j+1}^{(q,s)} + s, \max A_{j+1}^{(q,s)} -s]$, for $j\in \{1, \dots, k-1\}$, is bounded above by $3s+1$, while the difference between $\max A -s$ and $\max A_{k}^{(q,s)} -s$ is bounded by $\#A +s -kq$. Then, the triangle inequality, \eqref{NewIneExt1} and \eqref{NewIneExt1II} imply that, for $y > 2( 3s +1 +q^{2}/(\# A +s))$,
\begin{equation}
\begin{aligned} 
\label{NewIneExt2bb}
&\mathbb{P}_\varrho\Big(\Big|\# {\rm R}^W\big({\rm NE}^{W, s}(A)\big)- \mathbb{E}_{\pi^{G}}\big[ \# {\rm R}^W\big({\rm NE}^{W, s}(A)\big)\big]\Big| \geq\frac{(\#A+s)y}{q}\Big)  \\
&\quad \quad \leq 
\mathbb{P}_\varrho\Big(\Big|\sum_{j=1}^{k}\Big(\# {\rm R}^{W}\big({\rm NE}^{W, s}( A_{j}^{(q,s)})\big) 
- \mathbb{E}_{\pi^{G}}\big[\# {\rm R}^{W}\big({\rm NE}^{W, s}( A_{j}^{(q,s)})\big)\big]\Big) \Big| 
\geq \frac{(\#A+s)y}{2q} \Big),
\end{aligned}
\end{equation} 
\noindent Let $W_{1}, \dots, W_{k}$ be i.i.d.\ lazy random walks on $G$ with initial distribution given by $\pi^{G}$. Then, \eqref{NewIneExt2bb}, \eqref{conditionsKey} and \eqref{eq3supb2} in Lemma \ref{strongestima2} (using the latter with $q'$) imply that, for $y > 2( 3s +1 +q^{2}/(\# A +s))$,
\begin{equation}
\begin{aligned} 
\label{NewIneExt2}
&\mathbb{P}_\varrho\Big(\Big|\# {\rm R}^W\big({\rm NE}^{W, s}(A)\big)- \mathbb{E}_{\pi^{G}}\big[ \# {\rm R}^W\big({\rm NE}^{W, s}(A)\big)\big]\Big| \geq\frac{(\#A+s)y}{q}\Big)   \\
&\quad \leq \mathbb{P}_{\pi^{G}\otimes\cdots \otimes\pi^{G}}\Big(\Big|\sum_{j=1}^{k}\Big(\# {\rm R}^{W_{j}}\big({\rm NE}^{W_{j}, s}( A_{j}^{(q,s)})\big) 
- \mathbb{E}_{\pi^{G}}\big[\# {\rm R}^{W_{j}}\big({\rm NE}^{W_{j}, s}( A_{j}^{(q,s)})\big)\big]\Big) \Big| 
\geq \frac{(\#A+s)y}{2q} \Big) \\
& \quad \quad \quad \quad + \frac{\#A+s}{q \# V}.
\end{aligned}
\end{equation} 
\noindent  Note that for any path $\gamma$ on $G$, we have $\# {\rm R}^{\gamma}({\rm NE}^{\gamma, s}( A_{j}^{(q,s)})) \leq q$.
Since, under $\mathbb{P}_{\pi^{G}\otimes\cdots \otimes\pi^{G}}$, $\# {\rm R}^{W_{1}}({\rm NE}^{W_{1}, s}( A_{1}^{(q,s)})), \dots, \# {\rm R}^{W_{k}}({\rm NE}^{W_{k}, s}( A_{k}^{(q,s)}))$ are i.i.d.\ random variables, \eqref{NewIneExt2} and Hoeffding's inequality \cite[Theorem~1.3 in~Chapter 3]{Gut2013} imply that, for $y > 2( 3s +1 +q^{2}/(\# A +s))$,
\begin{equation}
\begin{aligned} 
\label{NewIneExt4}
& \mathbb{P}_\varrho\Big(\Big|\# {\rm R}^W\big({\rm NE}^{W, s}(A)\big)- \mathbb{E}_{\pi^{G}}\big[ \# {\rm R}^W\big({\rm NE}^{W, s}(A)\big)\big]\Big| \geq\frac{(\#A+s)y}{q}\Big)  \\
& \quad \quad \quad \quad  \leq  2\exp\Big( - \frac{\#A+s}{2q} \Big(\frac{y}{q}\Big)^{2}\Big) + \frac{\#A+s}{q \# V}.
\end{aligned}
\end{equation} 
\end{proof}

\subsection{Asymptotic estimates}
\label{AsymptoticEst}

In this subsection, we establish asymptotic estimates for the probabilities and quantities introduced in the preceding sections, which will be used in the proof of our main result.

Let $(G_N;N\in \mathbb{N}_{0})$ be a sequence of finite simple, connected, regular graphs, with $G_N=(V_N,E_N)$, such that $\# V_N\to \infty$ as $N\to \infty$. For every $N \in \mathbb{N}_{0}$, let $W_N = (W_N(n))_{n \in \mathbb{N}_{0}}$ be a lazy random walk on $G_N$. Suppose further that there are sequences $(s_{N}^{\prime})_{N \in \mathbb{N}}$, $(s_{N})_{N \in \mathbb{N}}$ and $(r_{N})_{N \in \mathbb{N}}$ in $\mathbb{N}$ such that
\begin{align} 
\label{THESeqNewII}
1 \ll \Big( \frac{\ln (2 \# V_{N})}{\ln 4}  \Big)^{2}t_{\mbox{\tiny mix}}^{G_{N}} \ll \frac{\ln (2\# V_{N})}{\ln 4} s_{N}^{\prime} \ll s_{N} \ll r_{N} \ll (\#V_{N})^{\frac{1}{2}}. 
\end{align}
\noindent and
\begin{equation}\label{THESeqNewIIs_Nr_N}
\paren{\ln \paren{\frac{(\# V_{N})^{1/2}}{r_N}}}^6 s_{N}\ll r_{N}.
\end{equation}

Next, we introduce Assumption \hyperref[assumptionTransientRandomWalk]{Transient random walk}, which bounds $H^{W_N}_\rho$ in \eqref{Hfunct} up to the mixing time $t_{{\mbox{\tiny mix}}}^{G_{N}}$. (Compare this with Condition (2) in \cite{PeresRevelle}.) 

\begin{description}\label{assumptionTransientRandomWalk}
\item[Transient random walk] 
There exists $\theta > 0$ such that
\begin{equation}
\sup_{N\geq 1 }\sup_{\rho\in V_N}H^{W_N}_\varrho(t_{{\mbox{\tiny mix}}}^{G_N })=\sup_{N\geq 1}\sup_{\rho\in V_N} \sum_{i=0}^{t_{{\mbox{\tiny mix}}}^{G_N}}\sum_{j=0}^{t_{{\mbox{\tiny mix}}}^{G_N}} \mathbb{P}_{\varrho}\big(W_N(i+j)=\varrho\big) \leq \theta.
\end{equation}
\end{description}

From \eqref{Hfunct}, it follows that necessarily $\theta\geq 1$ since $H_\rho^{W_N}(m)\geq 1$ for all $m\in \mathbb{N}_{0}$.

The next result is an immediate consequence of our bound on the intersection between a random walk and an independent locally non-erased chain (Lemma \ref{corollaryExt3}), together with our bound on the expected number of elements in a locally non-erased chain (Lemma \ref{NewLemmaSizeRangeL}).

\begin{corollary}[Uniform bounds on $(A,s)$-locally non erased indices]\label{corollaryUniformBoundsOfNonErasedVertices}
Let $(G_N;N\in \mathbb{N}_{0})$ be a sequence of finite simple, connected, regular graphs such that are $(s_{N}^{\prime})_{N \in \mathbb{N}}$, $(s_{N})_{N \in \mathbb{N}}$ and $(r_{N})_{N \in \mathbb{N}}$ in $\mathbb{N}$ that satisfy \eqref{THESeqNewII}. Suppose also that $(G_N;N\in \mathbb{N}_{0})$ satisfies the Assumption \hyperref[assumptionTransientRandomWalk]{Transient random walk}. Let $(A_{N})_{N \in \mathbb{N}}$ and $(B_{N})_{N \in \mathbb{N}}$ be a sequence of finite intervals of $\mathbb{N}_{0}$ (i.e., $A_{N}, B_{N}\subset \mathbb{N}_{0}$, for all $N \in \mathbb{N}_{0}$) such that $B_{N} \subset A_{N}$, for all $N \in \mathbb{N}_{0}$, 
\begin{align} \label{corollaryUniformBoundsOfNonErasedVerticesCond}
\lim_{N \rightarrow \infty} \frac{\# A_{N} }{r_{N}} =1 \quad \text{and} \quad \lim_{N \rightarrow \infty} \frac{\# B_{N} }{r_{N}} =1 
\end{align}
\noindent Then		
\begin{equation}\label{eqnUniformBoundsOfNonErasedVertices1}
\begin{aligned} 
\mathbb{E}_{\pi^{G_N}}\big[\# {\rm R}^{W_N}({\rm NE}^{W_N,s_N}(A_N))\big] = \Theta(r_N ), \quad \text{as} \quad N \rightarrow\infty,
\end{aligned}
\end{equation}
\noindent and 
\begin{equation}\label{eqnUniformBoundsOfNonErasedVertices2}
\begin{aligned}
&\mathbb{P}_{\pi^{G_N}\otimes\pi^{G_N}}\big(\mathrm{R}^{W^1_N}(B_N)\cap{\rm R}^{W^2_N}({\rm NE}^{W^2_N,s_{N}}(A_N))\neq \emptyset\big)  = \Theta\left(\frac{r_{N}^{2}}{\# V_N } \right), \quad \text{as} \quad N \rightarrow\infty,
\end{aligned}
\end{equation}
\noindent where $W^{1}_N = (W_{N}^{1}(n))_{n \in \mathbb{N}_{0}}$ and $W^2_N =(W_{N}^{2}(n))_{n \in \mathbb{N}_{0}}$ are two independent lazy random walks on $G_{N}$. 
\end{corollary}

\begin{proof}
By \eqref{THESeqNewII} and \eqref{corollaryUniformBoundsOfNonErasedVerticesCond}, we can assume throughout the proof that $N$ is sufficiently large such that $s_{N} \geq t_{\mbox{\tiny mix}}^{G_{N}} +1$ and $\# A_{N} \geq 2s_{N}+1$. 

First, we prove \eqref{eqnUniformBoundsOfNonErasedVertices1}. It follows from the upper bound in \eqref{NewLemmaSizeRangeLeq1} of Lemma \ref{NewLemmaSizeRangeL} that
\begin{align} \label{eqnUniformBoundsOfNonErasedVertices1A}
\mathbb{E}_{\pi^{G_N}}\big[\# {\rm R}^{W_N}({\rm NE}^{W_N,s_N}(A_N))\big] \leq \# A_{N}. 
\end{align}
On the other hand, by \eqref{eqnUniformBoundsOfNonErasedVertices3III} (with $m =s_{N}$) and the lower bound  in  \eqref{NewLemmaSizeRangeLeq1} of Lemma \ref{NewLemmaSizeRangeL},
\begin{equation}
\label{eqnUniformBoundsOfNonErasedVertices1ii}
\begin{aligned} 
& \mathbb{E}_{\pi^{G_N}}\big[\# {\rm R}^{W_N}({\rm NE}^{W_N,s_N}(A_N))\big] \\
& \quad \quad \quad \geq \left( 1 - \frac{10s^{2}_{N}}{\#V_{N}} \right)\frac{\# A_{N} -2s_{N}}{\sup_{\varrho \in V_{N}} H^{W_{N}}_\rho(t_{\mbox{\tiny mix}}^{G_{N}})+ \frac{5}{2}\frac{s^2_{N}}{\# V_{N}}} - \frac{2 (\# A_{N}-2s_{N})^{2}}{\# V_{N}}.
\end{aligned}
\end{equation}
\noindent Then, by \eqref{THESeqNewII}, \eqref{corollaryUniformBoundsOfNonErasedVerticesCond} and Assumption \hyperref[assumptionTransientRandomWalk]{Transient random walk}, we have that
\begin{equation}
\label{eqnUniformBoundsOfNonErasedVertices1iv}
\begin{aligned} 
\mathbb{E}_{\pi^{G_N}}\big[\# {\rm R}^{W_N}({\rm NE}^{W_N,s_N}(A_N))\big]  \geq \frac{\# A_{N}}{\theta} + o(\#A_N),  \quad \text{as} \quad N \rightarrow\infty.
\end{aligned}
\end{equation}
\noindent Thus, \eqref{eqnUniformBoundsOfNonErasedVertices1A} and \eqref{eqnUniformBoundsOfNonErasedVertices1iv} imply \eqref{eqnUniformBoundsOfNonErasedVertices1}. 

Now, we prove \eqref{eqnUniformBoundsOfNonErasedVertices2}. It follows immediately from the upper bounds in \eqref{ExtIne1} and \eqref{NewLemmaSizeRangeLeq1} in Lemma \ref{corollaryExt3} and Lemma \ref{NewLemmaSizeRangeL}, respectively, that 
\begin{align}
\label{eqnUniformBoundsOfNonErasedVertices2IV}
\begin{aligned}
\mathbb{P}_{\pi^{G_N}\otimes\pi^{G_N}}\big(\mathrm{R}^{W^1_N}(B_N)\cap{\rm R}^{W^2_N}({\rm NE}^{W^2_N,s_{N}}(A_N)) \leq \frac{\# B_N\# A_{N}}{\# V_{N}}.
\end{aligned}
\end{align}
\noindent On the other hand, by the lower bound in  \eqref{ExtIne1Low} of Lemma \ref{corollaryExt3}, and \eqref{eqnUniformBoundsOfNonErasedVertices3IV}  (with $m= \#A_{N} -1=\max \{\#A_{N} -1,\#B_{N} -1\}$)
and  we have
\begin{equation} \label{eqnUniformBoundsOfNonErasedVertices2ii}
\begin{aligned}
&\mathbb{P}_{\pi^{G_N}\otimes\pi^{G_N}}\big(\mathrm{R}^{W^1_N}(B_N)\cap{\rm R}^{W^2_N}({\rm NE}^{W^2_N,s_{N}}(A_N))  \\
& \quad  \geq \frac{\# B_N}{4\# A_N \# V_{N}} \left(\sup_{\varrho \in V_{N}} H^{W_{N}}_\rho(t_{\mbox{\tiny mix}}^{G_{N}})+ \frac{5}{2}\frac{(\#A_{N} -1)^2}{\# V_{N}} \right)^{-1} \left( \mathbb{E}_{\pi^{G_N}}\big[\# {\rm R}^{W_N}({\rm NE}^{W_N,s_N}(A_N))\big] \right)^{2}.
\end{aligned}
\end{equation}
\noindent Then, by \eqref{THESeqNewII}, \eqref{corollaryUniformBoundsOfNonErasedVerticesCond},  \eqref{eqnUniformBoundsOfNonErasedVertices1ii} (or \eqref{eqnUniformBoundsOfNonErasedVertices1iv}), and Assumption \hyperref[assumptionTransientRandomWalk]{Transient random walk} we have that
\begin{equation} \label{eqnUniformBoundsOfNonErasedVertices2v}
\begin{aligned}
\mathbb{P}_{\pi^{G_N}\otimes\pi^{G_N}}\big(\mathrm{R}^{W^1_N}(B_N)\cap{\rm R}^{W^2_N}({\rm NE}^{W^2_N,s_{N}}(A_N)) \geq \frac{(\# A_{N})^2}{4 \theta^{3}\# V_{N}} + o\left( \frac{(\# A_{N})^2}{\# V_{N}} \right),
\end{aligned}
\end{equation}
\noindent as $N \rightarrow \infty$. Thus, \eqref{eqnUniformBoundsOfNonErasedVertices2IV} and \eqref{eqnUniformBoundsOfNonErasedVertices2v} imply \eqref{eqnUniformBoundsOfNonErasedVertices2}.
\end{proof}

Recall the definition of  $\overline{q}^{G_{N}}$ in \eqref{LPeq8} and $H^{W_N}_\rho$ in \eqref{Hfunct}. 
In the next corollary we bound $\overline{q}^{G_{N}}$. 
It also proves that 
Assumption \hyperref[assumptionTransientRandomWalk]{Transient random walk} and \eqref{THESeqNewII}, give a bound on the number of times that a random walk on $G_N$, comes back to its starting position between the first $i+j$ steps, for $i,j\leq 2s'_N$.
The latter can be thought as a condition for the random walk on the graph to be transient.

\begin{corollary} \label{CorollaryBoundingQBar}
Let $(G_N;N\in \mathbb{N}_{0})$ be a sequence of finite simple, connected, regular graphs, and consider $(s_{N}^{\prime})_{N \in \mathbb{N}}$, $(s_{N})_{N \in \mathbb{N}}$ and $(r_{N})_{N \in \mathbb{N}}$ in $\mathbb{N}$ that satisfy \eqref{THESeqNewII}. Suppose also that $(G_N;N\in \mathbb{N}_{0})$ satisfies the Assumption \hyperref[assumptionTransientRandomWalk]{Transient random walk}. Then,
\begin{equation} \label{eqnLimInfBarq1}
\limsup_{N \rightarrow \infty}\overline{q}^{G_N}(2s_N^{\prime})\leq  1-\frac{1}{\theta}.
\end{equation}
\noindent where $\theta \geq 1$ is the constant defined in the Assumption \hyperref[assumptionTransientRandomWalk]{Transient random walk}.

Moreover, 
\begin{equation} \label{eqnLimInfBarq1II}
(\# V_N)^{\frac{1}{2}}\paren{\overline{q}^{G_N}(2s'_N)}^{ \lfloor \frac{s_N}{6s'_N}  \rfloor} = o(1), \quad \text{as} \quad N \rightarrow \infty. 
\end{equation}
\end{corollary}

\begin{proof}
By \eqref{THESeqNewII}, we can assume throughout the proof that $N$ is sufficiently large such that $2s_{N}^{\prime} \geq t_{\mbox{\tiny mix}}^{G_{N}} +1$. It follows from the definition of $\overline{q}^{G_{N}}$ in \eqref{LPeq8}, Lemma \ref{P:HGq} (with $s = 2s_N^{\prime}$) and the Assumption \hyperref[assumptionTransientRandomWalk]{Transient random walk} that
\begin{equation} \label{eqnLimInfBarq}
\overline{q}^{G_N}(2s_N^{\prime}) \leq  1- \paren{1-\frac{10 (2s'_N)^2}{\# V_N}}\frac{1}{\theta+ \frac{5}{2}\frac{(2s_N^{\prime})^{2}}{\# V_N}}.
\end{equation}
\noindent Then \eqref{eqnLimInfBarq1} follows from \eqref{eqnLimInfBarq} and \eqref{THESeqNewII}. 

Next, we prove \eqref{eqnLimInfBarq1II}. Note that
\begin{align} \label{eqnLimInfBarqII}
(\# V_N)^{\frac{1}{2}}\paren{\overline{q}^{G_N}(2s'_N)}^{\lfloor \frac{s_N}{6s'_N}  \rfloor} = \exp \paren{  \frac{s_N}{6s'_N} \paren{  \frac{3s_{N}^{\prime}}{s_{N}} \ln (\# V_N ) + \frac{6s_N^{\prime}}{s_N} \Big\lfloor \frac{s_N}{6s_N^{\prime}}  \Big\rfloor \ln (\overline{q}^{G_N}(2s'_N))}}.
\end{align}
\noindent Since by \eqref{THESeqNewII}, $(s_{N}^{\prime}/s_N)\ln (\# V_N )  = o(1)$ and $s_{N}^{\prime}/s_N = o(1)$, then \eqref{eqnLimInfBarq1II} follows from \eqref{eqnLimInfBarqII} and \eqref{eqnLimInfBarq1}. 
\end{proof}

Fix $r,s \in \na$ with $r \geq 3s+1$.
For each $i\in \na$, define the following sub-intervals of $\na_0$,
\begin{equation} \label{eqnDefinitionAiBi}
A_{i}^{(r,s)} \coloneqq [(i-1)r +s, ir-1], \quad \text{and} \quad B_{i}^{(r,s)} \coloneqq [(i-1)r +2s, ir-s-1].
\end{equation}

The following result gives us the asymptotic order of $\gamma_N$ and $c_N$.
Its proof is simply an application of \eqref{eqnUniformBoundsOfNonErasedVertices1} and \eqref{eqnUniformBoundsOfNonErasedVertices2} in Corollary \ref{corollaryUniformBoundsOfNonErasedVertices}, and \eqref{THESeqNewII}. 

\begin{corollary}[Uniform bounds on $(A^{(r_N,s_N)}_i,s_N)$-locally non erased indices]\label{corollaryUniformBoundsOfNonErasedVerticesSegmentAi}
Let $(G_N;N\in \mathbb{N}_{0})$ be a sequence of finite simple, connected, regular graphs such that are $(s_{N}^{\prime})_{N \in \mathbb{N}}$, $(s_{N})_{N \in \mathbb{N}}$ and $(r_{N})_{N \in \mathbb{N}}$ in $\mathbb{N}$ that satisfy \eqref{THESeqNewII}. Suppose also that $(G_N;N\in \mathbb{N}_{0})$ satisfies the Assumption \hyperref[assumptionTransientRandomWalk]{Transient random walk}. 
Further, let  $W_{N}, W^1_N$ and $W^2_N$ be independent lazy random walks on $G_{N}$. 
Define the sequences $(\gamma_N)_{N \in \mathbb{N}}$ in $\mathbb{R}_{+}$ and $(c_{N})_{N \in \mathbb{N}}$ by letting
\begin{equation}\label{eqnDefinitionOfGammaN}
{\gamma_N} \coloneqq \mathbb{E}_{\pi^{G_N}}\big[\# {\rm R}^{W_N}({\rm NE}^{W_N,s_N}(A^{(r_N,s_N)}_1))\big],
\end{equation}and 
\begin{equation} \label{eqnDefinitionOfcN}
(c_N)^2 \coloneqq -\ln \mathbb{P}_{\pi^{G_N}\otimes\pi^{G_N}}\big(\mathrm{R}^{W^1_N}(A^{(r_N,s_N)}_1)\cap{\rm R}^{W^2_N}({\rm NE}^{W^2_N,s_N}(A^{(r_N,s_N)}_1))= \emptyset\big),
\end{equation}
\noindent for $N \in \mathbb{N}$. Then, as $N \rightarrow \infty$,
\begin{equation}\label{eqnAsymptoticApproximationsForGammaNandCN}
\gamma_N=\Theta(r_N)\qquad \mbox{and}\qquad 
c_N^2 = \Theta\left(\frac{r_N ^{2}}{\# V_N } \right). 
\end{equation}
\end{corollary}

We now prove that the length of the range of all  $(A^{(r_N,s_N)}_i,s_N)$-locally non-erased indices is asymptotically close to $\gamma_N$, as $N \rightarrow \infty$, uniformly for all $i$ within a specified interval.

\begin{corollary}[Asymptotic concentration of the length of a locally non-erased chain] \label{ProbaEst1}
Let $(G_N;N\in \mathbb{N}_{0})$ be a sequence of finite simple, connected, regular graphs such that $G_{N} = (V_{N}, E_{N})$, for all $N \in \mathbb{N}$.
Suppose also that $(G_N;N\in \mathbb{N}_{0})$ satisfies the Assumption \hyperref[assumptionTransientRandomWalk]{Transient random walk}. 
Assume that we can find sequences $(r_{N})_{N \in \mathbb{N}}$ and $(s_{N})_{N \in \mathbb{N}}$ such that \eqref{THESeqNewII} and \eqref{THESeqNewIIs_Nr_N} hold. Let $(\gamma_{N})_{N \in \mathbb{N}}$ and $(c_{N})_{N \in \mathbb{N}}$ be the sequences defined in \eqref{eqnDefinitionOfGammaN} and \eqref{eqnDefinitionOfcN}, respectively. Further, let  $W_{N} = (W_{N}(n))_{n \in \mathbb{N}_{0}}$ be a lazy random walk on $G_{N}$ such that $W_{N}(0) = \varrho_N \in V_{N}$. Fix $0 <T < \infty$. Then, for $N$ large enough, we have that
\begin{equation}\begin{split}\label{eqnProbaEst2}
& \mathbb{P}_{\rho_N}\Big( \sup_{0 \leq i \leq \lfloor T/c_{N} \rfloor} \Big|\# {\rm R}^{W_N}({\rm NE}^{W_N,s_N}(A^{(r_N,s_N)}_i))- \gamma_{N} \Big| \leq r_{N}  \Big( \frac{s_{N}}{r_{N}} \Big)^{1/6} \Big) = 1-  o(1).
\end{split}
\end{equation} 
\end{corollary}

\begin{proof}
Our approach involves applying the result from Lemma \ref{PereRevellConcentration}. However, prior to its application, we must carefully select the appropriate sequences to ensure that the conditions outlined in Lemma \ref{PereRevellConcentration} are satisfied. 
First note that Corollary \ref{corollaryUniformBoundsOfNonErasedVerticesSegmentAi} allows us to select a sufficiently large value for $N$ such that $T/c_{N} \geq 2$.

Let $q_{N}: = \lfloor (r_{N}s_{N})^{\frac{1}{2}} \rfloor$, for all $N \in \mathbb{N}$. Observe that $\# A^{(r_N,s_N)}_i=r_N-s_N$ and $\min A^{(r_N,s_N)}_i= s_N$, for all $N \in \mathbb{N}$. By \eqref{THESeqNewII}, we can (and will) choose $N$ large enough (independently of $i \in \mathbb{N}$) such that $\# A^{(r_N,s_N)}_{i} \geq 2s_{N}+1$ and the interval $A^{(r_N,s_N)}_{i}$ satisfies the conditions in \eqref{conditionsKey} of Lemma \ref{PereRevellConcentration} with $q^{\prime} = \frac{\ln (2 \# V_{N})}{\ln 4}$. Define
\begin{align}
y_{N}: = s_{N}^{2/3} r_{N}^{1/3}.
\end{align} 
\noindent By \eqref{THESeqNewII} we can (and will) choose $N$ even larger and also (independently of $i \in \mathbb{N}$) so that
\begin{align}
y_{N} > 2( 3s_{N} +1 +q^{2}_{N}/(\# A^{(r_N,s_N)}_{i} +s_{N})) = 2( 3s_{N} +1 +(\lfloor (r_{N}s_{N})^{\frac{1}{2}} \rfloor)^{2}/r_{N}). 
\end{align}
\noindent On the one hand, observe that
\begin{align} \label{eq1AsymConceIne}
\frac{(\# A^{(r_N,s_N)}_{i} +s_{N})y_{N}}{q_{N}} \geq \frac{r_{N}s_{N}^{2/3} r_{N}^{1/3}}{(r_{N}s_{N})^{\frac{1}{2}}} = r_{N}\left(\frac{s_{N}}{r_{N}}\right)^{1/6},
\end{align}
\noindent and
\begin{align} \label{eq2AsymConceIne}
\frac{(\# A^{(r_N,s_N)}_{i} +s_{N})}{2q_{N}} \left( \frac{y_{N}}{q_{N}}\right)^{2} \geq  \frac{r_{N}}{2 (r_{N}s_{N})^{\frac{1}{2}}} \frac{s_{N}^{4/3} r_{N}^{2/3}}{r_{N}s_{N}} = \frac{1}{2} \left(\frac{r_{N}}{s_{N}}\right)^{1/6},
\end{align}
\noindent On the other hand, by \eqref{THESeqNewII},
\begin{align}  \label{eq3AsymConceIne}
\frac{(\# A^{(r_N,s_N)}_{i} +s_{N})}{q_{N} \# V_{N}} \sim \frac{r_{N}^{1/2}}{s_{N}^{1/2} \# V_{N}}, \quad \text{as} \quad N \rightarrow \infty. 
\end{align}
\noindent Thus, by \eqref{eq3AsymConceIne}, we can (and will) select a larger $N$ (independently of $i \in \mathbb{N}$)  such that
\begin{align} \label{eq4AsymConceIne}
\frac{(\# A^{(r_N,s_N)}_{i} +s_{N})}{q_{N} \# V_{N}} \leq \frac{2r_{N}^{1/2}}{s_{N}^{1/2} \# V_{N}}. 
\end{align}

Recall the definition of the sequence $(\gamma_{N})_{N \in \mathbb{N}}$ given in \eqref{eqnDefinitionOfGammaN}. By combining \eqref{eq1AsymConceIne}, \eqref{eq2AsymConceIne}, \eqref{eq4AsymConceIne} and  applying \eqref{e:039} in Lemma \ref{PereRevellConcentration}, we obtain that, for $N$ large enough (independently of $i \in \mathbb{N}$),
\begin{align} \label{eq5AsymConceIne}
& \mathbb{P}_{\rho_N}\Big( \Big|\# {\rm R}^{W_N}({\rm NE}^{W_N,s_N}(A^{(r_N,s_N)}_i))- \gamma_{N} \Big| \leq   r_{N}\left(\frac{s_{N}}{r_{N}}\right)^{1/6} \Big) \nonumber \\
& \quad \quad \leq  \mathbb{P}_{\rho_N}\Big( \Big|\# {\rm R}^{W_N}({\rm NE}^{W_N,s_N}(A^{(r_N,s_N)}_i))- \gamma_{N} \Big| \leq  (\# A^{(r_N,s_N)}_{i} +s_{N}) \frac{y_{N}}{q_{N}} \Big) \nonumber \\
& \quad \quad \leq 2 e^{-\frac{1}{2} \left(\frac{r_{N}}{s_{N}}\right)^{1/6}} + \frac{2r_N^{1/2}}{s_N^{1/2}\# V_N}.
\end{align}

Then, by using the union bound and \eqref{eq5AsymConceIne}, 
\begin{equation}\begin{split} \label{eqnProbaEst2III}
& \mathbb{P}_{\rho_N}\Big( \sup_{0 \leq i \leq \lfloor T/c_{N} \rfloor} \Big|\# {\rm R}^{W_N}({\rm NE}^{W_N,s_N}(A^{(r_N,s_N)}_i))- \gamma_{N} \Big| \leq r_{N}  \Big( \frac{s_{N}}{r_{N}} \Big)^{1/6} \Big) \\
& \quad \quad \quad \quad  \geq 1-  \frac{2T}{c_{N}} \left( e^{-\frac{1}{2} \left(\frac{r_{N}}{s_{N}}\right)^{1/6}} + \frac{r_N^{1/2}}{s_N^{1/2}\# V_N}\right).
\end{split}\end{equation} 
\noindent Next, \eqref{eqnAsymptoticApproximationsForGammaNandCN} in Corollary \ref{corollaryUniformBoundsOfNonErasedVerticesSegmentAi} implies that for sufficiently large $N$ there is a constant $C>0$,
\begin{align} \label{eqnProbaEst2IV}
\frac{2T}{c_{N}} \left( e^{-\frac{1}{2} \left(\frac{r_{N}}{s_{N}}\right)^{1/6}} + \frac{r_N^{1/2}}{s_N^{1/2}\# V_N}\right) & \leq\frac{ CT (\# V_N)^{1/2}}{r_{N}} \left( e^{-\frac{1}{2} \left(\frac{r_{N}}{s_{N}}\right)^{1/6}} + \frac{r_N^{1/2}}{s_N^{1/2}\# V_N}\right) \nonumber \\
& \leq C T \left(  e^{\ln((\# V_N)^{1/2}/ r_{N})-\frac{1}{2} \left(\frac{r_{N}}{s_{N}}\right)^{1/6}}  + \frac{1}{(s_Nr_N\# V_N)^{1/2}}  \right).
\end{align}
\noindent Therefore, our claim \eqref{eqnProbaEst2} follows by combining \eqref{eqnProbaEst2III} and \eqref{eqnProbaEst2IV} together with \eqref{THESeqNewII} and \eqref{THESeqNewIIs_Nr_N}.
\end{proof}

In the following corollary, we establish a bound on the number of intersections between the range of a segment and the range of locally non-erased indices from a preceding segment. To that end, for any $n\in \mathbb{N}$ define 
\begin{equation} \label{HERRRRn}
\mathcal{N}^{\gamma,(r, s)}(n):=\sum_{1\leq i <j\leq n}\mathbf{1}_{\{{\rm R}^\gamma ({\rm NE}^{\gamma,s}(A_{i}^{(r,s)})) \cap {\rm R}^\gamma (B_{j}^{(r,s)})  \neq \emptyset \}},
\end{equation}
\noindent with $\mathcal{N}^{\gamma,(r, s)}(0) = \mathcal{N}^{\gamma,(r, s)}(1) = 0$.

\begin{corollary} \label{coroBoundOnTheNumberOfIntersections}
Let $(G_N;N\in \mathbb{N}_{0})$ be a sequence of finite simple, connected, regular graphs such that $G_{N} = (V_{N}, E_{N})$, for all $N \in \mathbb{N}$.
Suppose also that $(G_N;N\in \mathbb{N}_{0})$ satisfies the Assumption \hyperref[assumptionTransientRandomWalk]{Transient random walk}. 
Assume that we can find sequences $(r_{N})_{N \in \mathbb{N}}$ and $(s_{N})_{N \in \mathbb{N}}$ such that \eqref{THESeqNewII} holds. Let $(c_{N})_{N \in \mathbb{N}}$ be the sequence defined in \eqref{eqnDefinitionOfcN}. Further, let  $W_{N} = (W_{N}(n))_{n \in \mathbb{N}_{0}}$ be a lazy random walk on $G_{N}$ such that $W_{N}(0) = \varrho_N \in V_{N}$. Then, for any $\alpha \in (0,1)$ and $0 <T < \infty$, as $N \rightarrow \infty$,
\begin{equation}\begin{aligned} 
\mathbb{P}_{\varrho_N}\Big( \mathcal{N}^{W_{N},(r_{N}, s_{N})}(\lfloor T/c_{N} \rfloor) \leq c_{N}^{-\alpha} \Big) = 1 - o(1).
\end{aligned}\end{equation} 
\end{corollary}

\begin{proof}
Corollary \ref{corollaryUniformBoundsOfNonErasedVerticesSegmentAi} allows us to select a sufficiently large value for $N$ such that $T/c_{N} \geq 2$. 
By applying \eqref{THESeqNewII}, \eqref{Exteq4} in Lemma \ref{lemma1NewE}, and possibly increasing $N$ further as needed, we obtain that
\begin{equation}\begin{aligned} 
&\mathbb{P}_{\rho_N}\big({\rm R}^{W_N}({\rm NE}^{W_N,s_N}(A^{(r_N,s_N)}_i)\cap \mathrm{R}^{W_N}(B^{(r_N,s_N)}_j))\neq \emptyset \big)\\
&\qquad\qquad  \leq 4\mathbb{P}_{\pi^{G_N}\otimes\pi^{G_N}}\big({\rm R}^{W^2_N}({\rm NE}^{W^2_N,s_N}(A^{(r_N,s_N)}_1)\cap \mathrm{R}^{W^1_N}(B^{(r_N,s_N)}_1))\neq \emptyset \big),
\end{aligned}\end{equation} 
\noindent for $1 \leq i < j \leq \floor{T/c_{N}}$, and $W^1_N, W^2_N$ two independent lazy random walks on $G_N$. 
Thus, by using Markov's inequality and the union bound, we have
\begin{equation}\begin{split} 
& \mathbb{P}_{\varrho_N}\Big( \mathcal{N}^{W_{N},(r_{N}, s_{N})}(\lfloor T/c_{N} \rfloor) \geq c_{N}^{-\alpha} \Big) \\
& \quad \quad \quad \leq  4\frac{T^{2}}{c_{N}^{2 - \alpha}}\mathbb{P}_{\pi^{G_N}\otimes\pi^{G_N}}\big({\rm R}^{W^2_N}({\rm NE}^{W^2_N,s_N}(A^{(r_N,s_N)}_1)\cap \mathrm{R}^{W^1_N}(B^{(r_N,s_N)}_1))\neq \emptyset \big).
\end{split}\end{equation} 
\noindent  Then, by combining \eqref{eqnAsymptoticApproximationsForGammaNandCN} with \eqref{eqnUniformBoundsOfNonErasedVertices2} from Corollary \ref{corollaryUniformBoundsOfNonErasedVertices}, we conclude that for sufficiently large $N$ there is a constant $C>0$ (which may vary from line to line) such that
\begin{equation}\begin{split} 
&\mathbb{P}_{\varrho_N}\Big( \mathcal{N}^{W_{N},(r_{N}, s_{N})}(\lfloor T/c_{N} \rfloor) \geq c_{N}^{-\alpha} \Big)\\
& \qquad \leq  CT^2\frac{(\# V_N)^{1-\alpha/2}}{r_{N}^{2-\alpha}} \mathbb{P}_{\pi^{G_N}\otimes\pi^{G_N}}\big({\rm R}^{W^2_N}({\rm NE}^{W^2_N,s_N}(A^{(r_N,s_N)}_1)\cap \mathrm{R}^{W^1_N}(B^{(r_N,s_N)}_1))\neq \emptyset \big) \\
& \qquad \leq  C T^2\frac{(\# V_N)^{1-\alpha/2}}{r_{N}^{2-\alpha}} \frac{r_N^2}{\# V_N} \\
& \qquad \leq  C T^2 \frac{r_{N}^{\alpha}}{(\# V_{N})^{\alpha/2}}.
\end{split}\end{equation}
\noindent Finally, our claim follows from \eqref{THESeqNewII}. 
\end{proof}

\section{The RGRG driven by a Poisson point process}
\label{S:RGRG}

In this section, we recall the construction of the RGRG from \cite{EvansPitmanWinter2006} and obtain estimates for the probability of the process to be \emph{decomposable} up to a given time. When the process is decomposable, we will be able to establish a coupling between indices of intersection of paths, coming from the Aldous-Broder chain and from the RGRG (see Section~\ref{Sub:regraftindex}). 

\begin{definition}[Nice point cloud]\label{Def:002}
We refer to a set
\begin{equation}\label{e:048}
\pi\subset\Delta_+^2:=\{(t,x) \in \mathbb{R}_+^2:\,t\in\mathbb{R}_+,x\in(0,t]\}\subseteq\mathbb{R}_+^2
\end{equation}
as a nice point cloud if it has the following properties:

\begin{enumerate}[label=(\alph*)]
\item For all $t>0$, $\#\pi \cap (\{t\} \times [0,t])\le 1$ and $\#\pi \cap (\mathbb{R}_+ \times \{t\})\le 1$. \label{Def:002Pro1}
\item For all $t>0$, the set $\pi \cap  \{(t^{\prime},x)\in \mathbb{R}_{+}^{2} : x \le t^{\prime} \le t\}$
is finite. \label{Def:002Pro2}
\end{enumerate}
\end{definition}

For a nice point cloud $\pi\in\Delta_+^2$,
put $\tau_0(\pi)=0$ and define recursively for all $i\in\mathbb{N}$,
\begin{equation} \label{e:038}
\tau_{i}(\pi):=\inf\big\{x>\tau_{i-1}(\pi):\,\exists \, y\in[0,x]\mbox{ with }(x,y)\in\pi\big\}.
\end{equation}
For each $i \in\mathbb{N}$, let $p_i(\pi)$ be the unique point such that $(\tau_i(\pi),p_i(\pi))\in\pi$.

We shall introduce the {\em root growth with re-grafting map} (RGRG map), denote by $\mathrm{RGRG}^\pi:=(\mathrm{RGRG}^\pi(t))_{t\ge 0}$, that takes a nice point cloud and maps it to a path with values in the space of rooted $\mathbb{R}$-trees.
For $t \geq 0$, let
\begin{equation} \label{e:132}
T_t:=[0,t]\quad \mbox{ and } \quad \varrho_t:=t,
\end{equation}
and define $\mathrm{RGRG}^\pi(t):=([0,t],d^\pi_t,\varrho_t)$ inductively as follows. Assume we have already defined the process for all $t'\in[0,\tau_{i}(\pi)]$ for some $i\in\mathbb{N}_0$. We then define the metric $d^\pi_t$ for all $t\in(\tau_{i}(\pi),\tau_{i+1}(\pi)]$. We will distinguish the two cases between $t$ being a root growth or a re-graf time. \\

\noindent
{\bf Root growth.} Assume that $t\in(\tau_{i}(\pi),\tau_{i+1}(\pi))$ and
define $d_t^{\pi}$ by
\begin{equation}
\begin{aligned}
\label{e:030}
d^\pi_{t}(x,y) \coloneqq
\left\{ \begin{array}{ll}
d_{\tau_{i}(\pi)}^{\pi}(x,y), & \mbox{if }x,y \in [0,\tau_{i}(\pi)],
\\
|y-x|, & \mbox{if }x,y\in(\tau_{i}(\pi),t],
\\
|x-\tau_{i}(\pi)|+d_{\tau_{i}(\pi)}^{\pi}(\tau_{i}(\pi),y), &
\mbox{if }x\in(\tau_{i}(\pi),t], y\in [0,\tau_{i}(\pi)]
\\
|y-\tau_{i}(\pi)|+d_{\tau_{i}(\pi)}^{\pi}(\tau_{i}(\pi),x), &
\mbox{if }y\in(\tau_{i}(\pi),t], x\in [0,\tau_{i}(\pi)].
\end{array}
\right.
\end{aligned}
\end{equation}

\noindent
{\bf Re-Grafting.} Assume that $t=\tau_{i+1}(\pi)$. Given a rooted metric tree $(T,d,\varrho)$ and a point $x\in T$, we denote by
\begin{equation}
\label{e:123}
S^{T}_{x}:=S^{(T, d, \varrho)}_{x} = \big\{y\in T:\,x\in \llbracket \varrho,y \rrbracket \big\}
\end{equation}
the rooted subtree of $(T,d,\varrho)$ above $x$ and rooted at $x$ (here we denote by $\llbracket \varrho,y \rrbracket$ all the points between $\rho$ and $y$ in $T$).
Define the metric $d_t^{\pi}$ by
\begin{equation}
\label{pru}
\begin{aligned}
d^\pi_t(x,y) & :=
\begin{cases}
d^\pi_{t-}(x,y), & \mbox{if }x,y\in S^{T_{t-}}_{p_{i+1}(\pi)},
\\
d^\pi_{t-}(x,y), & \mbox{if }x,y\in [0,t)
\setminus  S^{T_{t-}}_{p_{i+1}(\pi)},
\\
d^\pi_{t-}(x,\varrho_{t-})+ d^\pi_{t-}(p_{i+1}(\pi),y),
& \mbox{if }x \in [0,t)
\setminus S^{T_{t-}}_{p_{i+1}(\pi)}, y \in S^{T_{t-}}_{p_{i+1}(\pi)},
\\
d^\pi_{t-}(y,\varrho_{t-})+ d^\pi_{t-}(p_{i+1}(\pi),x),
& \mbox{if }y \in [0,t)
\setminus S^{T_{t-}}_{p_{i+1}(\pi)}, x \in S^{T_{t-}}_{p_{i+1}(\pi)}.
\end{cases}
\end{aligned}
\end{equation}
In words, at time $t=\tau_{i+1}(\pi)$, $\mathrm{RGRG}^{\pi}(t)$ is obtained from
$\mathrm{RGRG}^{\pi}(t-)$ by pruning off the subtree $S^{T_{t-}}_{p_{i+1}(\pi)}$
and re-attaching it to the root.

Note that the process $\mathrm{RGRG}^{\pi}(t)$ is c\`adl\`ag with respect to convergence in the Gromov-Hausdorff topology.

\begin{remark}[Poisson point processes are nice]\label{rem:001}
The conditions on $\pi$ for being a nice point cloud will hold almost surely if $\pi$ is a realization of a Poisson point processes on $\Delta_+^2$ with the Lebesgue measure $\lambda$ as the intensity measure. It is this random mechanism that will produce a stochastic process having the root growth with re-grafting dynamics.
\hfill$\qed$
\end{remark}

\begin{definition}[Root growth with re-grafting (RGRG) process]
We define the {\em root growth with re-grafting (RGRG)} process $X=(X(t))_{t\ge 0}$ by letting for each $t\ge 0$,
\begin{equation}\label{e:051a}
X(t):={\mathrm{RGRG}^\Pi(t)},
\end{equation}
where $\Pi$ is the Poisson process on $\Delta_+^2$ with the Lebesgue measure as intensity measure.
\label{Def:RGRG}
\end{definition}

For each $c\in(0,\infty)$ and  $(i,j)\in \mathbb{N}\times \mathbb{N}$ with $1\le i<j$, consider the square
\begin{equation}\begin{aligned} 
\label{e:060b}
B_{(i,j)}^{c}\coloneqq \big\{(x,y)\in\mathbb{R}_{+}^{2}: \,(j-1)c\le x< jc,(i-1)c\leq  y<ic\big\}.
\end{aligned}\end{equation}

For that we need that the point cloud is $c$-decomposable up to time $T$ in the following sense.   
\begin{definition}[$c$-decomposable point cloud]\label{Def:004} Fix $c>0$ and $T\in\mathbb{R}_+$ such that $\lfloor\tfrac{T}{c}\rfloor\geq 3$. 
We will refer to a nice point cloud $\pi\subset \Delta_+^2$ as a $c$-decomposable point cloud up to time $T$ if the following holds: 
\begin{enumerate}[label=(\textbf{\roman*})]
\item \label{Def:004Pro1} Up to time $T$ each square contains at most one point, i.e., 
\begin{equation}
\label{e:086}
{\mathcal N}^{\pi,c}_1(T):=\sum\nolimits_{1 \leq i < j \leq \lfloor\tfrac{T}{c}\rfloor}\mathbf{1}_{ \{ \# \pi \cap B_{(i,j)}^{c} \geq 2 \}}=0.
\end{equation}
\item \label{Def:004Pro2} There are no upper triangles up to time $T$ that contain a point, i.e., 
\begin{equation}
\label{e:087}
{\mathcal N}^{\pi,c}_2(T):=\sum\nolimits_{1 \leq j \leq \lfloor\tfrac{T}{c}\rfloor}\mathbf{1}_{ \{ \# \pi \cap E_{j}^{c} \geq 1 \}}=0,
\end{equation}
where upper triangles are subsets of $\Delta^2_+$ of the form
\begin{equation}\begin{aligned} 
	\label{e:082}
	E_{j}^{c}\coloneqq \big\{(x,y)\in\Delta^2_+:\,jc\le x< (j+1)c,jc\le y \leq x\big\}, \quad \quad j\in\mathbb{N}_{0}.
\end{aligned}\end{equation} 
\item \label{Def:004Pro3} Up to time $T$ there are no two squares in the same row that contain a point, i.e,  
\begin{equation} 
\label{e:116pi}
\begin{aligned}
	{\mathcal M}^{\pi,c}_1(T) & \coloneqq \sum_{1\le i < k <j\le \lfloor\tfrac{T}{c}\rfloor} \mathbf{1}_{ \{ \#\pi \cap B_{(i,j)}^{c}\geq 1, \, \#\pi \cap B_{(k,j)}^{c}\geq 1\}}=0.
\end{aligned}
\end{equation}
\item \label{Def:004Pro4} Up to time $T$ there are no two squares in the same column that contain a point, i.e,  
\begin{equation} \label{e:116piII}
\begin{aligned}
	{\mathcal M}^{\pi,c}_2(T) & \coloneqq \sum_{1\le i<j < k\le \lfloor\tfrac{T}{c}\rfloor} \mathbf{1}_{ \{ \#\pi \cap B_{(i,j)}^{c}\geq 1, \, \#\pi \cap B_{(i,k)}^{c}\geq 1\}}=0.
\end{aligned}
\end{equation}
\end{enumerate}
\end{definition}

In the following we also write for all $T \in \mathbb{R}_{+}$,
\begin{equation}\label{e:126a}
{\mathcal E}^{c}(T):=\big\{\pi\subset\Delta_+^2:\,\pi\mbox{ is }c\mbox{-decomposable up to time $T$}\big\},
\end{equation}
and let for a nice point cloud $\pi\subset\Delta_+^2$,
\begin{equation} \label{e:131}
\sigma^{\pi,c}:=\max\big\{T \in \mathbb{R}_{+}:\,\pi\in{\mathcal E}^{c}(T)\big\}.
\end{equation}

\begin{lemma}[Decomposability up to time $N$]\label{Lemma:002} 
Let  $\Pi$ be a Poisson point process on $\Delta_+^2$ with the Lebesgue measure $\lambda$ as the intensity measure. Then for all $c\in(0,\frac{1}{2})$ and $T \in \mathbb{R}_{+}$ such that $\lfloor\tfrac{T}{c}\rfloor\geq 3$,
\begin{equation}
\label{e:088}
\begin{aligned} 
\mathbb{P}\big(\sigma^{\Pi,c} > T\big) 
\geq 
1-\tfrac{3}{4}c^2T^2-\tfrac{1}{2}cT-\tfrac{1}{3}cT^3.
\end{aligned}
\end{equation} 
In particular,  $\lim_{c\to 0}\mathbb{P}\big(\sigma^{\Pi,c} > T\big)= 1$ for all $T \in \mathbb{R}_{+}$. 
\end{lemma}
\begin{proof} Note that by the union bound and Markov's inequality, 
\begin{equation}
\begin{aligned} 
\label{e:090}
\mathbb{P}\big(\Pi\in{\mathcal E}^{c}(T)\big) & \geq 
1-\sum_{1\le i<j\le\lfloor\frac{T}{c}\rfloor}\mathbb{P}\big(\# \Pi\cap B_{(i,j)}^{c} \ge 2\big)
-\sum_{1\le j\leq \lfloor\frac{T}{c}\rfloor}\mathbb{P}\big(\# \Pi\cap E_{j}^{c} \ge 1\big) \\
&\quad \quad 
-\sum_{1\le i < k <j\le \lfloor\tfrac{T}{c}\rfloor} \mathbb{P}(\#\Pi\cap B_{(i,j)}^{c}\ge 1, \, \#\Pi \cap B_{(k,j)}^{c}\ge 1 \big)
\\
&\quad \quad 
-\sum_{1\le i<j < k\le \lfloor\tfrac{T}{c}\rfloor} \mathbb{P}(\#\Pi \cap B_{(i,j)}^{c}\ge 1, \, \#\Pi \cap B_{(i,k)}^{c}\ge 1\big).
\end{aligned}
\end{equation}
\noindent Next, by using that for all $x>0$, $1-e^{-x}\le x$ and $1-e^{-x}-xe^{-x}\le \frac{x^2}{2}$, we obtain that
\begin{equation}
\begin{aligned} 
\mathbb{P}\big(\Pi\in{\mathcal E}^{c}(T)\big)   &\ge
1-\tfrac{1}{2}c^{-2}T^2\big(1-e^{-c^2}-c^2e^{-c^2}\big) \\
& \quad \quad \quad -c^{-1}T \big(1-e^{-\frac{c^2}{2}}\big)-(\tfrac{1}{3}c^{-3}T^3+\tfrac{1}{2}c^{-2}T^2)\big(1-e^{-c^2}\big)^2 \\
&\ge
1-\tfrac{3}{4}c^2T^2-\tfrac{1}{2}cT-\tfrac{1}{3}cT^3.
\end{aligned}
\end{equation} 
\end{proof}

For $c >0$ we introduce the $c$-RGRG, denoted by  $X^{(c)}=(X^{(c)}(t))_{t \geq 0}$, by letting
\begin{equation} 
\label{e:053II}
\begin{aligned}
X^{(c)}(t) 
&:=\mbox{RGRG}^{\pi}\big( t \wedge \sigma^{\pi,c}\big), \quad \text{for} \quad t \in \mathbb{R}_{+},
\end{aligned}
\end{equation}
\noindent where $\pi$ is a nice point cloud on $\Delta^2_+$. 
When instead of a nice point cloud, a Poisson point process on $\Delta_+^2$ with the Lebesgue measure $\lambda$ as intensity measure is used, we denote the $c$-RGRG as $X^{\Pi,(c)}$.

\begin{corollary}[Convergence to the RGRG]\label{P:001II} 
Let $\Pi$ be a Poisson point process on $\Delta_+^2$ with the Lebesgue measure $\lambda$ as its intensity measure. 
Then, for all $0 \leq T < \infty$, we have that,
\begin{equation}\begin{aligned}  \label{e:033IIb}
\lim_{c\downarrow 0 }\sup_{t \in [0,T]} d_{\mbox{\tiny {\rm GH}}}\big(X^{\Pi,(c)}(t), X(t)\big) = 0, \quad \text{in probability}.
\end{aligned}\end{equation} 
\noindent In particular,
\begin{equation} \label{e:033II}
(X^{\Pi,(c)}( t))_{t\ge 0} \xRightarrow[c\downarrow 0 ]{} (X(t))_{t\ge 0},
\end{equation}
where $\Rightarrow$ stands for weak convergence of random variables with values in the Skorohod space ${\mathcal D}(\mathbb{R}_{+},\mathbb{T})$.
\end{corollary}

\begin{proof}
Note that \eqref{e:033IIb} is a mere consequence of \eqref{e:053II}, Lemma \ref{Lemma:002}, and the union bound. Clearly, \eqref{e:033IIb} implies \eqref{e:033II}.
\end{proof}

\subsection{Behavior of the RGRG map between re-grafting events}\label{subsectionBehaviorOfRGRGBetweenRegrafting}
In this subsection, we show that between re-grafting events, the RGRG map undergoes exclusively root growing transitions, as outlined in \eqref{e:030}. 

Fix $c > 0$ and let $\pi\subset \Delta_+^2$ be a nice point cloud; see Definition \ref{Def:002}. Recall the definition of the sets $B_{(i,j)}^{c}$'s in \eqref{e:060b}. For $(i,j)\in \mathbb{N}\times \mathbb{N}$ with $1\le i<j$, we define
\begin{equation}\begin{aligned} \label{e:060ii}
Z^{\pi,c}_{(i,j)} \coloneqq \mathbf{1}_{\{\pi \cap B_{(i,j)}^{c}\neq \emptyset \}}.
\end{aligned}\end{equation} 

If $Z^{\pi,c}_{(i,j)}=1$, then we call $j$ a cut-time and $i$ a cut-point. 
The following lemma proves that if a $c$-decomposable nice point cloud contains no points within the interval $[0,kc]$, for any $k \in \mathbb{N}_{0}$, then the RGRG map exhibits path-like behavior, characterized solely by root-growth transitions.

\begin{lemma}\label{lemmaBranchescRGRG}
Fix $c > 0$ and $0 \leq T < \infty$ such that $T/c \geq 2$. Let $\pi\subset \Delta_+^2$ be a nice point cloud. Fix any $0 \leq k\leq \lfloor T/c \rfloor$, suppose that 
\begin{equation} \label{IndizeroII}
Z^{\pi,c}_{(i,j)} = 0, \quad \mbox{for all} \quad  1 \leq i < j \leq k.
\end{equation}
(If $k < 2$, then \eqref{IndizeroII} is satisfied by vacuity.) If $\sigma^{\pi,c} > T$, then ${\rm RGRG}^{\pi}(kc)$ is the metric space equivalent to the interval $[0, kc]$ endowed with the Euclidean distance. 
\end{lemma}

\begin{proof}
If $k=0$, then our claim is trivial since in this case ${\rm RGRG}^{\pi}(0)$ is a point. For $k = 1$, given that $\sigma^{\pi,c} > T$, the nice point cloud is $c$-decomposable. Consequently, it fulfils property \ref{Def:004Pro2} of Definition \ref{Def:004} and \ref{Def:002Pro1} in Definition \ref{Def:002}, namely that $\pi$ contains no points within the set
\begin{align}
\big\{(x,y)\in\mathbb{R}_{+}^{2}: x \in [0, c], y \in (0, x]\}
\end{align}
\noindent (that is, $\tau_1(\pi) > c$). Therefore, ${\rm RGRG}^{\pi}$ undergoes only root-growth transitions within the interval $[0, c]$, resulting in ${\rm RGRG}^{\pi}(c)$ being a metric space equivalent to the interval $[0, c]$ endowed with the Euclidean distance.

For $2\leq k\leq \floor{T/c}$, the nice point cloud $\pi$ is $c$-decomposable, given that $\sigma^{\pi,c} > T$. Consequently, properties \ref{Def:004Pro2} of Definition \ref{Def:004}, \ref{Def:002Pro1} of Definition \ref{Def:002}, and \eqref{IndizeroII} collectively imply that $\pi$ contains no points within the set 
\begin{align}
\big\{(x,y)\in\mathbb{R}_{+}^{2}: x \in [0, kc], y \in (0, x]\}
\end{align}
\noindent  (i.e., $\tau_1(\pi) > c$). As a result, ${\rm RGRG}^{\pi}$ undergoes exclusively root-growth transitions within the interval $[0, kc]$, leading to ${\rm RGRG}^{\pi}(kc)$ being a metric space equivalent to the interval $[0, kc]$ endowed with the Euclidean distance.
\end{proof}

The next step involves generalizing the preceding lemma. Intuitively, we show that the RGRG map, during any time interval devoid of points from the nice point cloud, undergoes exclusively root-growing transitions. To formally state this result, we require the following definitions.

Fix $c > 0$ and $0 <T < \infty$ such that $T/c \geq 2$. For $2 \leq k \leq \lfloor T/c \rfloor$, suppose that there exists an integer $1 \leq n \leq \lfloor k/2 \rfloor$ and indices $1 \leq i_{m} < j_{m} \leq k$, for $1 \leq m \leq n$, such that $i_{1}, \dots, i_{n}, j_{1}, \dots, j_{n}$ are all distinct, 
\begin{align}
Z^{\pi,c}_{(i_{m},j_{m})} = 1 \quad \text{and} \quad Z^{\pi,c}_{(i,j)} = 0,
\end{align}
\noindent for all $1 \leq i < j \leq k$ with either $i$ or $j$ not in $\{ i_{1}, \dots, i_{n}, j_{1}, \dots, j_{n} \}$.

In particular, the latter implies the existence of at least one point of $\pi$ within the set 
\begin{align}
B_{(i_m,j_m)}^{c}= \big\{(x,y)\in\mathbb{R}_{+}^{2}: \,(j_m-1)c\le x< j_mc,(i_m-1)c\leq  y<i_mc\big\},
\end{align}
\noindent for $1\leq m\leq n$. Moreover, by property \ref{Def:002Pro1} of Definition \ref{Def:002}, $\pi$ has not point within the sets
\begin{align}
\{(j_m-1)c\} \times [0,i_mc] \quad  \text{and} \quad \mathbb{R}_+ \times \{(i_m-1)c\},
\end{align}
\noindent for $1\leq m\leq n$. 

Now, we remove the intervals $\cup_{m=1}^n ((i_m-1)c,i_mc)\cup ((j_m-1)c,j_mc)$ from $[0,kc]$. More precisely, let $i_{1}^{\prime}, \ldots, i_{2n}^{\prime}$ be the sequence $i_{1}, \ldots, i_{n}, j_{1}, \ldots, j_{n}$ ordered in increasing order, that is, $i_{1}^{\prime} < \cdots < i_{2n}^{\prime}$. Set $i_{0}^{\prime} =0$ and $i_{2n+1}^{\prime} =k +1$. For $m=0, \dots, 2n$, define the following sub-intervals,
\begin{align}
I_{m}^{c} = \begin{cases}
\emptyset  & \quad \text{if} \quad [i^{\prime}_{m}c, (i^{\prime}_{m+1}-1)c)= \emptyset, \\
[i^{\prime}_{m}c, (i^{\prime}_{m+1}-1)c] &\quad  \text{otherwise}.
\end{cases}
\end{align}
\noindent For $m=0, \dots, 2n$, let ${\rm Branch}^{c}(I_{m}^{c})$ be the metric subspace of $\mbox{RGRG}^{\pi}(kc)$ restricted to the interval $I_{m}^{c}$ (that is, we equip ${\rm Branch}^{c}(I_{m}^{c})$ with the metric induced by the restriction of $d_{kc}^{\pi}$ to $I_{m}^{c}$). Note that due to the way the intervals $I_{m}^{c}$ are defined, then ${\rm Branch}^{c}(I_{m}^{c})$ is a compact metric space when equipped with the euclidean distance. Finally, we consider the metric space $\bigcup_{m=0}^{2n} {\rm Branch}^{c}(I_{m}^{c})$ formed by the union of the metric spaces ${\rm Branch}^{c}(I_{m}^{c})$. This union is naturally endowed with the metric induced by the restriction of $d_{kc}^{\pi}$ to this set.

\begin{lemma}\label{lemmaBranchescRGRGGeneralCase} 
Fix $c > 0$ and $0 <T < \infty$ such that $T/c \geq 2$ . Let $\pi\subset \Delta_+^2$ be a nice point cloud. Fix any $2 \leq k \leq \lfloor T/c \rfloor$ and suppose that there exists an integer $1 \leq n \leq \lfloor k/2 \rfloor$ and indices $1 \leq i_{m} < j_{m} \leq k$, for $1 \leq m \leq n$, such that $i_{1}, \dots, i_{n}, j_{1}, \dots, j_{n}$ are all distinct, satisfy
\begin{align}\label{eqnBranchescRGRGGeneralCase}
Z^{\pi,c}_{(i_{m},j_{m})} = 1 \quad \text{and} \quad Z^{\pi,c}_{(i,j)} = 0,
\end{align}
\noindent for all $1 \leq i < j \leq k$ with either $i$ or $j$ not in$\{ i_{1}, \dots, i_{n}, j_{1}, \dots, j_{n} \}$. 

Suppose that $\sigma^{\pi,c} > T$. Then, 
\begin{enumerate}[label=(\textbf{\roman*})]
\item \label{lemmaBranchescRGClaim1} For $m=0, \dots, 2n$, ${\rm Branch}^{c}(I_{m}^{c})$ is the metric space equivalent to the interval $I_{m}^{c}$ endowed with the Euclidean distance, with length $(i'_{m+1}-i'_m-1)c \vee 0$.

\item \label{lemmaBranchescRGClaim2}
\begin{align} \label{eqCompaHoleRGRG}
d_{\rm H}\Big({\rm RGRG}^{\pi}(kc), \bigcup_{m=0}^{2n} {\rm Branch}^{c}(I_{m}^{c})\Big)\leq 2c.
\end{align}

\item \label{lemmaBranchescRGClaim3} Consider $m=0, \dots, 2n$ such that $I_{m}^{c} \neq \emptyset$. Let $y_{1}, y_{2} \in I_{m}^{c}$ such that $y_{1}<y_{2}$. Then,
\begin{align} 
c (u_{2}-u_{1}-1) \vee 0 \leq d_{{\rm RGRG}^{\pi}(kc)}(y_{1},y_{2}) \leq c (u_{2}-u_{1}-1) \vee 0 + 2c,
\end{align}
\noindent where $u_{1}, u_{2} \in \{i_{m}^{\prime}+1, \dots, i_{m+1}^{\prime} \}$ are such that $u_{1}\leq u_{2}$, $y_{1} \in [(u_{1}-1)c,u_{1}c )$ and $y_{2} \in [(u_{2}-1)c,u_{2}c)$.

\item \label{lemmaBranchescRGClaim4} Consider different $m_{1}, m_{2}\in \{0, \dots, 2n\}$ such that $I_{m_{1}}^{c} \neq \emptyset$ and $I_{m_{2}}^{c} \neq \emptyset$.  Let $y_{1}  \in I_{m_{1}}^{c}$ and $y_{2}  \in I_{m_{2}}^{c}$ such that $y_{1}<y_{2}$. For $i =1,2$, 
let $\tilde{y}_{i}$  be the boundary point in $I_{m_{i}}^{c}$ that is in the path in ${\rm RGRG}^{\pi}(kc)$ from  $y_{1}$ to $y_{2}$.
Let 
\begin{align}
M_{y_{1},y_{2}}^{c} \coloneqq \{m \in  \{0, \dots, 2n\}: \, {\rm Branch}^{c}(I_{m}^{c}) \, \, \text{is contained in the path from} \, y_{1} \, \,  \text{to} \, \, y_{2} \}.
\end{align}
Then,
\begin{align}
& d_{{\rm RGRG}^{\pi}(kc)}(y_{1},y_{2}) \geq  c (u_{2}+u_{4}-u_{1}-u_{3}-2) \vee 0 + \sum_{m \in M_{y_{1},y_{2}}^{c}} c (i_{m+1}^{\prime} - i_{m}^{\prime}-1) \vee 0
\end{align}
\noindent and 
\begin{align}
&  d_{{\rm RGRG}^{\pi}(kc)}(y_{1},y_{2}) \leq c (u_{2}+u_{4}-u_{1}-u_{3}-2) \vee 0 + \sum_{m \in M_{y_{1},y_{2}}^{c}} c (i_{m+1}^{\prime} - i_{m}^{\prime}-1)\vee 0 + 4c + 2cn,
\end{align}
\noindent where $u_{1}, u_{2} \in \{i_{m_{1}}^{\prime}+1, \dots, i_{m_{1}+1}^{\prime}\}$ and $u_{3}, u_{4} \in \{i_{m_{2}}^{\prime}+1, \dots, i_{m_{2}+1}^{\prime}\}$ are such that $u_{1}\leq u_{2}$, $u_{3}\leq u_{4}$, $y_{1} \wedge \tilde{y}_{1} \in [(u_{1}-1)c,u_{1}c)$, $y_{1} \vee \tilde{y}_{1} \in [ (u_{2}-1)c,u_{2}c)$, $y_{2} \wedge \tilde{y}_{2} \in [( u_{3}-1)c,u_{3}c)$ and $y_{2} \vee \tilde{y}_{2}  \in [ (u_{4}-1)c,u_{4}c)$. 
\end{enumerate}
\end{lemma}

\begin{proof}
First we prove \ref{lemmaBranchescRGClaim1}. If $I_{m}^{c}=\emptyset$, then our claim follows immediately. Then, we suppose that $I_{m}^{c} \neq \emptyset$. Recall that the nice point cloud $\pi$ is $c$-decomposable since $\sigma^{\pi,c} > T$. Thus, by properties \ref{Def:004Pro2} of Definition \ref{Def:004}, \ref{Def:002Pro1} of Definition \ref{Def:002}, and \eqref{eqnBranchescRGRGGeneralCase}, $\pi$ contains no points within the set
\begin{align}
\big\{(x,y)\in\mathbb{R}_{+}^{2}: x \in I_{m}^{c}, y \in (0, x]\}.
\end{align}
\noindent Then, ${\rm RGRG}^{\pi}$ undergoes exclusively root-growth transitions within the interval $I_{m}^{c}$, leading to ${\rm RGRG}^{\pi}(\max I_{m}^{c})$ restricted to $I_{m}^{c}$ being a metric space equivalent to the interval $I_{m}^{c}$ endowed with the Euclidean distance.

To conclude that ${\rm Branch}^{c}(I_{m}^{c})$ is the metric space equivalent to the interval $I_{m}^{c}$ endowed with the Euclidean distance, it enough to justify that $\pi$ contains no points within the set
\begin{align}
\big\{(x,y)\in\mathbb{R}_{+}^{2}: x \in (\max I_{m}^{c}, kc], y \in I_{m}^{c}\}.
\end{align}
\noindent But this is a consequence of properties \ref{Def:004Pro2} towards \ref{Def:004Pro4} of Definition \ref{Def:004}, \ref{Def:002Pro1} of Definition \ref{Def:002}, and \eqref{eqnBranchescRGRGGeneralCase}. 

Next, we prove \ref{lemmaBranchescRGClaim2}. This follows directly from an equivalent definition of the Hausdorff distance between metric subspaces $\mathcal{X}$ and $\mathcal{Y}$ of a metric space $(\mathcal{Z},d)$ (see e.g., \cite[Exercise 7.3.2]{MR1835418})
\begin{equation}\label{eqnEquivalentDefinitionHausdorffDistance}
d_{\rm H}(\mathcal{X},\mathcal{Y}):=\max \left\{\sup_{x\in \mathcal{X}}d(x,\mathcal{Y}),\sup_{y\in \mathcal{Y}}d(\mathcal{X},y) \right\}.
\end{equation}

Now we prove Parts \ref{lemmaBranchescRGClaim3} and \ref{lemmaBranchescRGClaim4} of the lemma. 
First, we remark that if, say $y_2=\max I^c_m$, then necessarily $y_{2}=(i_{m+1}^{\prime} -1)c$, which justifies  that $u_2$ can take the value $i_{m+1}^{\prime}$. 
Now, the proof of \ref{lemmaBranchescRGClaim3} and \ref{lemmaBranchescRGClaim4} follows by decomposing the path connecting $y_{1}$ with $y_{2}$ within ${\rm RGRG}^{\pi}(kc)$. 
Indeed, note that $\sum_{m \in M_{y_{1},y_{2}}^{c}} c (i_{m+1}^{\prime} - i_{m}^{\prime}-1) \vee 0$ is the size of each ${\rm Branch}^{c}(I_{m}^{c})$ contained in the path from $y_{1}$ to $y_2$, by \ref{lemmaBranchescRGClaim1}. The term $c (u_{2}-u_{1}-1) \vee 0+2c$ accounts for the distance between $y_1$ and $\tilde y_1$ in $I^c_m$, by \ref{lemmaBranchescRGClaim3}. Finally, the term $2cn$ bounds the removed intervals $\cup_{m=1}^n ((i_m-1)c,i_mc)\cup ((j_m-1)c,j_mc)$.
\end{proof}

\section{The Aldous--Broder chain and the skeleton chain}
\label{S:AldousBroder}
Recall from the introduction the map $\mathrm{AB}^{\gamma}$ such that ${\mathrm{AB}}^{\gamma}(0):=(\gamma(0), \emptyset,\gamma(0))$ (that is, a single-vertex tree) and for $n\in\mathbb{N}$, is the rooted tree graph given by
\begin{equation}\begin{aligned} 
\label{e:102AB}
{\mathrm{AB}}^{\gamma}(n)
:=
\Big({\rm R}^\gamma([0,n]),\big\{\{x,\gamma(L_x(n-1)+1)\};x\in {\rm R}^\gamma([0,n-1])\setminus\{\gamma(n)\}\big\},\gamma(n)\Big),
\end{aligned}
\end{equation}
where for  a path $\gamma:\,\mathbb{N}_0\to V$ 
on a finite simple, connected graph $G=(V,E)$, $n\in\mathbb{N}_{0}$ and $x\in {\rm R}^\gamma([0,n])$, 
\begin{equation} 
\label{e:136AB}
L_x(n)=L_x\big([0,n]\big)
:=
\max\big\{k\in[0,n]:\,\gamma(k)=x\big\}.
\end{equation}

The goal of this section is to construct with the {\em skeleton chain} another discrete-time chain with values in rooted tree graphs that is close in 
Skorokhod distance with respect to the pointed Gromov-Hausdorff-distance when rooted tree graphs are encoded as pointed metric spaces. For that purpose we shall use that for typical paths we can separate short from long loops, where a loop is short if it scales down to a point under our scaling. Thus for time indices within short loops it should not matter whether we erase one edge as required in the Aldous-Broder move or erase the whole loop. We therefore think of those time indices as ghost indices.     
In Subsection~\ref{Sub:ABPoisson} we define the sets of {\em ghost indices} up to a certain time, and prove in Lemma~\ref{InteinSke} that under certain typical events, the ghost indices are indeed those that would be locally erased on that interval.  
In Subsection~\ref{subsecSkeletonChain} we then introduce the {\em skeleton chain} as the chain which for a given time is the rooted subgraphs obtained from the rooted trees of the Aldous-Broder chain at that time 
spanned by the non-ghost time indices up to that time. For this skeleton chain is then easy to prove the convergence to the limit dynamics of root growth and re-grafting after our rescaling. In Subsection~\ref{Sub:rsAB} we then bound the Gromov-Hausdorff distance between  the skeleton chain and the Aldous-Broder chain uniformly in time (Proposition~\ref{DerterLemma1}).

\subsection{The sets of ghost indices}
\label{Sub:ABPoisson}
Fix $s\in\mathbb{N}$ with $s\ge 2$, and
let $\gamma:\na_0\to V$ be a path on a finite, simple, connected graph $G=(V,E)$.

\begin{definition}[Ghost index chain]\label{DefGhostI}
Set ${\mathcal G}^{\gamma,s}(0) := \emptyset$. Suppose that we have constructed ${\mathcal G}^{\gamma,s}(n-1)$, for some $n \geq 1$. Define 
\begin{equation}\label{eqnGhostIndexUpTo_nAndUpTo_n-1}
{\mathcal G}^{\gamma,s}(n):={\mathcal G}^{\gamma,s}(n-1)\cup 
\big\{m\in [0,n-1]\setminus {\mathcal G}^{\gamma,s}(n-1):\,\mbox{$m$ satisfies ($\mathbf{G}_{n}$)}\big\},
\end{equation} 	
\noindent where we say that $m$ satisfies ($\mathbf{G}_{n}$) if
\begin{itemize}
\item[ ($\mathbf{G}_{n}$) ] It holds that
\begin{equation}\label{eqnGhostIndexDefinitionSmallLoop}
\gamma(n)\in {\rm R}^\gamma ([n-s+1,m]\setminus \mathcal{G}^{\gamma,s}(n-1)),
\end{equation} 
\noindent and that for $m_{1} \coloneqq \max \{ h \in [n-s+1,m]\setminus \mathcal{G}^{\gamma,s}(n-1): \gamma(h) = \gamma(n)\}$, we have that
\begin{align} \label{eqnGhostIndexDefinitionNoLongLoops}
\gamma(k) \not \in{\rm R}^\gamma\big([0,k-1]\setminus {\mathcal G}^{\gamma,s}(n-1)\big), \quad \text{for all} \quad k \in [m_1,n-1]\setminus {\mathcal G}^{\gamma,s}(n-1).
\end{align}			
\end{itemize}
\end{definition}

\begin{figure}
\includegraphics[width=14cm]{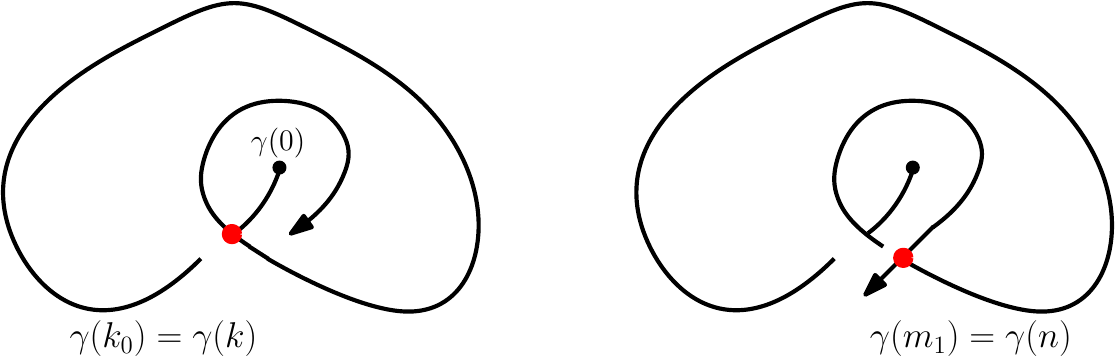}
\caption{We illustrate why we impose \eqref{eqnGhostIndexDefinitionNoLongLoops} to later ensure that the subgraph restricted to the non-ghost indices is connected. A loop shorter than $s$ is closing at time index $n+1$. There exists $k\in[m_1,n]\setminus{\mathcal G}^{\gamma,s}(n)$ and $k_0\in[0,k-1]\setminus{\mathcal G}^{\gamma,s}(n)$ with $\gamma(k_0)=\gamma(k)$. If $k-k_0>s$, a loop longer than $s$ is closed at time index $k$. If we would erase the recently formed loop $[m_1,...,n]$, the graph would disconnect. We therefore need a condition that prevents that indices in $\{m_1,...,n\}$ are declared to be in ${\mathcal G}^{\gamma,s}(n+1)$.}     
\label{figGhostIndexPretzelV2}
\end{figure}

If $m\in [0,n-1]\setminus {\mathcal G}^{\gamma,s}(n-1)$  is a ghost index in ${\mathcal G}^{\gamma,s}(n)$, then we will say that it satisfies \eqref{eqnGhostIndexDefinitionSmallLoop} and \eqref{eqnGhostIndexDefinitionNoLongLoops} in ($\mathbf{G}_{n}$).

From the definition of ghost index chain, note that it will be important to differentiate between loops of the path $\gamma$ with length smaller than $s$.
Recall from the discussion above Lemma \ref{lemmaLongLoop}, the definition of short and long loops. We will frequently consider paths that satisfy the following property for some fixed $s',r,N \in \mathbb{N}$ such that $s'<r$:

\begin{description}
\item[No loops of intermediate length]\label{assumptionNoLoopsOfIntermediateLength} If $\gamma(m_1)=\gamma(m_2)$, for some $m_1,m_2\in [0,N]$, then $|m_2-m_1|\notin [s^{\prime},r]$. 
\end{description}
\noindent In subsequent sections, we will show that if the path is a lazy random walk on $G$, then with high probability it has no loops of intermediate length over an adequate interval (see Definition \ref{Def:006} and Proposition \ref{ProDescompo}).

\begin{remark}[laziness indices immediately become ghost indices] Let $\gamma:\na_0\to V$ be a path on a finite, simple, connected graph $G=(V,E)$. If $\gamma(n)=\gamma(n+1)$ for some $n\in\mathbb{N}_0$, then $n\in{\mathcal G}^{\gamma,s}(n+1)$, or equivalently, $n \in{\mathcal G}^{\gamma,s}(m)$ for all $m\ge n+1$.
\label{Rem:002}
\end{remark}

Recall from \eqref{e:023} the set ${\rm NE}^{\gamma, s}(A)$ of $(A,s)$-locally non-erased indices. Recall also from \eqref{e:023II} that ${\rm R}^{\gamma}({\rm NE}^{\gamma, s}(A))$ denotes the $(A,s)$-locally non-erased path. 
The next result will play a crucial role later in the proof of the main result of this work.
We will use it in Corollary \ref{LemmaIdenSeDecomp}, to compare segments of the subgraphs spanned by the non-ghost time indices of the Aldous-Broder chain and the locally non-erased paths, whenever the path on such segments satisfy certain \emph{only ghost indices}
conditions.

\begin{lemma} \label{InteinSke}
Fix $r,s, s^{\prime} \in \na$ with $r \geq 3s+1 \geq 18s^{\prime} +1$. Let $A \subset \mathbb{N}_{0}$ be a finite interval such that $\min A \geq s$ and $\# A \geq 3s+1$. Fix $N \in \mathbb{N}$ such that $N \geq \max A$. Let $\gamma:\na_0\to V$ be a path on a finite, simple, connected graph $G=(V,E)$ that satisfies Assumption  \hyperref[assumptionNoLoopsOfIntermediateLength]{No loops of intermediate length}, and the following properties:
\begin{description}
\item[Local cut points]\label{assumptionLocalCutPoints} For all $m \in [0, N-s]$, the interval $[m,m+s]$ contains a $2s^{\prime}$-local cut point.
\item[Only ghost indices (1)]\label{assumptionOnlyGhostIndices1} For all $m \in [\min A +s, \max A -s]$ and $2s^{\prime}$-local cut-point $\ell \in [m-s,m]$, if $n_{1} \in [\ell +1, m]$ and $\widetilde{m} \in [\ell, n_{1}-1] \setminus {\mathcal G}^{\gamma,s}(n_1-1)$ are such that $\gamma(n_1)\in {\rm R}^\gamma([n_1-s+1, \widetilde m] \setminus {\mathcal G}^{\gamma,s}(n_1-1))$, then 
\begin{align}\label{eqnAssumptionOnlyGhostIndices1Eqn1}
\gamma(k) \not \in{\rm R}^\gamma\big([0,k-1]\setminus {\mathcal G}^{\gamma,s}(n_1-1)\big), \quad \text{for all} \quad k \in [m_1,n_1-1]\setminus {\mathcal G}^{\gamma,s}(n_1-1),
\end{align}
\noindent where $m_{1} = \max\{ h \in [n_1-s+1, \widetilde{m}] \setminus {\mathcal G}^{\gamma,s}(n_1-1):  \gamma(h) = \gamma(n_1)\}$.
\item[Only ghost indices  (2)]\label{assumptionOnlyGhostIndices2} For all $m \in [\min A +s, \max A -s]$, if $n_2 \in [m+1, m+s-1]$ and  $m \in [\min A +s, \max A -s] \setminus {\mathcal G}^{\gamma,s}(n_2-1)$ are such that $\gamma(n_2)\in {\rm R}^\gamma([n_2-s+1, m] \setminus {\mathcal G}^{\gamma,s}(n_2-1))$, then
\begin{align}\label{eqnAssumptionOnlyGhostIndices2Eqn1} 
\gamma(k) \not \in{\rm R}^\gamma\big([0,k-1]\setminus {\mathcal G}^{\gamma,s}(n_2-1)\big), \quad \text{for all} \quad k \in [m_1,n_2-1] \setminus {\mathcal G}^{\gamma,s}(n_2-1),
\end{align}
\noindent where $m_{1} = \max\{ h \in [n_2-s+1, m] \setminus {\mathcal G}^{\gamma,s}(n_2-1):  \gamma(h) = \gamma(n_2)\}$.
\end{description}
\noindent Then, 
\begin{align} \label{InteinSkeA}
{\rm N E}^{\gamma,s}(A) = [\min A +s, \max A -s] \setminus {\mathcal G}^{\gamma,s}(n), \quad \text{for all} \quad n \in [\max A, N].
\end{align}
\end{lemma}	

\begin{remark}
In Assumption  \hyperref[assumptionOnlyGhostIndices1]{Only ghost indices (1)}, whenever $\ell=m$, then \eqref{eqnAssumptionOnlyGhostIndices1Eqn1} holds by vacuity.
\end{remark}

\begin{remark}
The assumptions of only ghost indices in Lemma \ref{InteinSke} will simplify to prove when a non-ghost index transitions to a ghost index as the path $\gamma$ evolves. Given these assumptions, 
if a small loop has been just created at time $n$ and $m$ is a non-ghost index at time $n-1$ satisfying \eqref{eqnGhostIndexDefinitionSmallLoop}, then it will be a ghost index at time $n$. 
The latter can be deduced from the following reformulation.
\begin{description}
\item[Only ghost indices (1)]\label{assumptionOnlyGhostIndices11}
Fix $m \in [\min A +s, \max A -s]$ and a $2s^{\prime}$-local cut-point $\ell \in [m-s,m]$.  Assume $n_1\in [\ell+1,m]$ and that $\widetilde{m} \in [\ell, n_1-1] \setminus {\mathcal G}^{\gamma,s}(n_1-1)$ satisfy \eqref{eqnGhostIndexDefinitionSmallLoop} in $(\bold G_{n_1})$.  If furthermore \eqref{eqnAssumptionOnlyGhostIndices1Eqn1} holds, then $\widetilde m$ also satisfies \eqref{eqnGhostIndexDefinitionNoLongLoops} in $(\bold G_{n_1})$. In particular, by Definition \ref{DefGhostI}, we have that $\widetilde m\in {\mathcal G}^{\gamma,s}(n_1)$.
\item[Only ghost indices (2)]
Fix $m \in [\min A +s, \max A -s]$. Assume $n_2 \in [m+1, m+s-1]$ and $m\in [\min A +s, \max A -s]\setminus {\mathcal G}^{\gamma,s}(n_2-1)$ satisfy \eqref{eqnGhostIndexDefinitionSmallLoop} in $(\bold G_{n_2})$.
If furthermore \eqref{eqnAssumptionOnlyGhostIndices2Eqn1} holds, then $m$ also satisfies \eqref{eqnGhostIndexDefinitionNoLongLoops} in $(\bold G_{n_2})$. In particular, by Definition \ref{DefGhostI}, we have that $m\in {\mathcal G}^{\gamma,s}(n_2)$.
\end{description}
\end{remark}

Before we prove Lemma \ref{InteinSke}, we need the following two technical results. The first result implies that whenever every index in the interval $[m+1,n]$, for some $n > m \geq s$, does not satisfy \eqref{eqnGhostIndexDefinitionSmallLoop}, then the ghost indices $\mathcal{G}^{\gamma,s}(\cdot)$ in that interval do not change.

\begin{lemma} \label{ConstSke}
Fix $r,s, m, n \in \na$ with $r \geq 3s+1$ and $n > m \geq s$. Let $\gamma:\na_0\to V$ be a path on a finite, simple, connected graph $G=(V,E)$. Suppose that, for all $k \in [m+1, n]$, $\gamma(k) \not \in {\rm R}^{\gamma}([k-s+1,m] \setminus {\mathcal G}^{\gamma,s}(k-1))$. Then, 
\begin{align} \label{ConstSkeA}
[m-s,m] \setminus {\mathcal G}^{\gamma,s}(m) = [m-s,m] \setminus {\mathcal G}^{\gamma,s}(k), \quad \text{for all} \quad k\in [m+1, n]. 
\end{align}
\end{lemma}

\begin{proof}
Since ${\mathcal G}^{\gamma,s}(m) \subseteq {\mathcal G}^{\gamma,s}(k)$, we have that $[m-s,m] \setminus {\mathcal G}^{\gamma,s}(k) \subseteq [m-s,m] \setminus {\mathcal G}^{\gamma,s}(m)$. Then, it only remains to prove that 
\begin{align} \label{eqnConstSke1}
[m-s,m] \setminus {\mathcal G}^{\gamma,s}(m) \subseteq [m-s,m] \setminus {\mathcal G}^{\gamma,s}(k).
\end{align}
\noindent The above is proved by showing that every $\widetilde m\in [m-s,m] \setminus {\mathcal G}^{\gamma,s}(m)$ does not satisfy \eqref{eqnGhostIndexDefinitionSmallLoop} in $(\bold G_k)$. Then, we proceed by induction.

First, we consider the case $k=m+1$. Note that $[m-s+2, m-s] = [m-s+2, m-s+1]= \emptyset$ and thus, by Definition \ref{DefGhostI}, $[m-s,m-s+1] \setminus {\mathcal G}^{\gamma,s}(m) \subseteq [m-s,m] \setminus {\mathcal G}^{\gamma,s}(m+1)$. On the other hand, we have by our assumption that $\gamma(m+1) \not \in {\rm R}^{\gamma}([m-s+2,\widetilde{m}] \setminus {\mathcal G}^{\gamma,s}(m))$, for all $\widetilde{m} \in [m-s+2,m] \setminus {\mathcal G}^{\gamma,s}(m)$. Thus, by Definition \ref{DefGhostI}, $[m-s+2,m] \setminus {\mathcal G}^{\gamma,s}(m) \subseteq [m-s,m]  \setminus {\mathcal G}^{\gamma,s}(m+1)$. By combining the two inclusions established earlier, our claim \eqref{eqnConstSke1} for $k=m+1$ is proven.

Suppose that $n > m+1$ and that 
\begin{align} \label{ConstSkeB}
[m-s,m] \setminus {\mathcal G}^{\gamma,s}(m) = [m-s,m] \setminus {\mathcal G}^{\gamma,s}(k-1),
\end{align}
\noindent for some  $k \in [m+2, n]$. If $k-s+1 > m$, we have that $[k-s+1, \widetilde{m}] = \emptyset$, for all $\widetilde{m} \in [m-s,m] \setminus {\mathcal G}^{\gamma,s}(m)$. Thus, by Definition \ref{DefGhostI}, $[m-s,m] \setminus {\mathcal G}^{\gamma,s}(m) \subseteq [m-s,m]  \setminus {\mathcal G}^{\gamma,s}(k)$, i.e., \eqref{eqnConstSke1} holds.

Now, suppose that $k-s+1 \leq m$. Note that $[k-s+1, \widetilde{m}] = \emptyset$, for all $\widetilde{m} \in [m-s,k-s] \setminus {\mathcal G}^{\gamma,s}(m)$. Thus, by Definition \ref{DefGhostI}, $[m-s,k-s] \setminus {\mathcal G}^{\gamma,s}(m) \subseteq [m-s,m] \setminus {\mathcal G}^{\gamma,s}(k)$. On the other hand, we have by our assumption and the induction hypothesis (i.e., \eqref{ConstSkeB}) that
\begin{align} \label{ConstSkeC}
\gamma(k) \not \in {\rm R}^{\gamma}([k-s+1,\widetilde{m}] \setminus {\mathcal G}^{\gamma,s}(k-1)),
\end{align}
\noindent for all $\widetilde{m} \in [k-s+1,m] \setminus {\mathcal G}^{\gamma,s}(m) = [k-s+1,m] \setminus {\mathcal G}^{\gamma,s}(k-1)$. Thus, $[k-s+1,m] \setminus {\mathcal G}^{\gamma,s}(m) \subseteq [m-s,m]  \setminus {\mathcal G}^{\gamma,s}(k)$. By combining the two inclusions obtained, we conclude \eqref{eqnConstSke1}.
\end{proof}

\begin{lemma}  \label{LemmaSimiSke}
Fix $r,s, s^{\prime}, m, N \in \na$ with $r \geq 3s+1 \geq 18s^{\prime} +1$ and $N \geq m \geq 2s$. Let $\gamma:\na_0\to V$ be a path on a finite, simple, connected graph $G=(V,E)$. Let $\ell \in [m-s, m]$ be a $2s^{\prime}$-local cut point of $\gamma$. Suppose also that $\gamma$ satisfies Assumption  \hyperref[assumptionNoLoopsOfIntermediateLength]{No loops of intermediate length}, and for the given $m$ and $\ell$:
\begin{description}
\item[Only ghost indices]\label{assumptionOnlyGhostIndicesII} If $n_1 \in [\ell +1, m]$  and $\widetilde{m} \in [\ell, n_1-1] \setminus {\mathcal G}^{\gamma,s}(n_1-1)$ are such that $\gamma(n_1)\in {\rm R}^\gamma([n_1-s+1, \widetilde{m}] \setminus {\mathcal G}^{\gamma,s}(n_1 -1))$, then 
\begin{align}\label{eqnassumptionOnlyGhostIndicesIIEqn1}
\gamma(k) \not \in{\rm R}^\gamma\big([0,k-1]\setminus {\mathcal G}^{\gamma,s}(n_1-1)\big), \quad \text{for all} \quad k \in [m_1,n_1-1] \setminus {\mathcal G}^{\gamma,s}(n_1-1), 
\end{align}
\noindent where $m_{1} = \max\{ h \in [n_1-s+1, \widetilde{m}] \setminus {\mathcal G}^{\gamma,s}(n_1 -1):  \gamma(h) = \gamma(n_1)\}$.
\end{description}
\noindent Then we have,
\begin{align}
[\ell,n] \setminus {\mathcal G}^{\gamma,s}(n) =  [\ell, n]\cap {\rm NE}^{\gamma,[m-s, m]}(n) , \quad \text{for all} \quad n \in [\ell,m].
\end{align}
\end{lemma}
\begin{remark}\label{remarkComparisonOnlyGhostIndices1AndOnlyGhostIndices}
Note that the only difference between Assumption  \hyperref[assumptionOnlyGhostIndices1]{Only ghost indices (1)} and Assumption  \hyperref[assumptionOnlyGhostIndicesII]{Only ghost indices} is that in the former we take \emph{any} $m\in [\min A+s,\max A-s]$ whereas in the latter we take a fixed $m\geq 2s$. 
\end{remark}
\begin{proof}
Note that $\{ \ell \}= [\ell,\ell] \setminus {\mathcal G}^{\gamma,s}(\ell) = [\ell, \ell]\cap {\rm NE}^{\gamma,[\ell-s, \ell]}(\ell) $ (i.e., our claim for $m =\ell$  holds). Then, we assume that $m-s \leq \ell < m$ and proceed by induction. The case $n =\ell$  should be clear. Then, suppose that 
\begin{align}\label{eqnLemmaSimiSke1}
[\ell,n-1] \setminus {\mathcal G}^{\gamma,s}(n-1) =  [\ell, n-1]\cap {\rm NE}^{\gamma,[m-s, m]}(n-1) , 
\end{align}
\noindent for some $n \in [\ell+1, m]$. Note that $n \in [\ell,n] \setminus {\mathcal G}^{\gamma,s}(n)$ and $n \in [\ell, n]\cap {\rm NE}^{\gamma,[m-s, m]}(n) $. Thus, it is enough to prove that $[\ell,n-1] \setminus {\mathcal G}^{\gamma,s}(n) = [\ell, n-1]\cap {\rm NE}^{\gamma,[m-s, m]}(n) $.

First, we show that $[\ell,n-1] \setminus {\mathcal G}^{\gamma,s}(n) \subseteq [\ell, n-1]\cap {\rm NE}^{\gamma,[m-s, m]}(n)$. Suppose that $[\ell,n-1] \setminus {\mathcal G}^{\gamma,s}(n) \neq \emptyset$ and consider $\widetilde{m} \in [\ell,n-1] \setminus {\mathcal G}^{\gamma,s}(n)$. In particular, $\widetilde{m} \in [\ell,n-1] \setminus {\mathcal G}^{\gamma,s}(n-1)$. 
To reach a contradiction, assume that $\widetilde m$ satisfies \eqref{eqnGhostIndexDefinitionSmallLoop} in $(\bold G_n)$ for such $n$, that is $\gamma(n) \in {\rm R}^\gamma ([n-s+1,\widetilde{m}]\setminus \mathcal{G}^{\gamma,s}(n-1))$. 
Recall Remark \ref{remarkComparisonOnlyGhostIndices1AndOnlyGhostIndices} and the reformulation of Assumption \hyperref[assumptionOnlyGhostIndices11]{Only ghost indices (1)}. 
By Assumption \hyperref[assumptionOnlyGhostIndicesII]{Only ghost indices} (using it with $n_{1}=n$),  we have that $\widetilde m$ satisfies \eqref{eqnGhostIndexDefinitionNoLongLoops} in $(\bold G_n)$ and thus, by Definition \ref{DefGhostI} we have $\widetilde{m} \in  {\mathcal G}^{\gamma,s}(n)$. The latter is a contradiction. 
Hence, by Definition \ref{DefGhostI}, we must have that
\begin{equation} \label{LemmaSimiSkeA}
\gamma(n) \not \in {\rm R}^\gamma ([n-s+1,\widetilde{m}]\setminus \mathcal{G}^{\gamma,s}(n-1)).
\end{equation}
\noindent Note that $n-s \leq \ell \leq \widetilde{m} \leq n-1$. Then, since $\ell \in [m-s,m]$ is a $2s^{\prime}$-local cut point (see Definition \ref{LocaCutP}), the Assumption  \hyperref[assumptionNoLoopsOfIntermediateLength]{No loops of intermediate length} implies that
\begin{equation} \label{LemmaSimiSkeB}
\gamma(n) \not  \in {\rm R}^\gamma ([\ell \vee (n-s^{\prime}+1),\widetilde{m}]\setminus \mathcal{G}^{\gamma,s}(n-1)).
\end{equation}
\noindent On the other hand, since by the induction hypothesis \eqref{eqnLemmaSimiSke1},
\begin{align} \label{LemmaSimiSkeC}
[\ell \vee (n-s^{\prime}+1),\widetilde{m}]\setminus \mathcal{G}^{\gamma,s}(n-1) & = [\ell \vee (n-s^{\prime}+1),\widetilde{m}]\cap ([\ell,n-1] \setminus \mathcal{G}^{\gamma,s}(n-1))  \nonumber \\
& = [\ell \vee (n-s^{\prime}+1),\widetilde{m}]\cap ({\rm NE}^{\gamma,[m-s, m]}(n-1) \cap [\ell, n-1]) ,
\end{align}
\noindent we conclude that 
\begin{equation} \label{LemmaSimiSkeD}
\gamma(n) \not \in {\rm R}^\gamma ([\ell \vee (n-s^{\prime}+1),\widetilde{m}]\cap {\rm NE}^{\gamma,[m-s, m]}(n-1)).
\end{equation}We also conclude from \eqref{LemmaSimiSkeC} that $\widetilde m\in {\rm NE}^{\gamma,[m-s, m]}(n-1)$. 
Moreover, again by Assumption  \hyperref[assumptionNoLoopsOfIntermediateLength]{No loops of intermediate length} and since $\ell$ is a $2s'$-local cut point, we have that
\begin{equation} \label{LemmaSimiSkeE}
\gamma(n) \not \in {\rm R}^\gamma ([m-s,\ell \vee (n-s^{\prime}+1)]).
\end{equation}
\noindent Therefore, it follows from \eqref{LemmaSimiSkeD} and \eqref{LemmaSimiSkeE} that $\widetilde{m} \in [\ell, n-1]\cap {\rm NE}^{\gamma,[m-s, m]}(n) $ (recall \eqref{e:102}).  \\

Next, we show that $[\ell, n-1]\cap{\rm NE}^{\gamma,[m-s, m]}(n)  \subseteq [\ell,n-1] \setminus {\mathcal G}^{\gamma,s}(n)$. The idea is to show that every $\widetilde{m}\in [\ell, n-1]\cap{\rm NE}^{\gamma,[m-s, m]}(n)$ does not satisfy \eqref{eqnGhostIndexDefinitionSmallLoop} in $(\bold G_n)$. Suppose that $[\ell, n-1] \cap {\rm NE}^{\gamma,[m-s, m]}(n) \neq \emptyset$ and consider $\widetilde{m} \in [\ell, n-1] \cap {\rm NE}^{\gamma,[m-s, m]}(n) $. In particular, by \eqref{e:102} and the induction hypothesis (i.e., \eqref{eqnLemmaSimiSke1}), note that
\begin{align} \label{LemmaSimiSkeF}
\widetilde{m} \in [\ell, n-1] \cap {\rm NE}^{\gamma,[m-s, m]}(n-1) = [\ell,n-1] \setminus {\mathcal G}^{\gamma,s}(n-1). 
\end{align}
\noindent On the other hand, by \eqref{e:102}, 
\begin{equation}
\gamma(n) \not \in {\rm R}^\gamma ({\rm NE}^{\gamma,[m-s, m]}(n-1) \cap [m-s,\widetilde{m}]).
\end{equation}
\noindent Since $\ell \in [m-s,m]$ is a $2s^{\prime}$-local cut point (see Definition \ref{LocaCutP}), Assumption  \hyperref[assumptionNoLoopsOfIntermediateLength]{No loops of intermediate length} implies  \eqref{LemmaSimiSkeE}, and thus \eqref{LemmaSimiSkeD} holds. Moreover, by using the induction hypothesis again and \eqref{LemmaSimiSkeC}, we deduce that \eqref{LemmaSimiSkeB} also holds. Finally, by using again Assumption  \hyperref[assumptionNoLoopsOfIntermediateLength]{No loops of intermediate length} and that $n-s\leq \ell\leq \widetilde m\leq n-1$, we obtain \eqref{LemmaSimiSkeA}. That is, by \eqref{LemmaSimiSkeF} and Definition \ref{DefGhostI}, $\widetilde{m} \in [\ell,n-1] \setminus {\mathcal G}^{\gamma,s}(n)$.
\end{proof}

We have now all the ingredients to prove Lemma \ref{InteinSke}.

\begin{proof}[Proof of Lemma \ref{InteinSke}]
Throughout the proof, we fix $n \in [\max A, N]$.

First, we show that ${\rm N E}^{ \gamma,s}(A) \subseteq [\min A +s, \max A -s] \setminus {\mathcal G}^{\gamma,s}(n)$. We proceed by contradiction. Suppose that ${\rm N E}^{\gamma,s}(A) \neq \emptyset$ and that there exits $m \in {\rm N E}^{\gamma,s}(A)$ such that $m \not \in  [\min A +s, \max A -s] \setminus {\mathcal G}^{\gamma,s}(n)$. Then, $m \in [\min A +s, \max A -s]  \cap {\mathcal G}^{\gamma,s}(n)$, i.e., by Definition \ref{DefGhostI}, there exists $\widetilde{n} \in \mathbb{N}$ such that $m < \widetilde{n} \leq n$, $m \in [\min A +s, \max A -s]  \setminus {\mathcal G}^{\gamma,s}(\widetilde{n}-1)$,
$m \in [\min A +s, \max A -s] \cap  {\mathcal G}^{\gamma,s}(\widetilde{n})$ and $m$ satisfies ($\mathbf{G}_{\widetilde{n}}$). Thus, by \eqref{eqnGhostIndexDefinitionSmallLoop} (with $n = \widetilde{n}$),
\begin{equation} \label{InteinSkeB}
\gamma(\widetilde{n})\in {\rm R}^\gamma ([\widetilde{n}-s+1,m] \setminus \mathcal{G}^{\gamma,s}(\widetilde{n}-1)). 
\end{equation}Note that such a time index $\widetilde n$ must satisfy $m<\widetilde n\leq m+s$ for the above to be true.
\noindent Moreover, by the Assumption  \hyperref[assumptionNoLoopsOfIntermediateLength]{No loops of intermediate length},
\begin{equation} \label{InteinSkeC}
\gamma(\widetilde{n})\in {\rm R}^\gamma ([\widetilde{n}-s^{\prime}+1,m] \setminus \mathcal{G}^{\gamma,s}(\widetilde{n}-1)). 
\end{equation} 
\noindent On the one hand, by Assumption  \hyperref[assumptionLocalCutPoints]{Local cut points}, we know that the interval $[m-s,m]$ contains a $2s^{\prime}$-local cut point, say $\ell$. Then, by the Definition \ref{LocaCutP} (of $2s^{\prime}$-local cut point), and \eqref{InteinSkeC}, we have that
\begin{equation} \label{InteinSkeD}
\gamma(\widetilde{n})\in {\rm R}^\gamma ([\ell \vee (\widetilde{n}-s^{\prime}+1),m] \setminus \mathcal{G}^{\gamma,s}(\widetilde{n}-1)). 
\end{equation} 
\noindent On the other hand, since $\ell \in [m-s, m]$ is a $2s^{\prime}$-local cut point, the Assumption  \hyperref[assumptionNoLoopsOfIntermediateLength]{No loops of intermediate length}, the inclusion $\mathcal{G}^{\gamma,s}(m) \subseteq \mathcal{G}^{\gamma,s}(\widetilde{n}-1)$, the Assumption  \hyperref[assumptionOnlyGhostIndices1]{Only ghost indices (1)} (which implies Assumption \hyperref[assumptionOnlyGhostIndicesII]{Only ghost indices} of Lemma \ref{LemmaSimiSke}), and Lemma \ref{LemmaSimiSke} imply that 
\begin{align} \label{InteinSkeE}
& [\ell \vee (\widetilde{n}-s^{\prime}+1),m] \setminus \mathcal{G}^{\gamma,s}(\widetilde{n}-1)\\
& = \Big(\big([\ell ,m] \setminus \mathcal{G}^{\gamma,s}(m)\big) \setminus \mathcal{G}^{\gamma,s}(\widetilde{n}-1) \Big)\cap ([\ell \vee (\widetilde{n}-s^{\prime}+1),m] \setminus  \mathcal{G}^{\gamma,s}(\widetilde{n}-1))\\
& = {\rm NE}^{\gamma,[m-s, m]}(m) \cap ([\ell \vee (\widetilde{n}-s^{\prime}+1),m] \setminus  \mathcal{G}^{\gamma,s}(\widetilde{n}-1)).
\end{align}\noindent By \eqref{InteinSkeD} and \eqref{InteinSkeE},
\begin{equation} \label{InteinSkeF}
\begin{split}
\gamma(\widetilde{n})& \in {\rm R}^\gamma ({\rm NE}^{\gamma,[m-s, m]}(m) \cap ([\ell \vee (\widetilde{n}-s^{\prime}+1),m] \setminus  \mathcal{G}^{\gamma,s}(\widetilde{n}-1))),
\end{split}
\end{equation}
\noindent which implies that $\gamma(\widetilde{n}) \in {\rm R}^\gamma ({\rm NE}^{\gamma,[m-s, m]}(m) )$ and thus, $m \not \in {\rm N E}^{s, \gamma}(A)$ (recall \eqref{e:023}). This is a contradiction and thus, ${\rm N E}^{s, \gamma}(A) \subseteq [\min A +s, \max A -s] \setminus {\mathcal G}^{\gamma,s}(n)$. \\

Next, we show that $[\min A +s, \max A -s] \setminus {\mathcal G}^{\gamma,s}(n) \subseteq {\rm N E}^{s, \gamma}(A)$. Suppose that $[\min A +s, \max A -s] \setminus {\mathcal G}^{\gamma,s}(n) \neq \emptyset$ and that $m \in [\min A +s, \max A -s] \setminus {\mathcal G}^{\gamma,s}(n)$. In particular, $m \in [\min A +s, \max A -s] \setminus {\mathcal G}^{\gamma,s}(\widetilde{n})$, for all $\widetilde{n} \in [m+1, m+s-1]$ (recall that $n \geq \max A$). 
To reach a contradiction, assume that such an $m$ satisfies \eqref{eqnGhostIndexDefinitionSmallLoop} in $(\bold G_{\widetilde n})$ for some $\widetilde{n} \in [m+1, m+s-1]$, that is $\gamma(\widetilde n)\in {\rm R}^\gamma([\widetilde n-s+1, m] \setminus {\mathcal G}^{\gamma,s}(\widetilde n-1))$. 
But then, by the Assumption  \hyperref[assumptionOnlyGhostIndices2]{Only ghost indices (2)}  (with $n_{2}=\widetilde{n}$), we have that $m$ also satisfies \eqref{eqnGhostIndexDefinitionNoLongLoops} in $(\bold G_{\widetilde n})$. 
By Definition \ref{DefGhostI}, this contradicts that $m \not \in {\mathcal G}^{\gamma,s}(\widetilde{n})$.
Thus,
\begin{align} \label{InteinSkeG}
\gamma(\widetilde n) \not \in {\rm R}^{\gamma}([\widetilde{n}-s+1,m] \setminus {\mathcal G}^{\gamma,s}(\widetilde{n}-1)), \quad \text{for all} \quad \widetilde{n} \in [m+1, m+s-1]. 
\end{align}
\noindent In particular, by \eqref{InteinSkeG}, Lemma \ref{ConstSke}, 
and since $[m-s,m]\setminus {\mathcal G}^{\gamma,s}(\widetilde{n})\subseteq [m-s,m]\setminus {\mathcal G}^{\gamma,s}(\widetilde{n}-1)$
we conclude that
\begin{align} \label{InteinSkeH}
\gamma(\widetilde n) \not \in {\rm R}^{\gamma}([\widetilde{n}-s+1,m] \setminus {\mathcal G}^{\gamma,s}(m)), \quad \text{for all} \quad \widetilde{n} \in [m+1, m+s-1].
\end{align}
\noindent On the one hand, by the Assumption  \hyperref[assumptionLocalCutPoints]{Local cut points}, we know that the interval $[m-s,m]$ contains a $2s^{\prime}$-local cut point, say $\ell$. Then, by Assumption  \hyperref[assumptionLocalCutPoints]{Local cut points}, Definition \ref{LocaCutP} (of $2s^{\prime}$-local cut point), and \eqref{InteinSkeH}, we have that
\begin{equation} \label{InteinSkeI}
\gamma(\widetilde{n}) \not \in {\rm R}^\gamma ([\ell \vee (\widetilde{n}-s+1),m] \setminus {\mathcal G}^{\gamma,s}(m)),  \quad \text{for all} \quad \widetilde{n} \in [m+1, m+s-1].
\end{equation}
\noindent On the other hand, Assumption \hyperref[assumptionNoLoopsOfIntermediateLength]{No loops of intermediate length}, Assumption \hyperref[assumptionOnlyGhostIndices1]{Only ghost indices (1)} and Lemma \ref{LemmaSimiSke} (with $n=m$) imply that
\begin{align} \label{InteinSkeJ}
[\ell \vee (\widetilde{n}-s+1),m] \setminus \mathcal{G}^{\gamma,s}(m) = [\ell \vee (\widetilde{n}-s+1),m]\cap {\rm NE}^{\gamma,[m-s, m]}(m) ,
\end{align}
\noindent for all $\widetilde{n} \in [m+1, m+s-1]$. 
Hence, by \eqref{InteinSkeI} and \eqref{InteinSkeJ},
\begin{equation} \label{InteinSkeK}
\gamma(\widetilde{n}) \not \in {\rm R}^\gamma ([\ell \vee (\widetilde{n}-s+1),m]\cap {\rm NE}^{\gamma,[m-s, m]}(m) ), \quad \text{for all} \quad \widetilde{n} \in [m+1, m+s-1].
\end{equation} 
\noindent Finally, Assumption \hyperref[assumptionNoLoopsOfIntermediateLength]{No loops of intermediate length} and \eqref{InteinSkeK} imply that
\begin{equation} \label{InteinSkeL}
\gamma(\widetilde{n}) \not \in {\rm R}^\gamma ({\rm NE}^{\gamma,[m-s, m]}(m)), \quad \text{for all} \quad \widetilde{n} \in [m+1, m+s].
\end{equation}In the above expression, to include $m+s$ recall that ${\rm NE}^{\gamma,[m-s, m]}(m)\subseteq [m-s, m]$.
Therefore from \eqref{InteinSkeL} we conclude that $m \in {\rm N E}^{\gamma,s}(A)$ (recall \eqref{e:023}), and $[\min A +s, \max A -s] \setminus {\mathcal G}^{\gamma,s}(n) \subseteq {\rm N E}^{\gamma,s}(A)$. 
\end{proof}

The next result of this subsection states that $2s^{\prime}$-local cut points are not ghost indices. 

\begin{lemma} \label{NewLemmaInclII}
Fix $r,s, s^{\prime} \in \na$ with $r \geq 3s+1 \geq 18s^{\prime} +1$. Let $\gamma:\na_0\to V$ be a path on a finite, simple, connected graph $G=(V,E)$ that satisfies Assumption \hyperref[assumptionNoLoopsOfIntermediateLength]{No loops of intermediate length}. Let $\ell \in [2s^{\prime}, N -2s^{\prime}]$ be a $2s^{\prime}$-local cut point of the path $\gamma$. Then, $\ell \in [0, n] \setminus {\mathcal G}^{\gamma,s}(n)$, for $n \in [ \ell, N]$.
\end{lemma}

\begin{proof}
Clearly $\ell \in [0, \ell] \setminus {\mathcal G}^{\gamma,s}(\ell)$, by Definition \ref{DefGhostI}. On the other hand, since $\ell \in [2s^{\prime}, N -2s^{\prime}]$ is a $2s^{\prime}$-local cut point (recall Definition \ref{LocaCutP}), Assumption \hyperref[assumptionNoLoopsOfIntermediateLength]{No loops of intermediate length} implies that $\gamma(n) \not \in {\rm R}^{\gamma}([n - s^{\prime} +1, \ell])$, for all $n \in [\ell+1,N]$. Then, it is straightforward from Assumption \hyperref[assumptionNoLoopsOfIntermediateLength]{No loops of intermediate length} and Definition \ref{DefGhostI} that $\ell \in [0, n] \setminus {\mathcal G}^{\gamma,s}(n)$, for $n \in [ \ell, N]$. 
\end{proof}

To conclude this subsection, we prove a key result showing that if the path $\gamma$ satisfies certain non-intersection properties, then Lemma \ref{InteinSke}'s conditions \hyperref[eqnAssumptionOnlyGhostIndices1Eqn1]{Only ghost indices (1)} and \hyperref[assumptionOnlyGhostIndices2]{Only ghost indices (2)} hold.

Recall from \eqref{eqnDefinitionAiBi}, the definition of the sub-intervals $A_{i}^{(r,s)}$ and $B_{i}^{(r,s)}$. 
\noindent We also consider the sequence of times $(\tau_{i}^{\gamma, (r,s)})_{i \in \mathbb{N}_{0}}$ defined recursively as follows. Let $N \in \mathbb{N}$ as in Assumptions \hyperref[assumptionNoLoopsOfIntermediateLength]{No loops of intermediate length} and \hyperref[assumptionLocalCutPoints]{Local cut points}. Set $\tau_{0}^{\gamma, (r,s)} = 0$ and for each $i \in \mathbb{N}$,
\begin{align} \label{InterTime1}
\tau_{i}^{\gamma, (r,s)} \coloneqq  \inf \{  k\in \{\tau_{i-1}^{\gamma, (r,s)}+1,\ldots, \lfloor N/r \rfloor +1\}: {\rm R}^\gamma ({\rm NE}^{\gamma,s}(A_{j}^{(r,s)})) \cap {\rm R}^\gamma (B_{k}^{(r,s)})  \neq \emptyset \, \, \text{some} \, \, j <k  \};
\end{align}
\noindent with the convention $\inf \emptyset = \lfloor N/r \rfloor+1$. Recall that $\gamma:\na_0\to V$ is a path on a finite, simple, connected graph $G=(V,E)$. Consider the following properties: 
\begin{enumerate}[label=(\textbf{P.\arabic*})]
\item  \label{Pro1} If there exists $i,j \in [0, \lfloor N/r \rfloor]$ such that ${\rm R}^\gamma (B_{i}^{(r,s)}) \cap {\rm R}^\gamma (B_{j}^{(r,s)})  \neq \emptyset$, then 
\begin{align}
{\rm R}^\gamma (B_{k_{1}}^{(r,s)}) \cap {\rm R}^\gamma (B_{k_{2}}^{(r,s)}) = \emptyset,
\end{align}
\noindent for all $k_{1}, k_{2} \in [0, \lfloor N/r \rfloor]$ with $k_1\in \{i,j\}$ and $k_2\notin \{i,j\}$'.

\item  \label{Pro3} For every $i \in [2, \lfloor N/r \rfloor]$, we have that
\begin{align}
{\rm R}^\gamma ([(j-1)r, jr-1] \setminus B_{j}^{(r,s)}) \cap {\rm R}^\gamma ([\min B^{(r,s)}_{i},\max B^{(r,s)}_{i}+s]) = \emptyset,
\end{align}
\noindent for all $j \in [0, \lfloor N/r \rfloor]$ such that $j < i$. 
			
\item  \label{Pro4} For every $i \in [2, \lfloor N/r \rfloor]$, we have that
\begin{align}
{\rm R}^\gamma ([0,\max B^{(r,s)}_{i-1}])\cap {\rm R}^\gamma ([(i-1)r, ir-1] \setminus B_{i}^{(r,s)}) = \emptyset.
\end{align}
\setcounter{Cond}{\value{enumi}}
\end{enumerate}

\begin{corollary} \label{LemmaIdenSeDecomp}
Fix $r,s, s^{\prime} \in \na$ with $r \geq 3s+1 \geq 18s^{\prime} +1$. Let $\gamma:\na_0\to V$ be a path on a finite, simple, connected graph $G=(V,E)$ that satisfies Assumptions \hyperref[assumptionNoLoopsOfIntermediateLength]{No loops of intermediate length},  \hyperref[assumptionLocalCutPoints]{Local cut points} and properties \ref{Pro1}-\ref{Pro4}. \noindent Then,  if $N \geq r$, for $\widetilde k \in [1, \lfloor N/r \rfloor]$ and $i \in [\tau_{\widetilde k-1}^{\gamma, (r,s)}+1, \tau_{\widetilde k}^{\gamma, (r,s)}-1]$,
\begin{align}  \label{LemmaIdenSeDecompeq1}
{\rm N E}^{\gamma,s}(A_{i}^{(r,s)}) = B_{i}^{(r,s)} \setminus {\mathcal G}^{\gamma,s}(n), \quad \text{for all} \quad n \in [ir-1, N]. 
\end{align}
\end{corollary}

Before we prove Corollary \ref{LemmaIdenSeDecomp}, we need the following auxiliary result.
\begin{lemma} \label{AuxiliarLemmaI}
Fix $r,s, s^{\prime} \in \na$ with $r \geq 3s+1 \geq 18s^{\prime} +1$. Let $\gamma:\na_0\to V$ be a path on a finite, simple, connected graph $G=(V,E)$ that satisfies the Assumption  \hyperref[assumptionLocalCutPoints]{Local cut points}. Fix $m \in [s,N]$ and $\ell \in [m-s,m]$ a $2s^{\prime}$-local cut-point. Suppose that
\begin{align} \label{AuxiliarLemmaIeq1}
{\rm R}^\gamma\big([0, \ell]\setminus {\mathcal G}^{\gamma,s}(k-1)\big) \cap {\rm R}^\gamma\big([\ell+1, k-1]\setminus {\mathcal G}^{\gamma,s}(k-1)\big) = \emptyset,
\end{align}
\noindent for all $k \in [\ell+1, m-1]$. Then, for all $n \in [\ell +1, m]$ such that $\gamma(n) \in {\rm R}^\gamma([\ell+1, n-1]\setminus {\mathcal G}^{\gamma,s}(n-1))$, we have that
\begin{align} \label{AuxiliarLemmaIeq2}
\gamma(k) \not \in {\rm R}^\gamma\big([\ell+1, k-1]\setminus {\mathcal G}^{\gamma,s}(n-1)\big), \quad \text{for all} \quad k \in [m_{1},n-1]\setminus {\mathcal G}^{\gamma,s}(n-1),
\end{align}
\noindent where $m_{1} \coloneqq \max\{ h \in [n-s+1, n-1]\setminus {\mathcal G}^{\gamma,s}(n-1): \gamma(h) = \gamma(n) \}$. 
\end{lemma}

\begin{proof}
If $m \in \{ \ell, \dots, \ell +3\}$ our claim holds by vacuity. Then, suppose that $m > \ell +3$ and proceed by induction on $n$. For $n=\ell +1, \ell+2$ our claim also holds by vacuity. Suppose that $n = \ell +3$ is such that $\gamma(n) = \gamma(h)$ for some $h \in [\ell+1, \ell+2] \setminus {\mathcal G}^{\gamma,s}(\ell+2)$. Indeed, we must have $h= \ell+1$. Since $[\ell+1, k-1]\setminus {\mathcal G}^{\gamma,s}(n-1) = \emptyset$, for $k \in [\ell+1,\ell+2]\setminus {\mathcal G}^{\gamma,s}(\ell+2)$, we conclude that \eqref{AuxiliarLemmaIeq2} is satisfied.  

Next, suppose that for some $n \in [\ell +1, m]$, we have shown that for all $\widetilde{n} \in [\ell +1, n-1]$ such that $\gamma(\widetilde{n}) \in {\rm R}^\gamma([\ell+1, \widetilde{n}-1]\setminus {\mathcal G}^{\gamma,s}(\widetilde{n}-1))$,  we have that
\begin{align} \label{AuxiliarLemmaIeq3}
\gamma(k) \not \in{\rm R}^\gamma\big([\ell+1, k-1]\setminus {\mathcal G}^{\gamma,s}(\widetilde{n}-1)\big), \quad \text{for all} \quad k \in [\widetilde{m}_{1},\widetilde{n}-1]\setminus {\mathcal G}^{\gamma,s}(\widetilde{n}-1),
\end{align}
\noindent where $\widetilde{m}_{1} \coloneqq \max\{ h \in [\widetilde{n}-s+1, \widetilde{n}-1]\setminus {\mathcal G}^{\gamma,s}(\widetilde{n}-1): \gamma(h) = \gamma(\widetilde{n}) \}$. Then, we prove that if $\gamma(n) \in  {\rm R}^\gamma([\ell+1, n-1]\setminus {\mathcal G}^{\gamma,s}(n-1))$, we have that
\begin{align}
\gamma(k) \not \in{\rm R}^\gamma\big([\ell+1, k-1]\setminus {\mathcal G}^{\gamma,s}(n-1)\big), \quad \text{for all} \quad k \in [m_{1},n-1]\setminus {\mathcal G}^{\gamma,s}(n-1), 
\end{align}
\noindent where $m_{1} \coloneqq \max\{ h \in [n-s+1, n-1]\setminus {\mathcal G}^{\gamma,s}(n-1): \gamma(h) = \gamma(n) \}$.

We proceed by contradiction. Suppose that 
\begin{align}
\gamma(k) \in{\rm R}^\gamma\big([\ell+1, k-1]\setminus {\mathcal G}^{\gamma,s}(n-1)\big),
\end{align}
\noindent for some $k \in [m_{1},n-1]\setminus {\mathcal G}^{\gamma,s}(n-1)$. Hence, there exists $\widetilde{k} \in [m_{1},n-1]\setminus {\mathcal G}^{\gamma,s}(n-1)$ and $\widetilde{h} \in [\ell+1, \widetilde{k}-1]\setminus {\mathcal G}^{\gamma,s}(n-1)$ such that $\gamma(\widetilde{k}) = \gamma(\widetilde{h})$. In particular, since $\widetilde{k}-s+1 \leq n-s \leq m-s \leq \ell$, $\widetilde{h} \in [\ell+1, \widetilde{k}-1]\setminus {\mathcal G}^{\gamma,s}(\widetilde{k}-1)$  and it satisfies \eqref{eqnGhostIndexDefinitionSmallLoop} in ($\mathbf{G}_{\widetilde{k}}$). Define 
$m_{2} = \max\{ h \in [\widetilde{k}-s+1, \widetilde{k}-1]\setminus {\mathcal G}^{\gamma,s}(\widetilde{k}-1): \gamma(h) = \gamma(\widetilde{k}) \}$. However, by the induction hypothesis \eqref{AuxiliarLemmaIeq3} (with $\widetilde{n} = \widetilde{k}$ and $\widetilde{m}_{1} = \widetilde{m}_{2}$), we have that
\begin{align} \label{AuxiliarLemmaIeq4}
\gamma(\widetilde{k}_{1}) \not \in{\rm R}^\gamma\big([\ell+1, \widetilde{k}_{1}-1]\setminus {\mathcal G}^{\gamma,s}(\widetilde{k}-1)\big),  
\end{align}
\noindent for all $\widetilde{k}_{1} \in [\widetilde{m}_{2},\widetilde{k}-1]\setminus {\mathcal G}^{\gamma,s}(\widetilde{k}-1)$.  Moreover, our assumption \eqref{AuxiliarLemmaIeq1} (with $k = \widetilde{k}$) implies that 
\begin{align} \label{AuxiliarLemmaIeq5}
\gamma(\widetilde{k}_{1}) \not \in{\rm R}^\gamma\big([0, \ell]\setminus {\mathcal G}^{\gamma,s}(\widetilde{k}-1)\big), \quad \text{for all} \quad \widetilde{k}_{1} \in [\widetilde{m}_{2},\widetilde{k}-1]\setminus {\mathcal G}^{\gamma,s}(\widetilde{k}-1). 
\end{align}
\noindent Thus, \eqref{AuxiliarLemmaIeq4} and \eqref{AuxiliarLemmaIeq5} show that $\widetilde{h}$ does not satisfies \eqref{eqnGhostIndexDefinitionNoLongLoops} in ($\mathbf{G}_{\widetilde{k}}$) and 
we conclude that $\widetilde{h} \in {\mathcal G}^{\gamma,s}(\widetilde{k})$.  This contradicts the fact that $\widetilde{h} \in [\ell+1, \widetilde{k}-1]\setminus {\mathcal G}^{\gamma,s}(n-1)$ (recall $\widetilde{k} \leq n-1$) and finishes our proof. 
\end{proof}

\begin{proof}[Proof of Corollary \ref{LemmaIdenSeDecomp}]
First, we consider the case $r \leq N \leq 2r-1$. Then  $\lfloor N/r \rfloor =1$ and $\tau_{1}^{\gamma, (r,s)} = 2$ (that is, $\widetilde k =1$ and $i=1$). We verify that $A_{1}^{(r,s)}$ satisfies \hyperref[eqnAssumptionOnlyGhostIndices1Eqn1]{Only ghost indices (1)} and \hyperref[assumptionOnlyGhostIndices2]{Only ghost indices (2)}. 

We start by proving that the \hyperref[eqnAssumptionOnlyGhostIndices1Eqn1]{Only ghost indices (1)} property holds. Let $m \in B_{1}^{(r,s)}$ and $\ell \in [m-s,m]$ a $2s^{\prime}$-local cut-point. For  $n_{1} \in [\ell +1, m]$ and $\widetilde{m} \in [\ell, n_{1}-1] \setminus {\mathcal G}^{\gamma,s}(n_1-1)$ such that $\gamma(n_1)\in {\rm R}^\gamma([n_1-s+1, \widetilde{m}] \setminus {\mathcal G}^{\gamma,s}(n_1-1))$, define $m_{1} = \max\{ h \in [n_1-s+1, \widetilde{m}] \setminus {\mathcal G}^{\gamma,s}(n_1-1):  \gamma(h) = \gamma(n_1)\}$. 
By the Assumptions \hyperref[assumptionNoLoopsOfIntermediateLength]{No loops of intermediate length} and Definition \ref{LocaCutP} of $2s^{\prime}$-local cut-point, we must have that
\begin{align}  \label{LemmaIdenSeDecompeq2I}
m_{1} \in [\ell +1, \widetilde{m}] \setminus {\mathcal G}^{\gamma,s}(n_1-1).
\end{align}
\noindent On the other hand, since $\ell \in [m-s,m]$ is a $2s^{\prime}$-local cut-point (see Definition \ref{LocaCutP}), $m \in B_{1}^{(r,s)}$, the Assumption \hyperref[assumptionNoLoopsOfIntermediateLength]{No loops of intermediate length} implies that
\begin{align}  \label{LemmaIdenSeDecompeq3I}
{\rm R}^\gamma\big([0, \ell]\setminus {\mathcal G}^{\gamma,s}(k-1)\big) \cap {\rm R}^\gamma\big([\ell+1, k-1]\setminus {\mathcal G}^{\gamma,s}(k-1)\big) = \emptyset,
\end{align}
\noindent for all $k\in [\ell+1,m-1]$. Therefore, by \eqref{LemmaIdenSeDecompeq2I} and \eqref{LemmaIdenSeDecompeq3I}, Lemma \ref{AuxiliarLemmaI} implies that 
\begin{align}
\gamma(k) \not \in{\rm R}^\gamma\big([0,k-1]\setminus {\mathcal G}^{\gamma,s}(n_1-1)\big), \quad \text{for all} \quad k \in [m_1,n_1-1]\setminus {\mathcal G}^{\gamma,s}(n_1-1).
\end{align}
\noindent This proves that $A_{1}^{(r,s)}$ satisfies the \hyperref[eqnAssumptionOnlyGhostIndices1Eqn1]{Only ghost indices (1)} property. 

Next, we show that the \hyperref[assumptionOnlyGhostIndices2]{Only ghost indices (2)} property holds. Let $m \in B_{1}^{(r,s)}$. For $n_2 \in [m+1, m+s-1]$ and  $m \in B_{1}^{(r,s)}\setminus {\mathcal G}^{\gamma,s}(n_2-1)$ such that $\gamma(n_2)\in {\rm R}^\gamma([n_2-s+1, m] \setminus {\mathcal G}^{\gamma,s}(n_2-1))$, define $m_{1} = \max\{ h \in [n_2-s+1, m] \setminus {\mathcal G}^{\gamma,s}(n_2-1):  \gamma(h) = \gamma(n_2)\}$. On the one hand, by the Assumption \hyperref[assumptionLocalCutPoints]{Local cut points}, there exists $\ell \in [n_{2}-s,n_{2}]$ a $2s^{\prime}$-local cut-point. Then, by the Assumptions \hyperref[assumptionNoLoopsOfIntermediateLength]{No loops of intermediate length} and Definition \ref{LocaCutP} of $2s^{\prime}$-local cut-point, we must have that $\gamma(n_2)\in {\rm R}^\gamma([\ell +1, n_{2}-1] \setminus {\mathcal G}^{\gamma,s}(n_2-1))$ and so,
\begin{align} \label{LemmaIdenSeDecompeq4}
m_{1} \in [\ell +1, n_{2}-1] \setminus {\mathcal G}^{\gamma,s}(n_2-1).
\end{align}
\noindent On the other hand, since $\ell \in [n_{2}-s,n_{2}]$ is a $2s^{\prime}$-local cut-point (see Definition \ref{LocaCutP}), $m \in B_{1}^{(r,s)}$, and the Assumption \hyperref[assumptionNoLoopsOfIntermediateLength]{No loops of intermediate length} imply that
\begin{align} \label{LemmaIdenSeDecompeq5}
{\rm R}^\gamma\big([0, \ell]\setminus {\mathcal G}^{\gamma,s}(k-1)\big) \cap {\rm R}^\gamma\big([\ell+1, k-1]\setminus {\mathcal G}^{\gamma,s}(k-1)\big) = \emptyset,  
\end{align}
\noindent for all $k \in [\ell+1,n_{2}-1]$. Therefore, by \eqref{LemmaIdenSeDecompeq4} and \eqref{LemmaIdenSeDecompeq5}, Lemma \ref{AuxiliarLemmaI} implies that 
\begin{align}
\gamma(k) \not \in{\rm R}^\gamma\big([0,k-1]\setminus {\mathcal G}^{\gamma,s}(n_2-1)\big), \quad \text{for all} \quad k \in [m_1,n_2-1]\setminus {\mathcal G}^{\gamma,s}(n_2-1).
\end{align}
\noindent This proves that $A_{1}^{(r,s)}$ satisfies the \hyperref[assumptionOnlyGhostIndices2]{Only ghost indices (2)}  property. 

Then, we have proved that, for $\widetilde k =1$, $i=1$, Lemma \ref{InteinSke} implies that 
\begin{align} 
{\rm N E}^{\gamma,s}(A_{1}^{(r,s)}) = B_{1}^{(r,s)} \setminus {\mathcal G}^{\gamma,s}(n), \quad \text{for all} \quad n \in [r-1, N]. 
\end{align} \\

We henceforth suppose that $N \geq 2r$ and thus, $\lfloor N/r \rfloor \geq 2$. Consider the case $\widetilde k =1$ and proceed by induction on $i \in [\tau_{0}^{\gamma, (r,s)}+1, \tau_{1}^{\gamma, (r,s)}-1] =[1, \tau_{1}^{\gamma, (r,s)}-1]$. For $i=1$, a similar argument to the one used previously, based on Lemma \ref{InteinSke}, shows that 
\begin{align}  \label{LemmaIdenSeDecompeq6}
{\rm N E}^{\gamma,s}(A_{1}^{(r,s)}) = B_{1}^{(r,s)} \setminus {\mathcal G}^{\gamma,s}(n), \quad \text{for all} \quad n \in [r-1, N]. 
\end{align}
\noindent Suppose that, for all $j \in [0, i-1]$,
\begin{align}   \label{LemmaIdenSeDecompeq7}
{\rm N E}^{\gamma,s}(A_{j}^{(r,s)}) = B_{j}^{(r,s)} \setminus {\mathcal G}^{\gamma,s}(n), \quad \text{for all} \quad n \in [jr-1, N].
\end{align}
\noindent We will prove that 
\begin{align}   \label{LemmaIdenSeDecompeq8}
{\rm N E}^{\gamma,s}(A_{i}^{(r,s)}) = B_{i}^{(r,s)} \setminus {\mathcal G}^{\gamma,s}(n), \quad \text{for all} \quad n \in [ir-1, N]. 
\end{align}	
\noindent To do so, it is enough to prove that the interval $A_{i}^{(r,s)}$ satisfies the properties \hyperref[eqnAssumptionOnlyGhostIndices1Eqn1]{Only ghost indices (1)} and \hyperref[assumptionOnlyGhostIndices2]{Only ghost indices (2)} in Lemma \ref{InteinSke}. 

First, we verify \hyperref[eqnAssumptionOnlyGhostIndices1Eqn1]{Only ghost indices (1)}. Fix $m \in B_{i}^{(r,s)}$ and a $2s^{\prime}$-local cut-point $\ell \in [m-s,m]$. For $n_{1} \in [\ell +1, m]$ and  $\widetilde{m} \in [\ell, n_{1}-1] \setminus {\mathcal G}^{\gamma,s}(n_1-1)$ such that  $\gamma(n_1)\in {\rm R}^\gamma([n_1-s+1, \widetilde{m}] \setminus {\mathcal G}^{\gamma,s}(n_1-1))$, define $m_{1} = \max\{ h \in [n_1-s+1, \widetilde{m}] \setminus {\mathcal G}^{\gamma,s}(n_1-1):  \gamma(h) = \gamma(n_1)\}$.  By the Assumptions \hyperref[assumptionNoLoopsOfIntermediateLength]{No loops of intermediate length} and Definition \ref{LocaCutP} of $2s^{\prime}$-local cut-point, we must have that
\begin{align} \label{LemmaIdenSeDecompeq9}
m_{1} \in [\ell +1, \widetilde{m}] \setminus {\mathcal G}^{\gamma,s}(n_1-1).
\end{align}
\noindent By the definition of $\tau_{1}^{\gamma, (r,s)}$ in \eqref{InterTime1},  the induction hypothesis \eqref{LemmaIdenSeDecompeq7} and \ref{Pro1},
\begin{align}  \label{LemmaIdenSeDecompeq10}
{\rm R}^\gamma ({\rm N E}^{\gamma,s}(A_{j}^{(r,s)})) \cap {\rm R}^\gamma (B_{i}^{(r,s)}) = {\rm R}^\gamma (B_{j}^{(r,s)} \setminus {\mathcal G}^{\gamma,s}(k -1)) \cap {\rm R}^\gamma (B_{i}^{(r,s)})  = \emptyset,
\end{align}
\noindent for all $j \in [1, i-1]$ and $k \in [\ell+1, m-1]$.
We remark that the previous holds for $k\geq \ell+1$ since $\ell+1\geq (i-1)r-1$.
Note that if $i=1$, then \eqref{LemmaIdenSeDecompeq10} holds by vacuity. 
Moreover, from  \eqref{LemmaIdenSeDecompeq10}, \ref{Pro3} and \ref{Pro4} we obtain
\begin{align}  \label{LemmaIdenSeDecompeq11}
{\rm R}^\gamma ([(j-1)r, jr-1] \setminus B_{j}^{(r,s)}) \cap {\rm R}^\gamma (B_{i}^{(r,s)})  = \emptyset, \quad \text{for all}  \quad j \in [1, i-1],
\end{align}
\noindent and
\begin{align}  \label{LemmaIdenSeDecompeq12}
{\rm R}^\gamma ( [0, (i-1)r-s-1]) \cap {\rm R}^\gamma ([(i-1)r, (i-1)r+2s-1]) = \emptyset.
\end{align}
\noindent On the other hand, since $\ell \in [m-s,m]$ is a $2s^{\prime}$-local cut-point (see Definition \ref{LocaCutP}), $m \in B_{i}^{(r,s)}$, the Assumption \hyperref[assumptionNoLoopsOfIntermediateLength]{No loops of intermediate length}, \eqref{LemmaIdenSeDecompeq10}, \eqref{LemmaIdenSeDecompeq11} and \eqref{LemmaIdenSeDecompeq12} imply that
\begin{align}  \label{LemmaIdenSeDecompeq13}
{\rm R}^\gamma\big([0, \ell]\setminus {\mathcal G}^{\gamma,s}(k-1)\big) \cap {\rm R}^\gamma\big([\ell+1, k-1]\setminus {\mathcal G}^{\gamma,s}(k-1)\big) = \emptyset, 
\end{align}
\noindent for all $k \in [\ell+1,m-1]$. Therefore, by \eqref{LemmaIdenSeDecompeq9} and \eqref{LemmaIdenSeDecompeq13}, Lemma \ref{AuxiliarLemmaI} implies that 
\begin{align}
\gamma(k) \not \in{\rm R}^\gamma\big([0,k-1]\setminus {\mathcal G}^{\gamma,s}(n_1-1)\big), \quad \text{for all} \quad k \in [m_1,n_1-1]\setminus {\mathcal G}^{\gamma,s}(n_1-1).
\end{align}
\noindent This proves that $A_{i}^{(r,s)}$ satisfies the \hyperref[assumptionOnlyGhostIndices2]{Only ghost indices (1)}  property. 

Next, we verify the \hyperref[assumptionOnlyGhostIndices2]{Only ghost indices (2)} property. Let $m \in B_{i}^{(r,s)}$. For $n_2 \in [m+1, m+s-1]$ and  $m \in B_{i}^{(r,s)} \setminus {\mathcal G}^{\gamma,s}(n_2-1)$ such that $\gamma(n_2)\in {\rm R}^\gamma([n_2-s+1, m] \setminus {\mathcal G}^{\gamma,s}(n_2-1))$, define $m_{1} = \max\{ h \in [n_2-s+1, m] \setminus {\mathcal G}^{\gamma,s}(n_2-1):  \gamma(h) = \gamma(n_2)\}$. On the other hand, by the Assumption \hyperref[assumptionLocalCutPoints]{Local cut points}, there exists $\ell \in [n_{2}-s,n_{2}]$ a $2s^{\prime}$-local cut-point. Then, by the Assumptions \hyperref[assumptionNoLoopsOfIntermediateLength]{No loops of intermediate length} and Definition \ref{LocaCutP} of $2s^{\prime}$-local cut-point, we must have that $\gamma(n_2)\in {\rm R}^\gamma([\ell +1, n_{2}-1] \setminus {\mathcal G}^{\gamma,s}(n_2-1))$ and so,
\begin{align} \label{LemmaIdenSeDecompeq14}
m_{1} \in [\ell +1, n_{2}-1] \setminus {\mathcal G}^{\gamma,s}(n_2-1).
\end{align}
\noindent By the definition of $\tau_{1}^{\gamma, (r,s)}$ in \eqref{InterTime1},  the induction hypothesis \eqref{LemmaIdenSeDecompeq7} and \ref{Pro1},
\begin{align}  \label{LemmaIdenSeDecompeq10II}
{\rm R}^\gamma ({\rm N E}^{\gamma,s}(A_{j}^{(r,s)})) \cap {\rm R}^\gamma (B_{i}^{(r,s)}) = {\rm R}^\gamma (B_{j}^{(r,s)} \setminus {\mathcal G}^{\gamma,s}(k -1)) \cap {\rm R}^\gamma (B_{i}^{(r,s)})  = \emptyset,
\end{align}
\noindent for all $j \in [1, i-1]$ and $k \in [\ell+1, n_{2}-1]$. 
If $i=1$, then \eqref{LemmaIdenSeDecompeq10} holds by vacuity. 
Moreover, by the Assumption \hyperref[assumptionNoLoopsOfIntermediateLength]{No loops of intermediate length}, \ref{Pro3} and \ref{Pro4},
\begin{align}   \label{LemmaIdenSeDecompeq15}
{\rm R}^\gamma ([0, (i-1)r+2s-1 ]) \cap {\rm R}^\gamma ([ir-s, ir-1])  = \emptyset.
\end{align}
\noindent On the other hand, since $\ell \in [n_{2}-s,n_{2}]$ is a $2s^{\prime}$-local cut-point (see Definition \ref{LocaCutP} ), $m \in B_{i}^{(r,s)}$, the Assumption \hyperref[assumptionNoLoopsOfIntermediateLength]{No loops of intermediate length}, \eqref{LemmaIdenSeDecompeq11}, \eqref{LemmaIdenSeDecompeq12} and \eqref{LemmaIdenSeDecompeq15} imply that
\begin{align}  \label{LemmaIdenSeDecompeq16}
{\rm R}^\gamma\big([0, \ell]\setminus {\mathcal G}^{\gamma,s}(k-1)\big) \cap {\rm R}^\gamma\big([\ell+1, k-1]\setminus {\mathcal G}^{\gamma,s}(k-1)\big) = \emptyset, 
\end{align}
\noindent for all $k \in [\ell+1,n_{2}-1]$. Therefore, by \eqref{LemmaIdenSeDecompeq14} and \eqref{LemmaIdenSeDecompeq16}, Lemma \ref{AuxiliarLemmaI} implies that 
\begin{align}
\gamma(k) \not \in{\rm R}^\gamma\big([0,k-1]\setminus {\mathcal G}^{\gamma,s}(n_2-1)\big), \quad \text{for all} \quad k \in [m_1,n_2-1]\setminus {\mathcal G}^{\gamma,s}(n_2-1).
\end{align}
\noindent This proves that $A_{i}^{(r,s)}$ satisfies the \hyperref[assumptionOnlyGhostIndices2]{Only ghost indices (2)}  property. 

Then, we have proved that, for $k =1$ and  $i \in [\tau_{0}^{\gamma, (r,s)}+1, \tau_{1}^{\gamma, (r,s)}-1]$, Lemma \ref{InteinSke} implies that 
\begin{align} 
{\rm N E}^{\gamma,s}(A_{i}^{(r,s)}) = B_{i}^{(r,s)} \setminus {\mathcal G}^{\gamma,s}(n), \quad \text{for all} \quad n \in [ir-1, N]. 
\end{align} \\

Finally, to prove the general case \eqref{LemmaIdenSeDecompeq1} for all $\widetilde k  \in [1, \lfloor N/r \rfloor]$,  one proceeds by induction on $\widetilde k$. Recall that we have assumed $\widetilde k \geq 2r$. Then, as in the case of $\widetilde k =1$ and  $i \in [\tau_{0}^{\gamma, (r,s)}+1, \tau_{1}^{\gamma, (r,s)}-1]$, a similar argument will prove our claim.
\end{proof}

\subsection{The skeleton chain}\label{subsecSkeletonChain}

In this subsection we introduce the skeleton chain which is obtained from the Aldous-Broder chain by restricting the latter to the non-ghost indices. 

\begin{definition}[The skeleton chain] Fix $s\in\mathbb{N}$ and a path $\gamma:\,\mathbb{N}_0\to V$ on a simple, connected graph $G=(V,E)$. For each $n\in\mathbb{N}_0$, let
$\mathrm{Skel}^{\gamma,s}(n)$ be the subgraph of $\mathrm{AB}^\gamma(n)$ restricted to the vertex set ${\rm R}^\gamma([0,n]\setminus\mathcal {G}^{\gamma,s}(n))$ rooted at $\gamma(n)$. We refer to 
\begin{equation}
\label{e:skel}
\mathrm{Skel}^{\gamma,s}:=\big(\mathrm{Skel}^{\gamma,s}(n)\big)_{n\in\mathbb{N}_0}
\end{equation} 
as the skeleton chain.
\label{Def:005}
\end{definition}

Note that $\mathrm{Skel}^{\gamma,s}(0)$ is the trivial tree with only one vertex. Note also that $\mathrm{Skel}^{\gamma,s}(n)$ is not necessarily connected but that its connected components are trees. The first result states that $\mathrm{Skel}^{\gamma,s}(n)$ is connected. 

\begin{lemma}[Connectedness]
Fix $s\in \mathbb{N}$. Let $\gamma:\,\mathbb{N}_0\to V$ be a path on a finite simple, connected graph $G=(V,E)$. 
Then, for each $n \in\mathbb{N}_0$, $\mathrm{Skel}^{\gamma,s}(n)$ is connected. In particular,  $\mathrm{Skel}^{\gamma,s}(n)$ is a subtree of $\mathrm{AB}^\gamma(n)$ rooted at $\gamma(n)$. 
\label{SkeLemma2}
\end{lemma}
\begin{proof}
We proceed by induction on $n\in\mathbb{N}_0$. Recall from Definition \ref{DefGhostI} that $\mathcal{G}^{\gamma,(r,s)}(0)=\emptyset$, and that therefore $\mathrm{Skel}^{\gamma,s}(0)$ is the trivial tree rooted at $\gamma(0)$. Now, we suppose that, for some $n_0 \in \mathbb{N}$, $\mathrm{Skel}^{\gamma,s}(\ell-1)$ is a connected tree rooted at $\gamma(\ell-1)$, for all $\ell\in [1,n_0]$. We shall show that $\mathrm{Skel}^{\gamma,s}(n_0)$  is a connected tree rooted at $\gamma(n_0)$. 
We consider the following three cases: 
\begin{itemize}
\item {\em Case~1. } $\gamma(n_0) \not \in {\rm R}^\gamma([0,n_{0}-1]\setminus{\mathcal G}^{\gamma,s}(n_0-1))$.
\item {\em Case~2. } $\gamma(n_0) \in {\rm R}^\gamma([0,n_{0}-1]\setminus{\mathcal G}^{\gamma,s}(n_0-1))$
but ${\mathcal G}^{\gamma,s}(n_0-1)={\mathcal G}^{\gamma,s}(n_0)$.
\item {\em Case~3. } $\gamma(n_0) \in {\rm R}^\gamma([0,n_{0}-1]\setminus{\mathcal G}^{\gamma,s}(n_0-1))$
and ${\mathcal G}^{\gamma,s}(n_0)\setminus{\mathcal G}^{\gamma,s}(n_0-1)\not=\emptyset$.
\end{itemize}

\noindent{\em Case~1 (Root growth). } Assume that $\gamma(n_0) \not \in {\rm R}^\gamma([0,n_0-1]\setminus{\mathcal G}^{\gamma,s}(n_0-1)$. Then ${\mathcal G}^{\gamma,s}({n_0})={\mathcal G}^{\gamma,s}({n_0}-1)$ by Definition~\ref{DefGhostI}. As $\mathrm{Skel}^{\gamma,s}({n_0}-1)$ is a connected subtree of $\mathrm{AB}^\gamma({n_0}-1)$ rooted at $\gamma({n_0}-1)$ by the induction hypothesis, $\mathrm{Skel}^{\gamma,s}({n_0})$ is obtained from $\mathrm{Skel}^{\gamma,s}(n_0-1)$ by adding the vertex $\gamma(n_0)$ and the edge $\{\gamma(n_0-1),\gamma(n_0)\}$ and letting $\gamma(n_0)$ be the new root. In particular, $\mathrm{Skel}^{\gamma,s}(n_0)$ is a connected tree rooted at $\gamma(n_0)$ which yields the induction claim. \\

\noindent{\em Case~2 (Aldous-Broder move). } Assume that $\gamma(n_0)\in{\rm R}^\gamma([0,n_0-1]\setminus{\mathcal G}^{\gamma,s}(n_0-1)$ but ${\mathcal G}^{\gamma,s}(n_0)={\mathcal G}^{\gamma,s}(n_0-1)$. The latter implies that $\gamma(n_0)\not =\gamma(n_0-1)$ by Remark~\ref{Rem:002}. 
As $\mathrm{Skel}^{\gamma,s}({n_0}-1)$ is a connected subtree of $\mathrm{AB}^\gamma({n_0}-1)$ rooted at $\gamma({n_0}-1)$ by the induction hypothesis and as $\gamma(n_0-1) \neq \gamma(n_0)$, $\{ \gamma(L_{\gamma(n_0)}(n_0-1)+1), \gamma(n_0)\}$ must be an edge in $\mathrm{Skel}^{\gamma,s}(n_0-1)$ and $\gamma(L_{\gamma(n_0)}(n_0-1)+1) \in \mathrm{Skel}^{\gamma,s}(n_0-1)$. Therefore $\mathrm{Skel}^{\gamma,s}(n_0)$ is obtained by removing the edge $\{ \gamma(L_{\gamma(n_0)}(n_0-1)+1), \gamma(n_0) \}$ and adding the new edge $\{\gamma(n_0-1),\gamma(n_0)\}$. In particular, $\mathrm{Skel}^{\gamma,s}(n_0)$ is a connected tree rooted at $\gamma(n_0)$ which yields the induction claim. \\
\smallskip

\noindent {\em Case~3 (ghost loop erasure). }Assume that $\gamma(n_0)\in{\rm R}^\gamma([0,n_0-1]\setminus{\mathcal G}^{\gamma,s}(n_0-1)$ and that ${\mathcal G}^{\gamma,s}(n_0)\setminus{\mathcal G}^{\gamma,s}(n_0-1)\not=\emptyset$. 
We first claim that $n_0-1\in {\mathcal G}^{\gamma,s}(n_0)\setminus{\mathcal G}^{\gamma,s}(n_0-1)$. Indeed $n_0-1\not\in {\mathcal G}^{\gamma,s}(n_0-1)$ by Definition~\ref{DefGhostI}. Moreover, by the assumption in Case~3 (in particular, since ${\mathcal G}^{\gamma,{s}}(n_0)\setminus{\mathcal G}^{\gamma,{s}}(n_0-1)\not=\emptyset$), we have that $n_{0}-1$ satisfies $(\mathbf{G}_{n_{0}})$ in Definition~\ref{DefGhostI}, that is, $\gamma(n_0)\in {\rm R}^\gamma([n_0-s+1,n_0-1]\setminus{\mathcal G}^{\gamma,s}(n_0-1))$ and with 
\begin{equation}
\label{e:mast}
   m^\ast(n_0-1) \coloneqq \max \{ h \in [n_0-s+1,n_0-1]\setminus \mathcal{G}^{\gamma,s}(n_0-1): \gamma(h) = \gamma(n_0)\},
\end{equation}   
\noindent  we also have that
\begin{equation} \label{eqnGhostIndexDefinitionNoLongLoopsm}
\gamma(k) \not \in{\rm R}^\gamma\big([0,k-1]\setminus {\mathcal G}^{\gamma,s}(n_0-1)\big), \mbox{ for all } k \in [m^\ast(n_0-1),n_0-1]\setminus {\mathcal G}^{\gamma,s}(n_0-1).
\end{equation}

We claim that 
\begin{equation} \label{SkeLemma2eq8}
   {\mathcal G}^{\gamma,s}(n_0)\setminus{\mathcal G}^{\gamma,s}(n_0-1)=[m^\ast(n_0-1),n_0-1],
\end{equation}
and thus that $\mathrm{Skel}^{\gamma,s}(n_0)$ is obtained from  $\mathrm{Skel}^{\gamma,s}(n_0-1)$ by erasing the loop of length $n_0-1-m^\ast(n_0-1)$ (that is, the vertices ${\rm R}^{\gamma}([m^\ast(n_0-1),n_0-1])$ and letting the new root be $\gamma(n_0)$. Provided we have shown (\ref{SkeLemma2eq8}), it remains to show that the graph obtained from this ghost loop erasure is still connected.

To show (\ref{SkeLemma2eq8}), consider first $m\in[m^\ast(n_0-1),n_0-1] \setminus {\mathcal G}^{\gamma,s}(n_0-1)$. Note that 
$m$ clearly satisfy the condition $(\mathbf{G}_{n_{0}})$, i.e., $\gamma(n_0)\in R^\gamma([n_0-s+1,m]\setminus{\mathcal G}^{\gamma,s}(n_0-1))$ (because $\gamma(n_0)=\gamma(m^\ast(n_0-1))$, and thus $\gamma(n_0)\in R^\gamma([n_0-s+1,m^\ast(\gamma(n_0-1))]\setminus{\mathcal G}^{\gamma,s}(n_0-1))$) as well as  
$\gamma(k)\not\in {\rm R}^\gamma([0,k-1]\setminus{\mathcal G}^{\gamma,s}(n_0-1))$ for all $k\in[m^\ast(n_0-1),n_0-1]\setminus{\mathcal G}^{\gamma,s}(n_0-1)$ (due to (\ref{eqnGhostIndexDefinitionNoLongLoopsm})). Thus $[m^\ast(n_0-1),n_0-1]\subseteq{\mathcal G}^{\gamma,s}(n_0)\setminus{\mathcal G}^{\gamma,s}(n_0-1)$. Consider next $m\in[n_0-s+1,m^\ast(n_0-1)-1] \setminus{\mathcal G}^{\gamma,s}(n_0-1))$. In this case $m$ does not satisfy
  the condition $(\mathbf{G}_{n_{0}})$ (because if $\gamma(n_{0}) \in {\rm R}^{\gamma}([n_0-s+1,m^\ast(n_0-1)-1] \setminus{\mathcal G}^{\gamma,s}(n_0-1)))$, then $\gamma(m^\ast(n_0-1))\in {\rm R}^\gamma([0,k-1]\setminus{\mathcal G}^{\gamma,s}(n_0-1))$ which violates (\ref{eqnGhostIndexDefinitionNoLongLoops}) with $k=m^\ast(n_0-1)$), and thus $m\not\in{\mathcal G}^{\gamma,s}(n_0)\setminus{\mathcal G}^{\gamma,s}(n_0-1)=\emptyset$.
  As $ {\mathcal G}^{\gamma,s}(n_0)\setminus{\mathcal G}^{\gamma,s}(n_0-1)\subseteq[n_0-s+1,n_0-1]$ by Definition~\ref{DefGhostI}, (\ref{SkeLemma2eq8}) is confirmed. 
  
  We close the proof by verifying that deleting the vertices ${\rm R}^{\gamma}([m^\ast(n_0-1),n_0-1])$ from $\mathrm{Skel}^{\gamma,s}(n_0)$ yields a connected graph. Assume to the contrary that there are $k_1,k_2\in [0,m^\ast(n_0-1)]\setminus{\mathcal G}^{\gamma,s}(n_0-1)$ such that the unique path in $\mathrm{Skel}^{\gamma,s}(n_0-1)$ connecting $\gamma(k_1)$ with $\gamma(k_2)$ goes through a vertex in ${\rm R}^{\gamma}([m^\ast(n_0-1),n_0-1])$. In this case there would be an $\ell_{1} \in [0, m^\ast(n_0-1)] \setminus {\mathcal G}^{\gamma,s}(n_0-1)$ such that $\gamma(\ell_{1}) \neq \gamma(n_0)$ and  $\ell_{2} \in [ m^\ast(n_0-1)+1, n_0-1]$ such that $\gamma(\ell_{2}) = \gamma(\ell_{1})$. However, by (\ref{SkeLemma2eq8}), $\ell_2\in [ m^\ast(n_0-1)+1, n_0-1]$ implies that $\ell_2$ must satisfy $(\mathbf{G}_{n_{0}})$ which means in particular that 
  $\gamma(\ell_2)\not\in {\rm R}^\gamma([0,\ell_2-1]\setminus {\mathcal G}^{\gamma,s}(n_0-1))\supseteq\{\gamma(\ell_1)\}$, which contradicts that $\gamma(\ell_1)=\gamma(\ell_2)$.  
  \end{proof}
 
From the proof of Lemma~\ref{SkeLemma2} we can easily derive the dynamics of the skeleton chain 
(see  Figure~\ref{figEvolutionOfSkeleton} for an illustration). 
Recall the definition of the times $L_x(n)$ in \eqref{e:136AB} for $x \in {\rm R}^{\gamma}([0,n])$ and $n \in \mathbb{N}_{0}$.

\begin{remark}[Dynamics of the skeleton chain]
Fix $s\in \mathbb{N}$.  Let $\gamma:\,\mathbb{N}_0\to V$ be a path on a finite simple, connected graph $G=(V,E)$. The {\em skeleton chain}  $\big(\mathrm{Skel}^{\gamma,s}(n)\big)_{n\in\mathbb{N}_0}$ has the following dynamics. $\mathrm{Skel}^{\gamma,s}(0)$ is the trivial tree. Then, given $\mathrm{Skel}^{\gamma,s}(n-1)$, for some $n \in\mathbb{N}$, we construct $\mathrm{Skel}^{\gamma,s}(n)$ by distinguishing three cases:
\begin{itemize}
\item {\bf Root growth. } If $\gamma(n)$ is not a vertex in ${\rm Skel}^{\gamma,s}(n-1)$  
     then $\mathrm{Skel}^{\gamma,s}(n)$ is 
     obtained from $\mathrm{Skel}^{\gamma,s}(n-1)$ by adding the vertex $\gamma(n)$ and the edge $\{\gamma(n-1),\gamma(n)\}$ letting the root be $\gamma(n)$.  
\item {\bf Aldous-Broder move. } If $\gamma(n)$ is a vertex in ${\rm Skel}^{\gamma,s}(n-1)$ and ${\mathcal G}^{\gamma,s}(n)\setminus {\mathcal G}^{\gamma,s}(n-1)=\emptyset$ then a cycle is created and  ${\rm Skel}^{\gamma,s}(n)$ is constructed from ${\rm Skel}^{\gamma,s}(n-1)$ by applying the Aldous-Broder algorithm. That is, the edge $\{\gamma(L_{\gamma(n)}(n-1)+1),\gamma(n) \}$ is replaced by the edge $\{\gamma(n-1),\gamma(n) \}$. 
\item {\bf Ghost loop erasure. } If $\gamma(n)$ is a vertex in ${\rm Skel}^{\gamma,s}(n-1)$ and ${\mathcal G}^{\gamma,s}(n)\setminus {\mathcal G}^{\gamma,s}(n-1)\neq \emptyset$ then the ghost indices ${\mathcal G}^{\gamma,s}(n)\setminus {\mathcal G}^{\gamma,s}(n-1)$ created at time $n$ form a loop of length at most $s$, and this loop is erased in the skeleton. 
\end{itemize}
\label{remarkDynamicsOfSkeletonChain} 
\end{remark}

\begin{figure}
\includegraphics[width=12cm]{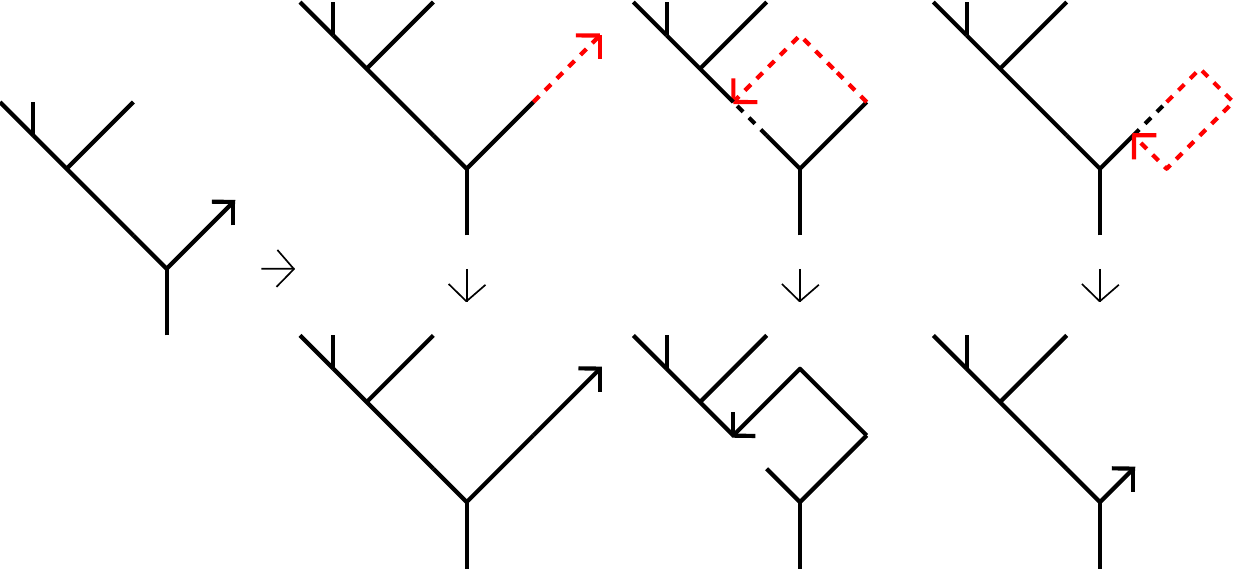}
\caption{Illustrates the dynamics of the skeleton chain. In the left-hand side, the skeleton up to a certain time is shown. In the first row, the last steps of the random walk are depicted in red. From left to right, there is root-growing, the Aldous-Broder algorithm is applied, or a small cycle is erased. In the bottom row the resulting skeleton is shown for each of the corresponding possibilities. }
\label{figEvolutionOfSkeleton}
\end{figure}

In the context of Remark \ref{remarkDynamicsOfSkeletonChain}, consider the specific instance where the skeleton chain undergoes a {\bf Ghost indices erasure} move. That is, $\gamma(n)\in {\rm Skel}^{\gamma,s}(n-1)$ and ${\mathcal G}^{\gamma,s}(n)\setminus {\mathcal G}^{\gamma,s}(n-1)\neq \emptyset$. In this case, ${\rm AB}^{\gamma}(n)$ performs an Aldous-Broder move, that is, ${\rm AB}^{\gamma}(n)$ is constructed from ${\rm AB}^{\gamma}(n-1)$ by replacing the edge $\{\gamma(L_{\gamma(n)}(n-1)+1),\gamma(n) \}$ by the edge $\{\gamma(n-1),\gamma(n) \}$. ($\gamma(n)$ is the new root of ${\rm AB}^{\gamma}(n)$.) 

By removing the edge $\{\gamma(L_{\gamma(n)}(n-1)+1),\gamma(n) \}$ from ${\rm AB}^{\gamma}(n-1)$ (before adding the new edge) we disconnect ${\rm AB}^{\gamma}(n-1)$ into two connected components (or two subtrees). Namely, the subtree above $\gamma(n)$ rooted at $\gamma(n)$, denoted by $S_{\gamma(n)}^{{\rm AB}}$, and the subtree that contains the root $\gamma(n-1)$ of ${\rm AB}^{\gamma}(n-1)$, denoted by $C_{\gamma(n-1)}^{{\rm AB}}$. Then, by adding the edge $\{\gamma(n-1),\gamma(n) \}$, one only attach to $C_{\gamma(n-1)}^{{\rm AB}}$ the vertex $\gamma(n)$. Denote this new subtree by $\widetilde{C}_{\gamma(n)}^{{\rm AB}}$. Thus, ${\rm AB}^{\gamma}(n)$ is obtained by gluing together $\widetilde{C}_{\gamma(n)}^{{\rm AB}}$ and $S_{\gamma(n)}^{{\rm AB}}$ at the vertex $\gamma(n)$. Indeed, by the Aldous-Broder algorithm, $C_{\gamma(n-1)}^{{\rm AB}}$ is the subtree of ${\rm AB}^{\gamma}(n)$ above $\gamma(n-1)$, rooted at $\gamma(n-1)$.

To conclude this section we show that, under certain assumptions, the number of vertices of the subtree $C_{\gamma(n-1)}^{{\rm AB}}$ can be bounded. 
To this end, we establish some key properties of its vertex composition.
Recall from (\ref{e:cutset}) the set of $s'$-local cut points, ${\rm CP}^{\gamma,s'}(A)$, in an finite interval $A$. 
In the following result, for $v\in V$ we use $\gamma^{-1}(v)$ to denote the inverse mapping $\{n\in \na_0:\gamma(n)=v\}$.

\begin{lemma}[Key properties of ghost loop erasure] 
Fix $s\in \na$. Let $\gamma:\na_0\to V$ be a path on a finite, simple, connected graph $G=(V,E)$. 
If  $n\in\mathbb{N}_0$ is such that $\gamma(n)\in {\rm Skel}^{\gamma,s}(n-1)$ and ${\mathcal G}^{\gamma,s}(n)\setminus {\mathcal G}^{\gamma,s}(n-1)\neq \emptyset$, then the following holds:
\begin{enumerate}[label=(\textbf{\roman*})]
\item \label{claim2} If $v$ is a vertex of $C_{\gamma(n-1)}^{{\rm AB}}$ then $\gamma^{-1}(v)\subseteq{\mathcal G}^{\gamma,s}(n)$.
\item \label{claim3} Let  $n \geq 2$, $m \in [1,n-1]\cap{\mathcal G}^{\gamma,s}(n)$, 
$\ell \in [0, m-1]$, and $k\in[0,\ell]$. If $h_1\in [k, \ell] \cap {\mathcal G}^{\gamma,s}(n)$ and $h_2\in [\ell+1,m]\cap {\mathcal G}^{\gamma,s}(n)$ with $\gamma(h_2) \not \in  {\rm R}^{\gamma}([k,\ell] \cap {\mathcal G}^{\gamma,s}(n))$ such that 
$\{\gamma(h_1), \gamma(h_2)\}$ is an edge in $C_{\gamma(n-1)}^{{\rm AB}}$, then $\gamma(h_1)\in {\rm R}^{\gamma}([\ell+1,n-1] \cap {\mathcal G}^{\gamma,s}(n))$.

\item \label{claim4}  Let $n \geq 2$ and $l_1,l_2\in  [0,n-1] \cap {\mathcal G}^{\gamma,s}(n)$ with $l_1<l_2$ and such that 
$\gamma(l_1),\gamma(l_2)$ are vertices in $C_{\gamma(n-1)}^{{\rm AB}}$. If we can choose $\ell\in[l_1,l_2-1]$ such that ${\rm R}^\gamma([l_{1}, \ell]) \cap {\rm R}^\gamma([l_{2}+1, n]) = \emptyset$, then 
there exists $l_{2}^{\ast} \in [\ell+1, l_{2}]$ and a path connecting $\gamma(l_{1})$ and $\gamma(l_{2}^{\ast})$ in $C_{\gamma(n-1)}^{{\rm AB}}$ with edges $\{ \{\gamma(i_{j}), \gamma(i_{j+1}) \} \}_{0 \leq j \leq \tilde{k}-1}$ such that $i_0<i_1<\cdots<i_{\tilde{k}}=l_{2}^{\ast}$, $\gamma(i_{0}) = \gamma(l_{1})$, $i_{0} \in [l_{1}, \ell]$ and $\tilde{k}\in\mathbb{N}$.
\end{enumerate}
\label{SeparateClaims}
\end{lemma}

\begin{proof}{\em \ref{claim2}} follows immediately since  ${\rm Skel}^{\gamma,s}(n)$ is connected by Lemma~\ref{SkeLemma2}.
\smallskip

To prove \ref{claim3}, assume that $h_1\in[k, \ell] \cap {\mathcal G}^{\gamma,s}(n)$ and $h_2\in [\ell+1,m]\cap {\mathcal G}^{\gamma,s}(n)$ with $\gamma(h_2) \not \in  {\rm R}^{\gamma}([k,\ell] \cap {\mathcal G}^{\gamma,s}(n))$ such that
$\{\gamma(h_1), \gamma(h_2)\}$ is an edge in $C_{\gamma(n-1)}^{{\rm AB}}$. 
Then  $\gamma(h_1)$ and $\gamma(h_2)$ are different from $\gamma(n)$ because $\gamma(n)\in S^{{\mathrm{AB}}}_{\gamma(n)}$ and $S^{{\mathrm{AB}}}_{\gamma(n)}\cap C^{{\mathrm{AB}}}_{\gamma(n-1)}=\emptyset$.
Clearly, $\gamma(h_1)=\gamma(L_{\gamma(h_1)}(n))$. If $L_{\gamma(h_1)}(n) > h_{2}$, then our claim follows immediately because $L_{\gamma(h_1)}(n) \in {\mathcal G}^{\gamma,s}(n)$ by part~\ref{claim2} and  $h_2 \geq \ell+1$ by assumption.

Now, assume that $h_{1}\leq L_{\gamma(h_1)}(n) < h_{2}$. Then, $\gamma(h_2) = \gamma(L_{\gamma(h_1)}(n)+1)$ by \eqref{e:102AB}.  
Thus  $L_{\gamma(h_1)}(n), L_{\gamma(h_1)}(n)+1 \in {\mathcal G}^{\gamma,s}(n)$ by part~\ref{claim2}. Since $\gamma(L_{\gamma(h_1)}(n)+1)=\gamma(h_2)\not\in {\rm R}^\gamma([k,\ell]\cap {\mathcal G}^{\gamma,s}(n))$, it follows that $L_{\gamma(h_1)}(n)+1\in ([0,k-1]\cup[\ell+1,n])\cap {\mathcal G}^{\gamma,s}(n)$. Further, since $L_{\gamma(h_1)}(n)+1\ge h_1+1\ge k+1$, it follows that $L_{\gamma(h_1)}(n)\in [\ell+1,n-1]$, and thus $\gamma(h_1)=\gamma(L_{\gamma(h_1)}(n))\in {\rm R}^\gamma([\ell+1,n])$. This together with $L_{\gamma(h_1)}(n)\in {\mathcal G}^{\gamma,s}(n)$ gives the claim. \smallskip

To prove~\ref{claim4}, recall from the Aldous-Broder algorithm (see \eqref{e:102AB}) that there exists a path connecting $\gamma(l_{1})$ and $\gamma(l_{2})$ in ${\rm AB}^\gamma(l_2)$ with edges $\{ \{\gamma(i_{j}), \gamma(i_{j+1}) \} \}_{0 \leq j \leq \tilde{k}-1}$ such that $i_0<i_1<\cdots<i_{\tilde{k}}=l_2$, $\gamma(i_{0}) = \gamma(l_{1})$ and $\tilde{k}\in\mathbb{N}$. Suppose $\ell\in[l_1,l_2-1]$ is such that ${\rm R}^\gamma([l_{1}, \ell]) \cap {\rm R}^\gamma([l_{2}+1, n]) = \emptyset$. Since $\gamma(i_{0}) = \gamma(l_{1})$, we have $i_{0} \in [l_{1}, \ell]$. Furthermore, we can and will choose $i_{0} = L_{\gamma(l_{1})}(l_{2})$.
Put $k(\ell) \coloneqq {\rm max}\{j\in \{1,\dots,\tilde{k}\}:\, i_{j}\le\ell\}$ and $l_{2}^{\ast} \coloneqq i_{k(\ell) +1}$. Then, in particular, 
\begin{equation} \label{claim4eq1}
   {\rm R}^\gamma\big(\{i_{0},..., i_{k(\ell)} \}\big) \cap {\rm R}^\gamma\big([l_{2}+1, n]\big) = \emptyset,
\end{equation} 
and $\gamma(i_{0}), \dots, \gamma(i_{k(\ell)}), \gamma(l_{2}^{\ast})$ are vertices in the path connecting $\gamma(l_1)$ with $\gamma(l_2)$  in ${\rm AB}^\gamma(n)$. It follows from \eqref{e:102AB} that the collection of edges $\{ \{\gamma(i_{j}), \gamma(i_{j+1}) \} \}_{0 \leq j \leq \tilde{k}-1}$ form a path in ${\rm AB}^\gamma(n)$ connecting  $\gamma(l_1)$ with $\gamma(l_2^{\ast})$ such that it is contained in the path connecting $\gamma(l_1)$ with $\gamma(l_2)$  in ${\rm AB}^\gamma(n)$.
Since $C^{\mathrm{AB}}_{\gamma(n-1)}$ is a connected subtree of ${\rm AB}^\gamma(n)$, the path with edges $\{ \{\gamma(i_{j}), \gamma(i_{j+1}) \} \}_{0 \leq j \leq \tilde{k}-1}$ is also a path in $C^{\mathrm{AB}}_{\gamma(n-1)}$.   
\end{proof}

\begin{remark}
In the setting of Lemma \ref{SeparateClaims}.  Note that \ref{claim2} shows that the vertices in $C_{\gamma(n-1)}^{{\rm AB}}$ are only images of  ghost indexes under the path  $\gamma$.  
\end{remark}

Recall the definition of the interval  $B_{i}^{(r,s)}$, for $i \in \mathbb{N}$,  as given in  \eqref{eqnDefinitionAiBi}. Recall that $\gamma:\na_0\to V$ is a path on a finite, simple, connected graph $G=(V,E)$. Consider the following additional properties that complement properties \ref{Pro1}-\ref{Pro4}: 
\begin{enumerate}[label=(\textbf{P.\arabic*})]
\setcounter{enumi}{\value{Cond}}
\item \label{Pro5II} We have that 
\begin{align}
{\rm R}^\gamma ([0,r \lfloor N/r \rfloor -1]) \cap {\rm R}^\gamma ( [r\lfloor N/r \rfloor+2s, N] )  =\emptyset
\end{align}
\noindent and
\begin{align}
{\rm R}^\gamma ([0,r \lfloor N/r \rfloor -s-1]) \cap {\rm R}^\gamma ( [r\lfloor N/r \rfloor, r\lfloor N/r \rfloor+2s-1] )  =\emptyset.
\end{align}
\end{enumerate}

\begin{lemma}\label{ControlSubTreeSize}
Fix $r,s, s^{\prime} \in \na$ with $r \geq 3s+1 \geq 18s^{\prime} +1$. Let $\gamma:\na_0\to V$ be a path on a finite, simple, connected graph $G=(V,E)$ that satisfies Assumptions \hyperref[assumptionNoLoopsOfIntermediateLength]{No loops of intermediate length}, \hyperref[assumptionLocalCutPoints]{Local cut points} and the properties \ref{Pro1}-\ref{Pro5II}. 
Then, for all $n\in[0,N]$, such that $\gamma(n)\in {\rm Skel}^{\gamma,s}(n-1)$ and ${\mathcal G}^{\gamma,s}(n)\setminus {\mathcal G}^{\gamma,s}(n-1)\neq \emptyset$,
\begin{align} \label{ControlSubTreeSizeeq1}
\# C_{\gamma(n-1)}^{{\rm AB}} \leq r.
\end{align}
\end{lemma}

\begin{proof}
We assume $n \geq r+1$, otherwise \eqref{ControlSubTreeSizeeq1} holds trivially. By Lemma \ref{SeparateClaims} \ref{claim2}, the vertices of $C_{\gamma(n-1)}^{{\rm AB}}$ belong to ${\rm R}^{\gamma}({\mathcal G}^{\gamma,s}(n))$, where ${\mathcal G}^{\gamma,s}(n)$ satisfies \eqref{SkeLemma2eq8} (with $n=n_0$), and $C_{\gamma(n-1)}^{{\rm AB}}$ does not contain vertices of ${\rm R}^{\gamma}([0,n]\setminus {\mathcal G}^{\gamma,s}(n))$. 

Recall that $n\geq r+1$ and $n-1 \in {\mathcal G}^{\gamma,s}(n)$ by \eqref{SkeLemma2eq8}. Recall also that $C_{\gamma(n-1)}^{{\rm AB}}$ is the subtree of ${\rm AB}^{\gamma}(n)$ above $\gamma(n-1)$, rooted at $\gamma(n-1)$.  By the Definition of  $C_{\gamma(n-1)}^{{\rm AB}}$ and  Lemma \ref{SeparateClaims} \ref{claim2}, the subtree $C_{\gamma(n-1)}^{{\rm AB}}$ is a subgraph of the subgraph of ${\rm AB}(n)$ restricted to ${\rm R}^{\gamma}([0, n-1] \cap {\mathcal G}^{\gamma,s}(n))$. For simplicity, we will also use $C_{\gamma(n-1)}^{{\rm AB}}$ to denote the vertex set of $C_{\gamma(n-1)}^{{\rm AB}}$.

Next we use the properties established in Lemma \ref{SeparateClaims} to control the size of  $C_{\gamma(n-1)}^{{\rm AB}}$, or more precisely, its vertex set. Let $\ell \in [n-s-2, n-2]$ be a $2s^{\prime}$-local cutpoint (by Assumption  \hyperref[assumptionLocalCutPoints]{Local cut points}, this $2s^{\prime}$-local cutpoint exists). Note that $C_{\gamma(n-1)}^{{\rm AB}} \cap {\rm R}^{\gamma}([\ell+1,n-1] \cap {\mathcal G}^{\gamma,s}(n)) \neq \emptyset$. \\

\noindent {\bf Case (i).} Suppose that ${\rm R}^{\gamma}([0,\ell] \cap {\mathcal G}^{\gamma,s}(n))\cap {\rm R}^{\gamma}([\ell+1,n-1] \cap {\mathcal G}^{\gamma,s}(n))  = \emptyset$. 

We claim that in this case $C_{\gamma(n-1)}^{{\rm AB}}$ is a subgraph of the subgraph of ${\rm AB}(n)$ restricted to ${\rm R}^{\gamma}([\ell+1, n-1] \cap {\mathcal G}^{\gamma,s}(n)))$. (That is, $C_{\gamma(n-1)}^{{\rm AB}}  \subseteq {\rm R}^{\gamma}([\ell+1, n-1] \cap {\mathcal G}^{\gamma,s}(n)) )$). 

To prove the above claim, we proceed by contradiction. Suppose that there exists $l_1 \in [0,\ell] \cap {\mathcal G}^{\gamma,s}(n)$ such that $\gamma(l_1) \in C_{\gamma(n-1)}^{{\rm AB}}$ and $\gamma(l_1) \not \in {\rm R}^{\gamma}([\ell+1, n-1] \cap {\mathcal G}^{\gamma,s}(n))$. 
Since $\gamma(n-1) \in C_{\gamma(n-1)}^{{\rm AB}} \cap {\rm R}^{\gamma}([\ell+1, n-1] \cap {\mathcal G}^{\gamma,s}(n))$, then Lemma \ref{SeparateClaims} \ref{claim4} implies that there are $h_1 \in [0,\ell] \cap {\mathcal G}^{\gamma,s}(n)$ and $h_2 \in [\ell+1,n-1] \cap {\mathcal G}^{\gamma,s}(n)$ such that $\{\gamma(h_1), \gamma(h_2) \}$ is an edge in $C_{\gamma(n-1)}^{{\rm AB}}$. But this is a contradiction to Lemma \ref{SeparateClaims} \ref{claim3} with $k=0$. Indeed, since ${\rm R}^{\gamma}([0,\ell] \cap {\mathcal G}^{\gamma,s}(n))\cap {\rm R}^{\gamma}([\ell+1,n-1] \cap {\mathcal G}^{\gamma,s}(n))  = \emptyset$ then $\gamma(h_1) \not \in {\rm R}^{\gamma}([\ell+1,n-1] \cap {\mathcal G}^{\gamma,s}(n))$ and $\gamma(h_2) \not \in {\rm R}^{\gamma}([0,\ell] \cap {\mathcal G}^{\gamma,s}(n))$, which are conditions required in Lemma \ref{SeparateClaims} \ref{claim3}.

This finishes the proof of the claim in {\bf Case (i)}. \\

Next, we will consider the case 
\begin{align}
{\rm R}^{\gamma}([0,\ell] \cap {\mathcal G}^{\gamma,s}(n)) \cap {\rm R}^{\gamma}([\ell+1,n-1] \cap {\mathcal G}^{\gamma,s}(n))\neq \emptyset.
\end{align}
\noindent However, our assumptions allow us to simplify this case further. Recall that $\ell \in [n-s-2, n-2]$ is a $2s^{\prime}$-local cutpoint. Then, by Definition \ref{LocaCutP} and the Assumption \hyperref[assumptionNoLoopsOfIntermediateLength]{No loops of intermediate length},
\begin{align} \label{ControlSubTreeSizeEq5}
{\rm R}^{\gamma}([n-r-1,\ell]) \cap{\rm R}^{\gamma}([\ell+1,n-1]) = \emptyset.
\end{align}
\noindent The above implies that we only need to consider the case 
\begin{align} \label{ControlSubTreeSizeEq6b}
{\rm R}^{\gamma}([0,n-r-2] \cap {\mathcal G}^{\gamma,s}(n)) \cap {\rm R}^{\gamma}([\ell+1,n-1] \cap {\mathcal G}^{\gamma,s}(n)) \neq \emptyset.
\end{align}

Recall the definition of the interval $B_{i}^{(r,s)}$ in \eqref{eqnDefinitionAiBi}. Since the interval $[\ell+1,n-1]$ is of length at most $s+1$, if \eqref{ControlSubTreeSizeEq6b} holds, then by \ref{Pro1}-\ref{Pro5II}, there must exists a unique pair $j_{1}, j_{2} \in [0, \lfloor N/r \rfloor]$ such that $j_{1} < j_{2}$ and 
\begin{align} \label{ControlSubTreeSizeEq6}
{\rm R}^{\gamma}(B^{(r,s)}_{j_{1}} \cap  [0,n-r-2] \cap {\mathcal G}^{\gamma,s}(n))\cap {\rm R}^{\gamma}(B^{(r,s)}_{j_{2}} \cap [\ell+1,n-1] \cap {\mathcal G}^{\gamma,s}(n))  \neq \emptyset,
\end{align}
\noindent 
\begin{align} \label{ControlSubTreeSizeEq7}
{\rm R}^{\gamma}((B^{(r,s)}_{j_{1}})^{c} \cap  [0,\min B^{(r,s)}_{j_2}-2s-1]) \cap{\rm R}^{\gamma}(B^{(r,s)}_{j_{2}} \cap [\ell+1,n-1]) = \emptyset,
\end{align}
\noindent 
\begin{align} \label{ControlSubTreeSizeEq8}
{\rm R}^{\gamma}([0,\max B^{(r,s)}_{j_{2}-1}] ) \cap {\rm R}^{\gamma}([\min B^{(r,s)}_{j_{2}}-2s,\min B^{(r,s)}_{j_{2}}-1] \cap [\ell+1,n-1])= \emptyset,
\end{align}
\noindent and
\begin{align} \label{ControlSubTreeSizeEq9}
{\rm R}^{\gamma}([0,\min B^{(r,s)}_{j_{2}}-1])\cap {\rm R}^{\gamma}([\max B^{(r,s)}_{j_{2}}+1,\max B^{(r,s)}_{j_{2}}+s] \cap [\ell+1,n-1]) = \emptyset. 
\end{align}
Note that in \eqref{ControlSubTreeSizeEq7} we denote by $(B^{(r,s)}_{j_{1}})^{c}$ the complement of $B^{(r,s)}_{j_{1}}$. 
\noindent Thus, we only need to consider the following case. 
By the Assumption \hyperref[assumptionLocalCutPoints]{Local cut points}, there exists a $2s^{\prime}$-local cut point $\widetilde{\ell} \in [\min B^{(r,s)}_{j_1}-s-1, \min B^{(r,s)}_{j_1}-1]$. Define $\widetilde B^{(r,s)}_{j_1}:=[\widetilde \ell+1,\max B^{(r,s)}_{j_{1}}]$. \\

\noindent {\bf Case (ii).} Suppose that there exists a unique pair $j_{1}, j_{2} \in [0, \lfloor N/r \rfloor]$ such that $j_{1} < j_{2}$ and \eqref{ControlSubTreeSizeEq6}-\eqref{ControlSubTreeSizeEq9} hold. 

We claim that $C_{\gamma(n-1)}^{{\rm AB}}$ is a subgraph of the subgraph of ${\rm AB}(n)$ restricted to ${\rm R}^{\gamma}((\widetilde B^{(r,s)}_{j_1}\cup [\ell+1, n-1] ) \cap {\mathcal G}^{\gamma,s}(n))$. (That is, $C_{\gamma(n-1)}^{{\rm AB}} \subseteq {\rm R}^{\gamma}((\widetilde B^{(r,s)}_{j_1}\cup [\ell+1, n-1] ) \cap {\mathcal G}^{\gamma,s}(n))$). \\

Before proving our claim in {\bf Case (ii)}, we make some useful observations. Since \eqref{ControlSubTreeSizeEq6} holds and the interval $[\ell+1,n-1]$ is of length at most $s+1$, we necessary have
\begin{align} \label{ControlSubTreeSizeEq10}
[\ell+1,n] \subseteq [(j_{2}-1)r, j_{2}r-1] =[\min B^{(r,s)}_{j_2}-2s,\max B^{(r,s)}_{j_2}+s].
\end{align}
\noindent In particular, we have that $n -r -2 \leq (j_{2}-1)r-3$. Then, by \eqref{ControlSubTreeSizeEq5} and \eqref{ControlSubTreeSizeEq7}-\eqref{ControlSubTreeSizeEq9},
\begin{align} \label{ControlSubTreeSizeEq11}
{\rm R}^{\gamma}([0, \widetilde{\ell}] \cup [\min B^{(r,s)}_{j_1}+1,\ell]) \cap {\rm R}^{\gamma}([\ell+1,n-1] ) = \emptyset.
\end{align}
\noindent This implies that the segment ${\rm R}^{\gamma}([\ell+1,n-1])$ intersects only with ${\rm R}^{\gamma}(\widetilde B^{(r,s)}_{j_1})$. 
On the other hand, since $\widetilde{\ell} \in [\min B^{(r,s)}_{j_1}-s-1, \min B^{(r,s)}_{j_1}-1]$ is a $2s^{\prime}$-local cut point (see Definition \ref{LocaCutP}), the Assumption \hyperref[assumptionNoLoopsOfIntermediateLength]{No loops of intermediate length} and \ref{Pro1}-\ref{Pro3} imply that
\begin{align} \label{ControlSubTreeSizeEq12}
{\rm R}^{\gamma}([0,\widetilde{\ell}]) \cap {\rm R}^{\gamma}(\widetilde B^{(r,s)}_{j_1}) = \emptyset.
\end{align}

We now turn to the proof of the claim stated after {\bf Case (ii)}. Namely, that $C_{\gamma(n-1)}^{{\rm AB}}$ is a subgraph of the subgraph of ${\rm AB}(n)$ restricted to ${\rm R}^{\gamma}((\widetilde B^{(r,s)}_{j_1}\cup [\ell+1, n-1] ) \cap {\mathcal G}^{\gamma,s}(n))$. We proceed by contradiction. Suppose that there exists $m \in ([0,\ell] \setminus \widetilde B^{(r,s)}_{j_1}) \cap {\mathcal G}^{\gamma,s}(n)$ such that $\gamma(m) \in C_{\gamma(n-1)}^{{\rm AB}}$ and $\gamma(m) \not \in {\rm R}^{\gamma}((\widetilde B^{(r,s)}_{j_1}\cup [\ell+1, n-1]) \cap {\mathcal G}^{\gamma,s}(n))$.

There are two sub-cases:

\begin{itemize}
\item[\bf 1.] Assume $m \in [\max B^{(r,s)}_{j_1}+1, \ell] \cap {\mathcal G}^{\gamma,s}(n)$. Since $\gamma(n-1) \in C_{\gamma(n-1)}^{{\rm AB}} \cap  {\rm R}^{\gamma}([\ell+1, n-1])$,  
then Lemma \ref{SeparateClaims} \ref{claim4} implies that there are $h_{1} \in [\max B^{(r,s)}_{j_1}+1, \ell] \cap {\mathcal G}^{\gamma,s}(n)$ and $h_{2} \in [\ell+1,n-1] \cap {\mathcal G}^{\gamma,s}(n)$ such that $\{\gamma(h_{1}), \gamma(h_{2}) \}$ is an edge in $C_{\gamma(n-1)}^{{\rm AB}}$. 
But this is a contradiction to Lemma \ref{SeparateClaims} \ref{claim3} (with $k=\max B_{j_1}^{(r,s)}+1$) since by \eqref{ControlSubTreeSizeEq11},
\begin{align}
{\rm R}^{\gamma}([\max B^{(r,s)}_{j_1}+1,\ell] \cap {\mathcal G}^{\gamma,s}(n))\cap {\rm R}^{\gamma}([\ell+1,n-1] \cap {\mathcal G}^{\gamma,s}(n))  = \emptyset.
\end{align}

\item[\bf 2.] Assume $m\in [0, \widetilde{\ell}] \cap {\mathcal G}^{\gamma,s}(n)$ and $\gamma(m) \not \in {\rm R}^{\gamma}([\max B^{(r,s)}_{j_1}+1, \ell] \cap {\mathcal G}^{\gamma,s}(n))$. Again, since $\gamma(n-1) \in C_{\gamma(n-1)}^{{\rm AB}} \cap  {\rm R}^{\gamma}([\ell+1, n-1])$, Lemma \ref{SeparateClaims} \ref{claim4} implies that there are $h_{1} \in [m,\ell] \cap {\mathcal G}^{\gamma,s}(n)$ and $h_{2} \in [\ell+1,n-1] \cap {\mathcal G}^{\gamma,s}(n)$ such that $\{\gamma(h_{1}), \gamma(h_{2}) \}$ is an edge in $C_{\gamma(n-1)}^{{\rm AB}}$. (Note that $\gamma(h_{1})$ must be an ancestor of $\gamma(m)$ in the tree ${\rm AB}(n)$.) Then, there are three sub-cases: 

\item[\bf 2.1.] If $h_{1} \in [\max B^{(r,s)}_{j_1}+1, \ell] \cap {\mathcal G}^{\gamma,s}(n)$, then a similar argument as in sub-case {\bf 1.}\ yields a contradiction to Lemma \ref{SeparateClaims} \ref{claim3}.

\item[\bf 2.2.] If $h_{1} \in [\widetilde{\ell} +1, \max B^{(r,s)}_{j_1}] \cap {\mathcal G}^{\gamma,s}(n)$ and $\gamma(h_{1}) \not \in {\rm R}^{\gamma}([\max B^{(r,s)}_{j_1}+1, \ell] \cap {\mathcal G}^{\gamma,s}(n))$, then since $\gamma(h_{1}) \in C_{\gamma(n-1)}^{{\rm AB}} \cap {\rm R}^{\gamma}([\widetilde{\ell} +1, \max B^{(r,s)}_{j_1}] \cap {\mathcal G}^{\gamma,s}(n))$ and $\gamma(h_{1})$ is an ancestor of $\gamma(m)$, by Lemma \ref{SeparateClaims} \ref{claim4}  there are $l_{3} \in [h_1, \widetilde{\ell}] \cap {\mathcal G}^{\gamma,s}(n)$ and $l_{4}\in [\widetilde{\ell}+1,h_{1}] \cap {\mathcal G}^{\gamma,s}(n)$ such that $\{\gamma(l_{3}), \gamma(l_{4}) \}$ is an edge in $C_{\gamma(n-1)}^{{\rm AB}}$. 

First, if $\gamma(l_{3}) \in {\rm R}^{\gamma}([\max B^{(r,s)}_{j_1}+1, \ell]\cap {\mathcal G}^{\gamma,s}(n))$, then there exists $l_{5} \in [\max B^{(r,s)}_{j_1}+1, \ell]$ such that $\gamma(l_{3}) = \gamma(l_{5})$ and $\gamma(l_{5}) \not \in 
{\rm R}^{\gamma}(([\ell+1, n-1] \cup [\widetilde{\ell}+1, \max B^{(r,s)}_{j_1}]) \cap {\mathcal G}^{\gamma,s}(n))$ by \eqref{ControlSubTreeSizeEq11} and \eqref{ControlSubTreeSizeEq12}. Then a similar argument as in sub-case {\bf 1.}\ yields a contradiction to Lemma \ref{SeparateClaims} \ref{claim3}. On the other hand, if $\gamma(l_{3}) \not \in {\rm R}^{\gamma}([\max B^{(r,s)}_{j_1}+1, \ell]\cap {\mathcal G}^{\gamma,s}(n))$, then we get a contradiction to Lemma \ref{SeparateClaims} \ref{claim3} since by \eqref{ControlSubTreeSizeEq11} and \eqref{ControlSubTreeSizeEq12}, we have have
\begin{align}
{\rm R}^{\gamma}([0,\widetilde{\ell}] \cap {\mathcal G}^{\gamma,s}(n)) \cap {\rm R}^{\gamma}(([\widetilde{\ell}+1,\max B^{(r,s)}_{j_1}] \cup [\ell +1, n-1]) \cap {\mathcal G}^{\gamma,s}(n)) = \emptyset.
\end{align}

\item[\bf 2.3.] If $h_{1} \in [0, \widetilde{\ell}] \cap {\mathcal G}^{\gamma,s}(n)$ and $\gamma(h_{1}) \not \in {\rm R}^{\gamma}([\widetilde{\ell}+1, \ell] \cap {\mathcal G}^{\gamma,s}(n))$, then we have two further cases to consider (by \eqref{ControlSubTreeSizeEq6}). Namely, $\gamma(h_{2}) \in {\rm R}^{\gamma}([\widetilde{\ell}+1, \max B^{(r,s)}_{j_1}] \cap {\mathcal G}^{\gamma,s}(n))$ and $\gamma(h_{2}) \not \in {\rm R}^{\gamma}([\widetilde{\ell}+1, \max B^{(r,s)}_{j_1}] \cap {\mathcal G}^{\gamma,s}(n))$. 

First, suppose that there exists $\widetilde{l}_{2} \in [\widetilde{\ell}+1, \max B^{(r,s)}_{j_1}] \cap {\mathcal G}^{\gamma,s}(n)$ such that $\gamma(\widetilde{l}_{2}) = \gamma(h_{2})$ (i.e., $\gamma(h_{2}) \in {\rm R}^{\gamma}([\widetilde{\ell}+1, \max B^{(r,s)}_{j_1}] \cap {\mathcal G}^{\gamma,s}(n))$. Then a similar argument as in the previous sub-case {\bf 2.2.}\ yields a contradiction to Lemma \ref{SeparateClaims} \ref{claim3}. On the other hand, if $\gamma(h_{2}) \not \in {\rm R}^{\gamma}([\widetilde{\ell}+1, \max B^{(r,s)}_{j_1}] \cap {\mathcal G}^{\gamma,s}(n))$ (i.e., $\gamma(l_{2}) \in {\rm R}^{\gamma}([\ell+1, n-1] \cap {\mathcal G}^{\gamma,s}(n))$), we get a contradiction to Lemma \ref{SeparateClaims} \ref{claim3} since by \eqref{ControlSubTreeSizeEq11},
\begin{align}
{\rm R}^{\gamma}(([0,\widetilde{\ell}] \cup [\max B^{(r,s)}_{j_1}+1,\ell]) \cap {\mathcal G}^{\gamma,s}(n)) \cap {\rm R}^{\gamma}([\ell+1,n-1] \cap {\mathcal G}^{\gamma,s}(n)) = \emptyset.
\end{align}

\end{itemize}
\noindent This finishes the proof of {\bf Case (ii)}. \\

Finally, recall that $n \geq r +1$, $n-1 \in {\mathcal G}^{\gamma,s}(n)$ by \eqref{SkeLemma2eq8} (with $n=n_0$), and that $C_{\gamma(n-1)}^{{\rm AB}}$ is the subtree of ${\rm AB}^{\gamma}(n)$ above $\gamma(n-1)$, rooted at $\gamma(n-1)$. Recall also that the interval $[\ell+1,n-1]$ is of length at most $s+1$. Since $\widetilde{\ell} \in [\min B^{(r,s)}_{j_1}-s-1, \min B^{(r,s)}_{j_1}-1]$ is a $2s^{\prime}$-local cut point, \eqref{ControlSubTreeSizeeq1} follows from the claims in {\bf Case (i)} and {\bf Case (ii)} by noticing that the interval $[\widetilde{\ell}+1, \max B^{(r,s)}_{j_1}]$ is of length at most $r-2s$.
\end{proof}

\subsection{Distance between the Aldous-Broder chain and the skeleton chain}
\label{Sub:rsAB} 

We have proven in Lemma \ref{SkeLemma2} that if the path $\gamma:\na_0\to V$  on a finite, simple, connected graph $G=(V,E)$ satisfies the Assumption \hyperref[assumptionNoLoopsOfIntermediateLength]{No loops of intermediate length}, then $\mathrm{Skel}^{\gamma,s}(n)$ is a rooted subtree of $\mathrm{AB}^{\gamma}(n)$. We then view $\mathrm{Skel}^{(r,s)}_{\gamma}(n)$ and $\mathrm{AB}^{\gamma}(n)$ as pointed metric measure spaces, where as usual a finite rooted graph-theoretical tree $(T,d,\varrho)$ is identified with the pointed metric space $(T,d_{\mbox{\tiny graph}},\varrho)$ with $d_{\mbox{\tiny graph}}$ being the graph distance. Recall the pointed Gromov-Hausdorff distance between two compact pointed metric spaces defined in (\ref{e:009b}).

\begin{proposition}[Bounding the GH-distance of the Aldous-Broder chain to the skeleton chain]\label{DerterLemma1}
Fix $r,s, s^{\prime} \in \na$ with $r \geq 3s+1 \geq 18s^{\prime} +1$. Let $\gamma:\na_0\to V$ be a path on a finite, simple, connected graph $G=(V,E)$ that satisfies Assumptions \hyperref[assumptionNoLoopsOfIntermediateLength]{No loops of intermediate length}, \hyperref[assumptionLocalCutPoints]{Local cut points} and the properties \ref{Pro1}-\ref{Pro5II}. Then, for all $n\in[0,N]$, 
\begin{equation}
\begin{aligned} \label{GromovIneD}
d_{\rm GH}\big(\mathrm{Skel}^{\gamma,s}(n), {\rm AB}^{\gamma}(n)\big) \leq r.
\end{aligned}
\end{equation} 
\end{proposition}

\begin{proof} 
Recall by Lemma \ref{SkeLemma2} that $\mathrm{Skel}^{\gamma,s}(n)$ is a rooted subtree of  ${\rm AB}^{\gamma}(n)$, for all $n\in[0,N]$. We claim that for all $x \in {\rm AB}^{\gamma}(n)$, there exists $y \in \mathrm{Skel}^{\gamma,s}(n)$ such that 
\begin{align} \label{DerterLemma1eq1}
d_{\mbox{\tiny graph}}^{(n)}(x, y) \leq r,
\end{align}
\noindent for all $n\in[0,N]$, where $d_{\mbox{\tiny graph}}^{(n)}$ denotes the graph distance relative to the tree ${\rm AB}^{\gamma}(n)$. Clearly, \eqref{DerterLemma1eq1} implies \eqref{GromovIneD}. Therefore, we proceed to prove \eqref{DerterLemma1eq1} by induction. 

If $n \in [0,r-1]$, then, by Lemma \ref{SkeLemma2}, for all $x \in {\rm AB}^{\gamma}(n)$, there exists $y \in \mathrm{Skel}^{\gamma,s}(n)$ such that $d_{\mbox{\tiny graph}}^{(n)}(x, y) \leq r$. This shows \eqref{DerterLemma1eq1} for $n \in [0,r-1]$. Then, we assume that for all $x \in {\rm AB}^{\gamma}(n-1)$, there exists $y \in \mathrm{Skel}^{\gamma,s}(n-1)$ such that 
\begin{align}   \label{DerterLemma1eq5}
d_{\mbox{\tiny graph}}^{(n-1)}(x, y) \leq r,
\end{align}
\noindent for $n \in [1,N]$. By using the dynamics of the Skeleton chain described in Remark  \ref{remarkDynamicsOfSkeletonChain} and the induction hypothesis \eqref{DerterLemma1eq5}, we will prove \eqref{DerterLemma1eq1} by considering the following cases: \\

{\bf Case 1 (Root-growth type I).} Suppose that $\gamma(n)\notin {\rm Skel}^{\gamma,s}(n-1)$ (i.e. $\gamma(n)\notin {\rm R}^\gamma ([0,n-1]\setminus \mathcal{G}^{\gamma,s}(n-1))$) and $\gamma(n)\notin {\rm AB}^{\gamma}(n-1)$. In this case, both trees $\mathrm{Skel}^{\gamma,s}(n)$ and ${\rm AB}^{\gamma}(n)$ are obtained by adding the same vertex $\gamma(n)$ and the edge $\{\gamma(n-1),\gamma(n)\}$ (the new root in both trees is $\gamma(n)$). Therefore, it should be clear that by the induction hypothesis \eqref{DerterLemma1eq5}, \eqref{DerterLemma1eq1} is satisfied. \\

{\bf Case 2 (Root-growth type II).} Suppose that $\gamma(n)\notin {\rm Skel}^{\gamma,s}(n-1)$ and $\gamma(n) \in {\rm AB}^{\gamma}(n-1)$. On the one hand, $\mathrm{Skel}^{\gamma,s}(n)$ is obtained by adding the vertex $\gamma(n)$ and the edge $\{\gamma(n-1),\gamma(n)\}$. On the other hand, ${\rm AB}^{\gamma}(n)$ is obtained by doing an Aldous-Broder move, that is, ${\rm AB}^{\gamma}(n)$ is constructed from ${\rm AB}^{\gamma}(n-1)$ by replacing the edge $\{\gamma(L_{\gamma(n)}(n-1)+1),\gamma(n) \}$  by the edge $\{\gamma(n-1),\gamma(n) \}$. (Again, $\gamma(n)$ is the new root in both trees.)

Note that the removed edge $\{\gamma(L_{\gamma(n)}(n-1)+1),\gamma(n) \}$ does not belong to ${\rm Skel}^{\gamma,s}(n-1)$ since in this case the Skeleton chain is doing a root-growth movement (recall Remark \ref{remarkDynamicsOfSkeletonChain}). Moreover, by removing the edge $\{\gamma(L_{\gamma(n)}(n-1)+1),\gamma(n) \}$ from ${\rm AB}^{\gamma}(n-1)$ (before adding the new edge) we disconnect ${\rm AB}^{\gamma}(n-1)$ into two connected components (or two subtrees). One is the subtree above $\gamma(n)$ rooted at $\gamma(n)$, denoted by $S_{\gamma(n)}^{{\rm AB}}$. The other component contains the root $\gamma(n-1)$ of ${\rm AB}^{\gamma}(n-1)$ and is denoted by $C_{\gamma(n-1)}^{{\rm AB}}$. 
Then, by adding the edge $\{\gamma(n-1),\gamma(n) \}$, one only attach to $C_{\gamma(n-1)}^{{\rm AB}}$ the vertex $\gamma(n)$. Denote this new subtree by $\widetilde{C}_{\gamma(n)}^{{\rm AB}}$. (That is,  ${\rm AB}^{\gamma}(n)$ is obtained by gluing together $\widetilde{C}_{\gamma(n)}^{{\rm AB}}$ and $S_{\gamma(n)}^{{\rm AB}}$ at the vertex $\gamma(n)$.)

Clearly, ${\rm Skel}^{\gamma,s}(n)$ is a subtree of $\widetilde{C}_{\gamma(n)}^{{\rm AB}}$.
Then, by the induction hypothesis \eqref{DerterLemma1eq5}, we have that for all $x \in \widetilde{C}_{\gamma(n)}^{{\rm AB}}$, there exists $y \in \mathrm{Skel}^{\gamma,s}(n)$ such that $d_{\mbox{\tiny graph}}^{(n)}(x, y) \leq r$. On the other hand, from the induction hypothesis \eqref{DerterLemma1eq5}, we know that for all $x \in S_{\gamma(n)}^{{\rm AB}}$ there exists $y \in {\rm Skel}^{\gamma,s}(n-1)$ such that $d_{\mbox{\tiny graph}}^{(n-1)}(x, y) \leq r$. However, note that necessary $\gamma(n)$ is in the path connecting $x$ and $y$ in ${\rm AB}^{\gamma}(n-1)$. Then, $d_{\mbox{\tiny graph}}^{(n)}(x, \gamma(n)) = d_{\mbox{\tiny graph}}^{(n-1)}(x, \gamma(n)) \leq r$. But $\gamma(n) \in {\rm Skel}^{\gamma,s}(n)$ (indeed, $\gamma(n)$ is the root of ${\rm Skel}^{\gamma,s}(n)$).
Thus, the above shows that \eqref{DerterLemma1eq1} also holds in this case. \\

{\bf Case 3 (Aldous-Broder move).} Suppose that $\gamma(n)\in {\rm Skel}^{\gamma,s}(n-1)$ and ${\mathcal G}^{\gamma,s}(n)\setminus {\mathcal G}^{\gamma,s}(n-1)=\emptyset$.  In this case, both trees $\mathrm{Skel}^{\gamma,s}(n)$ and ${\rm AB}^{\gamma}(n)$ are obtained by replacing the edge $\{\gamma(L_{\gamma(n)}(n-1)+1),\gamma(n) \}$ by the edge $\{\gamma(n-1),\gamma(n) \}$. (Again, $\gamma(n)$ is the new root in both trees.) 

Note that by removing the edge $\{\gamma(L_{\gamma(n)}(n-1)+1),\gamma(n) \}$ from ${\rm AB}^{\gamma}(n-1)$ (before adding the new edge) we disconnect ${\rm AB}^{\gamma}(n-1)$ into two connected components (or two subtrees). Namely, $S_{\gamma(n)}^{{\rm AB}}$ and $C_{\gamma(n-1)}^{{\rm AB}}$ defined as in the previous {\bf Case 2 (Root-growth type II)}. Let $\widetilde{C}_{\gamma(n)}^{{\rm AB}}$ denote the subtree defined as in {\bf Case 2 (Root-growth type II)}.

Similarly, by removing the edge $\{\gamma(L_{\gamma(n)}(n-1)+1),\gamma(n) \}$ from ${\rm Skel}^{\gamma,s}(n-1)$ (before adding the new edge) we disconnect ${\rm Skel}^{\gamma,s}(n-1)$ into two connected components (or two subtrees). One is the subtree above $\gamma(n)$ rooted at $\gamma(n)$, denoted by $S_{\gamma(n)}^{{\rm Skel}}$. The other component contains the root $\gamma(n-1)$ of ${\rm Skel}^{\gamma,s}(n-1)$ and is denoted by $C_{\gamma(n-1)}^{{\rm Skel}}$. Then, by adding the edge $\{\gamma(n-1),\gamma(n) \}$, one only attach to $C_{\gamma(n-1)}^{{\rm Skel}}$ the vertex $\gamma(n)$. Denote this new subtree by $\widetilde{C}_{\gamma(n)}^{{\rm Skel}}$. (That is,  ${\rm Skel}^{\gamma,s}(n)$ is obtained by gluing together $\widetilde{C}_{\gamma(n)}^{{\rm Skel}}$ and $S_{\gamma(n)}^{{\rm Skel}}$ at the vertex $\gamma(n)$.)

Clearly, $S_{\gamma(n)}^{{\rm Skel}}$ is a subtree of $S_{\gamma(n)}^{{\rm AB}}$ and $\widetilde{C}_{\gamma(n)}^{{\rm Skel}}$ is a subtree of $\widetilde{C}_{\gamma(n)}^{{\rm AB}}$. Since $\gamma(n) \in {\rm Skel}^{\gamma,s}(n)$ (indeed, $\gamma(n)$ is the root of ${\rm Skel}^{\gamma,s}(n)$), it follows from the induction hypothesis \eqref{DerterLemma1eq5}, that necessary for all $x \in S_{\gamma(n)}^{{\rm AB}}$ and $x^{\prime} \in \widetilde{C}_{\gamma(n)}^{{\rm AB}}$, there are $y \in S_{\gamma(n)}^{{\rm Skel}}$ and $y^{\prime} \in \widetilde{C}_{\gamma(n)}^{{\rm Skel}}$ such that $d_{\mbox{\tiny graph}}^{(n)}(x, y ) \leq r$ and $d_{\mbox{\tiny graph}}^{(n)}(x^{\prime}, y^{\prime} )  \leq r$. Thus, \eqref{DerterLemma1eq1} also holds in this case. \\

{\bf Case 4 (Ghost indices erasure).} Suppose that $\gamma(n)\in {\rm Skel}^{\gamma,s}(n-1)$ and ${\mathcal G}^{\gamma,s}(n)\setminus {\mathcal G}^{\gamma,s}(n-1)\neq \emptyset$. 
In this case, all the ghost indices created ${\mathcal G}^{\gamma,s}(n)\setminus {\mathcal G}^{\gamma,s}(n-1)$ form a small loop at time $n$ (that is, a loop of length less or equal to $s^{\prime}-1$ by the Assumption \hyperref[assumptionNoLoopsOfIntermediateLength]{No loops of intermediate length}). This loop is erased in the skeleton. 

On the other hand, ${\rm AB}^{\gamma}(n)$ is obtained by doing an Aldous-Broder move, that is, ${\rm AB}^{\gamma}(n)$ is constructed from ${\rm AB}^{\gamma}(n-1)$ by replacing the edge $\{\gamma(L_{\gamma(n)}(n-1)+1),\gamma(n) \}$  by the edge $\{\gamma(n-1),\gamma(n) \}$. ($\gamma(n)$ is the new root of ${\rm AB}^{\gamma}(n)$.) As in {\bf Case 2} {\bf (Root-growth type II)} and {\bf Case 3} {\bf (Aldous-Broder move)}, by removing the edge $\{\gamma(L_{\gamma(n)}(n-1)+1),\gamma(n) \}$ from ${\rm AB}^{\gamma}(n-1)$ (before adding the new edge) we disconnect ${\rm AB}^{\gamma}(n-1)$ into two connected components (or two subtrees). Namely, $S_{\gamma(n)}^{{\rm AB}}$ and $C_{\gamma(n-1)}^{{\rm AB}}$ defined as in {\bf Case 2 (Root-growth type II)}. Let $\widetilde{C}_{\gamma(n)}^{{\rm AB}}$ denote the subtree defined as in {\bf Case 2 (Root-growth type II)}.

Similarly, by removing the edge $\{\gamma(L_{\gamma(n)}(n-1)+1),\gamma(n) \}$ from ${\rm Skel}^{\gamma,s}(n-1)$ we disconnect ${\rm Skel}^{\gamma,s}(n-1)$ into two connected components (or two subtrees). Namely, $S_{\gamma(n)}^{{\rm Ske}}$ and $C_{\gamma(n-1)}^{{\rm Ske}}$ defined as in {\bf Case 3 (Aldous-Broder move)}. 
Then, ${\rm Skel}^{\gamma,s}(n)$ is exactly the connected component $S_{\gamma(n)}^{{\rm Skel}}$.

Clearly, $S_{\gamma(n)}^{{\rm Skel}}$ is a subtree of $S_{\gamma(n)}^{{\rm AB}}$. Then, by the induction hypothesis \eqref{DerterLemma1eq5}, we have that for all $x \in S_{\gamma(n)}^{{\rm AB}}$, there exists $y \in S_{\gamma(n)}^{{\rm Skel}}$ such that $d_{\mbox{\tiny graph}}^{(n)}(x, y) \leq r$. On the other hand, by Lemma \ref{ControlSubTreeSize}, $\# C_{\gamma(n)}^{{\rm AB}} \leq r$. Then, since $\gamma(n) \in {\rm Skel}^{\gamma,s}(n)$ (indeed, $\gamma(n)$ is the root of ${\rm Skel}^{\gamma,s}(n)$), it follows that necessary for all $x \in \widetilde{C}_{\gamma(n)}^{{\rm AB}}$, $d_{\mbox{\tiny graph}}^{(n)}(x, \gamma(n)) \leq r$. Thus, the above shows that \eqref{DerterLemma1eq1} also holds in this case.  \\

This concludes our proof.
\end{proof}

\subsection{Behavior of the skeleton chain between long loops}\label{subsectionBehaviorOfTheSkeletonBetweenLongLoops}
This subsection establishes analogous results to those in Subsection \ref{subsectionBehaviorOfRGRGBetweenRegrafting}. Informally, we prove that between long loops of the path $\gamma$ (those with a length of $r+1$ or greater), the Skeleton chain exhibits purely path-like behavior, equivalently, it only exhibits root-growth or ghost indices erasure (recall the dynamics of Remark \ref{remarkDynamicsOfSkeletonChain}).

Fix $s, s^{\prime}, r \in\mathbb{N}$ such that $r \geq 3s+1 \geq 18s^{\prime}+1$. Recall that $\gamma:\na_0\to V$ be a path on a finite, simple, connected graph $G=(V,E)$ and the definition of the intervals $A_{i}^{(r,s)}$'s and $B_{i}^{(r,s)}$'s in \eqref{eqnDefinitionAiBi}, for $i \in \mathbb{N}$. For each $(i,j)\in \mathbb{N}\times \mathbb{N}$ with $1\le i<j$, we define
\begin{equation}\begin{aligned} \label{e:061ii}
Z^{\gamma,(r,s)}_{(i,j)} \coloneqq \mathbf{1}_{\{{\rm R}^\gamma({\rm NE}^{\gamma,s}(A_{i}^{(r,s)})) \cap {\rm R}^\gamma(B_{j}^{(r,s)}) \neq \emptyset \}}.
\end{aligned}
\end{equation}

Just as in \eqref{e:060ii}, if $Z^{\gamma,(r,s)}_{(i,j)}=1$, then we call $j$ a cut-time and $i$ a cut-point. 
First, we show that if the path $\gamma$ does not form a long loop (with a length of $r+1$ or greater) within the interval $[0, kr]$ for some $k \in \mathbb{N}_{0}$, then $\mathrm{Skel}^{\gamma,s}(kr)$ is a connected path. 

\begin{lemma} \label{DerterLemma2iii}
Assume $r,s,s' \in \mathbb{N}$ with $r \geq 3s +1\ge 18s'+1$. Fix $N \in \mathbb{N}$ such that $N/r \geq 2$. Let $\gamma: \mathbb{N}_{0} \rightarrow V$ be a path on a finite, simple, connected graph $G=(V,E)$ that satisfies Assumptions \hyperref[assumptionNoLoopsOfIntermediateLength]{No loops of intermediate length}, \hyperref[assumptionLocalCutPoints]{Local cut points} and Properties \ref{Pro1}-\ref{Pro5II} on $[0, N]$. Fix any $0 \leq k\leq \floor{N/r}$ and suppose that 
\begin{equation} \label{Indizero}
Z^{\gamma,(r,s)}_{(i,j)} = 0, \quad \mbox{for all} \quad  1 \leq i < j \leq k.
\end{equation}
(If $k < 2$, then \eqref{Indizero} is satisfied by vacuity.) Then, $\mathrm{Skel}^{\gamma,s}(kr)$ is a connected path, and 
\begin{equation}\begin{aligned} \label{Deteq1ii}
\sum_{i=1}^{k} \# {\rm NE}^{\gamma,s}(A_{i}^{(r,s)})  \leq \# \mathrm{Skel}^{\gamma,s}(kr)  \leq \sum_{i=1}^{k} \# {\rm NE}^{\gamma,s}(A_{i}^{(r,s)}) + 3ks+1.
\end{aligned}\end{equation} 
\end{lemma}

\begin{proof}
First, we prove that $\mathrm{Skel}^{\gamma,s}(kr)$ is connected, for $0 \leq k\leq \floor{N/r}$. In fact, by Lemma \ref{SkeLemma2}, we know that $\mathrm{Skel}^{\gamma,s}(n)$ is connected, for all $n \in [0, \floor{N/r}r]$. Subsequently, we focus on demonstrating that $\mathrm{Skel}^{\gamma,s}(kr)$ is a path.

We fix $k \in \{1, \dots, \floor{N/r}\}$ and proceed by induction to prove that ${\rm Skel}^{\gamma,s}(n)$ is a path for every $n \in \{0, \dots, kr \}$. For $n=0,1$, note that $\mathrm{Skel}^{\gamma,s}(n)$ is clearly a path. Then, we assume that ${\rm Skel}^{\gamma,s}(m)$ is a path for every $m \in \{0, \dots, n-1\}$ and prove that ${\rm Skel}^{\gamma,s}(n)$ is a path. 

We will prove that at time $n$, the Skeleton chain undergoes either a root-growth or a ghost indices erasure movement (refer to Remark \ref{remarkDynamicsOfSkeletonChain}). Consequently, by the inductive hypothesis, ${\rm Skel}^{\gamma,s}(n)$ maintains its path structure. We consider two cases:

{\bf Case 1.} Suppose that $\gamma(n)\notin {\rm R}^\gamma ([0,n-1]\setminus \mathcal{G}^{\gamma,s}(n-1))$. In this case, the Skeleton undergoes a root-growth movement at time $n$ (refer to Remark \ref{remarkDynamicsOfSkeletonChain}). Consequently, ${\rm Skel}^{\gamma,s}(n)$ unequivocally remains a path.

{\bf Case 2.} Suppose that
\begin{equation}\label{eqnDerterLemma21}
\gamma(n)\in {\rm R}^\gamma ([0,n-1]\setminus \mathcal{G}^{\gamma,s}(n-1)).
\end{equation} 
\noindent We start by showing that 
\begin{equation} \label{EqInteNot1}
\gamma(n)\in {\rm R}^\gamma ([n-s+1,n-1]\setminus \mathcal{G}^{\gamma,s}(n-1))
\end{equation}
\noindent and
\begin{equation} \label{EqInteNot2}
\gamma(n) \not \in {\rm R}^\gamma ([0,n-s]\setminus \mathcal{G}^{\gamma,s}(n-1)). 
\end{equation}  
\noindent If $n \in [0,s-1]$, then \eqref{EqInteNot1} and \eqref{EqInteNot2} are satisfied. Suppose that $n \in [s, kr]$. By the assumption of \hyperref[assumptionNoLoopsOfIntermediateLength]{No loops of intermediate length}, we know that
\begin{equation} \label{EqInteNot3}
\gamma(n) \not \in {\rm R}^\gamma ([0 \vee (n-r),n-s]). 
\end{equation} 
\noindent If $n \in [s,r-1]$, then \eqref{EqInteNot3} implies that \eqref{EqInteNot1} and \eqref{EqInteNot2} are satisfied. So, we assume that $n \in [r, kr]$. Recall the definitions of the intervals $B_{i}^{(r,s)}$'s in \eqref{eqnDefinitionAiBi}. Since $n \in [r, kr]$, there exists $j \in \{2, \dots, k+1\}$ such that $n \in [(j-1)r, jr -1]$. On the one hand, if $n \in [(j-1)r, jr -1] \setminus B_{j}^{(r,s)}$, then \ref{Pro4} implies that
\begin{equation} \label{EqInteNot5}
\gamma(n) \not \in  {\rm R}^\gamma ([0,n-r-1]). 
\end{equation} 
\noindent On the other hand, if $n \in B_{j}^{(r,s)}$, it necessarily follows that $j\leq k$ since $n \in [r, kr]$. Thus, \eqref{Indizero} and Corollary \ref{LemmaIdenSeDecomp} (note that, $k<\tau_{1}^{\gamma, (r,s)}$) imply that
\begin{equation} \label{EqInteNot4}
\gamma(n) \not \in  \bigcup_{i=1}^{j-1} {\rm R}^\gamma({\rm NE}^{\gamma,s}(A_{i}^{(r,s)})) = \bigcup_{i=1}^{j-1} {\rm R}^\gamma(B_{i}^{(r,s)} \setminus {\mathcal G}^{\gamma,s}(n-1)). 
\end{equation} 
\noindent Hence \eqref{EqInteNot4} and \ref{Pro3} imply that
\begin{equation} \label{EqInteNot6}
\gamma(n) \not \in  {\rm R}^\gamma ([0,n-r-1] \setminus \mathcal{G}^{\gamma,s}(n-1)). 
\end{equation} 
\noindent Therefore, if $n \in [r, kr]$, then \eqref{EqInteNot1} and \eqref{EqInteNot2} follows from \eqref{EqInteNot3}, \eqref{EqInteNot5} and \eqref{EqInteNot6}. 

Next, we prove that there exists $m_{1} \in [n-s+1,n-1]\setminus \mathcal{G}^{\gamma,s}(n-1)$ such that $\gamma(m_{1}) = \gamma(n)$ (that is, $m_{1}$ satisfies \eqref{eqnGhostIndexDefinitionSmallLoop} in ($\mathbf{G}_{n}$)) and
\begin{align} \label{EqInteNot7}
\gamma(k) \not \in {\rm R}^\gamma\big([0,k-1]\setminus {\mathcal G}^{\gamma,s}(n-1)\big), \quad \text{for all} \quad k \in [m_{1},n-1] \setminus {\mathcal G}^{\gamma,s}(n-1)
\end{align}
\noindent (that is, $m_{1}$ satisfies \eqref{eqnGhostIndexDefinitionNoLongLoops} in ($\mathbf{G}_{n}$)). But by the induction hypothesis, we know that ${\rm Skel}(n-1)$ is a path. Recall from Definition \ref{Def:005} that ${\rm Skel}(n-1)$ is the subgraph of $\mathrm{AB}^\gamma(n-1)$ restricted to the vertex set ${\rm R}^\gamma([0,n-1]\setminus\mathcal {G}^{\gamma,s}(n-1))$. In particular, we have that any $m_{1} \in [n-s+1,n-1]\setminus \mathcal{G}^{\gamma,s}(n-1)$ such that $\gamma(m_{1}) = \gamma(n)$ satisfies \eqref{EqInteNot7}. (Note that, by \eqref{EqInteNot1}, there is at least one $m_{1} \in [n-s+1,n-1]\setminus \mathcal{G}^{\gamma,s}(n-1)$ such that $\gamma(m_{1}) = \gamma(n)$.). 

Consequently, by \eqref{EqInteNot1}, \eqref{EqInteNot2} and \eqref{EqInteNot7}, it is established that at time $n$, the scenario aligns with {\bf Case 3 } within the proof of Lemma \ref{SkeLemma2}. This implies that at time $n$, the Skeleton undergoes a Ghost indices erasure movement (refer to Remark \ref{remarkDynamicsOfSkeletonChain}). As a result, ${\rm Skel}(n)$ maintains its path structure.

This concludes the proof of the first part of Lemma \ref{DerterLemma2iii}. \\

Finally, we prove that \eqref{Deteq1ii}. Clearly, for $k=0, 1$, the inequality \eqref{Deteq1ii} holds. Suppose that $k \in \{2, \dots, \floor{N/r}\}$. Recall from Definition \ref{Def:005} that ${\rm Skel}(kr)$ is the subgraph of $\mathrm{AB}^\gamma(kr)$ restricted to the vertex set ${\rm R}^\gamma([0,kr]\setminus\mathcal {G}^{\gamma,s}(kr))$. On the other hand, note that, ${\rm R}^\gamma([0,kr]\setminus\mathcal {G}^{\gamma,s}(kr))$ is equal to
\begin{align} \label{EqInteNot8}
\{\gamma(kr) \} \cup \bigcup_{i=1}^{k} {\rm R}^\gamma(B_{i}^{(r,s)} \setminus\mathcal {G}^{\gamma,s}(kr)) \cup \bigcup_{i=1}^{k} {\rm R}^\gamma( ([(i-1)r, ir -1] \setminus B_{i}^{(r,s)}) \setminus\mathcal {G}^{\gamma,s}(kr)),
\end{align}
\noindent which by \eqref{Indizero} and Corollary \ref{LemmaIdenSeDecomp} (note that, $k<\tau_{1}^{\gamma, (r,s)}$) is equal to
\begin{align}  \label{EqInteNot9}
\{\gamma(kr) \} \cup \bigcup_{i=1}^{k} {\rm R}^\gamma({\rm NE}^{\gamma,s}(A_{i}^{(r,s)}))  \cup \bigcup_{i=1}^{k} {\rm R}^\gamma( ([(i-1)r, ir -1] \setminus B_{i}^{(r,s)}) \setminus\mathcal {G}^{\gamma,s}(kr)).
\end{align}
\noindent Note also that \eqref{Indizero} and Corollary \ref{LemmaIdenSeDecomp} imply that the sets
\begin{align}
{\rm R}^\gamma({\rm NE}^{\gamma,s}(A_{i}^{(r,s)}))  = {\rm R}^\gamma(B_{i}^{(r,s)} \setminus\mathcal {G}^{\gamma,s}(kr)),
\end{align}
\noindent for $i=1, \dots, k$, are disjoint. Therefore, the lower bound in \eqref{Deteq1ii} is a direct consequence of \eqref{EqInteNot9}. The upper bound in \eqref{Deteq1ii} also follows from \eqref{EqInteNot9} in conjunction with the observation that
\begin{align}
\# \left( \{\gamma(kr) \} \cup \bigcup_{i=1}^{k} {\rm R}^\gamma( ([(i-1)r, ir -1] \setminus B_{i}^{(r,s)}) \setminus\mathcal {G}^{\gamma,s}(kr)) \right) \leq 3ks+1. 
\end{align}
\noindent This concludes the proof of \eqref{Deteq1ii}. 
\end{proof}

The next objective is to generalize the preceding lemma. We aim to prove, under specific assumptions regarding the path $\gamma$, that the Skeleton chain, when restricted to an interval encompassing segments not involved in any intersection, exhibits a path structure. To formalize this, we require the following definitions.

Consider $N \in \mathbb{N}$ such that $N/r \geq 2$. Fix $2 \leq k \leq \floor{N/r}$ and suppose that there exists an integer $1 \leq n \leq \lfloor k/2 \rfloor$ and indices $1 \leq i_{m} < j_{m} \leq k$, for $1 \leq m \leq n$, such that $i_{1}, \dots, i_{n}, j_{1}, \dots, j_{n}$ are all distinct, 
\begin{align} \label{IndentityOnemoretime}
Z^{\gamma,(r,s)}_{(i_{m},j_{m})} = 1 \quad \text{and} \quad Z^{\gamma,(r,s)}_{(i,j)} = 0,
\end{align}
\noindent for all $1 \leq i < j \leq k$ with either $i$ or $j$ not in $\{ i_{1}, \dots, i_{n}, j_{1}, \dots, j_{n} \}$. 

Next, suppose that the path $\gamma: \mathbb{N}_{0} \rightarrow V$ satisfies the Assumption  \hyperref[assumptionLocalCutPoints]{Local cut points} on $[0,N]$. Let $i_{1}^{\prime}, \ldots, i_{2n}^{\prime}$ be the sequence $i_{1}, \ldots, i_{n}, j_{1}, \ldots, j_{n}$ ordered in increasing order, that is, $i_{1}^{\prime} < \cdots < i_{2n}^{\prime}$. Set $i_{0}^{\prime} =0$ and $i_{2n+1}^{\prime} =k+1$. For $m=0, \ldots, 2n$, define $I_{m}^{(r,s)}\subset [0,N]$, as
\begin{align}\label{eqnDefinitionIrsm}
I_{m}^{(r,s)} = \begin{cases}
\emptyset  & \text{if} \quad [\min B_{i_{m}^{\prime}+1}, \max B_{i_{m+1}^{\prime}-1}] = \emptyset, \\
[\ell_{m}^{1}+1, \ell_{m}^{2}] & \text{otherwise},
\end{cases}
\end{align}
\noindent where $\ell^1_m$ is the largest $2s^{\prime}$-local cutpoint in $[\min B_{i_{m}^{\prime}+1}-s, \min B_{i_{m}^{\prime}+1}]$, and $\ell_{m}^{2}$ is the smallest $2s^{\prime}$-local cutpoint in $[\max B_{i_{m+1}^{\prime}-1}+1, \max B_{i_{m+1}^{\prime}-1}+s+1]$, whenever $[\min B_{i_{m}^{\prime}+1}, \max B_{i_{m+1}^{\prime}-1}] \neq \emptyset$. (By Assumption  \hyperref[assumptionLocalCutPoints]{Local cut points}, those $2s^{\prime}$-local cutpoints exist.) 

For $m=0, \dots, 2n$, let ${\rm Branch}^{\gamma}(I_{m}^{(r,s)})$ be the subgraph of ${\rm AB}^{\gamma}(kr)$ restricted to the vertex set ${\rm R}^{\gamma}(I_{m}^{(r,s)} \setminus \mathcal {G}^{\gamma,s}(kr) )$. Note that ${\rm Branch}^{\gamma}(I_{m}^{(r,s)})$ is a subgraph (not necessary connected) of  $\mathrm{Skel}^{(r,s)}_{\gamma}(kr)$. Moreover, if $[\min B_{i_{m}^{\prime}+1}, \max B_{i_{m+1}^{\prime}-1}] = \emptyset$, then, by Lemma \ref{NewLemmaInclII}, $\gamma(\ell_{m}^{2}) \in {\rm R}^{\gamma}(I_{m}^{(r,s)} \setminus \mathcal {G}^{\gamma,s}(kr) )$. 

Finally, in the following result, we also consider the metric space $\bigcup_{m=0}^{2n} {\rm Branch}^{\gamma}(I_{m}^{(r,s)})$ formed by the union of the metric spaces ${\rm Branch}^{\gamma}(I_{m}^{(r,s)})$. This union is naturally endowed with the metric induced by the restriction of the metric of $\mathrm{Skel}^{(r,s)}_{\gamma}(kr)$ to this set. 

\begin{lemma} \label{lemmaComparingSkeletonAndcRGRGBetweenLongLoopsiii}
Assume $r,s,s' \in \mathbb{N}$ with $r \geq 3s +1\ge 18s'+1$. Fix $N \in \mathbb{N}$ such that $N/r \geq 2$. Let $\gamma: \mathbb{N}_{0} \rightarrow V$ be a path on a finite, simple, connected graph $G=(V,E)$ that satisfies Assumptions \hyperref[assumptionNoLoopsOfIntermediateLength]{No loops of intermediate length}, \hyperref[assumptionLocalCutPoints]{Local cut points} and Properties \ref{Pro1}-\ref{Pro5II} on $[0, N]$. Fix any $2 \leq k \leq \floor{N/r}$ and suppose that there exists an integer $1 \leq n \leq \lfloor k/2 \rfloor$ and indices $1 \leq i_{m} < j_{m} \leq k$, for $1 \leq m \leq n$, such that $i_{1}, \dots, i_{n}, j_{1}, \dots, j_{n}$ are all distinct, and
\begin{align}\label{InteCondiIIm}
Z^{(r,s)}_{(i_{m},j_{m})}(\gamma) = 1 \quad \text{and} \quad Z^{(r,s)}_{(i,j)}(\gamma) = 0,
\end{align}
\noindent for all $1 \leq i < j \leq k$ with either $i$ or $j$ not in $\{ i_{1}, \dots, i_{n}, j_{1}, \dots, j_{n} \}$. Then,
\begin{enumerate}[label=(\textbf{\roman*})]
\item \label{lemmaComparingPro1} For $m=0, \dots, 2n$, ${\rm Branch}^{\gamma}(I_{m}^{(r,s)})$ is a connected path. In particular, for each  $n \in I_{m}^{(r,s)} \setminus \mathcal {G}^{\gamma,s}(kr)$, there is no different $n^{\prime} \in I_{m}^{(r,s)} \setminus \mathcal {G}^{\gamma,s}(kr)$ such that $\gamma(n) = \gamma(n^{\prime})$. 

\item \label{lemmaComparingPro2}It holds that
\begin{align} \label{eqCompaHoleSke}
d_{\rm H}\Big({\rm Skel}^{\gamma,s}(kr), \bigcup_{m=0}^{2n} {\rm Branch}^{\gamma}(I_{m}^{(r,s)})\Big)\leq 2r+6s.
\end{align}

\item \label{lemmaComparingPro3} For $m=0, \dots, 2n$ such that $I_{m}^{(r,s)} \neq \emptyset$, we have that
\begin{equation}\label{eqnSizeOfTheBranchesBetweenLongLoops}
\begin{aligned} 
&\sum_{h=i'_m+1}^{i'_{m+1}-1} \# {\rm NE}^{\gamma,s}(A_{i}^{(r,s)})  \leq \# {\rm Branch}^{\gamma}(I_{m}^{(r,s)}) \leq \sum_{h=i'_m+1}^{i'_{m+1}-1}\# {\rm NE}^{\gamma,s}(A_{i}^{(r,s)}) + 3s (i_{m+1}^{\prime} - i_{m}^{\prime}-1) \vee 0 +2s.
\end{aligned}
\end{equation}

\item \label{lemmaComparingPro4} Consider $m=0, \dots, 2n$ such that $I_{m}^{(r,s)} \neq \emptyset$. Let $n_{1}, n_{2} \in I_{m}^{(r,s)} \setminus \mathcal {G}^{\gamma,s}(kr)$ be such that $n_{1}<n_{2}$. Then,
\begin{align} \label{eqnSizeOfTheBranchesBetweenLongLoopsII}
\sum_{h=u_{1}+1}^{u_{2}-1} \# {\rm NE}^{\gamma,s}(A_{i}^{(r,s)})  \leq d_{\mathrm{Skel}^{\gamma,s}(kr)}(\gamma(n_{1}),\gamma(n_{2})) \leq \sum_{h=u_{1}+1}^{u_{2}-1} \# {\rm NE}^{\gamma,s}(A_{i}^{(r,s)}) + 3s(u_{2}-u_{1}-1) \vee 0 + 2r,
\end{align}
\noindent where $u_{1}, u_{2} \in \{i_{m}^{\prime}+1, \dots, i_{m+1}^{\prime}\}$ are such that $u_{1}\leq u_{2}$, $n_{1} \in {[(u_{1}-1)r, u_{1}r-1]}$ and $n_{2} \in [(u_{2}-1)r, u_{2}r-1]$. 

\item \label{lemmaComparingPro5} Consider different $m_{1}, m_{2}\in \{0, \dots, 2n\}$ such that $I_{m_{1}}^{(r,s)} \neq \emptyset$ and $I_{m_{2}}^{(r,s)} \neq \emptyset$.  Let $n_{1}  \in I_{m_{1}}^{(r,s)} \setminus \mathcal {G}^{\gamma,s}(kr)$ and $n_{2}  \in I_{m_{2}}^{(r,s)} \setminus \mathcal {G}^{\gamma,s}(kr)$ such that $n_{1}<n_{2}$. For $i =1,2$, let $\tilde{n}_{i}$ 
be the boundary point in $I_{m_{i}}^{(r,s)} \setminus \mathcal {G}^{\gamma,s}(kr)$ that is in the path in $\mathrm{Skel}^{\gamma,s}(kr)$ from  $\gamma(n_{1})$ to $\gamma(n_{2})$.
Let 
\begin{align}
M_{n_{1},n_{2}}^{(r,s)} \coloneqq \{m \in  \{0, \dots, 2n\}: \, {\rm Branch}^{\gamma}(I_{m}^{(r,s)}) \, \, \text{is contained in the path from} \, \gamma(n_{1}) \, \,  \text{to} \, \, \gamma(n_{2}) \},
\end{align}
Then,
\begin{align}
& d_{\mathrm{Skel}^{\gamma,s}(kr)}(\gamma(n_{1}),\gamma(n_{2})) \nonumber \\
&  \quad \quad  \geq  \sum_{h=u_{1}+1}^{u_{2}-1} \# {\rm NE}^{\gamma,s}(A_{i}^{(r,s)}) + \sum_{h=u_{3}+1}^{u_{4}-1} \# {\rm NE}^{\gamma,s}(A_{i}^{(r,s)}) + \sum_{m \in M_{n_{1},n_{2}}^{(r,s)}} \# {\rm Branch}^{\gamma}(I_{m}^{(r,s)})
\end{align}
\noindent and 
\begin{align}
& d_{\mathrm{Skel}^{\gamma,s}(kr)}(\gamma(n_{1}),\gamma(n_{2}))\nonumber \\
& \quad \quad \leq \sum_{h=u_{1}+1}^{u_{2}-1} \# {\rm NE}^{\gamma,s}(A_{i}^{(r,s)}) + \sum_{h=u_{3}+1}^{u_{4}-1} \# {\rm NE}^{\gamma,s}(A_{i}^{(r,s)}) + \sum_{m \in M_{n_{1},n_{2}}^{(r,s)}} \# {\rm Branch}^{\gamma}(I_{m}^{(r,s)})  \nonumber \\
& \quad \quad \quad \quad  + 3s(u_{2}-u_{1} + u_{4}-u_{3}-2) \vee 0 + 4r + 2rn,
\end{align}
\noindent where $u_{1}, u_{2} \in \{i_{m_{1}}^{\prime}+1, \dots, i_{m_{1}+1}^{\prime}\}$ and $u_{3}, u_{4} \in \{i_{m_{2}}^{\prime}+1, \dots, i_{m_{2}+1}^{\prime}\}$ are such that $u_{1}\leq u_{2}$, $u_{3}\leq u_{4}$, $n_{1} \wedge \tilde{n}_{1}  \in [(u_{1}-1)r, u_{1}r-1]$, $n_{1} \vee \tilde{n}_{1} \in [(u_{2}-1)r, u_{2}r-1]$, $n_{2} \wedge \tilde{n}_{2}  \in [(u_{3}-1)r, u_{3}r-1]$ and $n_{2} \vee \tilde{n}_{2} \in [(u_{4}-1)r, u_{4}r-1]$. 
\end{enumerate}
\end{lemma}

\begin{proof}
Fix $k \in \{2, \dots, \floor{N/r}\}$ and $m \in \{0, \dots, 2n\}$. Assume that $I_{m}^{(r,s)} \neq \emptyset$.  Note that $\ell^{2}_{m} = \max I_{m}^{(r,s)}$ and $\ell^{1}_{m}+1 = \min I_{m}^{(r,s)}$. First, we prove that 
\begin{align} \label{Le2Cas2Eq13}
{\rm R}^{\gamma}((I_{m}^{(r,s)} \cap [0,n]) \setminus \mathcal {G}^{\gamma,s}(n) ) \cap {\rm R}^{\gamma}([0, \ell_{m}^{1} -1] \setminus \mathcal {G}^{\gamma,s}(n) ) = \emptyset. 
\end{align}
\noindent for all $n \in I_{m}^{(r,s)}$. This is equivalent to prove that for each $n \in I_{m}^{(r,s)}$, 
\begin{align} \label{Le2Cas2Eq14}
\gamma(h) \not \in {\rm R}^{\gamma}([0, \ell_{m}^{1} -1] \setminus \mathcal {G}^{\gamma,s}(n) ), \quad \text{for} \quad h \in (I_{m}^{(r,s)} \cap [0,n]) \setminus \mathcal {G}^{\gamma,s}(n),
\end{align}

Fix $n \in I_{m}^{(r,s)}$ and consider $h \in I_{m}^{(r,s)} \setminus \mathcal {G}^{\gamma,s}(n)$. Note that $h \geq \ell^{1}_{m} \geq s$. Since $\ell^{1}_{m}$ is a $2s^{\prime}$-local cutpoint (recall Definition \ref{LocaCutP}), the assumption of \hyperref[assumptionNoLoopsOfIntermediateLength]{No loops of intermediate length} implies that
\begin{equation} \label{Le2Cas2Eq4}
\gamma(h) \not \in {\rm R}^\gamma ([0 \vee (h-r), \ell^{1}_{m} \vee (h-s^{\prime})]). 
\end{equation} 
\noindent If $h \in [s,r-1]$, then \eqref{Le2Cas2Eq4} implies \eqref{Le2Cas2Eq14}. So, we assume that $h \in [r, \ell^{2}_{m}]$ and $\ell^{2}_{m} \geq r$. Since $h \in [r, \ell^{2}_{m}]$, there exists $j \in \{i^{\prime}_{m}+1, \dots, i^{\prime}_{m+1}-1\}$ such that $h \in [(j-1)r, jr -1]$. On the one hand, if $h \in [(j-1)r, jr -1] \setminus B_{j}^{(r,s)}$, then \ref{Pro4} implies that
\begin{equation} \label{Le2Cas2Eq5}
\gamma(h) \not \in  {\rm R}^\gamma ([0,h-r-1]). 
\end{equation} 
\noindent On the other hand, assume $h \in B_{j}^{(r,s)}$. Thus \eqref{InteCondiIIm} implies that
\begin{equation} \label{Le2Cas2Eq6}
\gamma(h) \not \in  \bigcup_{i=1}^{j-1} {\rm R}^\gamma({\rm NE}^{\gamma,s}(A_{i}^{(r,s)})).
\end{equation} 
\noindent Recall the definition of  $\tau_{i}^{\gamma, (r,s)}$ given in \eqref{InterTime1}. By \eqref{InteCondiIIm}, there exists $\widetilde k \in [1, \lfloor N/r \rfloor]$ such that $\{i^{\prime}_{m}+1, \dots, i^{\prime}_{m+1}-1\}\subseteq [\tau_{\widetilde k-1}^{\gamma, (r,s)}+1, \tau_{\widetilde k}^{\gamma, (r,s)}-1]$. Therefore, by Corollary \ref{LemmaIdenSeDecomp}
\begin{equation} \label{Le2Cas2Eq7}
\bigcup_{u=1}^{m-1} \bigcup_{i=i_{u}^{\prime}+1}^{i_{u+1}^{\prime}-1} {\rm R}^\gamma({\rm NE}^{\gamma,s}(A_{i}^{(r,s)})) = \bigcup_{u=1}^{m-1} \bigcup_{i=i_{u}^{\prime}+1}^{i_{u+1}^{\prime}-1} {\rm R}^\gamma(B_{i}^{(r,s)} \setminus {\mathcal G}^{\gamma,s}(n-1))
\end{equation} 
\noindent and
\begin{equation}  \label{Le2Cas2Eq8}
\bigcup_{i=i_{m}^{\prime}+1}^{j-1} {\rm R}^\gamma({\rm NE}^{\gamma,s}(A_{i}^{(r,s)})) = \bigcup_{i=i_{m}^{\prime}+1}^{j-1}{\rm R}^\gamma(B_{i}^{(r,s)} \setminus {\mathcal G}^{\gamma,s}(n-1)). 
\end{equation} 
\noindent Moreover, note that
\begin{equation} \label{Le2Cas2Eq9}
\bigcup_{u=1}^{m} {\rm R}^\gamma({\rm NE}^{\gamma,s}(A_{i_{u}^{\prime}}^{(r,s)})) \subseteq \bigcup_{u=1}^{m}  {\rm R}^\gamma(B_{i_{u}^{\prime}}^{(r,s)})
\end{equation} 
\noindent and thus, by \eqref{InteCondiIIm}, \eqref{Le2Cas2Eq6}, \eqref{Le2Cas2Eq9} and \ref{Pro1}, 
\begin{equation} \label{Le2Cas2Eq10}
\gamma(h) \not \in \bigcup_{u=1}^{m}  {\rm R}^\gamma(B_{i_{u}^{\prime}}^{(r,s)} \setminus {\mathcal G}^{\gamma,s}(n-1)). 
\end{equation} 
Hence \eqref{Le2Cas2Eq6}, \eqref{Le2Cas2Eq7}, \eqref{Le2Cas2Eq8}, \eqref{Le2Cas2Eq10} and \ref{Pro3} imply that
\begin{equation} \label{Le2Cas2Eq11}
\gamma(h) \not \in  {\rm R}^\gamma ([0,h-r-1] \setminus \mathcal{G}^{\gamma,s}(n-1)). 
\end{equation} 
\noindent Therefore, if $h \in [r, \ell^{2}_{m}]$, then \eqref{Le2Cas2Eq14} follows from \eqref{Le2Cas2Eq4}, \eqref{Le2Cas2Eq5} and \eqref{Le2Cas2Eq11}. 

This concludes that proof of \eqref{Le2Cas2Eq14} and thus of \eqref{Le2Cas2Eq13}.
\\

Next, we prove that 
\begin{align} \label{Le2Cas2Eq15}
{\rm R}^{\gamma}(I_{m}^{(r,s)} \setminus \mathcal {G}^{\gamma,s}(n) ) \cap {\rm R}^{\gamma}([\ell_{m}^{2}+1,n] \setminus \mathcal {G}^{\gamma,s}(n) ) = \emptyset,
\end{align}
\noindent for all $n \in [\ell_{m}^{2}+1,kr]$. This is equivalent to prove that for each $n \in [\ell_{m}^{2}+1,kr]$, 
\begin{align} \label{Le2Cas2Eq16}
\gamma(h) \not \in {\rm R}^{\gamma}(I_{m}^{(r,s)} \setminus \mathcal {G}^{\gamma,s}(n) ), \quad \text{for} \quad h \in [\ell_{m}^{2}+1,n] \setminus \mathcal {G}^{\gamma,s}(n) .
\end{align}

Fix $n \in [\ell_{m}^{2}+1,kr]$ and consider $h \in [\ell_{m}^{2}+1,n] \setminus \mathcal {G}^{\gamma,s}(n)$. Note that $h \geq \ell^{2}_{m}+1 \geq r-s$. Since $\ell^{2}_{m}$ is a $2s^{\prime}$-local cutpoint (recall Definition \ref{LocaCutP}), the assumption of \hyperref[assumptionNoLoopsOfIntermediateLength]{No loops of intermediate length} implies that
\begin{equation} \label{Le2Cas2Eq17}
\gamma(h) \not \in {\rm R}^\gamma ([0 \vee (h-r), \ell^{2}_{m} \vee (h-s^{\prime})]). 
\end{equation} 
\noindent If $h \in [r-s,r-1]$, then \eqref{Le2Cas2Eq17} implies \eqref{Le2Cas2Eq16}. So, we assume that $h \in [\ell_{m}^{2}+1,n] \cap [r, n]$ such that $[\ell_{m}^{2}+1,n] \cap [r, n] \neq \emptyset$.
Then, there exists $j \in \{i^{\prime}_{m}+1, \dots, k\}$ such that $h \in [(j-1)r, jr -1]$. On the one hand, if $h \in [(j-1)r, jr -1] \setminus B_{j}^{(r,s)}$, then \ref{Pro4} implies that
\begin{equation} \label{Le2Cas2Eq18}
\gamma(h) \not \in  {\rm R}^\gamma ([0,h-r-1]). 
\end{equation} 
\noindent On the other hand, assume $h \in B_{j}^{(r,s)}$. Thus, \eqref{InteCondiIIm} and Corollary \ref{LemmaIdenSeDecomp} imply that
\begin{equation} \label{Le2Cas2Eq19}
\gamma(h) \not \in  \bigcup_{i=i_{m}^{\prime}+1}^{i_{m+1}^{\prime}-1} {\rm R}^\gamma({\rm NE}^{\gamma,s}(A_{i}^{(r,s)})) = \bigcup_{i=i_{m}^{\prime}+1}^{i_{m+1}^{\prime}-1} {\rm R}^\gamma(B_{i}^{(r,s)} \setminus {\mathcal G}^{\gamma,s}(n-1)).
\end{equation} 
\noindent Thus, by \eqref{Le2Cas2Eq19} and \ref{Pro1}, 
\begin{equation}  \label{Le2Cas2Eq20}
\gamma(h) \not \in  {\rm R}^\gamma([0 \vee \ell_{m}^{1},h-r-1] \setminus {\mathcal G}^{\gamma,s}(n-1)). 
\end{equation} 
\noindent Therefore, if $h \in [\ell_{m}^{2}+1,n] \cap [r, n]$, then \eqref{Le2Cas2Eq16} follows from \eqref{Le2Cas2Eq17}, \eqref{Le2Cas2Eq18} and \eqref{Le2Cas2Eq20}. 

This concludes that proof of \eqref{Le2Cas2Eq16} and thus of \eqref{Le2Cas2Eq15}. \\

Having established the preceding results, we are now prepared to prove Lemma \ref{lemmaComparingSkeletonAndcRGRGBetweenLongLoopsiii}. We start with the proof of \ref{lemmaComparingPro1}.  We prove that ${\rm Branch}^{\gamma}(I_{m}^{(r,s)})$ is a connected path. The proof follows an argument similar to that of Lemma  \ref{DerterLemma2iii}, and hence we only outline the key ideas. Assume that $I_{m}^{(r,s)} \neq \emptyset$, otherwise our claim is trivial. For $n \in \mathbb{N}_{0}$, let ${\rm B}_{m}^{\gamma}(n)$ be the subgraph of ${\rm AB}^{\gamma}(n)$ restricted to the vertex set ${\rm R}^{\gamma}((I_{m}^{(r,s)} \cap [0,n]) \setminus \mathcal {G}^{\gamma,s}(n))$. Observe that ${\rm B}_{m}^{\gamma}(n) = {\rm Branch}^{\gamma}(I_{m}^{(r,s)})$, for $n = kr$. The goal is to prove that
\begin{align}
{\rm B}_{m}^{\gamma}(n) \quad \text{is a connected path for every} \quad n \in \{0, \dots,  kr\}.
\end{align}
\noindent Recall that $\ell^{2}_{m} = \max I_{m}^{(r,s)}$ and $\ell^{1}_{m}+1 = \min I_{m}^{(r,s)}$. Then, ${\rm B}_{m}^{\gamma}(n) = \emptyset$, for all $n=0, \dots, \ell^{1}_{m}$, and our claim follows, that is, ${\rm B}_{m}^{\gamma}(n)$ is a connected path. Next, an induction argument, analogous to the proof of the first claim in Lemma  \ref{lemmaComparingSkeletonAndcRGRGBetweenLongLoopsiii} but employing \eqref{Le2Cas2Eq13}, establishes that ${\rm B}_{m}^{\gamma}(n)$ is a connected path for every $n \in \{\ell^{1}_{m}+1, \dots,  \ell^{2}_{m} \}$. (The above is equivalent to prove that the Skeleton chain $\mathrm{Skel}^{\gamma,s}(\cdot)$ exhibits purely path-like behaviour in the interval $[\ell^{1}_{m}+1, \ell^{2}_{m}]$.) Finally, to prove that ${\rm B}_{m}^{\gamma}(n)$ remains a connected path for $n \in \{\ell^{2}_{m}+1, \dots,  kr\}$, it suffices to employ \eqref{Le2Cas2Eq15}. Notably, the latter also implies
\begin{align}
I_{m}^{(r,s)} \setminus \mathcal {G}^{\gamma,s}(\ell^{2}_{m}) = I_{m}^{(r,s)} \setminus \mathcal {G}^{\gamma,s}(n),
\end{align}
\noindent for all $n \in \{\ell^{2}_{m}+1, \dots,  kr\}$, as evident from Definition \ref{DefGhostI}. 

The ``in particular'' part of \ref{lemmaComparingPro1} follows from Definition \ref{DefGhostI}. This concludes the proof of \ref{lemmaComparingPro1}. \\

Now, we prove \ref{lemmaComparingPro2}.
By \eqref{eqnEquivalentDefinitionHausdorffDistance}, it is enough to bound
\begin{equation}
\sup \left\{d_{\rm H}\Big(x,\bigcup_{m=0}^{2n} {\rm Branch}^{\gamma}(I_{m}^{(r,s)})\Big): x\in {\rm Skel}^{\gamma,s}(kr)\setminus \bigcup_{m=0}^{2n} {\rm Branch}^{\gamma}(I_{m}^{(r,s)})\right\}.
\end{equation}The latter is bounded by the maximum size among the connected components of ${\rm Skel}^{\gamma,s}(kr)\setminus \bigcup_{m=0}^{2n} {\rm Branch}^{\gamma}(I_{m}^{(r,s)})$.
Those are formed by segments involving the creation of a long-loop (of length larger than $r+1$) in the trajectory of the path $\gamma$. By \eqref{eqnDefinitionIrsm} and the definition of $\ell^1_m$ and $\ell^2_m$, note that each of those connected components is contained in 
\begin{equation} \label{AnotherSet2}
\begin{split}
&\quad[\max B_{i_m-1}^{(r,s)}+1,\min B_{i_m}^{(r,s)}-1]\cup \big(B^{(r,s)}_{i_m}\setminus \mathcal{G}^{\gamma,s}(kr)\big)\cup [\max B_{i_m}^{(r,s)}+1,\min B_{i_m+1}^{(r,s)}-1]\\
&\quad\bigcup [\max B_{j_m-1}^{(r,s)}+1,\min B_{j_m}^{(r,s)}-1]\cup \big(B^{(r,s)}_{j_m}\setminus \mathcal{G}^{\gamma,s}(kr)\big)\cup [\max B_{j_m}^{(r,s)}+1,\min B_{j_m+1}^{(r,s)}-1],
\end{split}
\end{equation}
\noindent for some $m\in [1,n]$, where $i_{m}$ and $j_{m}$ are defined before \eqref{IndentityOnemoretime}. 
Since the size of the set in \eqref{AnotherSet2} is at most $2r+6s$, then \ref{lemmaComparingPro2} follows.

We continue with the proof of \ref{lemmaComparingPro3}. 
First note that $n_{2}$ could be equal to $\ell_{m}^{2}$, and this could be equal (by Definition \eqref{eqnDefinitionIrsm}) to $\max B_{i_{m+1}^{\prime}-1}+s+1 \in [(i_{m+1}^{\prime}-1)r, i_{m+1}^{\prime}r-1]$, which justifies $u_2=i'_{m+1}$. 
The proof of \eqref{eqnSizeOfTheBranchesBetweenLongLoops} follows a similar argument to that of \eqref{Deteq1ii} in Lemma \ref{lemmaComparingSkeletonAndcRGRGBetweenLongLoopsiii}. 
Recall that ${\rm Branch}^{\gamma}(I_{m}^{(r,s)})$ is defined as the subgraph of ${\rm AB}^{\gamma}(kr)$ restricted to the vertex set ${\rm R}^{\gamma}(I_{m}^{(r,s)} \setminus \mathcal {G}^{\gamma,s}(kr) )$. Define
\begin{align}
D^{(r,s)}_{i} \coloneqq {\rm R}^\gamma( ([(i-1)r, ir -1] \setminus B_{i}^{(r,s)}) \setminus\mathcal {G}^{\gamma,s}(kr)),
\end{align}
\noindent for $i = i_{m}^{\prime}+1, \dots, i_{m+1}^{\prime}-1$, 
\begin{align}
C^{(r,s)}_{1,m} = {\rm R}^\gamma([\ell_{m}^{1}+1, \min B_{i_{m}^{\prime}+1}-1]  \setminus\mathcal {G}^{\gamma,s}(kr))
\end{align} 
\noindent and
\begin{align}
C^{(r,s)}_{2,m} = {\rm R}^\gamma([\max B_{i_{m+1}^{\prime}-1}+1, \ell_{m}^{2}]  \setminus\mathcal {G}^{\gamma,s}(kr)).
\end{align} 
Note that ${\rm R}^{\gamma}(I_{m}^{(r,s)} \setminus \mathcal {G}^{\gamma,s}(kr) )$ is equal to
\begin{equation}\label{eqnComparingSkeletonAndcRGRGBetweenLongLoopsiii4} 
C^{(r,s)}_{1,m} \cup \bigcup_{i=i_{m}^{\prime}+1}^{i_{m+1}^{\prime}-1} {\rm R}^\gamma(B_{i}^{(r,s)} \setminus\mathcal {G}^{\gamma,s}(kr)) \cup D^{(r,s)}_{i} \cup C^{(r,s)}_{2,m}.
\end{equation}
Now, using \eqref{InteCondiIIm} and Corollary \ref{LemmaIdenSeDecomp}, the set in \eqref{eqnComparingSkeletonAndcRGRGBetweenLongLoopsiii4} is equal to
\begin{align} \label{Le2Cas2Eq21}
C^{(r,s)}_{1,m} \cup \bigcup_{i=i_{m}^{\prime}+1}^{i_{m+1}^{\prime}-1} {\rm R}^\gamma({\rm NE}^{\gamma,s}(A_{i}^{(r,s)}))  \cup D^{(r,s)}_{i} \cup C^{(r,s)}_{2,m}.
\end{align}
\noindent Note also that \eqref{InteCondiIIm} and Corollary \ref{LemmaIdenSeDecomp} imply that the sets
\begin{align}
{\rm R}^\gamma({\rm NE}^{\gamma,s}(A_{i}^{(r,s)}))  = {\rm R}^\gamma(B_{i}^{(r,s)} \setminus\mathcal {G}^{\gamma,s}(kr)),
\end{align}
\noindent for $i = i_{m}^{\prime}+1, \dots, i_{m+1}^{\prime}-1$, are disjoint. Therefore, the lower bound in \eqref{eqnSizeOfTheBranchesBetweenLongLoops} is a direct consequence of \eqref{Le2Cas2Eq21}. The upper bound in \eqref{eqnSizeOfTheBranchesBetweenLongLoops} also follows from \eqref{Le2Cas2Eq21} in conjunction with the observation that
\begin{align}
\# \left( C^{(r,s)}_{1,m} \cup \bigcup_{i=i_{m}^{\prime}+1}^{i_{m+1}^{\prime}-1} D^{(r,s)}_{i} \cup C^{(r,s)}_{2,m} \right) \leq 3s (i_{m+1}^{\prime} - i_{m}^{\prime}-1) +2s. 
\end{align}
\noindent This concludes the proof of \ref{lemmaComparingPro3}. \\

Now we prove \ref{lemmaComparingPro4}. The proof of similar to that of part \ref{lemmaComparingPro3}. Clearly, for $n_{1}, n_{2} \in I_{m}^{(r,s)} \setminus \mathcal {G}^{\gamma,s}(kr)$ such that $n_{1}<n_{2}$, there are $u_{1}, u_{2} \in \{i_{m}^{\prime}+1, \dots, i_{m+1}^{\prime}\}$ such that $u_{1}\leq u_{2}$, $n_{1} \in [(u_{1}-1)r, u_{1}r-1]$ and $n_{2} \in [(u_{2}-1)r, u_{2}r-1]$. Define
\begin{align}
C(n_{1}) = {\rm R}^{\gamma}([n_{1}, u_{1}r-1]\setminus  \mathcal {G}^{\gamma,s}(kr))
\end{align}
\noindent and 
\begin{align}
C(n_{2}) = {\rm R}^{\gamma}( [(u_{2}-1)r, n_{2}]\setminus \mathcal {G}^{\gamma,s}(kr)).
\end{align}
\noindent Note that ${\rm R}^{\gamma}([n_{1}, n_{2}] \setminus \mathcal {G}^{\gamma,s}(kr))$ is equal to
\begin{align}
C(n_{1}) \cup \bigcup_{i=u_{1}+1}^{u_{2}-1} {\rm R}^{\gamma}(B_{i}\setminus \mathcal {G}^{\gamma,s}(kr)) \cup D_{i}^{(r,s)}  \cup C(n_{2}),
\end{align}
\noindent which by \eqref{InteCondiIIm} and Corollary \ref{LemmaIdenSeDecomp}, is equal to 
\begin{align} \label{Le2Cas2Eq22}
C(n_{1}) \cup \bigcup_{i=u_{1}+1}^{u_{2}-1}{\rm R}^{\gamma}({\rm NE}^{\gamma,s}(A_{i}^{(r,s)})) \cup D_{i}^{(r,s)}  \cup C(n_{2}).
\end{align}
\noindent Then, \eqref{eqnSizeOfTheBranchesBetweenLongLoopsII} follows from \eqref{Le2Cas2Eq22}. \\

Finally, we prove \ref{lemmaComparingPro5}. The proof follows a similar argument to that in part \ref{lemmaComparingPro4}. 
Indeed, we decompose the path connecting $\gamma(n_{1})$ with $\gamma(n_{2})$ within $\mathrm{Skel}^{(r,s)}_{\gamma}(kr)$. Then, we utilize the fact that the branches ${\rm Branch}^{\gamma}(I_{m}^{(r,s)})$ are disjoint.
The term $\sum_{m \in M_{n_{1},n_{2}}^{(r,s)}} \# {\rm Branch}^{\gamma}(I_{m}^{(r,s)})$ is the size of each branch contained in the path from $n_{1}$ to $n_2$ in $\mathrm{Skel}^{(r,s)}_{\gamma}(kr)$. Also, the term $\sum_{h=u_{1}+1}^{u_{2}-1} \# {\rm NE}^{\gamma,s}(A_{i}^{(r,s)})+3s (u_{2}-u_{1}-1) \vee 0+2r$ accounts for the distance between $\gamma(n_1)$ and $\gamma(\tilde n_1)$ in ${\rm Branch}^{\gamma}(I_{m}^{(r,s)})$, by \ref{lemmaComparingPro4}. Finally, the term $2rn$ bounds the removed intervals $\cup_{m=1}^n {\rm R}^\gamma(B_{i_m}^{(r,s)} \setminus {\mathcal G}^{\gamma,s}(kr))\cup {\rm R}^\gamma(B_{j_m}^{(r,s)} \setminus {\mathcal G}^{\gamma,s}(kr))$.
\end{proof}

\section{Decomposable paths}
\label{Sub:decomposablepaths} 

In this section, we prove that the lazy random walk on a finite, simple, connected, regular graph $G=(V,E)$ satisfies all assumptions and properties introduced in the preceding section with high probability up to a given fixed time. Namely, Assumptions \hyperref[assumptionNoLoopsOfIntermediateLength]{No loops of intermediate length},  \hyperref[assumptionLocalCutPoints]{Local cut points} and properties \ref{Pro1}-\ref{Pro5II} with high probability up to a given fixed time. This will enable us to directly apply the results established in Corollary \ref{LemmaIdenSeDecomp} and Proposition~\ref{DerterLemma1}. 

We start by introducing the notion of \emph{decomposable paths} in the deterministic setting.

\begin{definition}[$(r,s,s')$-good paths up to $\mathbb{N}$]\label{Def:006} Let $r,s,s^\prime\in\mathbb{N}$ with $r \geq 3s+1 \geq 18s^{\prime} +1$. We say that a path $\gamma:\,\mathbb{N}_0\to V$ on a finite, simple, connected, regular graph $G=(V,E)$ is an $(r,s,s')$-good path on $[0,N]$ if it satisfies Assumptions \hyperref[assumptionNoLoopsOfIntermediateLength]{No loops of intermediate length} and  \hyperref[assumptionLocalCutPoints]{Local cut points} on $[0,N]$.

\end{definition}  

In the following, we write
\begin{equation} 
\label{e:126y}
{\mathcal G}^{G;(r,s, s^{\prime})}(N):=\big\{\gamma:\mathbb{N}_0\to V:\,\gamma\mbox{ is a }(r,s, s^{\prime})\mbox{-good path on $[0,N]$}\big\}. 
\end{equation}

Recall the definition of $B^{(r,s)}_i$ in \eqref{eqnDefinitionAiBi}.

\begin{definition}[$(r,s, s^{\prime})$-decomposable paths]\label{Def:003AB}
Let $r,s,s^\prime\in\mathbb{N}$ with $r \geq 3s+1 \geq 18s^{\prime} +1$, $s'\geq t_{{\mbox{\tiny mix}}}^{G}+1$ and $N\in \na$. We define a path $\gamma:\,\mathbb{N}_0\to V$ on a finite, simple, connected, regular graph $G=(V,E)$ as $(r,s, s^{\prime})$-decomposable on $[0,N]$ if $\gamma\in {\mathcal G}^{G,(r,s, s^{\prime})}(N)$ and satisfies \ref{Pro1}-\ref{Pro5II}. 
\end{definition}

In the following, we write
\begin{equation} 
\label{e:126}
{\mathcal D}^{G,(r,s, s^{\prime})}(N)
:=
\big\{\gamma\in {\mathcal G}^{G,(r,s, s^{\prime})}(N):\,\gamma\mbox{ is }(r,s, s^{\prime})\mbox{-decomposable on $[0,N]$}\big\}.
\end{equation}

Next, for $(r,s,s^\prime)$-decomposable  $\gamma$, let
\begin{equation} 
\label{e:131AB}
\sigma^{\gamma,(r,s, s^{\prime})} 
\coloneqq 
\max\big\{m \in\mathbb{N}_0:\,\gamma\in{\mathcal D}^{G,(r,s, s^{\prime})}(m)\big\}.
\end{equation}

The next result gives an estimate of the probability that the lazy random walk on a simple, connected, regular graph is $(r,s,s^\prime)$-decomposable until a given fixed time.

\begin{proposition}[Decomposability up to time $N$]\label{ProDescompo}
Let $W$ be the lazy random walk on a finite simple, connected, regular graph $G=(V,E)$. 
Let $r,s,s^\prime, q \in\mathbb{N}$ with $r\ge 3s+1\ge 18s'+1$,  $2s'\geq q t_{{\mbox{\tiny mix}}}^{G}+1$ and $N\in \na$. 
Then, for all $\varrho \in V$, we have
\begin{equation}
\begin{split} 
\label{e:093}
\mathbb{P}_\varrho\big(\sigma^{W,(r,s,s')} > N\big) &\geq 
1- N\big(\overline{q}^{G}(2s^{\prime}) \big)^{ \lfloor \frac{s}{6s'} \rfloor} - \frac{2(3r+2s)N}{\# V}-\frac{16rN^{3}}{(\# V)^{2}}\\
& \quad \quad \quad -\frac{8Ns(3N + 2s)}{r\# V} - \frac{sN}{3 \cdot  4^{q} s'} ,
\end{split}
\end{equation}
where $\overline{q}^{G}$ is defined in \eqref{LPeq8}.
\end{proposition}

\begin{proof} 
By Definition \ref{Def:006}, Lemma~\ref{lemmaLongLoop} and Lemma~\ref{Lemmacut point} together with the union bound, for all $\varrho\in V$,
\begin{equation}
\begin{split}\label{e:095}
\qquad \mathbb{P}_\varrho\big(W\in{\mathcal P}^{G,(r,s, s^{\prime})}( N)\big)
\ge&
1-\frac{2(r-s'+1)(N+1)}{\# V}-\big(N-s+1\big) \Big( \big(\overline{q}^{G}(2s^{\prime}) \big)^{\lfloor\frac{s}{6s'}\rfloor} +  \frac{s}{3 \cdot  4^{q} s'} \Big)\\
\geq& 1-\frac{4rN}{\# V}-N\big(\overline{q}^{G}(2s^{\prime}) \big)^{\lfloor\frac{s}{6s'}\rfloor} - \frac{sN}{3 \cdot  4^{q} s'} .
\end{split}
\end{equation}
Now, we define random variables that integrate the properties \ref{Pro1}-\ref{Pro5II}.
Let
\begin{equation}
\label{e:091}
\begin{split}
{\mathcal N}^{\gamma,(r,s)}_1(N):=&\sum_{ 1 \leq j_1,j_{2},j_3\leq \lfloor\tfrac{N}{r}\rfloor}\mathbf{1}_{\{ \# \{ j_1,j_{2},j_3\}=3\}}\mathbf{1}_{ \{  {\rm R}^\gamma(B_{j_1}^{(r,s)}) \cap {\rm R}^\gamma(B_{j_{2}}^{(r,s)})\neq \emptyset, {\rm R}^\gamma(B_{j_3}^{(r,s)}) \cap {\rm R}^\gamma(B_{j_{2}}^{(r,s)})\neq \emptyset\}}
\end{split}
\end{equation}
\begin{equation} 
\label{e:116pirgamma}
\begin{aligned}
{\mathcal N}^{\gamma,(r,s)}_2(N) 
\coloneqq \sum_{\substack{2\le i \le \lfloor\tfrac{N}{r}\rfloor\\ 1\leq j<i}} 
\mathbf{1}_{ \{R^\gamma([(j-1)r,jr-1]\setminus B^{(r,s)}_j)\cap R^\gamma([\min B^{(r,s)}_i,\max B^{(r,s)}_i+s]) \not=\emptyset\}},
\end{aligned}
\end{equation}
\begin{equation} 
\label{e:116pirgamma}
\begin{aligned}
{\mathcal N}^{\gamma,(r,s)}_3(N) 
\coloneqq \sum_{2\le i \le \lfloor\tfrac{N}{r}\rfloor} 
\mathbf{1}_{ \{R^\gamma([0,\max B^{(r,s)}_{i-1}])\cap  R^\gamma([(i-1)r,ir-1]\setminus B^{(r,s)}_i)\not=\emptyset\}},
\end{aligned}
\end{equation}
\noindent and
\begin{equation}
\begin{split}
{\mathcal N}^{\gamma,(r,s)}_4(N):=&
\mathbf{1}_{ \{  {\rm R}^\gamma ([0,r \lfloor N/r \rfloor -1]) \cap {\rm R}^\gamma ( [r\lfloor N/r \rfloor+2s,  N ] )  \neq \emptyset\}}\\
+&\mathbf{1}_{ \{ {\rm R}^\gamma ([0,r \lfloor N/r \rfloor -s-1]) \cap {\rm R}^\gamma ( [r\lfloor N/r \rfloor, r\lfloor N/r \rfloor+2s-1] ) \neq \emptyset\}}.
\end{split}
\end{equation}
\noindent Note that the variables $\mathcal{N}^{\gamma,(r,s)}_k(N)'s$ equalling zero corresponds to the fulfilment of the corresponding properties \ref{Pro1}-\ref{Pro5II}. That is, $\mathcal{N}^{\gamma,(r,s)}_1(N) = 0$ if an only if \ref{Pro1} is satisfied, and similarly for the other variables.
Now, for any fixed $\varrho\in V$, we bound $\mathbb{P}_\varrho\big({\mathcal N}^{W,(r,s)}_k(T) \geq 1 \big)$, for $k= 1, \dots, 4$.

First, we prove that 
\begin{equation} \label{ProDescompoEQ2}
\begin{aligned}
&\mathbb{P}_\varrho\big({\mathcal N}^{W,(r,s)}_1(N) \geq 1 \big)\leq \frac{16rN^3{}}{(\# V)^{2}}.
\end{aligned}   
\end{equation}
\noindent Observe that \eqref{ProDescompoEQ2} holds trivially when $N < 3r$, as ${\mathcal N}^{W,(r,s)}_1(N) =0$ in this case. Then, suppose that $N \geq 3r$. By \eqref{Exteq13} in Proposition \ref{P:004}, for any $i,j,k\in \mathbb{N}$ all different, we have that
\begin{equation} \label{ProDescompoEQ1}
\p_{\rho}\big({\rm R}^W(B_{j}^{(r,s)}) \cap {\rm R}^W(B_{k}^{(r,s)})\neq \emptyset, {\rm R}^W(B_{i}^{(r,s)}) \cap {\rm R}^W(B_{k}^{(r,s)})\neq \emptyset\big)\leq \frac{8(r-3s)^{4}}{(\# V)^{2}}.
\end{equation}
\noindent Therefore, \eqref{ProDescompoEQ2} follows by observing that \eqref{ProDescompoEQ1} and Markov's inequality imply that
\begin{equation} \label{ProDescompoEQ2II}
\begin{aligned}
&\mathbb{P}_\varrho\big({\mathcal N}^{W,(r,s)}_1(N) \geq 1 \big)\leq 2 \left \lfloor \frac{N}{r}\right \rfloor^{3}\frac{8(r-3s)^{4}}{(\# V)^{2}}.
\end{aligned}   
\end{equation}

Next, we prove that
\begin{equation} \label{ProDescompoEQ7}
\begin{aligned}
&\mathbb{P}_\varrho\big({\mathcal N}^{W,(r,s)}_2(N) \geq 1 \big) \leq \frac{12s N^{2}}{r\# V}.
\end{aligned}   
\end{equation}
\noindent Observe that \eqref{ProDescompoEQ7} holds trivially when $N < 2r$, as ${\mathcal N}^{W,(r,s)}_2(N) =0$ in this case. Then, suppose that $N \geq 2r$. For any $j \in \mathbb{N}$, note that
\begin{align} \label{ProDescompoEQ5}
R^W([(j-1)r,jr-1]\setminus B^{(r,s)}_j) = R^W([(j-1)r,(j-1)r+2s-1]) \cup R^W([jr-s,jr-1]).
\end{align}
\noindent Then, by \eqref{ProDescompoEQ5}, \eqref{Exteq12} in Proposition \ref{P:004} and the union bound, we have that, for any $i,j \in \mathbb{N}$ different, 
\begin{equation} \label{ProDescompoEQ6} 
\p_{\rho}\big(R^W([(j-1)r,jr-1]\setminus B^{(r,s)}_j)\cap R^W([\min B^{(r,s)}_i,\max B^{(r,s)}_i+s]) \not=\emptyset\big)\leq \frac{12s(r-2s)}{\# V}.
\end{equation}
\noindent Hence,  \eqref{ProDescompoEQ7} follows by observing that \eqref{ProDescompoEQ6} and Markov's inequality imply that
\begin{equation} \label{ProDescompoEQ7III}
\begin{aligned}
&\mathbb{P}_\varrho\big({\mathcal N}^{W,(r,s)}_2(N) \geq 1 \big)\leq \left \lfloor \frac{N}{r} \right \rfloor^{2} \frac{12s(r-2s)}{\# V}.
\end{aligned}   
\end{equation}

Now, we prove that
\begin{equation}
\begin{aligned} \label{ProDescompoEQ15}
\mathbb{P}_\varrho\big({\mathcal N}^{W,(r,s)}_3(N) \geq 1 \big) & \leq  \frac{12sN^2}{r \# V}+\frac{16s^2N}{r\# V}.
\end{aligned}   
\end{equation}
\noindent Observe that \eqref{ProDescompoEQ15} holds trivially when $N < 2r$, as ${\mathcal N}^{W,(r,s)}_3(N) =0$ in this case. Then, suppose that $N \geq 2r$. It follows from \eqref{ProDescompoEQ5} (with $j=i$) and the union bound that, for $i \geq 2$, 
\begin{align} \label{ProDescompoEQ8}
& \p_{\rho}\big(R^W([0,\max B^{(r,s)}_{i-1}])\cap  R^W([(i-1)r,ir-1]\setminus B^{(r,s)}_i) \not=\emptyset \big) \nonumber \\
& \quad \quad  = \p_{\rho}\big(R^W([0,\max B^{(r,s)}_{i-1}])\cap  R^W([(i-1)r,(i-1)r+2s-1]) \not=\emptyset \big) \nonumber \\
& \quad \quad \quad \quad \quad \quad +  \p_{\rho}\big(R^W([0,\max B^{(r,s)}_{i-1}])\cap  R^W([ir-s,ir-1]) \not=\emptyset \big).
\end{align}
\noindent Note that, once again, the union bound implies that
\begin{align} \label{ProDescompoEQ9}
& \p_{\rho}\big(R^W([0,\max B^{(r,s)}_{i-1}])\cap  R^W([(i-1)r,(i-1)r+2s-1]) \not=\emptyset \big) \nonumber \\
& \quad \quad  = \p_{\rho}\big(R^W([0,s])\cap  R^W([(i-1)r,(i-1)r+2s-1]) \not=\emptyset \big)  \nonumber \\
& \quad \quad \quad \quad \quad \quad +  \p_{\rho}\big(R^W([s+1,\max B^{(r,s)}_{i-1}])\cap  R^W([(i-1)r,(i-1)r+2s-1]) \not=\emptyset \big).
\end{align}
\noindent On the one hand, by \eqref{Exteq12} in Proposition \ref{P:004},
\begin{align} \label{ProDescompoEQ10}
\p_{\rho}\big(R^W([s+1,\max B^{(r,s)}_{i-1}])\cap  R^W([(i-1)r,(i-1)r+2s-1]) \not=\emptyset \big) \leq \frac{8s((i-1)r-1)}{\# V}.
\end{align}
\noindent 
On the other hand, by Corollary \ref{coroRangeSelfIntersectionProbabilitiesStartingAtZero} we have
\begin{equation} \label{ProDescompoEQ12} \p_{\rho}\big(R^W([0,s])\cap  R^W([(i-1)r,(i-1)r+2s-1]) \not=\emptyset \big) \leq  \frac{4s(s+1)}{\# V}.
\end{equation}

\noindent Then, by \eqref{ProDescompoEQ9}, \eqref{ProDescompoEQ10} and  \eqref{ProDescompoEQ12}, 
\begin{align} \label{ProDescompoEQ13}
& \p_{\rho}\big(R^W([0,\max B^{(r,s)}_{i-1}])\cap  R^W([(i-1)r,(i-1)r+2s-1]) \not=\emptyset \big) \nonumber \\
& \quad \quad  \quad \leq \frac{8s(i-1)r}{\# V} + \frac{8s^2}{\# V}.
\end{align}
\noindent Following a computation analogous to \eqref{ProDescompoEQ13}, we have that
\begin{align} \label{ProDescompoEQ14}
\p_{\rho}\big(R^W([0,\max B^{(r,s)}_{i-1}])\cap  R^W([ir-s,ir-1]) \not=\emptyset \big)  \leq \frac{4s(i-1)r}{\# V} + \frac{8s^2}{\# V}.
\end{align}
\noindent Hence, \eqref{ProDescompoEQ15} follows by observing that \eqref{ProDescompoEQ8}, \eqref{ProDescompoEQ13}, \eqref{ProDescompoEQ14} and Markov's inequality imply that
\begin{equation}
\begin{aligned} \label{ProDescompoEQ15II}
\mathbb{P}_\varrho\big({\mathcal N}^{W,(r,s)}_3(N) \geq 1 \big) & \leq \left \lfloor\frac{N}{r}  \right \rfloor^2 \frac{12sr}{\# V}+ \left \lfloor\frac{N}{r}  \right \rfloor \frac{16s^2}{\# V}.
\end{aligned}
\end{equation}

Finally, we prove that
\begin{equation} \label{ProDescompoEQ16}
\begin{split}
& \mathbb{P}_\varrho\big({\mathcal N}^{W,(r,s)}_4(N) \geq 1 \big) \leq \frac{2 rN}{\# V}+\frac{4sN}{\# V}.
\end{split}
\end{equation}
\noindent Observe that \eqref{ProDescompoEQ16} holds trivially when $N < r$, as ${\mathcal N}^{W,(r,s)}_4(N) =0$ in this case. Then, suppose that $N \geq r$. By applying Corollary \ref{coroRangeSelfIntersectionProbabilitiesStartingAtZero} and following a derivation analogous to that of \eqref{ProDescompoEQ15}, we observe that
\begin{equation} \label{ProDescompoEQ16III}
\begin{split}
& \mathbb{P}_\varrho\big({\mathcal N}^{W,(r,s)}_4(N) \geq 1 \big) \\
& \quad \quad \leq \p_\varrho\paren{ {\rm R}^\gamma ([0,r \lfloor N/r \rfloor -1]) \cap {\rm R}^\gamma ( [r\lfloor N/r \rfloor+2s,N ] )  \neq \emptyset}\\
&\quad \quad \quad \quad  +\p_\varrho\paren{ {\rm R}^\gamma ([0,r \lfloor N/r \rfloor -s-1]) \cap {\rm R}^\gamma ( [r\lfloor N/r \rfloor, r\lfloor N/r \rfloor+2s-1] ) \neq \emptyset}\\
& \quad \quad \leq \frac{2 r \lfloor N/r \rfloor (N -r\lfloor N/r \rfloor-2s+1)}{\# V} + \frac{4s (r\lfloor N/r \rfloor -s)}{\# V}.
\end{split}
\end{equation}
\noindent Clearly, \eqref{ProDescompoEQ16III} implies \eqref{ProDescompoEQ16}. 

Therefore, \eqref{e:093} follows from Definition \ref{Def:003AB} and the union bound by combing \eqref{e:095}, \eqref{ProDescompoEQ2}, \eqref{ProDescompoEQ7}, \eqref{ProDescompoEQ15}, \eqref{ProDescompoEQ16}. This finishes our proof.
\end{proof}

We conclude this section by proving a consequence of Proposition \ref{ProDescompo}. We show that, under suitable conditions, with high probability, a sequence of lazy random walks $(W_N)_{N \in \mathbb{N}}$ on a sequence of finite, simple, connected, regular graphs $(G_N;N\in \mathbb{N}_{0})$ converges to a decomposable path as the size of the sequence of graphs grows boundlessly.

\begin{corollary}
\label{corollaryAsymptoticDecomposability}
Let $(G_N;N\in \mathbb{N}_{0})$ be a sequence of finite simple, connected, regular graphs such that are $(s_{N}^{\prime})_{N \in \mathbb{N}}$, $(s_{N})_{N \in \mathbb{N}}$ and $(r_{N})_{N \in \mathbb{N}}$ in $\mathbb{N}$ that satisfy \eqref{THESeqNewII}. Suppose also that $(G_N;N\in \mathbb{N}_{0})$ satisfies the Assumption \hyperref[assumptionTransientRandomWalk]{Transient random walk} and let $(c_{N})_{N \in \mathbb{N}}$ be the sequence defined in \eqref{eqnDefinitionOfcN}. For every $N \in \mathbb{N}_{0}$, let $\rho_N \in G_N$  and let $W_N = (W_N(n))_{n \in \mathbb{N}_{0}}$ be the lazy random walk on $G_N$. Then, for any constant $T\in (0,\infty)$, we have that
\begin{equation}
\mathbb{P}_{\varrho_N}\big(\sigma^{W_N,(r_N,s_N,s'_N)} > Tr_N/c_N \big) = 1-o(1).
\end{equation}
\end{corollary}
\begin{proof}
Note that by  \eqref{eqnAsymptoticApproximationsForGammaNandCN} in Corollary \ref{CorollaryBoundingQBar}, we have $r_N/c_N=\Theta((\# V_N)^{1/2})$.
Also, by \eqref{THESeqNewII},
one can assume that $N$ is sufficiently large such that $2s_{N}^{\prime} \geq \big( \frac{\ln (2 \# V_{N})}{\ln 4}  \big)t_{\mbox{\tiny mix}}^{G_{N}} +1$.
Thus, the result is a consequence of \eqref{e:093} in Proposition \ref{ProDescompo}, \eqref{eqnLimInfBarq1II} in Corollary \ref{CorollaryBoundingQBar}, Corollary \ref{corollaryUniformBoundsOfNonErasedVerticesSegmentAi} and \eqref{THESeqNewII}.
\end{proof}

\section{Bounding the GH-distance between the Aldous-Broder chain and the $c$-RGRG map}\label{sectionBoundingGHBEtweenABChainAndRGRGR}
In this subsection, we compare the Gromov-Hausdorff distance between the Aldous--Broder chain and the $c$-RGRG map. Fix $0 \leq T < \infty$, $c>0$ and $r,s,s' \in \mathbb{N}$ such that $r \geq 3s +1\ge 18s'+1$. 
Let $\gamma: \mathbb{N}_{0} \rightarrow V$ be a path on a finite, simple, connected graph $G=(V,E)$ and let $\pi\subset \Delta_+^2$ be a nice point cloud.

For each $(i,j)\in \mathbb{N}\times \mathbb{N}$ with $1\le i<j$, recall the definition of the variables $Z^{\pi,c}_{(i,j)}$ and $Z^{\gamma,(r,s)}_{(i,j)}$ given in \eqref{e:060ii} and \eqref{e:061ii}, respectively. Recall also, the definition of the times $\sigma^{\pi,c}$ and $\sigma^{\gamma,(r,s, s^{\prime})}$ given in \eqref{e:131} and \eqref{e:131AB}, respectively.

Note that for $k \in \mathbb{N}_{0}$, from \eqref{HERRRRn} we can rewrite $\mathcal{N}^{\gamma,(r, s)}$ as 
\begin{align}
\mathcal{N}^{\gamma,(r, s)}(k) = \sum_{1\leq i < j \leq k} Z^{\gamma,(r,s)}_{(i,j)}.
\end{align}This variable counts the number of indices in $[1,k]$, involving a long loop between a segment and a locally non-erased segment. 
In the context of Lemma \ref{lemmaComparingSkeletonAndcRGRGBetweenLongLoopsiii}, whenever \eqref{InteCondiIIm} is satisfied we have $\mathcal{N}^{\gamma,(r, s)}(k)=2n$.

\begin{lemma}[Comparing the Aldous-Broder chain and the $c$-RGRG]  \label{lemmaComparingABAndcRGRG}
Let $\pi\subset \Delta_+^2$ be a nice point cloud and let $\gamma: \mathbb{N}_{0} \rightarrow V$ be a path on a finite, simple, connected graph $G=(V,E)$.
Fix $r,s, s^{\prime} \in \na$ with $r \geq 3s+1 \geq 18s^{\prime} +1$. 
Fix $c > 0$ and $0 \leq T < \infty$ such that $T/c \geq 2$ and $(\sigma^{\gamma,(r,s, s^{\prime})}/r) \wedge (\sigma^{\pi,c}/c) > T/c$. 
Suppose that 
\begin{equation}\label{eqnComparingSkeletonAndcRGRG}
Z^{\gamma,(r,s)}_{(i,j)} = Z^{\pi,c}_{(i,j)} \quad \mbox{ for all $1 \leq i < j \leq \lfloor T/c \rfloor$}.
\end{equation}Then, for any $a >0$,
\begin{equation}\begin{aligned} \label{eqnGHDistanceABcRGRG}
& \sup_{t\in [0,T]}d_{\rm GH}\Big( a \cdot {\rm AB}^{\gamma}(\lfloor t/c \rfloor r), {\rm RGRG}^{\pi}(\lfloor t/c  \rfloor c)\Big) \\
& \quad \quad \leq \sum_{i=1}^{\lfloor T/c \rfloor} \left|a \cdot \# {\rm NE}^{\gamma,s}(A_{i}^{(r,s)}) -c \right| + \frac{3asT}{c}+ 2a(r+s) + 2c  \\
& \quad \quad \quad \quad + \frac{1}{2}(ar +c) \mathcal{N}^{\gamma,(r, s)}(\lfloor T/c \rfloor)   + a r.
\end{aligned}\end{equation}
\end{lemma}
\begin{remark}
If \eqref{eqnComparingSkeletonAndcRGRG} is satisfied, then $\mathcal{N}^{\gamma,(r, s)}(k) = \mathcal{N}^{\pi,c}(k)$, for all $k = 0, \dots, \lfloor T/c \rfloor$, where $\mathcal{N}^{\pi,c}(k) \coloneqq \sum_{1\leq i < j \leq k} Z^{\pi,c}_{(i,j)}$. In particular, we can substitute  $\mathcal{N}^{\gamma,(r, s)}(k)$ by $\mathcal{N}^{\pi,c}(k)$ in \eqref{eqnGHDistanceABcRGRG}. 
\end{remark}

In preparation for the proof of Lemma \ref{lemmaComparingABAndcRGRG}, we need the next result.

\begin{lemma} \label{ComparisionGHSkeletoRGRG}
Let $\pi\subset \Delta_+^2$ be a nice point cloud and let $\gamma: \mathbb{N}_{0} \rightarrow V$ be a path on a finite, simple, connected graph $G=(V,E)$.
Fix $r,s, s^{\prime} \in \na$ with $r \geq 3s+1 \geq 18s^{\prime} +1$. 
Fix $c > 0$ and $0 \leq T < \infty$ such that $T/c \geq 2$ and $(\sigma^{\gamma,(r,s, s^{\prime})}/r) \wedge (\sigma^{\pi,c}/c) > T/c$.
Suppose that
\begin{equation} \label{IdentIndex}
Z^{\gamma,(r,s)}_{(i,j)} = Z^{\pi,c}_{(i,j)}, \quad \text{for all} \quad 1 \leq i < j \leq \lfloor T/c \rfloor.
\end{equation}
\begin{enumerate}[label=(\textbf{\roman*})]
\item \label{ClaimDist1} Fix any $0 \leq k\leq \lfloor T/c \rfloor$ and suppose that 
\begin{equation} 
Z^{\gamma,(r,s)}_{(i,j)} = 0, \quad \mbox{for all} \quad  1 \leq i < j \leq k.
\end{equation}
(If $k< 2$, then \eqref{Indizero} is satisfied by vacuity.) 
Then, for any $a >0$,
\begin{equation}\begin{aligned} \label{Deteq2iii}
d_{\rm GH}(a \cdot \mathrm{Skel}^{\gamma,s}(kr), {\rm RGRG}^{\pi}(kc)) \leq \sum_{i=1}^{k} \left|a \cdot \# {\rm NE}^{\gamma,s}(A_{i}^{(r,s)}) -c \right| +3ask+a.
\end{aligned}\end{equation}
\item \label{ClaimDist2} Fix any $2 \leq k \leq \lfloor T/c \rfloor$ and suppose also that there exists an integer $1 \leq n \leq \lfloor k/2 \rfloor$ and indices $1 \leq i_{m} < j_{m} \leq k$, for $1 \leq m \leq n$, such that $i_{1}, \dots, i_{n}, j_{1}, \dots, j_{n}$ are all distinct, 
\begin{align} \label{IdentIndexOneZero}
Z^{\gamma,(r,s)}_{(i_{m},j_{m})} = 1 \quad \text{and} \quad Z^{\gamma,(r,s)}_{(i,j)}= 0,
\end{align}
\noindent for all $1 \leq i < j \leq k$ with either $i$ or $j$ not in $\{ i_{1}, \dots, i_{n}, j_{1}, \dots, j_{n} \}$. Then, for any $a >0$, 
\begin{equation}\begin{aligned}  \label{DistBoundP}
& d_{\rm GH}(a \cdot \mathrm{Skel}^{\gamma,s}(kr), {\rm RGRG}^{\pi}(kc)) \\
&  \quad \quad \leq
\sum_{i=1}^{k} \left|a \cdot \# {\rm NE}^{\gamma,s}(A_{i}^{(r,s)}) -c \right| + 3ask+ 2a(r+s) + 2c + n (ar +c).
\end{aligned}\end{equation} 
\end{enumerate}
\end{lemma}

\begin{proof}
The proof relies on the results from Lemmas \ref{lemmaBranchescRGRG},  \ref{lemmaBranchescRGRGGeneralCase},  \ref{DerterLemma2iii} and  \ref{lemmaComparingSkeletonAndcRGRGBetweenLongLoopsiii}.  Observe that the conditions of those Lemmas are satisfied. Indeed, since $(\sigma^{\gamma,(r,s, s^{\prime})}/r) \wedge (\sigma^{\pi,c}/c) > T/c$, it follows that $\sigma^{\gamma,(r,s, s^{\prime})} > \floor{T/c}r$ and $\sigma^{\pi,c} > T$. Thus, $\gamma$ is $(r,s, s^{\prime})$-decomposable on the interval $[0, \floor{T/c}r]$ (recall Definition \ref{Def:003AB}), and $\pi$ is a $c$-decomposable point cloud up to time $T$ (recall Definition \ref{Def:004}). \\

First, we prove \ref{ClaimDist1}. Recall that an isomorphism between metric spaces must be distance preserving. Then, \ref{ClaimDist1} follows immediately from Lemmas \ref{lemmaBranchescRGRG} and \ref{DerterLemma2iii}. 
Indeed, it is enough to observe that, in this case, $\mathrm{Skel}^{\gamma,s}(kr)$ and ${\rm RGRG}^{\pi}(kc)$ can be isometrically embedded to the interval $[0, a \cdot (\# \mathrm{Skel}^{\gamma,s}(kr)-1) ]$ and $[0,kc]$, respectively.

Next, we prove \ref{ClaimDist2}. It follows from the triangle inequality, Lemma \ref{lemmaBranchescRGRGGeneralCase} and Lemma \ref{lemmaComparingSkeletonAndcRGRGBetweenLongLoopsiii} that
\begin{align}
& d_{\rm GH}(a \cdot \mathrm{Skel}^{\gamma,s}(kr), {\rm RGRG}^{\pi}(kc))   \nonumber \\
& \quad \quad \leq d_{\rm GH} \left(a \cdot \bigcup_{m=0}^{2n} {\rm Branch}^{\gamma}(I_{m}^{(r,s)}),  \bigcup_{m=0}^{2n} {\rm Branch}^{c}(I_{m}^{c}) \right)  + 2(c + ar +3as).
\end{align}
\noindent On the other hand,
\begin{equation}\begin{aligned} \label{GHdisIII}
d_{\rm GH}\Big(a \cdot \bigcup_{m=0}^{2n} {\rm Branch}^{\gamma}(I_{m}^{(r,s)}),  \bigcup_{m=0}^{2n} {\rm Branch}^{c}(I_{m}^{c}) \Big) = \frac{1}{2} \inf_{{\mathfrak R}} \mathrm{dis}({\mathfrak R}),
\end{aligned}\end{equation} 
\noindent where the infimum is taken over all correspondences ${\mathfrak R}$ between  $\bigcup_{m=0}^{2n} {\rm Branch}^{\gamma}(I_{m}^{(r,s)})$ and $\bigcup_{m=0}^{2n} {\rm Branch}^{c}(I_{m}^{c})$ and
$\mathrm{dis}({\mathfrak R})$, is the distortion of ${\mathfrak R}$, defined in \eqref{distortion}. 

By Lemmas \ref{lemmaBranchescRGRGGeneralCase} and \ref{lemmaComparingSkeletonAndcRGRGBetweenLongLoopsiii}, we know that for $m =0, \dots, 2n$, the set ${\rm Branch}^{\gamma}(I_{m}^{(r,s)})$ is a connected path and ${\rm Branch}^{c}(I_{m}^{c})$ is a compact interval. We can then construct a correspondence ${\mathfrak R} \subseteq \bigcup_{m=0}^{2n} {\rm Branch}^{\gamma}(I_{m}^{(r,s)})\times \bigcup_{m=0}^{2n} {\rm Branch}^{c}(I_{m}^{c})$ as follows: \\

$(x,y)\in {\mathfrak R}$ if and only if there exists $m\in [0,2n]$ and $h \in \{1, \dots, k\}$ such that for a unique $n_x \in (I_{m}^{(r,s)} \cap [(h-1)r, hr-1]) \setminus \mathcal {G}^{\gamma,s}(kr)$, we have $x=\gamma(n_x)\in {\rm Branch}^{\gamma}(I_{m}^{(r,s)})$ and $y \in I_{m}^{c} \cap [(h-1)c, hc)$. The uniqueness of $n_x$ follows from Lemma \ref{lemmaComparingSkeletonAndcRGRGBetweenLongLoopsiii}. \\

Let  $(x_1,y_1), (x_2, y_2) \in {\mathfrak R}$. Then, there exist $m_{1},m_{2}\in [0,2n]$ and $h_{1}, h_{2} \in \{1, \dots, k\}$ such that
\begin{itemize}
\item $n_{1} \in (I_{m_{1}}^{(r,s)} \cap [(h_{1}-1)r, h_{1}r-1]) \setminus \mathcal {G}^{\gamma,s}(kr)$ and $n_{2} \in (I_{m_{2}}^{(r,s)} \cap [(h_{2}-1)r, h_{2}r-1]) \setminus \mathcal {G}^{\gamma,s}(kr)$,
\item $x_{1}=\gamma(n_{1})\in {\rm Branch}^{\gamma}(I_{m_{1}}^{(r,s)})$ and $x_{2}=\gamma(n_{2})\in {\rm Branch}^{\gamma}(I_{m_{2}}^{(r,s)})$,
\item $y_{1} \in I_{m_{1}}^{c} \cap [(h_{1}-1)c, h_{1}c)$ and $y_{2} \in I_{m_{2}}^{c} \cap [(h_{2}-1)c, h_{2}c)$.
\end{itemize}

Furthermore, we assume that the correspondence ${\mathfrak R}$ is constructed such that $n_{1}<n_{2}$ if and only if $x_{1}<x_{2}$. We consider two cases: \\

{\bf Case 1.} Suppose that $m_{1}=m_{2}$. Then, we have $h_1=h_2$.  Without of generality we assume that $n_{1}<n_{2}$. Then, by Lemma \ref{lemmaComparingSkeletonAndcRGRGBetweenLongLoopsiii} \ref{lemmaComparingPro4} and Lemma \ref{lemmaBranchescRGRGGeneralCase} \ref{lemmaBranchescRGClaim3}, we have that
\begin{align}
& |a \cdot d_{\mathrm{Skel}^{\gamma,s}(kr)}(x_1,x_2) - d_{{\rm RGRG}^{\pi}(kc)}(y_1,y_2)| \nonumber \\
& \quad \quad \leq \sum_{h=u_{1}+1}^{u_{2}-1} \left|a \cdot \# {\rm NE}^{\gamma,s}(A_{i}^{(r,s)}) -c \right| + 3as (u_{2}-u_{1}-1) \vee 0 + 2ar + 2c \nonumber \\
& \quad \quad \leq \sum_{i=1}^{k} \left|a \cdot \# {\rm NE}^{\gamma,s}(A_{i}^{(r,s)}) -c \right| + 3as k + 2ar + 2c,
\end{align}\noindent where $u_{1}$ and $u_{2}$ are defined as in Lemma \ref{lemmaComparingSkeletonAndcRGRGBetweenLongLoopsiii} \ref{lemmaComparingPro4} and Lemma \ref{lemmaBranchescRGRGGeneralCase} \ref{lemmaBranchescRGClaim3}. 

{\bf Case 2.} Suppose that $m_{1} \neq m_{2}$. Without loss of generality we assume that $n_{1}<n_{2}$. A similar computation to that in {\bf Case 1}, using Lemma \ref{lemmaComparingSkeletonAndcRGRGBetweenLongLoopsiii} \ref{lemmaComparingPro3}-\ref{lemmaComparingPro5} and Lemma \ref{lemmaBranchescRGRGGeneralCase} \ref{lemmaBranchescRGClaim4} yields to
\begin{align}
& |a \cdot d_{\mathrm{Skel}^{\gamma,s}(kr)}(x_1,x_2) - d_{{\rm RGRG}^{\pi}(kc)}(y_1,y_2)| \nonumber \\
& \quad \quad \leq \sum_{i=1}^{k} \left|a \cdot \# {\rm NE}^{\gamma,s}(A_{i}^{(r,s)}) -c \right| +  6ask + 4ar + 2arn+ 2as+ 4c + 2cn. 
\end{align}

Then, by {\bf Case 1}, {\bf Case 2} and \eqref{distortion}, we have that
\begin{align} \label{GHdisII}
\mathrm{dis}({\mathfrak R})  \leq \sum_{i=1}^{k} \left|a \cdot \# {\rm NE}^{\gamma,s}(A_{i}^{(r,s)}) -c \right| +6ask+ 4a(r+s) + 4c + 2n (ar +c). 
\end{align}
\noindent Finally, \eqref{DistBoundP} follows from \eqref{GHdisIII} and \eqref{GHdisII}, completing the proof.
\end{proof}

\begin{proof}[Proof of Lemma \ref{lemmaComparingABAndcRGRG}]
Observe that
\begin{equation} \label{eq1eqnGHDistanceABcRGRG}
\begin{split} 
&\sup_{t\in [0,T]}d_{\rm GH}\Big( a \cdot {\rm AB}^{\gamma}(\lfloor t/c \rfloor r),{\rm RGRG}^{\pi}(\lfloor t/c  \rfloor c)\Big)  \\
& \quad \quad = \sup_{k\in [0,\floor{T/c}]} d_{\rm GH}\Big( a \cdot {\rm AB}^{\gamma}(kr),{\rm RGRG}^{\pi}(k c)\Big).
\end{split}
\end{equation}
\noindent Since $(\sigma^{\gamma,(r,s, s^{\prime})}/r) \wedge (\sigma^{\pi,c}/c) > T/c$, it follows that $\sigma^{\gamma,(r,s, s^{\prime})} > \floor{T/c}r$ and $\sigma^{\pi,c} > T$. Thus, $\gamma$ is $(r,s, s^{\prime})$-decomposable on the interval $[0, \floor{T/c}r]$ (recall Definition \ref{Def:003AB}), and $\pi$ is a $c$-decomposable point clouds up to time $T$ (recall Definition \ref{Def:004}).
It follows from \eqref{eq1eqnGHDistanceABcRGRG}, the triangle inequality and Proposition \ref{DerterLemma1}, that
\begin{equation}
\begin{split} \label{eqnGHDistanceBetweencRGRGAndAB1}
&\sup_{t\in [0,T]}d_{\rm GH}\Big( a \cdot {\rm AB}^{\gamma}(\lfloor t/c \rfloor r),{\rm RGRG}^{\pi}(\lfloor t/c  \rfloor c)\Big)  \\
& \quad \quad \leq  \sup_{k\in [0,\floor{T/c}]}d_{\rm GH}\Big( a \cdot {\rm AB}^{\gamma}( k r), a \cdot \mathrm{Skel}^{\gamma,s}(k r)\Big)\\
& \quad \quad  \quad \quad  +\sup_{k\in [0,\floor{T/c}]}d_{\rm GH}\Big(a \cdot \mathrm{Skel}^{\gamma,s}(kr), {\rm RGRG}^{\pi}(kc)\Big)\\
& \quad \quad  \leq \sup_{k\in [0,\floor{T/c}]}d_{\rm GH}\Big(a \cdot \mathrm{Skel}^{\gamma,s}(kr), {\rm RGRG}^{\pi}(kc)\Big) + ar.
\end{split}
\end{equation}

Next, we use Lemma \ref{ComparisionGHSkeletoRGRG} to bound,
\begin{equation}
d_{\rm GH}\Big(a \cdot \mathrm{Skel}^{\gamma,s}(kr), {\rm RGRG}^{\pi}(kc)\Big),
\end{equation}
\noindent for $k \in [0, \floor{T/c}]$.

Since $\gamma$ is $(r,s, s^{\prime})$-decomposable, it  satisfies \ref{Pro1}, and thus,
\begin{align} \label{eq2eqnGHDistanceABcRGRG}
Z^{\gamma,(r,s)}_{(i,j)} Z^{\gamma,(r,s)}_{(i^{\prime},j)} = 0 \quad \text{and} \quad Z^{\gamma,(r,s)}_{(i,j)} Z^{\gamma,(r,s)}_{(i,j^{\prime})} = 0,
\end{align}
\noindent for all $i,j,i^{\prime},j^{\prime} \in [1, k]$ such that $i \neq i^{\prime}$ and $j \neq j^{\prime}$. Then, by \eqref{eq2eqnGHDistanceABcRGRG}, we only need to consider two cases: 

{\bf Case 1.} $\mathcal{N}^{\gamma,(r, s)}(\floor{T/c})=0$. In this case, we are in the setting of Lemma \ref{ComparisionGHSkeletoRGRG} \ref{ClaimDist1}, and thus, \eqref{Deteq2iii} holds.

{\bf Case 2.} $\mathcal{N}^{\gamma,(r, s)}(\floor{T/c})=2n$ for some $n \in [1, \floor{k/2}]$. In this case, we are in the setting of Lemma \ref{ComparisionGHSkeletoRGRG} \ref{ClaimDist2}, and thus, \eqref{DistBoundP} holds. \\

Then, by {\bf Case 1} and {\bf Case 2}, we have
\begin{align} \label{eq3eqnGHDistanceABcRGRG}
d_{\rm GH}\Big(a \cdot \mathrm{Skel}^{\gamma,s}(kr), {\rm RGRG}^{\pi}(kc)\Big) & \leq \sum_{i=1}^{k} \left|a \cdot \# {\rm NE}^{\gamma,s}(A_{i}^{(r,s)}) -c \right|+ 3ask  + 2a(r+s)  \nonumber \\
& \quad \quad \quad \quad + 2c+ \frac{1}{2}\mathcal{N}^{\gamma,(r, s)}(\floor{T/c})(ar +c).
\end{align}

Finally, \eqref{eqnGHDistanceABcRGRG} follows from \eqref{eqnGHDistanceBetweencRGRGAndAB1} and \eqref{eq3eqnGHDistanceABcRGRG}. This concludes our proof. 
\end{proof}

\section{Coupling the cut-times and cut-points}
\label{Sub:regraftindex}

Recall the definition of 
$Z^{\pi,c}_{(i,j)}$ given in \eqref{e:060ii}, and $Z^{\gamma,(r,s)}_{(i,j)}$ given in \eqref{e:061ii}. 
We will use the definition of the times $\sigma^{\pi,c}$ and $\sigma^{\gamma,(r,s, s^{\prime})}$ given in \eqref{e:131} and \eqref{e:131AB}, respectively.

Recall the coupling assumption \eqref{eqnComparingSkeletonAndcRGRG} from Lemma \ref{lemmaComparingABAndcRGRG}, which couples the cut-points and cut-times of both the Aldous-Broder chain and the c-RGRG. In this section, we prove that this condition holds with high probability when the underlying path is a lazy random walk and the nice point cloud is generated by a Poisson point process.

\begin{proposition}[Coupling of the cut-times and cut-points]\label{propoCouplingTheRegrafAndCutIndices}
Let $(G_N;N\in \mathbb{N}_{0})$ be a sequence of finite simple, connected, regular graphs such that are sequence $(s_{N})_{N \in \mathbb{N}}$ and $(r_{N})_{N \in \mathbb{N}}$ in $\mathbb{N}$ that satisfy \eqref{THESeqNewII}. Suppose also that $(G_N;N\in \mathbb{N}_{0})$ satisfies the Assumption \hyperref[assumptionTransientRandomWalk]{Transient random walk}. For every $N \in \mathbb{N}_{0}$, let $\rho_N \in G_N$  and let $W_N = (W_N(n))_{n \in \mathbb{N}_{0}}$ be a lazy random walk on $G_N$. 
Let  $\Pi$ be a Poisson point process on $\Delta_+^2$ with the Lebesgue measure $\lambda$ as the intensity measure. 
Let $(c_N;N\in\na)$ be the sequence defined in \eqref{eqnDefinitionOfcN}. Fix any $T\in \re_+$ and $N$ large enough such that $T/c_N\geq 2$. 
Then, the random variables
\begin{equation}\begin{aligned} 
(Z^{W_N,(r_N,s_N)}_{(i,j)})_{1 \leq i <j \leq \floor{T/c_N}} \quad \text{and} \quad (Z^{\Pi,c_N}_{(i,j)})_{1 \leq i <j \leq \floor{T/c_N}}
\end{aligned}\end{equation} 
\noindent can be coupled in such a way that
\begin{equation}\begin{split} \label{e:054}
& \mathbb{P}_{(\rho_N,0)}\big(Z^{W_N,(r_N,s_N)}_{(i,j)}=Z^{\Pi,c_N}_{(i,j)},\,\forall\,\,1\le i<j\le  \floor{T/c_N} \big) = 1-o(1),
\quad \text{as} \quad N \rightarrow\infty,
\end{split}\end{equation}where $\mathbb{P}_{(\rho_N,0)}$ is the joint law of $W_N$ starting at $\rho_N$ and $X^{(c)}$ starting at $0$. 
\end{proposition}

The proof will make use of the following estimate which is similar to \cite[Lemma~4.5]{Schweinsberg2009}.
\begin{lemma}[Coupling on the hypercube]\label{e:057}
For $j \in \mathbb{N}$, let $Z=(Z_i;i\in [j])$ and  $Z^{\prime}=(Z^{\prime}_i;i\in [j])$ be two sequences of $\{0,1\}$-valued random variables. Then, there exists a coupling of $Z$ and $Z^{\prime}$ such that
\begin{equation} \label{e:058}
\mathbb{P}(Z=Z^{\prime}) \ge 1-\sum_{i =1}^{j} \sum_{k\neq i} \big(\mathbb{E}[Z_kZ_i]+\mathbb{E}[Z^{\prime}_kZ^{\prime}_i]\big)-\sum_{i =1}^{j}\big|\mathbb{E}[Z_i]-\mathbb{E}[Z^{\prime}_i]\big|.
\end{equation}
\end{lemma}
We will also use the following result, that gives us a formula to compute the \emph{total variation distance} between two measures $\mu$ and $\nu$ on a common space $\Omega$, that is
\begin{equation}
d_{{\rm TV}}(\mu,\nu)=\max_{A\subseteq X}|\mu(A)-\nu (A)|. 
\end{equation}Recall that a \emph{coupling} between $\mu$ and $\nu$ is a pair of random variables $(X,Y)$ on a common probability space such that the marginal distribution of $X$ is $\mu$ and of $Y$ is $\nu$.

\begin{propo}[Proposition 4.7 in \cite{Levin2017}]\label{propoFormulaTVAndCoupling}
Let $\mu$ and $\nu$ be two probability distributions on a common space $\Omega$. 
Then
\begin{equation}
d_{{\rm TV}}(\mu,\nu) =\inf\{\p(X\neq Y):(X,Y) \mbox{ is a coupling of $\mu$ and $\nu$} \}.
\end{equation}Furthermore, there exists a coupling, called \emph{optimal coupling}, which attains the infimum. 
\end{propo}

\begin{proof}[Proof of Proposition~\ref{propoCouplingTheRegrafAndCutIndices}]
The argument is similar to that of \cite[Proposition 4.7]{Schweinsberg2009}. We apply Lemma \ref{e:057} for fixed $j\in\na$, to couple $Z^{W_N,(r_N,s_N)}_{(i,j)}$ and $Z^{\Pi,c_N}_{(i,j)}$ for all $1 \leq i < j$, and then we proceed by induction on $j$.

Since the result holds by vacuity for $j=1$, assume it holds for every $2\leq j<k$ for some fixed $k\geq 3$. 
Recall the definition of $c_{N}$ in \eqref{eqnDefinitionOfcN} and note that, for $1 \leq i < j$,
\begin{equation}\begin{aligned} 
\mathbb{E}[Z^{\Pi,c_N}_{(i,j)}] &= \mathbb{P}(\Pi \cap B_{i,j}^{c_{N}} \neq \emptyset ) = 1-e^{-c_{N}^{2}} \\
&= \mathbb{P}_{\pi^{G_N}\otimes\pi^{G_N}}\big({\rm R}^{W^2_N}({\rm NE}^{W^2_N,s_N}(A^{(r_N,s_N)}_1)\cap \mathrm{R}^{W^1_N}(B^{(r_N,s_N)}_1))\neq \emptyset \big).
\end{aligned}\end{equation} 
\noindent Note also that
\begin{equation}\begin{aligned} 
\mathbb{E}_{\rho_N}[Z^{W_{N},(r_{N},s_{N})}_{(i,j)}] = \mathbb{P}_{\rho_N}\big({\rm R}^{W_N}({\rm NE}^{W_N,s_N}(A^{(r_N,s_N)}_j)\cap \mathrm{R}^{W_N}(B^{(r_N,s_N)}_i))\neq \emptyset \big).
\end{aligned}\end{equation} \noindent Then, by \eqref{THESeqNewII}, Lemma \ref{strongestima2} (with $k=2$ and $q= \frac{\ln (2 \# V_{N})}{\ln 4} $), and possibly increasing $N$ further as needed, we obtain that
\begin{equation}\begin{aligned} 
&\left|\mathbb{P}_{\rho_N}\big({\rm R}^{W_N}({\rm NE}^{W_N,s_N}(A^{(r_N,s_N)}_j)\cap \mathrm{R}^{W_N}(B^{(r_N,s_N)}_i))\neq \emptyset \big)\right.\\
&\qquad\qquad -\left.\mathbb{P}_{\pi^{G_N}\otimes\pi^{G_N}}\big({\rm R}^{W^2_N}({\rm NE}^{W^2_N,s_N}(A^{(r_N,s_N)}_1)\cap \mathrm{R}^{W^1_N}(B^{(r_N,s_N)}_1))\neq \emptyset \big)\right|\leq \frac{4}{2\# V_N},
\end{aligned}\end{equation} 
\noindent for $1 \leq i < j$.

The latter implies that
\begin{equation}\begin{aligned} \label{CoupIne1}
\sum_{i =1}^{j-1}\Big|\mathbb{E}_{\rho_N}\big[Z^{W_N,(r_N,s_N)}_{(i,j)}\big]-\mathbb{E}\big[Z^{\Pi,c_N}_{(i,j)}\big]\Big| \leq \frac{4j }{2\# V_{N}}.
\end{aligned}\end{equation} 

For $1 \leq i <k < j$, note that, by independence,
\begin{equation} \label{e:134}
\mathbb{E}[ Z^{\Pi,c_N}_{(i,j)}Z^{\Pi,c_N}_{(k,j)} ]= \mathbb{P}(\Pi \cap B_{i,j}^{c_{N}} \neq \emptyset, \Pi \cap B_{k,j}^{c_{N}} \neq \emptyset) = (1-e^{-c_{N}^{2}})^{2},
\end{equation}
\noindent and then,
\begin{equation}  \label{CoupIne2}
\sum_{i=1}^{j-1} \sum_{k \neq i}\mathbb{E}[ Z^{\Pi,c_N}_{(i,j)}Z^{\Pi,c_N}_{(k,j)} ] \leq  j^{2}(1-e^{-c_{N}^{2}})^{2}.
\end{equation}
\noindent On the other hand, for $1 \leq i <k < j$, \eqref{Exteq13} in Proposition \eqref{P:004} imply that
\begin{equation}\begin{aligned} \label{e:134II}
\mathbb{E}_{\rho_N}[ Z^{W_N,(r_N, s_{N})}_{(i,j)}Z^{W_N,(r_{N}, s_{N})}_{(k,j)} ] \leq  \frac{8r_{N}^{4}}{(\# V_{N})^{2}}.
\end{aligned}\end{equation} 
\noindent It follows that
\begin{equation}  \label{CoupIne3}
\sum_{i=1}^{j-1} \sum_{k \neq i}\mathbb{E}_{\rho_N }[ Z^{W_N,(r_N, s_{N})}_{(i,j)}Z^{W_N,(r_{N}, s_{N})}_{(k,j)} ] \leq  \frac{8j^{2}r_{N}^{4}}{(\# V_{N})^{2}}.
\end{equation}

For $j \geq 2$ fixed, let $\mu^{W_N}_{j}$ and $\mu^{\Pi}_{j}$ denote the probability distributions of
\begin{equation}\begin{aligned} 
(Z^{W_N,(r_N,s_N)}_{(i,j)};i\in [1,j-1]) \quad \text{and} \quad  (Z^{\Pi,c_N}_{(i,j)};i\in [1,j-1]),
\end{aligned}\end{equation} 
\noindent respectively. 
We now apply Lemma \ref{e:057} and Proposition \ref{propoFormulaTVAndCoupling}, together with \eqref{CoupIne1}, \eqref{CoupIne2} and \eqref{CoupIne3}. 
Thus, for every fixed $j \geq 2$, we can couple
\begin{equation}\begin{split} \label{CoupIne4}
d_{\rm TV}(\mu^{W_N}_{j}, \mu^{\Pi}_{j}) &\leq \mathbb{P}\big((Z^{W_N,(r_N,s_N)}_{(i,j)};i\in [1,j-1]) \neq (Z^{\Pi,c_N}_{(i,j)};i\in [1,j-1]) \big) \\
& \leq \frac{4 j}{2\# V_N} +\frac{8j^{2}r_{N}^{4}}{(\# V_{N})^{2}} + j^{2}(1-e^{-c_{N}^{2}})^{2},
\end{split}\end{equation} 
\noindent On the other hand, for $k \geq 2$, let $\mu^{W_N}_{\leq k}$ and $\mu^{\Pi}_{\leq k}$ denote the probability distributions of
\begin{equation}\begin{aligned} 
(Z^{W_N,(r_N,s_N)}_{(i,j)};1 \leq i <j \leq k) \quad \text{and} \quad (Z^{\Pi,c_N}_{(i,j)};1 \leq i <j \leq k),
\end{aligned}\end{equation} 
\noindent respectively. Then, Proposition \ref{propoFormulaTVAndCoupling} and the triangle inequality implies for every $k\geq 2$ that
\begin{equation}\begin{aligned} \label{CoupIne5}
d_{\rm TV}(\mu^{W_N}_{\leq k}, \mu^{\Pi}_{\leq k}) &\leq  d_{\rm TV}(\mu^{W_N}_{\leq k-1}, \mu^{\Pi}_{\leq k-1}) + d_{\rm TV}(\mu^{W_N}_{k}, \mu^{\Pi}_{k}).
\end{aligned}\end{equation} 
\noindent Therefore, by Proposition \ref{propoFormulaTVAndCoupling}, we can couple for every $k \geq 2$
\begin{equation}\begin{aligned} \label{CoupIne6}
d_{\rm TV}(\mu^{N}_{\leq k}, \mu^{\Pi}_{\leq k})\leq \mathbb{P}\big((Z^{W_N,(r_N,s_N)}_{(i,j)};1 \leq i <j \leq k)\neq (Z^{\Pi,c_N}_{(i,j)};1 \leq i <j \leq k) \big) .
\end{aligned}\end{equation} 
\noindent Finally, \eqref{CoupIne4}, \eqref{CoupIne5}, \eqref{CoupIne6} and an induction argument show that for every $k\geq 2$
\begin{equation}\begin{aligned} 
& \mathbb{P}\big((Z^{W_N,(r_N,s_N)}_{(i,j)};1 \leq i <j \leq k)\neq (Z^{\Pi,c_N}_{(i,j)};1 \leq i <j \leq k) \big) \\
& \quad \quad \quad \quad \leq  \frac{2 k^2}{\# V_N} +\frac{8k^{3}r_{N}^{4}}{(\# V_{N})^{2}} + k^{3}(1-e^{-c_{N}^{2}})^{2}.
\end{aligned}\end{equation} \noindent Recall the definition of $c_N$ in \eqref{eqnDefinitionOfcN} and the inequality $1-e^{-x} \leq x$, for $x \geq 0$. Then, by \eqref{eqnAsymptoticApproximationsForGammaNandCN} in Corollary \ref{corollaryUniformBoundsOfNonErasedVerticesSegmentAi}, there exists a constant $C_{T}>0$ (depending on $T$) such that for $N$ large enough,
\begin{equation}\begin{split} \label{e:054III}
&\mathbb{P}\big((Z^{W_N,(r_N,s_N)}_{(i,j)};1 \leq i <j \leq \floor{T/c_N}) \neq (Z^{\Pi,c_N}_{(i,j)};1 \leq i <j \leq \floor{T/c_N}) \big)\\
& \qquad \leq \frac{2 T^2}{c_N^2\# V_N} +\frac{8T^{3}r_{N}^{4}}{c_N^3(\# V_{N})^{2}} + \frac{T^{3}c_{N}^{4}}{c_N^3} \\
& \qquad \leq C_{T} \left( \frac{1}{r_N^2} + \frac{r_N }{(\# V_N)^{1/2}}  \right).
\end{split}\end{equation}
Finally, we obtain \eqref{e:054} from \eqref{e:054III} and \eqref{THESeqNewII}.
\end{proof}

\section{Proof of Main Result}\label{S:proof}

Recall from (\ref{f:001})
the Aldous--Broder chain $Y^G$ on a finite, simple, connected graph $G=(V,E)$, and from Definition~\ref{Def:RGRG} the RGRG $X$ as well as from  
\eqref{e:053II} the $c$-RGRG  $X^{\Pi,(c)}$ for some $c>0$. 
In this section we prove Theorem~\ref{T:002}. 

The proof of such a theorem relies on bounding the Gromov-Hausdorff distance between the rescaled Aldous--Broder chain and the $c_N$-RGRG and showing that they are close with high probability.

\begin{proposition}[Bounding the GH distance between AB chain and $c$-RGRG]
Assume that $(G_{N})_{N \in \mathbb{N}}$ is a sequence of finite simple, connected, regular graphs in the transient regime, i.e.,  $\limsup_{N\to\infty}H^{G_N}<\infty$, and such that we can choose $(s_N)$ and $(r_N)$ such that (\ref{f:021}) and (\ref{f:s_N}) holds. 
Let $\{Y^{G_N};\,N\in\mathbb{N}\}$ be a family of Aldous-Broder chains with values in rooted metric trees embedded in $G_N$ which start in the rooted trivial tree, i.e., $Y^{G_N}(0)=(\{\varrho_N\}, \varrho_N)$. 
Then, for any fixed $0 \leq T < \infty$, any $\alpha \in (0,1)$, there exists a constant $C_{T}$ (depending on $T$) such that
\begin{equation}\begin{aligned} \label{e:114}
\mathbb{P} \Big( \sup_{t \in [0,T]} d_{\rm GH} \Big(\frac{c_{N}}{\gamma_{N}} \mathrm{AB}^{W_{N}}(\floor{t/c_N} r_{N}), X^{\Pi,(c_N)}(\floor{t/c_N}c_N) \Big) \leq C_{T} h_{N} \Big) =  1-  o(1), 
\end{aligned}\end{equation} 
\noindent as $N \rightarrow \infty$, where 
\begin{equation}\begin{aligned} \label{e:114a}
h_{N} = \left(\frac{s_{N}}{r_{N}}\right)^{1/6} + \frac{s_{N}}{r_{N}}+ \frac{r_{N}}{(\# V_{N})^{1/2}} + \paren{\frac{r_{N}}{(\# V_{N})^{\frac{1}{2} }}}^{1-\alpha}.
\end{aligned}\end{equation}
\label{PropShapes}
\end{proposition}

\begin{proof}
First note that by Corollary \ref{corollaryUniformBoundsOfNonErasedVerticesSegmentAi} and \eqref{THESeqNewII}, we can select a sufficiently large value for $N$ such that $T/c_{N} \geq 2$. 

By Proposition~\ref{propoCouplingTheRegrafAndCutIndices}, the random variables
\begin{equation}\begin{aligned} 
(Z^{W_N,(r_N,s_N)}_{(i,j)})_{1 \leq i <j \leq \floor{T/c_N}} \quad \text{and} \quad (Z^{\Pi,c_N}_{(i,j)})_{1 \leq i <j \leq \floor{T/c_N}}
\end{aligned}\end{equation} 
\noindent can be coupled so that
\begin{equation}\begin{split} \label{Shapeq1}
& \mathbb{P}_{(\rho_N,0)}\big(Z^{W_N,(r_N,s_N)}_{(i,j)}=Z^{\Pi,c_N}_{(i,j)},\,\forall\,\,1\le i<j\le  \floor{T/c_N} \big) = 1-o(1), \quad \text{as} \quad 
N \rightarrow\infty.
\end{split}
\end{equation}

Define the events, 
\begin{equation}\begin{aligned} 
A_{N,T}^{(1)} & \coloneqq \Big \{ Z^{W_N,(r_N,s_N)}_{(i,j)}=Z^{\Pi,c_N}_{(i,j)},\forall\,1\le i<j\le  \lfloor T/c_{N} \rfloor  \Big \}, \\
A_{N,T}^{(2)} & \coloneqq \Big \{ \sigma^{\Pi, c_{N}} > T, \sigma^{W_{N}, (r_{N}, s_{N}, s_{N}^{\prime})} > Tr_{N} /c_{N}  \Big \},  \\
A_{N,T}^{(3)} & \coloneqq \Big \{ \sup_{0 \leq i \leq \lfloor T/c_{N} \rfloor} \Big|\# {\rm R}^{W_{N}} ({\rm NE}^{W_{N},s_{N}}(A_{i}^{(r_{N},s_{N})}))- \gamma_{N} \Big| \leq r_{N}  \Big( \frac{s_{N}}{r_{N}} \Big)^{1/6}   \Big \},  \\
A_{N,T}^{(4)} & \coloneqq  \Big \{ \mathcal{N}^{W_{N},(r_{N}, s_{N})}(\lfloor T/c_{N} \rfloor) \leq c_{N}^{-\alpha} \Big \},
\end{aligned}\end{equation} 
\noindent for some fixed $\alpha \in (0,1)$.

It follows from \eqref{Shapeq1}, Lemma \ref{Lemma:002}, Corollary \ref{corollaryUniformBoundsOfNonErasedVerticesSegmentAi}, Corollary \ref{corollaryAsymptoticDecomposability}, Corollary \ref{ProbaEst1}, Corollary \ref{coroBoundOnTheNumberOfIntersections} and \eqref{THESeqNewII} that
\begin{align} 
\mathbb{P}(A_{N,T}^{(1)} \cap A_{N,T}^{(2)} \cap A_{N,T}^{(3)} \cap A_{N,T}^{(4)}) = 1 - o(1), \quad \text{as} \quad N \rightarrow\infty.
\end{align}

Finally, observe that under the event $A_{N,T}^{(1)} \cap A_{N,T}^{(2)} \cap A_{N,T}^{(3)} \cap A_{N,T}^{(4)}$, \eqref{eqnGHDistanceABcRGRG} in Lemma \ref{lemmaComparingABAndcRGRG} yields to 
\begin{equation}\begin{aligned} 
& \sup_{t \in [0,T]} d_{\rm GH} \Big(\frac{c_{N}}{\gamma_{N}} \mathrm{AB}^{W_{N}}(\floor{t/c_N} r_{N}), X^{\Pi,(c_N)}(\floor{t/c_N}c_N) \Big) \\
& \quad \quad \leq \frac{c_{N}}{\gamma_{N}} \sum_{i=1}^{\lfloor T/c_{N} \rfloor} \left| {\rm R}^{W_{N}} ({\rm NE}^{W_{N},s_{N}}(A_{i}^{(r_{N},s_{N})})) - \gamma_{N} \right| + 3\frac{c_{N}}{\gamma_{N}}\frac{3s_{N}T}{c_{N}} + 2 \frac{c_{N}}{\gamma_{N}}(r_{N}+s_{N})   \\
& \quad \quad \quad \quad + 2c_{N}+ \left(\frac{c_{N}}{\gamma_{N}}r_{N} +c_{N} \right) \mathcal{N}^{W_{N},(r_{N}, s_{N})}(\lfloor T/c_{N} \rfloor)  +\frac{c_{N}}{\gamma_{N}} r_{N},
\end{aligned}\end{equation}
\noindent with probability $1-o(1)$. 
Therefore, our claim follows from Corollary \ref{corollaryUniformBoundsOfNonErasedVerticesSegmentAi} and the bounds in the event $A_{N,T}^{(1)} \cap A_{N,T}^{(2)} \cap A_{N,T}^{(3)} \cap A_{N,T}^{(4)}$.
\end{proof}

\begin{proof}[Proof of Theorem \ref{T:002}]
We will use the fact that convergence in the Skorohod space ${\mathcal D}(\mathbb{R}_{+},\mathbb{T})$, is equivalent to convergence in ${\mathcal D}([0,T],\mathbb{T})$, for any $T\in \re_+$ which is a continuity point of $X$ (see Theorem 16.2 in \cite{MR1700749}).
Consider any fixed $T\in \re_+$, which by standard results is a.s. a continuity point of $X$. 
Note that the function $g_{N}(t) = c_{N}\lfloor t/c_{N} \rfloor$, for $t \geq 0$, converges uniformly over compact intervals to the identity function. Then it follows from \eqref{e:033II} in Corollary \ref{P:001II} and for example, \cite[Theorem 3.1]{Whitt1980} that
\begin{equation}  \label{Convg1}
(X^{\Pi,(c_N)}( c_{N}\lfloor t/c_{N} \rfloor))_{t\ge 0} \xRightarrow[N\to \infty ]{} (X(t))_{t\ge 0},
\end{equation}
\noindent weakly as random variables with values in the Skorohod space ${\mathcal D}(\mathbb{R}_{+},\mathbb{T})$. On the other hand, it follows from our assumption \eqref{THESeqNewII} that $h_{N}$ defined in \eqref{e:114a}, satisfies $\lim_{N \rightarrow \infty} h_{N} = 0$. Hence, \eqref{Convg1}, Proposition \ref{PropShapes}, the triangle inequality and the union bound imply that 
\begin{equation}
\Big(\frac{c_{N}}{\gamma_{N}} \mathrm{AB}^{W_{N}}(\floor{t/c_N} r_{N})\Big)_{t\in [0,T]} \TNo \big(X(t)\big)_{t\in [0,T]},
\end{equation}
\noindent weakly as random variables with values in the Skorohod space ${\mathcal D}([0,T],\mathbb{T})$. 
Here we used Theorem 3.1 in Billingsley \cite{MR1700749} saying that if the distance between two sequences goes to zero in probability and one converges in law, then the other also converges to the same limit.

\end{proof}


\begin{remark}[Application to $\mathbb{Z}^d_N$, $d\ge 5$] Assume that $G_N:=\mathbb{Z}^d_N$, $d\ge 5$. Then it is known that for some $C\in (0,\infty)$, $\limsup_{N\to\infty}H^{\mathbb{Z}^d_N}<\infty$ if $d\ge 5$. To see this, use  one can show that for all $N\in \na$ and $\rho\in V_N$,
\begin{equation}
\label{f:030}
   \big|\p_\rho(W_N=\rho) -\pi^{\mathbb{Z}^d_N}(\rho)\big|\leq C\frac{1}{N^{d/2}}
\end{equation}
(compare, for example, \cite[Equation~(2.10)]{MR1048930}). Notice further that $\int_1^{N^{d/2}}x^{-d/2}\,\mathrm{d}x\leq C\,N^{d-d^2/4}<\infty$ whenever $d\geq 5$.

Moreover, we can conclude from \cite[Theorem 5.6]{Levin2017}
that $\tau_{\mathrm{mix}}^{\mathbb{Z}^d_N}={\mathcal O}(N^{2})\ll N^{d/2}=\# V_N^{1/2}$. So we can clearly find sequences $(s_N)$ and $(r_N)$ that satisfy  (\ref{f:021}) and (\ref{f:s_N}). 

Recall $c_N:=c^{\mathbb{Z}^d_N}(r_N)$ and $\gamma_N:=\Gamma^{\mathbb{Z}^d_N}(r_N)$ from (\ref{e:ccNN}) respectively (\ref{e:GammaNN}).
Recall further $\alpha(d)$ and $\gamma(d)$ from  (\ref{e:062})  respectively (\ref{e:063}). 
It is shown in  \cite[Lemma~8.1]{PeresRevelle} that $\frac{\gamma'_N}{r_N}\tNo\gamma(d)$ and $\frac{N^d c_N}{(r_N)^2}\tNo\alpha(d)$ and (\ref{e:063}), respectively. This yields the claim of Theorem~\ref{T:001}.
\label{Rem:torus}
\hfill$\qed$
\end{remark}

\subsection*{Aknowledgements}
OAH and AW were supported
by the Deutsche Forschungsgemeinschaft (through Project-ID444091549 within SPP-2265). 

\providecommand{\bysame}{\leavevmode\hbox to3em{\hrulefill}\thinspace}
\providecommand{\MR}{\relax\ifhmode\unskip\space\fi MR }
\providecommand{\MRhref}[2]{%
  \href{http://www.ams.org/mathscinet-getitem?mr=#1}{#2}
}
\providecommand{\href}[2]{#2}


\begin{thebibliography}{BLPS01}

\bibitem[Ald90]{Aldous1990}
David Aldous, \emph{The random walk construction of uniform spanning trees and
  uniform labeled trees}, SIAM Journal of discrete Mathematics (1990),
  450--465.

\bibitem[Ald91]{Aldous1991a}
\bysame, \emph{The continuum random tree {I}}, Ann. Probab. \textbf{19} (1991),
  1--28.

\bibitem[Ald93]{Aldous1993}
\bysame, \emph{The continuum random tree {III}}, Ann. Probab. \textbf{21}
  (1993), 248--289.

\bibitem[ALW16]{MR3522292}
Siva Athreya, Wolfgang L\"{o}hr, and Anita Winter, \emph{The gap between
  {G}romov-vague and {G}romov-{H}ausdorff-vague topology}, Stochastic Process.
  Appl. \textbf{126} (2016), no.~9, 2527--2553. \MR{3522292}

\bibitem[ALW17]{AthreyaLohrWinter2017}
Siva Athreya, Wolfgang L{\"o}hr, and Anita Winter, \emph{Invariance principle
  for variable speed random walks on trees}, Annals Probab. \textbf{45} (2017),
  no.~2, 625--667, arXiv:1404.6290. \MR{3630284}

\bibitem[ANS24]{MR4712854}
Eleanor Archer, Asaf Nachmias, and Matan Shalev, \emph{The {GHP} scaling limit
  of uniform spanning trees in high dimensions}, Comm. Math. Phys. \textbf{405}
  (2024), no.~3, Paper No. 73, 41. \MR{4712854}

\bibitem[AS24]{ArcherShalev2024}
Eleanor Archer and Matan Shalev, \emph{The {GHP} scaling limit of uniform
  spanning trees of dense graphs}, Random Structures and Algorithms \textbf{25}
  (2024), no.~1, 149--190.

\bibitem[AT89]{AnanthrahamTsoucas1989}
Venkat Ananthraham and Pantelis Tsoucas, \emph{A proof of the {M}arkov chain
  tree theorem}, Stat. Probab. Letters \textbf{8} (1989), 189–192.

\bibitem[BBI01a]{BuragoBuragoIvanov2001}
Dmitri Burago, Yuri Burago, and Sergei Ivanov, \emph{A course in metric
  geometry}, vol.~33, American Mathematical Soc., 2001.

\bibitem[BBI01b]{MR1835418}
\bysame, \emph{A course in metric geometry}, Graduate Studies in Mathematics,
  vol.~33, American Mathematical Society, Providence, RI, 2001. \MR{1835418}

\bibitem[Bil99]{MR1700749}
Patrick Billingsley, \emph{Convergence of probability measures}, second ed.,
  Wiley Series in Probability and Statistics: Probability and Statistics, John
  Wiley \& Sons, Inc., New York, 1999, A Wiley-Interscience Publication.
  \MR{1700749}

\bibitem[BLPS01]{BenjaminiLyonsSchramm2001}
Itai Benjamini, Russell Lyons, Yuval Peres, and Oded Schramm, \emph{Uniform
  spanning forests}, Ann. Probab. \textbf{29} (2001), 1–65.

\bibitem[BP93]{BurtonPemantle1993}
R.~Burton and Robin Pemantle, \emph{Local characteristics, entropy and limit
  theorems for spanning trees and domino tilings via transfer-impedances}, Ann.
  Probab. \textbf{21} (1993), 1329–1371. \MR{1235419}

\bibitem[Bro89]{Broder1989}
Andrei~Z. Broder, \emph{Generating random spanning trees}, 30th Annual
  Symposium on Foundations of Computer Science (1989), 442--447.

\bibitem[CDG04]{MR2074427}
J.~T. Cox, D.~A. Dawson, and A.~Greven, \emph{Mutually catalytic super
  branching random walks: large finite systems and renormalization analysis},
  Mem. Amer. Math. Soc. \textbf{171} (2004), no.~809, viii+97. \MR{2074427}

\bibitem[CG94]{CoxGreven1994}
J.~T. Cox and A.~Greven, \emph{The finite systems scheme: An abstract theorem
  and a new example}, Proceedings of CRM Conference on Measure valued
  processes. Stochastic partial differential equation and interacting Systems,
  CRM and Lecture Notes series of the American Mathematical Society, vol.~5,
  1994, pp.~55--68.

\bibitem[CGS95]{CoxGrevenShiga1995}
Ted Cox, Andreas Greven, and T.~Shiga, \emph{Finite and infinite systems of
  interacting diffusions}, Probab. Theory Rel. Fields \textbf{103} (1995),
  165--197.

\bibitem[Cox89]{MR1048930}
J.~T. Cox, \emph{Coalescing random walks and voter model consensus times on the
  torus in {${\bf Z}^d$}}, Ann. Probab. \textbf{17} (1989), no.~4, 1333--1366.
  \MR{1048930}

\bibitem[Dha90]{Dhar1990}
D.~Dhar, \emph{Self-organized critical state of sandpile automaton models},
  Phys. Rev. Lett. \textbf{64} (1990), 1613–1616. \MR{1044086}

\bibitem[EPW06]{EvansPitmanWinter2006}
Steven~N. Evans, Jim Pitman, and Anita Winter, \emph{Rayleigh processes, real
  trees, and root growth with re-grafting}, Probab. Theory Related Fields
  \textbf{134} (2006), no.~1, 81--126. \MR{2221786}

\bibitem[GLW05]{GrevenLimicWinter2005}
Andreas Greven, Vlada Limic, and Anita Winter, \emph{Representation theorems
  for interacting {M}oran models, interacting {F}isher-{W}right diffusions and
  applications}, Electronic {J}ournal {P}robability \textbf{10(39)} (2005),
  1286--1358.

\bibitem[Gro99]{MR1699320}
Misha Gromov, \emph{Metric structures for {R}iemannian and non-{R}iemannian
  spaces}, Progress in Mathematics, vol. 152, Birkh\"auser Boston, Inc.,
  Boston, MA, 1999, Based on the 1981 French original [MR0682063 (85e:53051)],
  With appendices by M.\ Katz, P.\ Pansu and S.\ Semmes, Translated from the
  French by Sean Michael Bates. \MR{1699320}

\bibitem[Gut13]{Gut2013}
Allan Gut, \emph{Probability: a graduate course}, second ed., Springer Texts in
  Statistics, Springer, New York, 2013. \MR{2977961}

\bibitem[HLT21]{HuaLyonsTang2021}
Yiping Hua, Russell Lyons, and Pengfei Tang, \emph{A reverse
  {A}ldous–{B}roder algorithm}, Ann. Inst. H. Poincaré Probab. Statist.
  \textbf{57} (2021), no.~2, 890--900.

\bibitem[Hut18]{Hutchcroft2018}
Tom Hutchcroft, \emph{Interlacements and the wired spanning forest}, Annals of
  Probability \textbf{46} (2018), no.~2, 1170--2000.

\bibitem[Jar18]{Jarai2018}
Antal~A. Jarai, \emph{Sandpile models}, Probab. Surveys \textbf{15} (2018),
  243--306.

\bibitem[Ken00]{Kenyon2000}
Rick Kenyon, \emph{The asymptotic determinant of the discrete laplacian}, Acta
  Math. \textbf{185} (2000), 239–286. \MR{1819995}

\bibitem[LP17]{Levin2017}
David~A. Levin and Yuval Peres, \emph{Markov chains and mixing times}, American
  Mathematical Society, Providence, RI, 2017, Second edition of [ MR2466937],
  With contributions by Elizabeth L. Wilmer, With a chapter on ``Coupling from
  the past'' by James G. Propp and David B. Wilson. \MR{3726904}

\bibitem[LSW04]{LawlerSchrammWerner2004}
Gregory~F. Lawler, Oded Schramm, and Wendelin Werner, \emph{Conformal
  invariance of planar loop-erased random walks and uniform spanning trees},
  Ann. Probab. \textbf{32} (2004), no.~1B, 939–995.

\bibitem[PR04]{PeresRevelle}
Yuval {Peres} and David {Revelle}, \emph{{Scaling limits of the uniform
  spanning tree and loop-erased random walk on finite graphs}}, arXiv
  Mathematics e-prints (2004), math/0410430.

\bibitem[Sch00]{Schramm2000}
Oded Schramm, \emph{Scaling limits of loop-erased random walks and uniform
  spanning trees}, Israel J. Math. \textbf{118} (2000), 221–288. \MR{1776084}

\bibitem[Sch08]{Schweinsberg2008}
Jason Schweinsberg, \emph{Loop-erased random walk on finite graphs and the
  {R}ayleigh process}, J. Theoret. Probab. \textbf{21} (2008), no.~2, 378--396.
  \MR{2391250}

\bibitem[Sch09]{Schweinsberg2009}
\bysame, \emph{The loop-erased random walk and the uniform spanning tree on the
  four-dimensional discrete torus}, Probab. Theory Related Fields \textbf{144}
  (2009), no.~3-4, 319--370. \MR{2496437}

\bibitem[Sym84]{Symer1984}
D.E. Symer, \emph{Expanded ergodic {M}arkov chains and cycling systems}, 1984,
  Senior thesis, Dartmouth College.

\bibitem[Szn10]{Sznitman2010}
Alain-Sol Sznitman, \emph{Vacant set of random interlacements and percolation},
  Annals of Mathematics \textbf{171} (2010), no.~3, 2039--2087.

\bibitem[Tei09]{Teixeira2009}
A.~Teixeira, \emph{Interlacement percolation on transient weighted graphs},
  Electron. J. Probab. \textbf{14} (2009), 1604–1628. \MR{2525105}

\bibitem[Whi80]{Whitt1980}
Ward Whitt, \emph{Some useful functions for functional limit theorems}, Math.
  Oper. Res. \textbf{5} (1980), no.~1, 67--85. \MR{561155}

\bibitem[Wil96]{Wilson1996}
David~B. Wilson, \emph{Generating spanning trees more quickly than the cover
  time}, Proceedings of the 28 {I}nternational {C}ongress of {M}athematicians
  Symposium on the Theory of Computing, 1996.

\end{thebibliography}
\end{document}